\documentclass[12pt]{amsart}

\usepackage[left=20mm,right=20mm]{geometry}
\usepackage{graphicx} % Required for inserting images
\usepackage{amsmath}
\usepackage{amssymb}
\usepackage{amscd}
\usepackage{amsthm}
\usepackage{enumerate}
\usepackage{setspace}
\usepackage{xcolor}
\usepackage[all]{xy}
\usepackage{url}
\definecolor{linkgold}{HTML}{966e09}
\usepackage[colorlinks=true,linkcolor=linkgold,citecolor=linkgold,urlcolor=linkgold,breaklinks,pagebackref]{hyperref}
\usepackage[lite,abbrev,alphabetic]{amsrefs}
\usepackage{ragged2e}
\AddToHook{env/biblist/begin}{\RaggedRight}
\usepackage{aliascnt}
\usepackage{cleveref}
\usepackage{comment}
\usepackage{todonotes}
\usepackage{tikz-cd}
\usepackage{quiver}

\makeatletter
\renewcommand{\subsection}{\@startsection{subsection}{2}%
  \z@{.7\linespacing\@plus.5\linespacing}{.35\linespacing}%
  {\normalfont\bfseries\raggedright}}
\makeatother

\theoremstyle{definition}
\newtheorem{definition}{Definition}[section]
\newaliascnt{example}{definition}
\newtheorem{example}[example]{Example}
\aliascntresetthe{example}

\theoremstyle{plain}
\newaliascnt{theorem}{definition}
\newtheorem{theorem}[theorem]{Theorem}
\aliascntresetthe{theorem}
\newaliascnt{proposition}{definition}
\newtheorem{proposition}[proposition]{Proposition}
\aliascntresetthe{proposition}
\newaliascnt{lemma}{definition}
\newtheorem{lemma}[lemma]{Lemma}
\aliascntresetthe{lemma}
\newaliascnt{corollary}{definition}
\newtheorem{corollary}[corollary]{Corollary}
\aliascntresetthe{corollary}
\newaliascnt{conjecture}{definition}
\newtheorem{conjecture}[conjecture]{Conjecture}
\aliascntresetthe{conjecture}
\newaliascnt{hypothesis}{definition}

\aliascntresetthe{hypothesis}
\newtheorem*{theorem*}{Theorem}
\newaliascnt{construction}{definition}
\newtheorem{construction}[construction]{Construction}
\aliascntresetthe{construction}
\newaliascnt{fact}{definition}

\aliascntresetthe{fact}

\theoremstyle{remark}
\newaliascnt{remark}{definition}
\newtheorem{remark}[remark]{Remark}
\aliascntresetthe{remark}

\crefname{theorem}{Theorem}{Theorems}
\crefname{proposition}{Proposition}{Propositions}
\crefname{lemma}{Lemma}{Lemmas}
\crefname{corollary}{Corollary}{Corollaries}
\crefname{conjecture}{Conjecture}{Conjectures}
\crefname{hypothesis}{Hypothesis}{Hypotheses}
\crefname{remark}{Remark}{Remarks}
\crefname{condition}{Condition}{Conditions}
\crefname{example}{Example}{Examples}
\crefname{definition}{Definition}{Definitions}
\crefname{section}{Section}{Sections}
\crefname{equation}{Equation}{Equations}
\crefname{construction}{Construction}{Constructions}

\def\N{\mathbb{N}}
\def\C{\mathbb{C}}% complex number field
\def\R{\mathbb{R}}% real number field
\def\Q{\mathbb{Q}}% rational number field
\def\Z{\mathbb{Z}}% ring of integers
\def\A{\mathbb{A}}% the ring of adeles
\def\H{\mathbb{H}}% Hamilton's quaternion field

\def\GL{\mathrm{GL}}
\def\PGL{\mathrm{PGL}}
\def\SL{\mathrm{SL}}

\def\O{\textnormal{O}}
\def\SO{\mathrm{SO}}
\def\GSO{\mathrm{GSO}}

\def\Sp{\mathrm{Sp}}
\def\GSp{\mathrm{GSp}}
\def\Spin{\mathrm{Spin}}
\def\PGSO{\mathrm{PGSO}}
\def\PGSp{\mathrm{PGSp}}
\def\PGO{\mathrm{PGO}}
\def\GO{\mathrm{GO}}
\def\HII{\mathrm{HII}}

\def\rd{\,\mathrm{d}}% measure

\def\Cc{C_c^{\infty}}% smooth functions

\DeclareMathOperator{\diag}{diag} % diagonal matrix
\DeclareMathOperator{\ord}{ord}% order 
\DeclareMathOperator{\vol}{vol}% volume

\DeclareMathOperator{\Gal}{Gal}% Galois group
\DeclareMathOperator{\ind}{ind}% compact induction
\DeclareMathOperator{\Hom}{Hom}% space of homomorphisms
\DeclareMathOperator{\End}{End}% space of endomorphisms
\DeclareMathOperator{\Aut}{Aut}% group of automorphisms
\DeclareMathOperator{\Lie}{Lie}% Lie algebra
\DeclareMathOperator{\Cent}{Cent}% centralizer
\DeclareMathOperator{\Irr}{Irr}% set of irreducible representations
\DeclareMathOperator{\Ad}{Ad}% adjoint action
\DeclareMathOperator{\ad}{ad}% adjoint action
\DeclareMathOperator{\Std}{Std}
\DeclareMathOperator{\Ker}{Ker}% kernel
\DeclareMathOperator{\triv}{triv}% the trivial representation
\DeclareMathOperator{\tr}{tr}% trace
\DeclareMathOperator{\id}{id}% identity operator
\DeclareMathOperator{\Sym}{Sym}% symmetrization

\DeclareMathOperator{\Out}{Out}
\DeclareMathOperator{\der}{der}

\makeatletter
\renewcommand\subsubsection{\@startsection{subsubsection}{3}{\z@}%
  {-18\p@ \@plus -4\p@ \@minus -4\p@}%
  {-0.6em}%
  {\normalfont\bfseries}} % ← 太字に変更
\makeatother

\newcommand{\fa}{\mathfrak{a}}

\newcommand{\fg}{\mathfrak{g}}

\newcommand{\fn}{\mathfrak{n}}
\newcommand{\fo}{\mathfrak{o}}
\newcommand{\fp}{\mathfrak{p}}

\newcommand{\fw}{\mathfrak{w}}

\newcommand{\bB}{\mathbb{B}}

\newcommand{\bG}{\mathbb{G}}

\newcommand{\bM}{\mathbb{M}}

\newcommand{\bP}{\mathbb{P}}
\newcommand{\bS}{\mathbb{S}}
\newcommand{\bT}{\mathbb{T}}

\newcommand{\bX}{\mathbb{X}}

\newcommand{\cA}{\mathcal{A}}

\newcommand{\cC}{\mathcal{C}}
\newcommand{\cD}{\mathcal{D}}

\newcommand{\cF}{\mathcal{F}}

\newcommand{\cL}{\mathcal{L}}

\newcommand{\cO}{\mathcal{O}}
\newcommand{\cP}{\mathcal{P}}

\newcommand{\cU}{\mathcal{U}}
\newcommand{\cV}{\mathcal{V}}
\newcommand{\cW}{\mathcal{W}}
\newcommand{\cX}{\mathcal{X}}

\numberwithin{equation}{section}
\newcommand{\disc}{\mathrm{disc}}

\newcommand{\reg}{\mathrm{reg}}
\newcommand{\temp}{\mathrm{temp}}

\newcommand{\unit}{\mathrm{unit}}
\newcommand{\spl}{\mathrm{spl}}
\newcommand{\red}{\mathrm{red}}

\newcommand{\esixroot}[6]{
  \begin{array}{ccccc}
     &   &  #2 &   &   \\
  #1 & #3 & #4 & #5 & #6 \\
  \end{array}
}

\newcommand{\Temp}{\mathrm{Temp}}
\newcommand{\abs}[1]{\lvert #1 \rvert}

\newcommand{\pr}{\mathrm{pr}}
\newcommand{\GSpin}{\mathrm{GSpin}}

\newcommand{\Norm}{\mathrm{Norm}}
\newcommand{\dist}{\mathrm{dist}}

\DefineSimpleKey{bib}{arxiv}
\newcommand\myurl[1]{\url{#1}}
\newcommand\arxiv[1]{available at \href{https://arxiv.org/abs/#1}{arXiv:#1}}

\BibSpec{article}{%
  +{}  {\PrintAuthors}                {author}
  +{,} { \textit}                     {title}
  +{.} { }                            {part}
  +{:} {\textit}                     {subtitle}
  +{,} {\PrintContributions}         {contribution}
  +{,} {\PrintConference}            {conference}
  +{}  {\PrintBook}                   {book}
  +{,} { }                            {booktitle}
  +{,} { }                            {journal}
  +{,} { \textbf}                     {volume}
  +{,} { \issuetext}                  {number}
  +{,} { \PrintDateB}                 {date}
  +{,} { pp.~}                       {pages}
  +{,} { available at \eprint}        {eprint}
  +{;} { \arxiv}               {arxiv}
  +{.\!\!}  {\SentenceSpace \PrintReviews} {review}
  +{,} { \PrintDOI}                   {doi}
}

\title[On the Hiraga--Ichino--Ikeda conjecture for $\mathrm{G_2}$]{On the Hiraga--Ichino--Ikeda conjecture on formal degrees for $\mathrm{G_2}$}
\author{Yugo Takanashi} 
\address{Department of Mathematics, Graduate School of Science, Kyoto University, Kitashirakawa Oiwake-cho, Sakyo-ku, Kyoto 606-8502, Japan}
\email{takanashi.yugo.7s@kyoto-u.ac.jp}
\date{}

\begin{document}

\begin{abstract}
    We prove that the twisted endoscopic character identities between triality $\mathrm{PGSO}_8$ and the endoscopic group $G_2$ imply the Hiraga--Ichino--Ikeda conjecture on formal degrees for $G_2$.
    In the course of proving the main result, we also establish some fundamental results on representation theory of $\mathrm{PGSO}_8$ with triality.
    A key feature of our approach is an application of the adjoint group of type $E_6$.
\end{abstract}

\maketitle
\setcounter{tocdepth}{1} 
\tableofcontents
\bigskip

\section{Introduction}\label{section:introduction}

Let $F$ be a local field of characteristic $0$. 
Let $G$ be a reductive group over $F$.
The local Langlands correspondence for $G$ is a conjectural finite-to-one correspondence between the set $\Pi(G)$ of isomorphism classes of smooth admissible irreducible representations of $G$ and the set $\Phi(G)$ of $L$-parameters for $G$.
In the archimedean case, it has already been established by Harish-Chandra and Langlands.
The local Langlands correspondence has been established for general linear groups by Harris--Taylor ~\cite{HarrisTaylor2001-simple-Shimura}, Henniart ~\cite{Henniart2000-simpleproof}, and Scholze ~\cite{Scholze2013-localLanglandsGLn}, and for classical groups by Arthur \cite{Art13}, Mok \cite{Mok15}, and others (see also \cite{kaletha2014endoscopicclassificationrepresentationsinner} and \cite{atobe2024localintertwiningrelationscotempered}).

From a harmonic-analytic point of view, we will also consider the Plancherel formula for $G$. It reads
\begin{align*}
     f(1) = \int_{\Irr_{\temp}(G)} \tr\pi(f^{\vee} \rd g) \rd \mu^G(\pi),
\end{align*}
for any $f \in \Cc(G)$.
Here, $dg$ is a Haar measure on $G$, and we set $f^{\vee}(g) = f(g^{-1})$ and
\[
  \pi(f^\vee \rd g)=\int_G f^\vee(g)\,\pi(g)\,\rd g .
\]
Furthermore, $\Irr_{\temp}(G) \subset \Pi(G)$ denotes the set of tempered irreducible representations of $G$ with the Fell topology, and $\rd \mu^G(\pi)$ is a Borel measure on $\Irr_{\temp}(G)$, called the Plancherel measure for $G$. Note that $\rd \mu^G(\pi)$ depends on the choice of the Haar measure $dg$.

Shahidi proved in \cite{Sha90}*{Proposition 9.3} that, if the local Langlands correspondence for $G$ satisfies certain desiderata, then the measure $\rd \mu^G(\pi)$ can be expressed via the local Langlands correspondence for $G$.
Hiraga, Ichino, and Ikeda proposed a conjectural formula for the measure $\rd \mu^G(\pi)$ in \cite{HII08} (see also \cite{HiragaIchinoIkeda2008-Correction}), assuming the local Langlands correspondence for $G$ (see \cref{conjecture:formal-degree-conjecture}).
They also proved the conjecture in some cases, including stable $L$-packets for $U_3$, by applying the endoscopic character identities proved by Rogawski \cite{Rogawski1990-U3} and a technique of Goldberg and Shahidi.

In the archimedean case, the conjecture follows from the results of Harish-Chandra \cite{Harish-Chandra1976-HarmonicanalysisIII}, Knapp--Stein \cite{Knapp-Stein1971intertwiningoperators}, and Arthur \cite{Art89IntResI}; see \cite{HII08}*{2}.
In \cite{BP21a}, Beuzart-Plessis announced that he had proved the Hiraga--Ichino--Ikeda conjecture for classical groups assuming the endoscopic character identities between classical groups and $\GL_n(F)$ with the (conjugate) self-dual involution, and the detailed proof has been given in \cite{beuzartplessis2025hiragaichinoikedaconjectureformaldegrees} recently.
He utilized his own results, such as those proved in \cite{Beuzart-Plessis2021-Plancherel-GLnE-GLnF}*{Proposition 3.41}, to substantially generalize the method of Hiraga--Ichino--Ikeda.

In this paper, we prove the Hiraga--Ichino--Ikeda conjecture \cite{HII08} (see also \cite{HiragaIchinoIkeda2008-Correction}) for the exceptional group $G_2(F)$ over any non-archimedean local field $F$ of characteristic $0$, assuming the conjectural endoscopic character identities between $G_2$ and $\PGSO_8$ with the triality automorphism $\tau \colon \PGSO_8 \xrightarrow{\sim} \PGSO_8$. 
We note that the local Langlands correspondence for $G_2(F)$ has been recently established by Gan--Savin in \cite{GS23LLC} and our work is based on this result.

We largely follow Beuzart-Plessis' method for classical cases, but several new issues arise in our setting:
\begin{itemize}
    \item 
    The local Langlands correspondence for $\PGSO_8(F)$ by Xu is somewhat weaker than the expected local Langlands correspondence. 
    \item    
    We have to develop the representation theory for $\PGSO_8(F)$ with the triality automorphism $\tau$.
    For example, as tempered $L$-packets for $\PGSO_8(F)$ need not be singletons, we have to construct an extension of each functorial lift from $G_2(F)$ to $\PGSO_8(F) \rtimes \tau$, which conjecturally satisfies the endoscopic character identities.
    \item      
    We need a certain irreducibility result for the semisimple adjoint group of type $E_6$, which seems not to have been considered previously.
\end{itemize}

Thus, we need to establish fundamental results on representation theory of $\PGSO_8(F)$ with the triality automorphism $\tau$, before proving the main analytic theorem.
Moreover, the proof of the main analytic theorem also requires similar harmonic analysis to that carried out in \cite{Beuzart-Plessis2021-Plancherel-GLnE-GLnF}*{Proposition 3.41}.
For example, see the proof of \cref{theorem:key-computation}.

Building on the present paper, the author proves the endoscopic character identities for $G_2$ and $\PGSO_8$ with the triality automorphism $\tau$ in \cite{Takanashi2026-EndoscopicDescriptionG2}.

\subsection*{Outline of the paper}

\begin{itemize}
    \item    
    In \Cref{section:2}, we will give the notation and assumptions which are used throughout this paper.
    \item    
    In \Cref{section:3}, we will give a form of the conjectural local Langlands correspondence for quasi-split $p$-adic reductive groups.
    \item    
    In \Cref{section:4}, we state the formal degree conjecture of Hiraga--Ichino--Ikeda, based on the conjectural form of the local Langlands correspondence given in the previous section.
    We will also review known results for general linear groups and classical groups.
    \item    
    In \Cref{section:5}, we will prove the Hiraga--Ichino--Ikeda conjecture for $\GSO_{2n}(F)$ and  $\PGSO_{2n}(F)$ by using the result of Beuzart-Plessis.
    \item    
    In \Cref{section:6}, we will give some results on the group $\PGSO_8(F)$ and its triality automorphism. 
    \item     
    In \Cref{section:7}, we will review some results on normalizing factors of intertwining operators.
    In particular, we study the relation between normalizing factors and isogenies.
    \item     
    In \Cref{section:8}, we study the relation between the group of type $E_6$ and the group $\PGSO_8(F)$ with the triality.
    \item       
    In \Cref{section:9}, we will study the endoscopy for $\PGSO_8(F)$ with the triality. 
    In particular, we will construct an extension of each functorial lift from $G_2(F)$ to $\PGSO_8(F)$, which conjecturally satisfies the endoscopic character identities. See \cref{theorem:canonical-extension-PGSO8-S3}.
    We will also state the conjectural endoscopic character identities, using the results of the previous section.
    \item      
    In \Cref{section:10}, we will prove the Hiraga--Ichino--Ikeda conjecture for $G_2(F)$ assuming the conjectural endoscopic character identities stated in the previous section.
\end{itemize}

\subsection*{Acknowledgements}
The author is deeply grateful to my advisor Yoichi Mieda for his constant encouragement and guidance throughout my doctoral studies.
The author also would like to express sincere gratitude to his family for their unwavering support.
Without their help, this work would not have been possible.

A substantial portion of this research was carried out during my visit to the National University of Singapore in 2024.
The author is grateful to Wee Teck Gan for hosting me and for many helpful discussions.
The author is also grateful to Rapha\"el Beuzart-Plessis for sharing his manuscript on the Hiraga--Ichino--Ikeda conjecture for classical groups and for his insightful comments on this work. 
Finally, the author thanks Atsushi Ichino, Masao Oi, and Satoshi Wakatsuki for their encouragement and valuable feedback.
This work was supported by JSPS KAKENHI Grant Number 23KJ0403.

\section{Notation and assumptions}\label{section:2}

\subsection{Local fields}

Let $F$ be a non-archimedean local field. 
We denote the ring of integers by $\fo_F$ and its maximal ideal by $\fp_F$.
Let $\varpi_F \in \fp_F$ be a uniformizer of $\fo_F$.
We denote the cardinality of the residue field $k_F = \fo_F/\fp_F$ of $F$ by $q = q_F$.
We denote the normalized absolute value of $F$ by $\abs{\cdot}_F$ which satisfies $\abs{\varpi_F}_F = q^{-1}$.
We denote a unique extension of $\abs{\cdot}_F$ to $\overline{F}$ by the same symbol.

Let $\Gamma_F$ (resp. $W_F$) be the Galois group (resp. the Weil group) of $F$.
Let $I_F$ be the inertia subgroup of $\Gamma_F$, and let $\mathrm{Frob}_F$ be a geometric Frobenius element of $F$.
Let $L_F = W_F \times \SL_2(\C)$ be the Weil-Deligne group of $F$.

Let $\psi = \psi_F$ be a nontrivial additive character of $F$.
Unless otherwise stated, we assume that $\psi_F$ is trivial on $\fo_F$ and nontrivial on $\varpi_F^{-1}\fo_F$.
We set 
\begin{align*}
    \bS^1 &= \{ z \in \C \mid \abs{z} = 1 \}, \\
    \C_{+} &= \{ z \in \C \mid \mathrm{Re}(z) > 0 \}.
\end{align*}

We define the function $\zeta_F(s)$ by
\begin{align*}
    \zeta_F(s) = \frac{1}{1-q^{-s}}.
\end{align*}

\subsection{Reductive groups}

Let $F$ be a field of characteristic 0.
Let $\bG$ be an algebraic group over $F$.
We will often write simply $G$ for $\bG(F)$, omitting explicit reference to $\bG$.
If $E/F$ is a field extension, then we write the base change $\bG \times_{\mathrm{Spec}(F)} \mathrm{Spec}(E)$ of $\bG$ to $E$ as $\bG_E$. We use similar notation for Lie algebras and morphisms of algebraic groups.
For example, for a Lie algebra $\fg$ over $F$, we write $\fg_E = \fg \otimes_F E$.

Let $G$ be an algebraic group over $F$.
Let $X^*(G)$ denote the group of algebraic characters of $G$ defined over $F$.
If $G$ is a torus, $X_*(G)$ denotes the group of algebraic cocharacters of $G$ defined over $F$.
Let $S \subset G$ be a split subtorus over $F$.
Then, the group $S$ acts on the Lie algebra $\Lie(G)$ via the restriction of the adjoint action and this action is diagonalizable.
Let $\Phi(G, S)$ be the set of nonzero weights of the action, and we call its elements the roots for $S$ in $G$.

Let $G$ be a reductive algebraic group over $F$.
Unless otherwise stated, when we speak of ``subgroups of $G$'', we mean the group of $F$-points of algebraic subgroups of $\bG$ defined over $F$.

Let $A_0$ be a maximal $F$-split torus of $G$ and let $P_0$ be a minimal parabolic subgroup of $G$ over $F$.
We assume that $A_0 \subset P_0$.
We call such a pair $(A_0, P_0)$ a minimal parabolic pair.
We denote $Z_G(A_0)$ by $M_0$, which is a Levi component of $P_0$.
We write the Levi decomposition $P_0 = M_0 \ltimes N_0$.

Let $P$ be a parabolic subgroup of $G$.
We write $P = MN$ for a parabolic subgroup $P$ of $G$ with the Levi decomposition $P = M \ltimes N$.
The parabolic subgroup is called standard (resp. semi-standard), if $P_0 \subset P$ (resp. $A_0 \subset P$). 

Let $M$ be a Levi subgroup of $G$.
We say that a Levi subgroup $M$ of $G$ is semi-standard if $M_0 \subset M$.
A semi-standard Levi subgroup $M$ is said to be standard if it is a Levi component of a standard parabolic subgroup $P$.
In this case, the group $P$ determines the group $M$ and let $M_P$ denote the group $M$.

We denote by $\cF^G(M)$ (resp. $\cL^G(M)$) the set of parabolic subgroups (resp. Levi subgroups) of $G$ containing $M$.
We write the set of parabolic subgroups of $G$ having a Levi component $M$ as $\cP^G(M)$.
Let $\cL^G_0$ denote a set of representatives of conjugacy classes of the standard Levi subgroups of $G$.

\subsection{Root systems}
    Let $G$ be a reductive group over $F$.
    Let $(A_0, P_0)$ be a minimal parabolic pair of $G$.
    
    Let $M$ be a standard Levi subgroup of $G$.
    Let $A_M$ be the maximal split torus in the center of $M$.
    We call the group $A_M$ the split component of $M$.
    We set
    \begin{align*}
        \fa^*_M = X^*(A_M) \otimes_{\Z} \R
    \end{align*}
    and 
    \begin{align*}
        \fa_M = X_*(A_M) \otimes_{\Z} \R.
    \end{align*}
    We have a canonical isomorphism 
    \begin{align*}
        X^*(M) \otimes_{\Z} \R \xrightarrow{\sim} \fa^*_M
    \end{align*}
    obtained from the restriction map.

    Let $M \subset M'$ be Levi subgroups of $G$.
    Then, we have the canonical maps 
    \begin{align*}
    X^*(M') \to X^*(M) \to X^*(A_M) \to X^*(A_{M'}).
    \end{align*}
    This gives the surjection
    \begin{align*}
    \fa^*_{M} \to \fa^*_{M'} 
    \end{align*}
    and its splitting.
    We denote the kernel by $\fa^{M', *}_M$.
    Thus, we have the canonical decomposition
    \begin{align}
    \fa^*_{M} \xrightarrow{\sim} \fa^*_{M'} \oplus \fa^{M', *}_M 
    \label{eq:decomposition-of-character-space}.
    \end{align}
    We set $\fa^{M', *}_{M, \C} = \fa^{M', *}_M \otimes_{\R} \C$.
    
    From a minimal parabolic pair $(A_0, P_0)$, we obtain the root system $\Phi(G, A_0)$ and the basis $\Delta_0$.
    We write $P = M P_0$.
    From the restriction of $\Delta_0$ to $A_M$, we obtain the basis $\Delta_P$ of $\fa^{G, *}_{M}$.
    We also denote the set of simple roots in $\Phi(M, A_0)$ by $\Delta^M_0$.
    Similarly, for an element $\alpha \in \Delta_P$, we can attach the coroot $\alpha^{\vee} \in \fa_M$ by using the lifting to $\Delta_0^{\vee}$.
    We denote the set of reduced elements in $\Phi(G, A_M)$ by $\Phi(G, A_M)_{\red}$.

\subsection{$L$-groups}
For details, see \cite{Borel1979-AutomorphicLfunctions}*{I.2}.
Assume that $F$ is a local field of characteristic $0$.
We denote the $L$-group of $G$ by ${}^LG = \widehat{G} \rtimes W_F$.
If $H$ is another reductive group over $F$ and there exists an $L$-morphism ${}^L H \to {}^L G$, then we usually denote it by
\begin{align*}
     \iota^{G}_{H} \colon {}^L H \to {}^L G
\end{align*}
in this paper. 
For example, this notation will be used if $H$ is an endoscopic group of $G$ or a Levi subgroup of $G$.

\subsection{Local factors}

Let $F$ be a non-archimedean local field of characteristic $0$.
Let $\phi \colon L_F \to \GL(V)$ be an $L$-parameter for a general linear group.
Then, we can write this parameter as 
\begin{align*}
    \phi = \bigoplus_{i \geq 0} \phi_i \otimes S_{i},
\end{align*}
where $\phi_i$ is a semisimple representation of $W_F$ and $S_{i}$ is the irreducible algebraic representation of $\SL_2(\C)$ of dimension $i+1$.
Then, we can define the Artin $L$-factor $L(s, \phi)$, the $\epsilon$-factor $\epsilon(s, \phi, \psi_F)$, and the $\gamma$-factor $\gamma(s, \phi, \psi_F) = \epsilon(s, \phi, \psi_F) \frac{L(1-s, \phi^{\vee})}{L(s, \phi)}$.
For example, we have 
\begin{align*}
    L(s, \phi) = \prod_{i \geq 0} L(s + \tfrac{i}{2}, \phi_i),
\end{align*}
where 
\begin{align*}
    L(s, \phi_i) = \det(1 - q^{-s} \phi_i(\mathrm{Frob}_F) \mid V^{I_F})^{-1}.
\end{align*}
These local factors are additive with respect to direct sums of representations.
We also use Arthur's version of local factors; see \cite{Art13}*{1.3.2}.
We set 
\begin{align*}
    \gamma_A(s, \phi, \psi_F) = \epsilon(s, \phi, \psi_F) \frac{L(1+s, \phi)}{L(s, \phi)}.
\end{align*}

Note that, for any tempered $L$-parameter $\phi$ of the general linear group, the local factor $L(s, \phi)$ (resp. $\epsilon(s, \phi, \psi_F)$) is holomorphic and nonzero at $s=1$ (resp. at $s=0$).
Thus, we have
\begin{align*}
    \ord_{s=0} \gamma(s, \phi, \psi_F) =
    \ord_{s=0} \gamma_A(s, \phi, \psi_F) = 
    -\ord_{s=0} L(s, \phi).
\end{align*}
We also note that 
\begin{align*}
    \lim_{s \to 0}
    \frac{\gamma(s, \mathbf{1}, \psi)}{\gamma_A(s, \mathbf{1}, \psi)}
    = 1.
\end{align*}

For a connected reductive group $G$ over $F$, an $L$-parameter $\phi \colon L_F \to {}^L G$ and a finite-dimensional representation $r \colon {}^L G \to \GL(V)$, we can define the local factors $L(s, \phi, r)$, $\epsilon(s, \phi, r, \psi_F)$, and $\gamma(s, \phi, r, \psi_F)$ by composing $r$ with $\phi$.
We also have Arthur's version $\gamma_A(s, \phi, r, \psi_F)$ defined similarly.

\subsection{Whittaker datum and the Tits liftings}

Let $F$ be a non-archimedean local field of characteristic $0$.
Let $G$ be a quasi-split reductive group over $F$.
Let $(T, B)$ be a Borel pair in $G$ defined over $F$.
We consider the set of roots $\Phi(\bG_{\overline{F}}, \bT_{\overline{F}})$.
Let $\Delta_B$ be the basis of $\Phi(\bG_{\overline{F}}, \bT_{\overline{F}})$ corresponding to $B$.
Let $\{X_{\alpha}\}_{\alpha}$ be a set of root vectors $X_{\alpha} \in (\fg_{\overline{F}})_{\alpha}$ for each simple root $\alpha \in \Delta_B$.
We call the triple $(T, B, \{X_{\alpha}\}_{\alpha \in \Delta_B})$ a splitting of $G$ over $F$, or an $F$-splitting of $G$, if the set $\{X_{\alpha}\}_{\alpha}$ is stable under the action of $\Gamma_F$.
We can dually define the notion of a $\Gamma_F$-splitting
$(\widehat{B}, \widehat{T}, \{ X_{\alpha^{\vee}} \}_{\alpha^{\vee}})$ for the dual group $\widehat{G}$.
In this situation, for any semi-standard Levi subgroup $M$, we have the canonical $L$-embedding $\iota^G_M \colon {}^L M \hookrightarrow {}^L G$.
Let $\mathbf{spl}_G = (T, B, \{X_{\alpha}\}_{\alpha \in \Delta_B})$ be an $F$-splitting of $G$.
Also, let $\mathbf{spl}_{\widehat{G}} = (\widehat{B}, \widehat{T}, \{ X_{\alpha^{\vee}} \}_{\alpha^{\vee}})$ be a $\Gamma_F$-splitting of $\widehat{G}$.

From $\mathbf{spl}_G$ and $\psi_F$, we obtain a Whittaker datum $\fw$ for $G$.
We follow the definitions in \cite{atobe2024localintertwiningrelationscotempered}*{1.7}.
Let $M_1, M_2$ be semi-standard Levi subgroups of $G$. 
Let $T(G, (M_1, M_2))$ be the transporter in $G$ from $M_1$ to $M_2$.
Then, we can define the Weyl set 
\begin{align*}
     W(G, (M_1, M_2)) 
     = M_2 \backslash T(G, (M_1, M_2)) /M_1.
\end{align*}
Then, for each element $w \in W(M_1, M_2)$, we obtain the Tits lifting $\widetilde{w} \in N_G(A_0) \subset G$ by using $\mathbf{spl}_G$; see \cite{atobe2024localintertwiningrelationscotempered}*{p.23, line 16}.
We denote $W(G, (M, M))$ by $W(G, M)$.

\subsection{Groups and representations}

For any group $G$ and subgroup $H$, we write the centralizer of $H$ in $G$ (resp. the normalizer of $H$ in $G$) as $\Cent_G(H)$ (resp. $\Norm_G(H)$).
For any group $G$ and its representation $\pi$ which are in an appropriate category, let $\pi^{\vee}$ denote the contragredient representation (in that category) of $\pi$.
We mainly consider the category of smooth admissible representations of a $p$-adic reductive group.
For a locally compact abelian group $G$, let $G^{\cD}$ denote the Pontryagin dual of $G$.

Let $F$ be a non-archimedean local field of characteristic $0$.
Let $G$ be a quasi-split reductive group over $F$.
In this paper, representations of $G$ are smooth representations of $G$ over $\C$.
Let 
\begin{align*}
    \Pi_{\disc}(G) \subset \Pi_{\temp}(G) \subset \Pi(G)
\end{align*}
denote the set of (isomorphism classes of) discrete series irreducible representations, the set of tempered irreducible representations, and the set of irreducible representations of $G$, respectively.

Let $H$ be another reductive group which is possibly equal to $G$.
Let $\theta \colon G \to H$ be an isomorphism.
For any representation $\sigma$ of $H$ and $\pi$ of $G$, we set
\begin{align*}
    &\sigma^{\theta} = \sigma \circ \theta, \\
    &{}^{\theta}\pi = \pi \circ \theta^{-1}.
\end{align*}

Let $P=MN$ be a parabolic subgroup of $G$ with a Levi decomposition.
We denote the normalized parabolic induction by $i^G_P$.
Later, we also consider the induction from an open subgroup $H$ to $G$.
In this situation, we write the usual induction functor as $\ind^G_{H}$.

\subsection{Unramified characters}\label{subsection:unramified-characters}

Let $F$ be a non-archimedean local field of characteristic $0$.
Let $G$ be a reductive group over $F$.

For an element $\chi \in X^*(G)$, we set $\abs{\chi}_F \colon G \to \R_{>0}$ by 
\begin{align*}
      \abs{\chi}_F \colon g \mapsto \abs{\chi(g)}_F.
\end{align*}
We define $G^1 = \bigcap_{\chi \in X^*(G)} \Ker(\abs{\chi}_F)$.
An unramified character of $G$ is a character of $G$ that factors through $G/G^1$.
We denote the set of unramified characters of $G$ by $\cX(G)$.
We have the surjective homomorphism
\begin{align}
       X^*(G) \otimes_{\Z} \C \to \cX(G) \colon \chi \otimes s \mapsto \abs{\chi}_F^{s}, \label{eq:character-to-unr-character}
\end{align}
whose kernel is a lattice in $i\fa^*_G$ containing the lattice
\begin{align*}
    \frac{2\pi i}{\log(q)} X^*(G).
\end{align*}
If $G$ is a split torus, the kernel coincides with the above group.

The map \cref{eq:character-to-unr-character} endows $\cX(G)$ with a structure of a complex analytic Lie group.
Let $\cX_{\unit}(G)$ be the subgroup of unitary unramified characters.
Thus, if $G$ is a split torus, we have the isomorphism 
\begin{align*}
    X^*(G) \otimes_{\Z} i\R / \frac{2\pi i}{\log(q)} X^*(G)
    \xrightarrow{\sim} 
    \cX_{\unit}(G),
\end{align*}
obtained from the homomorphism above.
The map \cref{eq:character-to-unr-character} endows $\cX_{\unit}(G)$ with the structure of a compact torus induced by the Euclidean space $i\fa^*_G$.

\subsection{Embedding of a torus in the product of linear groups}\label{subsection:embedding-tori}

Let $\prod_{i=1}^{n} G_i \subset \prod_{i=1}^{n} \GL_{n_i}$ be a product of linear groups such that $Z(\GL_{n_i}) \subset G_i$.
For $(a_i)_{i=1}^{n} \in \Z^n$, let $\GL_1^{(a_i)_{i=1}^n}$ denote the image of the morphism 
\begin{align*}
    \GL_1 \to \prod_{i=1}^n G_i \subset \prod_{i=1}^n \GL_{n_i} \colon z \mapsto (z^{a_i} \cdot \mathrm{id}_{n_i})_{i=1}^n.
\end{align*}
This morphism is a closed immersion if and only if the ideal $(a_1, \ldots, a_n)$ generated by the integers $a_i$ in $\Z$ is equal to $\Z$.

\section{Local Langlands correspondence and local theta correspondence}\label{section:3}

\subsection{Conjectural correspondence}

Let $F$ denote a non-archimedean local field of characteristic $0$.
In this paper, we consider the local Langlands correspondence only for tempered representations of quasi-split groups.

We fix an $F$-splitting $(T, B, \{X_{\alpha}\}_{\alpha})$ for $G$.

\begin{definition}[\cite{SilbergerZink2018-Classification-L-parameter}*{Definition 3.1}]
    Let $\phi \colon L_F \to {}^LG$ be a continuous homomorphism such that the composition 
    \begin{align*}
        W_F \to L_F \overset{\phi}{\to} {}^LG \to W_F
    \end{align*}
    is the identity.
    Here, the first map is the inclusion and the third map is the projection.
    The map $\phi$ is called an $L$-homomorphism for $G$ if its restriction to $W_F$ is semisimple and its restriction to $\SL_2(\C) \to \widehat{G}$ is an algebraic homomorphism.

    The $L$-homomorphism $\phi$ is called tempered if the image of the cocycle
    \begin{align*}
        W_F \to L_F \overset{\phi}{\to} {}^LG \to \widehat{G}
    \end{align*}
    is bounded.
\end{definition}

\begin{remark}
    We recall the definition of semisimplicity in this context.
    Let $E/F$ be the splitting field of $G$.
    Then, the action of $W_F$ on $\widehat{G}$ factors through the quotient $\Gal(E/F)$ and this defines the group $\widehat{G} \rtimes \Gal(E/F)$.
    Thus, we obtain the homomorphism 
    \begin{align*}
        \phi_{E/F} \colon L_F \to {}^L G \to \widehat{G} \rtimes \Gal(E/F).
    \end{align*}
    We say that the homomorphism $W_F \to {}^LG$ is semisimple if the induced homomorphism $W_F \to \widehat{G} \rtimes \Gal(E/F)$ is semisimple, in the sense that for any element $\gamma \in W_F$, the image $\phi_{E/F}(\gamma)$ is semisimple in the complex algebraic group $\widehat{G} \rtimes \Gal(E/F)$.
\end{remark}

\begin{definition}
    The set of $L$-homomorphisms has a conjugation action by $\widehat{G}$.
    Let $\phi_1, \phi_2$ be $L$-homomorphisms for $G$. 
    The $L$-homomorphisms $\phi_1$ and $\phi_2$ are said to be equivalent if they are conjugate by an element $s \in \widehat{G}$.
    We call an equivalence class of $L$-homomorphisms for $G$ an $L$-parameter for $G$.
    We denote by $\Phi(G)$ the set of $L$-parameters for $G$.

    We also denote by $\Phi_{\temp}(G) \subset \Phi(G)$ the subset of equivalence classes of tempered $L$-homomorphisms for $G$.
\end{definition}

\begin{remark}
    For a general (i.e., not necessarily quasi-split) reductive group $G$, we also have to assume that $\phi$ is relevant for $G$ to define the notion of $L$-homomorphisms for $G$; see \cite{SilbergerZink2018-Classification-L-parameter}*{p.319 line 13}.
\end{remark}

\begin{remark}
    In the situation that we distinguish an $L$-homomorphism $\phi$ from its equivalence class, we write the equivalence class as $[\phi]_{G}$.
\end{remark}

\begin{definition}
    For an $L$-homomorphism $\phi$ for $G$, we write $S_{\phi}$ for the centralizer subgroup $\Cent_{\widehat{G}}(\phi(L_F)) = \widehat{G} \cap \Cent_{{}^LG}(\phi(L_F))$ in $\widehat{G}$.
    
    For an $L$-parameter $[\phi]_G$ for $G$, the conjugacy class of the subgroup $S_{\phi}$ of $\widehat{G}$ is uniquely determined by $[\phi]_G$.
    We denote this conjugacy class of subgroups in $\widehat{G}$ by $S_{[\phi]_G}$.
\end{definition}

\begin{lemma}[\cite{Heiermann2006-unipotentpoles}*{Proposition 5.2}]
    For any $L$-homomorphism $\phi$ for $G$, the group $S_{\phi}$ is reductive.
\end{lemma}

\begin{lemma}\label{lemma:discrete-parameter}
    Let $[\phi]_G$ be an $L$-parameter for $G$.  
    Let $M$ be a standard Levi subgroup of $G$ such that $\phi'(L_F) \subset {}^LM$ for an $L$-homomorphism $\phi' \in [\phi]_G$.
    Assume that $M$ is minimal in the set of standard Levi subgroups which satisfy this condition.
    We fix an $L$-homomorphism $\phi' \in [\phi]_G$ satisfying the condition above.
    Then, we have 
    \begin{align*}
        \Cent_{{}^L G}(Z(S_{\phi'})^{\circ}) = {}^L M.
    \end{align*} 
\end{lemma}

\begin{proof}
    Let $\phi' \colon L_F \to {}^L G$ be an $L$-homomorphism for $G$.
    If we have $\phi'(L_F) \subset {}^LM$ for some standard Levi subgroup $M$ of $G$, then we have 
    \begin{align*}
         Z(\widehat{M})^{\Gamma_F} \subset \Cent_{\widehat{G}}(\phi'(L_F))
    \end{align*} 
    and this group is central in the right-hand side.
    Thus we have 
    \begin{align*}
          Z(\widehat{M})^{\Gamma_F} \subset Z(\Cent_{\widehat{G}}(\phi'(L_F))) = Z(S_{\phi'}).
    \end{align*} 
    As we have $\Cent_{{}^L G}(\{Z(\widehat{M})^{\Gamma_F}\}^{\circ}) = {}^L M$, we obtain 
    \begin{align*}
         \phi'(L_F) \subset \Cent_{{}^L G}(Z(S_{\phi'})^{\circ}) \subset {}^L M.
    \end{align*}
    The group in the middle is conjugate to a group ${}^L M'$ for a standard Levi subgroup $M'$ of $M$ by \cite{Borel1979-AutomorphicLfunctions}*{Lemma 3.5} in ${}^L M$. 
    If we assume the minimality of $M$, then we have $M = M'$.
    This completes the proof.
\end{proof}

\begin{definition}
    We say that an $L$-parameter $[\phi]_G$ is discrete if, for any representative $\phi \in [\phi]_G$, there are no proper standard Levi subgroups $M$ of $G$ such that $\phi(L_F) \subset {}^L M$.
    By the proof of \cite{Borel1979-AutomorphicLfunctions}*{Proposition 3.6}, this condition is equivalent to the condition 
    \begin{align*}
         S_{\phi}/Z(\widehat{G})^{\Gamma_F}
    \end{align*}
    is finite for any representative $\phi$ of $[\phi]_G$.
    We denote by $\Phi_{\disc}(G)$ the set of equivalence classes of discrete tempered $L$-parameters for $G$.
\end{definition}

We have the following local Langlands correspondence for unramified characters.

\begin{proposition}[\cite{Haines2014-StableBernstein}*{p.128 line 3}]\label{proposition:local-Langlands-for-unramified-characters}
    We have functorial isomorphisms
    \begin{align*}
        \cX(G) \xrightarrow{\sim} \{ (Z(\widehat{G})^{I_F})_{\mathrm{Frob}_F} \}^{\circ}
    \end{align*}
    in $G$, obtained from the Kottwitz homomorphisms
    \begin{align*}
       \kappa_G \colon G(F) \to (X^*(Z(\widehat{G}))_{I_F})^{\mathrm{Frob}_F}.
    \end{align*}
    The right-hand side is a group of cocycles $H^1(W_F/I_F, Z(\widehat{G})^{I_F})$, which can be seen as $L$-homomorphisms $W_F \to Z(\widehat{G})^{I_F}$.
    For any element $\chi \in \cX(G)$, we write the corresponding $L$-homomorphism as the same symbol $\chi$.

    In this correspondence, unitary unramified characters correspond to bounded cocycles.
\end{proposition}

Later, we consider split reductive groups.
To make things explicit, we give explicit description of the isomorphism above.
Recall that there is a unique isomorphism
\begin{align*}
    X^*(\bG_{\overline{F}}) \xrightarrow{\sim} X_*(Z(\widehat{G})^{\circ})
\end{align*}
which is compatible with $X^*(\bT_{\overline{F}}) \xrightarrow{\sim} X_*(\widehat{T})$, by \cite{SilbergerZink2018-Classification-L-parameter}*{Proposition 4.1}.

\begin{lemma}
    Assume that $G$ is split in \cref{proposition:local-Langlands-for-unramified-characters}.
    Then, the isomorphism in \cref{proposition:local-Langlands-for-unramified-characters} becomes a map
    \begin{align*}
        \cX(G) \xrightarrow{\sim} Z(\widehat{G})^{\circ}
    \end{align*}
    This isomorphism sends any unramified character $\abs{\chi}_F^s$ to the element $\chi(q^{-s}) \in Z(\widehat{G})^{\circ}$ via the isomorphism
    \begin{align*}
        X^*(G) \xrightarrow{\sim} X_*(Z(\widehat{G})^{\circ}).
    \end{align*}
\end{lemma}

\begin{proof}
    If $G$ is a torus, this follows from \cite{KalethaPrasad-Bruhat-Tits}*{Proposition 11.1}.
    If the derived group of $G$ is simply connected, the lemma follows from the construction of the Kottwitz isomorphisms \cite{KalethaPrasad-Bruhat-Tits}*{Construction 11.3.1}.
    The general case follows from the previous cases by taking the split $z$-extension of $G$, as described in \cite{KalethaPrasad-Bruhat-Tits}*{Construction 11.4.1}.
\end{proof}

\begin{example}
    If we assume that $G = \bG_m$ and $\chi$ is the identity map of $\bG_m$ in the previous lemma, then the character 
    \begin{align*}
        F^{\times} \to \C^{\times} \colon a \mapsto \abs{a}_F
    \end{align*}
    is sent to the element $q^{-1} = \abs{\varpi}_F \in \C^{\times}$.
    This corresponds to the unramified character 
    \begin{align*}
        W_F \to \C^{\times} \colon \mathrm{Frob}_F \mapsto \abs{\varpi}_F,
    \end{align*}
    where $\mathrm{Frob}_F$ is a geometric Frobenius element.
\end{example}

In this paper, we use the following weaker form of the local Langlands correspondence to state the Hiraga--Ichino--Ikeda conjecture.

\begin{conjecture}\label{conjecture:local-Langlands-conjecture}
    Let $F$ be a non-archimedean local field of characteristic $0$.
    Let $G$ be a quasi-split reductive group over $F$.
    There exists a surjection
    \begin{align*}
        \cL_M \colon \Pi_{\temp}(M) \to \Phi_{\temp}(M)
    \end{align*}
    for any standard Levi subgroup $M$ of $G$.
    The inverse image $\cL_G^{-1}(\phi)$ of an $L$-parameter $\phi$ is called the $L$-packet $\Pi_{\phi}=\Pi_{\phi}(G)$ associated to $\phi$.
    This map should satisfy the following properties.
    \begin{enumerate}
       \item(Compatibility with parabolic inductions)
       Let $P=MN$ be a standard parabolic subgroup of $G$ with $M$ standard.
       Let $\sigma$ be a tempered irreducible representation of $M$. 
       Let $\pi$ be an irreducible component of the parabolic induction $i^G_P(\sigma)$.
       Then, we have $\iota^G_M \circ \cL_M(\sigma) = \cL_G(\pi)$.
       Moreover, the set $\Pi_{\cL_G(\pi)}$ is the union of irreducible components of $i^G_P(\sigma')$ for $\sigma' \in \Pi_{\cL_M(\sigma)}$.

        \item (Compatibility with twists by unramified characters)
        Let $\chi$ be a unitary unramified character of $M$ and let $\sigma$ be a tempered irreducible representation of $M$.
        Then, we have 
        \begin{align*}
            \cL_M(\sigma \otimes \chi) = \cL_M(\sigma) \otimes \chi.
        \end{align*}
        Here, the $L$-parameter $\cL_M(\sigma) \otimes \chi$ is given as 
        \begin{align*}
            w \rtimes g \in L_F = W_F \times \SL_2(\C) \mapsto \cL_M(\sigma)(w, g) \chi(w).
        \end{align*}
        
        \item (Enhanced $L$-parameters)
        Let $\phi$ be an $L$-homomorphism for $G$.
        For each Whittaker datum $\fw$ of $G$, we have a bijection 
        \begin{align*}
            \iota_{\fw} \colon \Pi_{[\phi]_G}(G) \to \Irr(\pi_0(S_{\phi}/Z(\widehat{G})^{\Gamma_F})). 
        \end{align*}
        Here, $\Irr(\pi_0(S_{\phi}/Z(\widehat{G})^{\Gamma_F}))$ denotes the set of isomorphism classes of irreducible representations of the finite group $\pi_0(S_{\phi}/Z(\widehat{G})^{\Gamma_F})$ over $\C$.
        Also, for each representation $\pi \in \Pi_{[\phi]_G}(G)$, the dimension $\dim_{\C}(\iota_{\fw}(\pi))$ is independent of the Whittaker datum $\fw$.
        We denote this number by $\langle \pi, 1 \rangle$.

        \item (Characterization of discrete series representations)
        A tempered irreducible representation $\sigma$ of $M$ is a  discrete series representation if and only if the $L$-parameter $\cL_M(\sigma)$ is discrete.
        \item (Stability of $L$-packets)
        The character 
        \begin{align*}
            \Theta_{\phi} = \sum_{\pi \in \Pi_{\phi}} \langle \pi, 1 \rangle \, \tr \pi
        \end{align*}
        is a stable invariant distribution on $G$.
        In particular, this implies that the distribution $\Theta_{\phi}$ and the map $\cL_{G}$ are invariant under the action of $G_{\ad}(F)$.
    \end{enumerate}
\end{conjecture}

\begin{remark}
     In this paper, we may write the image $\cL_M(\sigma)$ of a tempered representation $\sigma$ of $M$ as $\phi^M_{\sigma}$.
\end{remark}

\begin{remark}[Direct products and the LLC]\label{remark:direct-product-and-LLC}
    If the groups $G_1$ and $G_2$ satisfy a property in \cref{conjecture:local-Langlands-conjecture}, then their direct product $G_1 \times G_2$ also satisfies this property.
\end{remark}

We will define the data assuming \cref{conjecture:local-Langlands-conjecture}, which will be used to state the Hiraga--Ichino--Ikeda conjecture. 

% Given an $L$-parameter $\phi \colon L_F \to {}^L G$ and a finite-dimensional complex representation $r \colon {}^L G \to \GL_{\C}(V)$, we have the Artin local factors
% \begin{align*}
%     \gamma(s, \phi, r, \psi), \quad L(s, \phi, r), \quad \epsilon(s, \phi, r, \psi).
% \end{align*}
% For example, see \cite{Beuzart-Plessis2021-Plancherel-GLnE-GLnF}*{2.12} for details.
% For $G = \GL_n$, we write the standard $\gamma$-factor (resp. $\epsilon$-factor) of an $L$-parameter $\phi$ simply as $\gamma(s, \phi, \psi)$ (resp. $\epsilon(s, \phi, \psi)$).

\begin{definition}[Invariants of an  $L$-parameter]\label{definition:invariants-in-formal-degree-conjecture}
    Let $G$ be a quasi-split reductive group over $F$.
    Assume that \cref{conjecture:local-Langlands-conjecture} holds for $G$.
    Let $M$ be a standard Levi subgroup of $G$.
    Let $\Ad_{M}$ denote the adjoint representation ${}^L M \to \Aut_{\C}(\Lie(\widehat{M}))$.
    \begin{enumerate}
        \item 
        We set $M^{\natural} = M/A_M$.
        There is a natural injection 
        \begin{align*}
            {}^L M^{\natural} \hookrightarrow {}^L M
        \end{align*}
        which is the dual map of the projection $M \twoheadrightarrow M^{\natural}$.
        \item 
        Let $\sigma$ be a discrete series representation of $M$.
        Then, we have the number $\langle \sigma, 1 \rangle $.
        We define the group $S^{\natural}_{\sigma}= S^{\natural}_{\phi^M_{\sigma}}$ by 
        \begin{align*}
            S^{\natural}_{\sigma}
            = 
            S^{\natural}_{\phi^M_{\sigma}}
            =
            \Cent_{\widehat{M^{\natural}}}(\phi^M_{\sigma}).
        \end{align*} 
        \item 
        Let $\sigma$ be a discrete series representation of $M$.
        The adjoint $\gamma$-factor for $\sigma$ is defined by 
        \begin{align*}
            \gamma(s, \sigma, \Ad_{M}, \psi) = \gamma(s, \phi^M_{\sigma}, \Ad_{M}, \psi).
        \end{align*}
        Let $\Ad_{G/M}$ denote the representation of ${}^L M$ given by the quotient $\Ad_G \circ \iota^G_M / \Ad_{M}$.
        We set 
        \begin{align*}
            \gamma(s, \sigma, \Ad_{G/M}, \psi) = \gamma(s, \phi^M_{\sigma}, \Ad_{G/M}, \psi).
        \end{align*}
    \end{enumerate}
\end{definition}

\begin{remark}\label{remark:products-of-adjoint-gamma-factors}
    We use the notation from \cref{definition:invariants-in-formal-degree-conjecture}.
    We have $\gamma(s, \sigma, \Ad_{M}, \psi) \gamma(s, \sigma, \Ad_{G/M}, \psi) = \gamma(s, \phi^M_{\sigma}, \Ad_{G} \circ \iota^G_M, \psi)$.
    For any standard parabolic subgroup $P=MN$ with $M$ standard and any irreducible component $\pi$ of $i^G_P(\sigma)$, we have 
    \begin{align*}
        \gamma(s, \phi^M_{\sigma}, \Ad_{G} \circ \iota^G_M, \psi)
        = 
        \gamma(s, \pi, \Ad_G, \psi)
    \end{align*}
    by compatibility with parabolic inductions in \cref{conjecture:local-Langlands-conjecture}.
\end{remark}

\begin{lemma}\label{lemma:well-definedness}
    Let $M_1, M_2$ be standard Levi subgroups of $G$.
    Let $\phi_1, \phi_2$ be discrete $L$-parameters for $M_1$ and $M_2$.
    Assume that $\iota^G_{M_1} \circ \phi_1$ and $\iota^G_{M_2} \circ \phi_2$ are equivalent.
    Then, we have 
    \begin{align*}
        \gamma(s, \phi_1, \Ad_{G/M_1}, \psi) =  \gamma(s, \phi_2, \Ad_{G/M_2}, \psi).
    \end{align*}
\end{lemma}

\begin{proof}
    Let $g \in \widehat{G}$ be an element such that 
    \begin{align*}
        \Ad(g) \circ \iota^G_{M_1} \circ \phi_1 = \iota^G_{M_2} \circ \phi_2.
    \end{align*}
    By \cref{lemma:discrete-parameter}, we have 
    \begin{align*}
        \Ad(g)({}^L M_1) = {}^L M_2.
    \end{align*}
    Thus, the representation $\Ad_{G/M_1} \circ \iota^G_{M_1} \circ \phi_{1}$ of $L_F$ is equivalent to the representation $\Ad_{G/M_2} \circ \iota^G_{M_2} \circ \phi_{2}$ of $L_F$.
    This completes the proof.
\end{proof}

\subsection{Local Langlands correspondence for $\GL_n$}

    \begin{theorem}[\cite{HarrisTaylor2001-simple-Shimura}*{Theorem VII.2.20}, \cite{Henniart2000-simpleproof}*{Th\'eor\`eme 4.2}, \cite{Scholze2013-localLanglandsGLn}*{Theorem 1.2}]\label{theorem:local-Langlands-GLn}
        \cref{conjecture:local-Langlands-conjecture} holds for $G = \GL_n$.
        Furthermore, the map $\cL_{\GL_n}$ has the following properties.
        \begin{enumerate}
        \item[(2-bis)] (Compatibility with character twists) 
        Let $\chi \colon F^{\times} \to \C^{\times}$ be a continuous unitary character. 
        Then, we have 
        \begin{align*}
            \cL_{\GL_n}(\pi \otimes (\chi \circ \det)) = \cL_{\GL_n}(\pi) \otimes \chi,
        \end{align*}
        for any $\pi \in \Pi_{\temp}(\GL_n)$.
        Here, $\chi$ on the right-hand side is the $L$-parameter corresponding to the character $W_F \to \GL_1(\C)$, obtained from $\chi$ by the local class field theory.
        \item[(5)] (Compatibility with duality)
        Let $\pi$ be a tempered irreducible representation of $\GL_n(F)$.
        We have $\cL_{\GL_n}(\pi^{\vee}) = \cL_{\GL_n}(\pi)^{\vee}$.
        Here, we set 
        \begin{align*}
            {\cdot}^\vee \colon (g, w) \in {}^L \GL_n \mapsto ({}^t g^{-1}, w) \in {}^L \GL_n
        \end{align*}
        on the right-hand side. 
        \item[(6)] 
        The central character of $\pi \in \Pi_{\temp}(\GL_n)$ corresponds to the character $\det \circ \cL_{\GL_n}(\pi)$ of $L_F$ and hence of $W_F$ via the local class field theory.
        \end{enumerate}
\end{theorem}

\begin{remark}
    It is also proved that the local Langlands correspondence for $\GL_n$ is compatible with twists by algebraic automorphisms of $\GL_n$ over $F$; see \cite{Haines2014-StableBernstein}*{Proposition 4.10}.
\end{remark}

\subsection{Arthur's local Langlands correspondence for $\SO_{2n}$ and $\Sp_{2n}$}

Let $\Sp_{2n}$ (resp. $\SO_{2n}$) be the symplectic group of a $2n$-dimensional symplectic space (resp. the split special orthogonal group of a $2n$-dimensional quadratic space) over $F$. 

\begin{remark}
    The standard Levi subgroups of $\Sp_{2n}$ are of the form 
    \begin{align*}
        \GL_{n_1} \times \cdots \times \GL_{n_r} \times \Sp_{2m}
    \end{align*}
    with $\sum_{i} 2n_i + 2m = 2n$.
    A similar result holds for $\SO_{2n}$.
\end{remark}

\begin{theorem}[\cite{Art13}*{Theorem 1.5.1, Proposition 6.6.1}]\label{theorem:local-Langlands-correspondence-Sp2n}
    \cref{conjecture:local-Langlands-conjecture} holds for $\Sp_{2n}(F)$.
\end{theorem}

\begin{proof}
    The existence of the map $\cL_{\Sp_{2n}}$ and the property $(3)$ follow from \cite{Art13}*{Theorem 1.5.1}.
    The property $(1)$ follows from the proof of \cite{Art13}*{Proposition 6.6.1}.
    The property $(2)$ follows from \cref{theorem:local-Langlands-GLn}.
    The property $(4)$ follows from \cite{Art13}*{Theorem 6.6.1, Proposition 6.6.5, Theorem 6.7.2}.
    The property $(5)$ follows from \cite{Art13}*{Theorem 2.2.1}.
\end{proof}

Let  $\O_{2n}$ denote the orthogonal group of a $2n$-dimensional split quadratic space over $F$.
We have an isomorphism $\O_{2n}(F)/\SO_{2n}(F) \xrightarrow{\sim} \Z/2\Z$.
To state the local Langlands correspondence for $\SO_{2n}(F)$, we need to introduce a slightly different set of $L$-parameters for $\SO_{2n}$. 

\begin{definition}
    Let $\Phi(\SO_{2n})_{/ \O_{2n}}$ be the set of $\O_{2n}(\C)$-conjugacy classes of $L$-parameters for $\SO_{2n}$. 
    We use similar notation for tempered $L$-parameters.
    For an  $L$-parameter $\phi$ for $\SO_{2n}$, we write its $\O_{2n}(\C)$-conjugacy class as $[\phi]_{\O_{2n}}$.
    
    Dually, we write the set of $\O_{2n}(F)$-conjugacy classes of irreducible representations of $\SO_{2n}(F)$ as $\Pi(\SO_{2n})_{/ \O_{2n}}$.
    We use similar notation for tempered representations.
    For an irreducible representation $\pi$ of $\SO_{2n}(F)$, we write its $\O_{2n}(F)$-conjugacy class as $[\pi]_{\O_{2n}}$.
\end{definition}

Let $\iota^{\GL_{2n}}_{\SO_{2n}} \colon {}^L \SO_{2n} \to {}^L \GL_{2n}$ be the standard $L$-embedding.
Note that the map $\Phi(\SO_{2n}) \to \Phi(\GL_{2n})$ induced by this embedding factors through the map $\Phi(\SO_{2n})_{/\O_{2n}} \to \Phi(\GL_{2n})$.

\begin{theorem}[\cite{Art13}*{Theorem 1.5.1, Proposition 6.6.1, Theorem 2.2.4}]
\label{theorem:local-Langlands-correspondence-SO2n}
     We have a surjection
     \begin{align*}
        {\cL_{\SO_{2n}}}_{/ \O_{2n}} \colon \Pi_{\temp}(\SO_{2n})_{/ \O_{2n}} \to \Phi_{\temp}(\SO_{2n})_{/ \O_{2n}},
     \end{align*}
     which satisfies the conditions of \cref{conjecture:local-Langlands-conjecture} if we replace 
     \begin{itemize}
        \item the representations of $\SO_{2n}(F)$ by $\O_{2n}(F)$-conjugacy classes of representations,
        \item $L$-parameters for $\SO_{2n}$ by $\O_{2n}(\C)$-conjugacy classes of $L$-parameters for $\SO_{2n}$, and
        \item the character 
            \begin{align*}
                \tr \pi
            \end{align*}
            by the average of characters
            \begin{align*}
                \frac{1}{2} (\tr \pi + \tr \pi^{\theta})
            \end{align*}
            for any $\theta \in \O_{2n}(F) \setminus \SO_{2n}(F)$.
     \end{itemize}
     We call the inverse image of the equivalence class $[\phi]_{\O_{2n}}$ under this map, the $L$-packet associated with $[\phi]_{\O_{2n}}$, and we write it as  $\overline{\Pi}_{[\phi]_{\O_{2n}}}(\SO_{2n}) = \overline{\Pi}_{[\phi]_{\O_{2n}}}$.
     
     Furthermore, if $\pi$ is a tempered irreducible representation of $\SO_{2n}(F)$, the following are equivalent.
     \begin{enumerate}
        \item[(i)] The set ${\cL_{\SO_{2n}}}_{/ \O_{2n}}([\pi]_{\O_{2n}})$ consists of a unique $L$-parameter for $\SO_{2n}$.
        \item[(ii)] The set $[\pi]_{\O_{2n}}$ consists of a unique irreducible representation $\pi$ of $\SO_{2n}(F)$.
        \item[(iii)] The representation $\pr_1 \circ \iota^{\GL_{2n}}_{\SO_{2n}} \circ {\cL_{\SO_{2n}}}_{/ \O_{2n}}([\pi]_{\O_{2n}})$ of $L_F$ contains an odd-dimensional orthogonal irreducible representation of $L_F$. 
     \end{enumerate}
\end{theorem}

\begin{definition} 
    Let $\pi$ be a tempered irreducible representation for $\SO_{2n}(F)$.
    We call an $L$-parameter $\phi$ in the class ${\cL_{\SO_{2n}}}_{/\O_{2n}}([\pi]_{\O_{2n}})$ an $L$-parameter of $\pi$.
    We choose such an $L$-parameter for each tempered irreducible representation $\pi$ of $\SO_{2n}(F)$ and denote it by $\phi^{\SO_{2n}}_{\pi}$.
    We use similar notation for Levi subgroups for $\SO_{2n}$.

    If the representation $\pi$ satisfies the equivalent conditions in the statement of \cref{theorem:local-Langlands-correspondence-SO2n}, this $L$-parameter is uniquely determined by $\pi$.
    In this situation, we call the $L$-parameter the $L$-parameter of $\pi$.
    
    Also, we call the $L$-parameter $\iota^{\GL_{2n}}_{\SO_{2n}} \circ [\phi^G_{\pi}]_{\O_{2n}} \in \Phi(\GL_{2n})$ the standard $L$-parameter of $\pi$.
    We denote it by $\phi^{\GL_{2n}}_{\pi}$.
\end{definition}

\begin{remark}\label{remark:invariant-in-formal-degree-even-orthogonal}
    Let $M$ be a standard Levi subgroup of $\SO_{2n}$ and let $\sigma$ be a discrete series representation of $M$.
    Then, we can easily check that the isomorphism class of the group $S^{\natural}_{\phi^M_{\sigma}}$ and the $\gamma$-factors $\gamma(s, \phi^{M}_{\sigma}, \Ad_M, \psi)$ and $\gamma(s, \phi^M_{\sigma}, \Ad_{G/M}, \psi)$ are independent of the choice of $\phi^M_{\sigma}$.
    We write them as $S^{\natural}_{\sigma}, \gamma(s, \sigma, \Ad_M, \psi)$ and $\gamma(s, \sigma, \Ad_{G/M}, \psi)$.
\end{remark}

\begin{definition}
    The central character for an irreducible representation of $\SO_{2n}(F)$ is invariant under the conjugation action of $\O_{2n}(F)$.
    Thus, the central character of the equivalent class $[\pi]_{\O_{2n}}$ is well-defined.
    We call it the central character of $[\pi]_{\O_{2n}}$. 
\end{definition}

Later, we use the following well-known result. 
\begin{proposition}\label{proposition:central-character-SO2n}
    The central characters of the elements in one $L$-packet $\overline{\Pi}_{[\phi]_{\O_{2n}}}$ for $\SO_{2n}(F)$ are the same and its value at $-1 \in \SO_{2n}(F)$ is given by the root number $\epsilon(\frac{1}{2}, \iota^{\GL_{2n}}_{\SO_{2n}} \circ \phi)$ of the standard $L$-parameter $\iota^{\GL_{2n}}_{\SO_{2n}} \circ \phi$.
\end{proposition}

\begin{proof}
    By \cite{LapidRallis2005-localfactors}*{Theorem 1}, the value $\chi_{\pi}(-1)$ of the central character $\chi_{\pi}$  at $-1$ is given by the root number $\epsilon^{\cV}(\frac{1}{2}, \pi)$ of the epsilon factor defined by Lapid--Rallis in \cite{LapidRallis2005-localfactors} for any irreducible representation $\pi$ of $\SO_{2n}(F)$.
    The coincidence of the local factor $\epsilon^{\cV}(s, \pi, \psi)$ defined by Lapid--Rallis and the local factor $\epsilon(s, \iota^{\GL_{2n}}_{\SO_{2n}} \circ \phi^{\SO_{2n}}_{\pi}, \psi)$ obtained from the $L$-parameter $\phi^{\SO_{2n}}_{\pi}$ follows from a globalization argument in \cite{atobe2024localintertwiningrelationscotempered}*{p.130, Remark A.4.1}.
\end{proof}

We denote the standard $L$-embedding ${}^L \Sp_{2n} \to {}^L \GL_{2n+1}$ by $\iota^{\GL_{2n+1}}_{\Sp_{2n}}$.
By the same proof, we can obtain the following proposition.
\begin{proposition}\label{proposition:central-character-Sp2n}
    The central characters of the elements in one $L$-packet $\Pi_{\phi}$ for $\Sp_{2n}(F)$ are the same and its value at $-1 \in \Sp_{2n}(F)$ is given by the root number $\epsilon(\frac{1}{2}, \iota^{\GL_{2n+1}}_{\Sp_{2n}} \circ \phi)$ of the standard $L$-parameter $\iota^{\GL_{2n+1}}_{\Sp_{2n}} \circ \phi$.
\end{proposition}

\subsection{Xu's local Langlands correspondence for $\GSO_{2n}$ and $\GSp_{2n}$}

Let $\GSO_{2n}$ and $\GSp_{2n}$ be the split even orthogonal similitude group of a $2n$-dimensional quadratic space and the symplectic similitude group of a $2n$-dimensional symplectic space over $F$, respectively.
We denote their similitude character by $\lambda$.

Xu \cite{BinXu2016liftingproblem,Xu18local} proved a form of the local Langlands correspondence for the similitude symplectic group $\GSp_{2n}$ and the split similitude orthogonal group $\GSO_{2n}$.
It is described using the Arthur local Langlands correspondence for $\Sp_{2n}$ and $\SO_{2n}$.

\begin{remark}
    The standard Levi subgroups of $\GSO_{2n}$ are of the form 
    \begin{align*}
        \GL_{n_1} \times \cdots \times \GL_{n_r} \times \GSO_{m}
    \end{align*}
    with $\sum_{i} 2n_i + m = 2n$.
    A similar result holds for $\GSp_{2n}$.
\end{remark}

Let $\iota^{\SO_{2n}}_{\GSO_{2n}} \colon {}^L \GSO_{2n} \to {}^L \SO_{2n}$ be the morphism of $L$-groups induced by the standard representation. 
Also, let $\nu \colon  {}^L \GSO_{2n} \to {}^L \GL_1$ be the morphism of $L$-groups induced by the spinor norm.

\begin{proposition}\label{proposition:amalgamation-to-GSO2n}
    Let $\phi^{\SO_{2n}}$ be an $L$-parameter for $\SO_{2n}$ and let $\chi$ an $L$-parameter for $\GL_1$, such that $\epsilon( \frac{1}{2}, \iota^{\GL_{2n}}_{\SO_{2n}} \circ \phi^{\SO_{2n}}) = \chi(-1)$.
    Here, $\chi(-1)$ is defined by the local class field theory.
    Then, there exists an $L$-parameter $\phi^{\GSO_{2n}}$ for $\GSO_{2n}$ such that
    \begin{align*}
        \iota^{\SO_{2n}}_{\GSO_{2n}} \circ \phi^{\GSO_{2n}}
        &=
        \phi^{\SO_{2n}}, \\
        \nu \circ \phi^{\GSO_{2n}} &= \chi.
    \end{align*}
    The $L$-parameter $\phi^{\GSO_{2n}}$ is determined up to twists by quadratic characters $W_F \to \mu_2(\C)$.
    We also have a corresponding result for $\GSp_{2n}$.
\end{proposition}

\begin{proof}
     We realize the group $\GSpin_{2n}$ as the quotient $\Spin_{2n} \times \GL_1 / \Delta \mu_2$, where $\Delta \colon \mu_2 \hookrightarrow Z(\Spin_{2n}) \times \GL_1$ is the diagonal embedding. 
     First, the morphism $\GSpin_{2n} \to \SO_{2n} \times \GL_1$ has kernel $(\mu_2 \times \mu_2) / \Delta \mu_2$, where $\Delta \colon \mu_2 \hookrightarrow \mu_2 \times \mu_2$ is the diagonal embedding. 
     Thus, the $L$-morphism $(\phi^{\SO_{2n}}, \chi) \colon L_F \to {}^L \SO_{2n} \times {}^L \GL_1$ lifts to $\GSpin_{2n}(\C) \times W_F$ if and only if the sum of the elements in $H^2(W_F, \mu_2)$ defined by $\phi^{\SO_{2n}}$ and $\chi$ is trivial in the group $H^2(W_F, (\mu_2 \times \mu_2) / \Delta \mu_2)$.
     
     We have a canonical isomorphism $H^2(W_F, \mu_2) \xrightarrow{\sim} \{\pm 1\}$; see \cite{Deligne1976-localconstant}*{(4.1), (4.2)}.
     The element in $\{\pm 1\}$ which corresponds to $\phi^{\SO_{2n}}$ is equal to the root number $\epsilon(\frac{1}{2}, \iota^{\GL_{2n}}_{\SO_{2n}} \circ \phi^{\SO_{2n}})$ by \cite{Deligne1976-localconstant}*{Proposition (5.2)} and the additivity of $\epsilon$-factors.
     The element in $\{\pm 1\}$ corresponding to $\chi$ is equal to $\chi(-1)$ by the local class field theory.
     Our hypothesis implies that the product of these elements is trivial.
     This proves the first statement.
     The second statement follows from the fact that the kernel of the projection $\GSpin_{2n} \to \SO_{2n} \times \GL_1$ is isomorphic to $\mu_2$. 
     The proof for the group $\GSp_{2n}$ is the same.
\end{proof}

\begin{definition} 
   We define the set $\Phi(\GSO_{2n})_{/ \mathrm{GO}_{2n}}$ analogously to the definition of $\Phi(\SO_{2n})_{/ \O_{2n}}$.
   The set $\Pi_{\temp}(\GSO_{2n})_{/ \mathrm{GO}_{2n}}$ is also defined analogously to the definition of $\Pi_{\temp}(\SO_{2n})_{/ \O_{2n}}$.

   Then, we have the maps $\iota^{\SO_{2n}}_{\GSO_{2n}} \colon \Phi(\GSO_{2n}) \to \Phi(\SO_{2n})_{/ \O_{2n}}$ and $\nu \colon \Phi(\GSO_{2n}) \to \Phi(\GL_1)$ which are induced by the $L$-homomorphisms $\iota^{\SO_{2n}}_{\GSO_{2n}}$ and $\nu$.
   They factor through the quotient set $\Phi(\GSO_{2n})_{/ \mathrm{GO}_{2n}}$.
   We define an equivalence relation $\sim_{\O_{2n} \times \GL_1}$ on the set $\Phi(\GSO_{2n})_{/ \mathrm{GO}_{2n}}$ by $\phi_1 \sim_{\O_{2n} \times \GL_1} \phi_2$ if and only if $\iota^{\SO_{2n}}_{\GSO_{2n}} \circ \phi_1 = \iota^{\SO_{2n}}_{\GSO_{2n}} \circ \phi_2$ and $\nu \circ \phi_1 = \nu \circ \phi_2$.
   We denote the quotient set defined by this equivalence relation by 
   \begin{align*}
    \Phi(\GSO_{2n})_{/ \O_{2n} \times \GL_1}.
   \end{align*}
   We denote the subset of tempered elements by $\Phi_{\temp}(\GSO_{2n})_{/ \O_{2n} \times \GL_1}$.
\end{definition}

\begin{lemma}\label{lemma:identification-parameter-GSO-and-SO-times-GL1}
    The natural map 
    \begin{align*}
        \Phi_{\temp}(\GSO_{2n})_{/ \O_{2n} \times \GL_1} 
        \to   
        \Phi_{\temp}(\SO_{2n})_{/ \O_{2n}} \times \Phi_{\temp}(\GL_1)
    \end{align*}
    is injective. 
    The image coincides with the subset consisting of the pairs $(\phi, \chi) \in \Phi_{\temp}(\SO_{2n})_{/ \O_{2n}} \times \Phi_{\temp}(\GL_1)$ which satisfies the hypothesis of \cref{proposition:amalgamation-to-GSO2n}.
    We have the corresponding result for $\GSp_{2n}$.
\end{lemma}

\begin{proof}
    This follows from \cref{proposition:amalgamation-to-GSO2n}.
\end{proof}

\begin{remark}
    We will identify the set $\Phi_{\temp}(\GSO_{2n})_{/ \O_{2n} \times \GL_1} $ with the subset of $\Phi_{\temp}(\SO_{2n})_{/ \O_{2n}} \times \Phi_{\temp}(\GL_1)$ via the natural map above.
    Let $[\phi]_{\O_{2n} \times \GL_1}$ denote the class of $\phi \in \Phi(\GSO_{2n})$.
\end{remark}

\begin{theorem}[\cite{Xu18local}*{Theorem 4.6, Lemma 6.2, Corollary 4.2}]\label{theorem:local-Langlands-correspondence-GSO2n}
    There exists a map 
    \begin{align*}
        {\cL_{\GSO_{2n}}}_{/ \O_{2n} \times \GL_1} 
        \colon   
        \Pi_{\temp}(\GSO_{2n})_{/ \mathrm{GO}_{2n}}
        \to
        \Phi_{\temp}(\GSO_{2n})_{/ \O_{2n} \times \GL_1}.
    \end{align*}
    This satisfies the analogues of $(1), (2)$, and $(4)$ in \cref{conjecture:local-Langlands-conjecture}, after replacing
    \begin{itemize}
    \item 
    The set $\Pi_{\temp}(G)$ by $\Pi_{\temp}(\GSO_{2n})_{/ \mathrm{GO}_{2n}}$.
    \item
    The set $\Phi_{\temp}(G)$ by $\Phi_{\temp}(\GSO_{2n})_{/ \O_{2n} \times \GL_1}$.
    \item     
    The character $\tr \pi$ by the average of characters $\frac{1}{2} (\tr \pi + \tr \pi^{\theta})$ for any $\theta \in \O_{2n}(F) \setminus \SO_{2n}(F)$.
    \end{itemize}
    Furthermore, the map has the following properties.
 
\begin{itemize}
    \item[(a)]
    If $\pi$ is a tempered irreducible representation of $\GSO_{2n}(F)$ with central character $\chi_{\pi}$ and if $\pi'$ is an irreducible component of the restriction of $\pi$ to $\SO_{2n}(F)$, then we have 
    \begin{align*}
        {\cL_{\GSO_{2n}}}_{/ \O_{2n} \times \GL_1}(\pi) = ({\cL_{\SO_{2n}}}_{/ \O_{2n}}(\pi'), \chi_{\pi})
    \end{align*}
    via the identification in \cref{lemma:identification-parameter-GSO-and-SO-times-GL1}.

    \item[(b)]
    We denote the inverse image of $(\phi, \chi) \in \Phi_{\temp}(\SO_{2n})_{/ \O_{2n}} \times \Phi_{\temp}(\GL_1)$ by $\overline{\Pi}_{\phi, \chi}$, via the identification in \cref{lemma:identification-parameter-GSO-and-SO-times-GL1}.
    Then, on Properties $(3), (5)$ of \cref{conjecture:local-Langlands-conjecture}, we have the following result: 
    There exists a subset $\overline{\Pi}^{X} \subset \overline{\Pi}_{\phi, \chi}$ such that   
    \begin{enumerate}
        \item[(3)] 
        For any $L$-parameter $\widetilde{\phi} \in \Phi(\GSO_{2n})$ which maps to $(\phi, \chi)$, we have a bijection
        \begin{align*}
            \overline{\Pi}^{X} \xrightarrow{\sim} \Irr(\pi_0(S_{\widetilde{\phi}}/Z(\widehat{\GSO_{2n}}))).
        \end{align*}
        Also, any element $[\pi]_{\GO_{2n}} \in \overline{\Pi}_{\phi, \chi}$ is contained in the set $\overline{\Pi}^{X} \otimes (\omega \circ \lambda)$ for some quadratic character $\omega \colon F^{\times} \to \mu_2(\C)$.
 
        \item[(5)] 
        The character defined by $\overline{\Pi}^{X}$ is stable.
    \end{enumerate}

    \item[(c)]
    Let $\pi$ be an irreducible representation such that $[\pi]_{\GO_{2n}} \in \overline{\Pi}^{X}$. 
    Let $\omega \colon F^{\times} \to \mu_{2}(\C)$ be a quadratic character.
    We take any $L$-parameter $\widetilde{\phi} \in \Phi(\GSO_{2n})$ which maps to $(\phi, \chi)$.
    Also, let $\Sigma$ be either $1 \in \GSO_{2n}(F)$ or any $\theta \in \O_{2n}(F) \setminus \SO_{2n}(F)$. 
    Dually, we set $\widehat{\Sigma} = 1$ if $\Sigma = 1$ and $\widehat{\Sigma} \in \mathrm{O}_{2n}(\C) \setminus \SO_{2n}(\C)$ if $\Sigma = \theta$. 
    The following two conditions are equivalent.
    \begin{enumerate}
        \item[(i)]   
        We have $\pi^{\Sigma} \otimes (\omega \circ \lambda) = \pi$. 
        \item[(ii)]    
        We have $(\Ad(\widehat{\Sigma}) \circ \widetilde{\phi}) \otimes \omega = \widetilde{\phi}$ in $\Phi_{\temp}(\GSO_{2n})$.
    \end{enumerate}

    \item[(d)]
    Finally, we write the set of the quadratic character $\omega \colon F^{\times} \to \mu_2(\C)$ satisfying the equivalent conditions $(i)$ and $(ii)$ of $(c)$ for some element $\Sigma$, as $\Hom(W_F, \mu_2(\C))_{\widetilde{\phi}}$.
    Then, we have 
    \begin{align*}
        \overline{\Pi}_{\phi, \chi} 
        =   
        \coprod_{\omega \in \Hom(W_F, \mu_2(\C))/\Hom(W_F, \mu_2(\C))_{\widetilde{\phi}}} \overline{\Pi}^{X} \otimes (\omega \circ \lambda),
    \end{align*} 
    by $(b)(3)$ and $(c)$.

\end{itemize}
\end{theorem}

\begin{definition}
    Let $\pi$ be a tempered irreducible representation of $\GSO_{2n}(F)$.
    We call an $L$-parameter $\phi^{\GSO_{2n}}_{\pi}$ for $\GSO_{2n}(F)$ which maps to the parameter ${\cL_{\GSO_{2n}}}_{/ \O_{2n} \times \GL_1}([\pi]_{\GSO_{2n}})$ an $L$-parameter of $\pi$.
    We use similar notation for representations of its Levi subgroups.
\end{definition}

\begin{remark}\label{remark:GSO2n-invariant-independence-of-choice}
    Let $M$ be a standard Levi subgroup of $\GSO_{2n}$ and let $\sigma$ be a discrete series representation of $M$.
    As in \cref{remark:invariant-in-formal-degree-even-orthogonal}, the adjoint $\gamma$-factors $\gamma(s, \phi^M_{\sigma}, \Ad_{M}, \psi)$, $\gamma(s, \phi^M_{\sigma}, \Ad_{G/M}, \psi)$ and the number $\abs{S^{\natural}_{\sigma}}$ are canonically defined by $\sigma$ and independent of the choice of $\phi^M_{\sigma}$.
    We denote them by $\gamma(s, \sigma, \Ad_{M}, \psi)$ and so on.
\end{remark}

\begin{definition}
    We call a twist $\overline{\Pi}^X \otimes (\omega \circ \lambda)$ of the set $\overline{\Pi}^X$ by a quadratic character $\omega$ an $L$-packet of Xu or a Xu $L$-packet for $\GSO_{2n}(F)$.
\end{definition}

The set of equivalence classes of $L$-parameters $\Phi_{\temp}(\GSp_{2n})_{/ \Sp_{2n} \times \GL_1}$ is defined analogously to the set $\Phi_{\temp}(\GSO_{2n})_{/ \O_{2n} \times \GL_1}$.

\begin{theorem}[\cite{Xu18local}*{Theorem 4.6, Lemma 6.2, Corollary 4.2}]
\label{theorem:local-Langlands-correspondence-GSp2n}

    There exists a map 
    \begin{align*}
        {\cL_{\GSp_{2n}}}_{/ \Sp_{2n} \times \GL_1} \colon \Pi_{\temp}(\GSp_{2n}) \to \Phi_{\temp}(\GSp_{2n})_{/ \Sp_{2n} \times \GL_1}.
    \end{align*}
    This map satisfies the analogues of $(1), (2), (4)$ in \cref{conjecture:local-Langlands-conjecture}, after replacing 
    \begin{itemize}
        \item The set $\Pi_{\temp}(G)$ by $\Pi_{\temp}(\GSp_{2n})$, and
        \item 
        The set $\Phi_{\temp}(G)$ by $\Phi_{\temp}(\GSp_{2n})_{/ \Sp_{2n} \times \GL_1}$.
    \end{itemize}
    
    Also, the map has the following properties.
\begin{itemize}
    \item[(a)]
    If $\pi$ is a tempered irreducible representation of $\GSp_{2n}(F)$ with central character $\chi_{\pi}$ and if $\pi'$ is an irreducible component of the restriction of $\pi$ to $\Sp_{2n}(F)$, then we have 
    \begin{align*}
        {\cL_{\GSp_{2n}}}_{/ \Sp_{2n} \times \GL_1}(\pi) = ({\cL_{\Sp_{2n}}}(\pi'), \chi_{\pi})
    \end{align*}
    via the identification in \cref{lemma:identification-parameter-GSO-and-SO-times-GL1}.
    
    \item[(b)]
    We denote the inverse image of $(\phi, \chi) \in \Phi_{\temp}(\Sp_{2n}) \times \Phi_{\temp}(\GL_1)$ by $\Pi_{\phi, \chi}$ via the identification in \cref{lemma:identification-parameter-GSO-and-SO-times-GL1}. 
    Then, on Properties $(3), (5)$ of \cref{conjecture:local-Langlands-conjecture}, we have the following result:
    There exists a subset $\Pi^{X} \subset \Pi_{\phi, \chi}$ such that 
    \begin{enumerate}
        \item[(3)]
        For any $L$-parameter $\widetilde{\phi} \in \Phi_{\temp}(\GSp_{2n})$ which maps to $(\phi, \chi)$, there exists a bijection
        \begin{align*}
            \Pi^{X} \xrightarrow{\sim} \Irr(\pi_0(S_{\widetilde{\phi}}/Z(\widehat{\GSp_{2n}}))).
        \end{align*}   
        \item[(5)] 
        The character defined by the set $\Pi^X$ is stable.
    \end{enumerate}
    
    \item[(c)]
    Let $\pi$ be an irreducible representation such that $\pi \in \Pi^{X}$. 
    Let $\omega \colon F^{\times} \to \mu_{2}(\C)$ be a quadratic character.
    We take an $L$-parameter $\widetilde{\phi} \in \Phi_{\temp}(\GSp_{2n})$ which maps to $(\phi, \chi)$.
    Then, the following two conditions are equivalent.
    \begin{enumerate}
        \item[(i)]   
        We have $\pi \otimes (\omega \circ \lambda) = \pi$. 
        \item[(ii)]    
        We have $\widetilde{\phi} \otimes \omega = \widetilde{\phi}$ in $\Phi_{\temp}(\GSp_{2n})$.
    \end{enumerate}
  
    \item[(d)]
    Finally, let $\Hom(W_F, \mu_2(\C))_{\widetilde{\phi}}$ denote the set of the quadratic character $\omega \colon F^{\times} \to \mu_2(\C)$ that satisfy the equivalent conditions $(i)$ and $(ii)$ of $(c)$.
    Then, we have 
    \begin{align*}
        {\Pi}_{\phi, \chi} 
        =   
        \coprod_{\omega \in \Hom(W_F, \mu_2(\C))/\Hom(W_F, \mu_2(\C))_{\widetilde{\phi}}} 
        {\Pi}^{X} \otimes (\omega \circ \lambda),
    \end{align*} 
    by $(b)(3)$ and $(c)$.
\end{itemize}
\end{theorem}

\begin{definition}
    We call a twist $\Pi^X \otimes (\omega \circ \lambda)$ of a set $\Pi^X$ by a quadratic character $\omega$, an $L$-packet of Xu or a Xu $L$-packet for $\GSp_{2n}(F)$. 
\end{definition}

\begin{remark}[Xu's local Langlands correspondence for $\PGSO_{2n}$ and $\PGSp_{2n}$]
    If $\chi=1$, we obtain the local Langlands correspondence for $\PGSO_{2n}$ and $\PGSp_{2n}$.
    In this case, the corresponding root numbers of the standard $L$-parameters are trivial.
\end{remark}

\subsection{Refinement of Xu's result for $\GSO_4$ and $\GSO_6$}

In this section, building on results by Ramakrishnan \cite{Ramakrishnan2000-Modularity-Rankin-Selberg}, Kim \cite{Kim2003-functoriality-exterior-square} and Henniart \cite{Henniart2009exterior}, we prove a refinement of Xu's local Langlands correspondence.
The results in this subsection seem to be well-known to experts, but for the sake of completeness, we include the proof.

\begin{remark} \label{remark:accidental-isomorphism}
    Related to $A_1 \times A_1 = D_2$ and $A_3 = D_3$, we have the following accidental isomorphisms given in \cite{GanTakeda2011LLC-for-GSp4}*{3}.
    The notation is as in \Cref{subsection:embedding-tori}.
    \begin{enumerate}
    \item    
    We consider the embedding 
    \begin{align*}
        \GL_1 \hookrightarrow \GL_2 \times \GL_2 \colon z \mapsto (z, z^{-1}).
    \end{align*}
    Then, we have an isomorphism 
    \begin{align}
        ac_{4} \colon (\GL_2 \times \GL_2) / \GL_{1}^{(1, -1)}  \xrightarrow{\sim} \GSO_{4}. 
        \label{eq:GSO4-accidental}
    \end{align} 
    \item   
    We consider the embedding 
    \begin{align*}
        \GL_1 \hookrightarrow \GL_4 \times \GL_1 \colon z \mapsto (z, z^{-2}).
    \end{align*} 
    Then, we have an isomorphism 
    \begin{align}
        ac_{6} \colon (\GL_4 \times \GL_1) / \GL_{1}^{(1, -2)}  \xrightarrow{\sim} \GSO_{6}.
        \label{eq:GSO6-accidental}
    \end{align}
    \end{enumerate}

    In the first case, we have the isomorphism of dual groups:
    \begin{align*}
        \{ (g_1, g_2) \in \GL_2 \times \GL_2 \mid \det g_1 = \det g_2 \} \xrightarrow{\sim} \GSpin_4.
    \end{align*}
    The connected centers of both groups are isomorphic to $\GL_1$ and we have the induced morphism 
    \begin{align*}
        \{ (g_1, g_2) \in \GL_2 \times \GL_2 \mid \det g_1 = \det g_2 \} \twoheadrightarrow \GSpin_4/\GL_1 \xrightarrow{\sim} \SO_4 \hookrightarrow \GL_4.
    \end{align*}
    This morphism induces the four-dimensional representation of the group $\{ (g_1, g_2) \in \GL_2 \times \GL_2 \mid \det g_1 = \det g_2 \}$ isomorphic to $\Std \otimes \Std^{\vee}$; see \cite{GanTakeda2011LLC-for-GSp4}*{3}.

    Similarly, in the second case, the induced map 
    \begin{align*}
        \{ (g_1, z) \in \GL_4 \times \GL_1 \mid \det g_1 = z^2 \} \xrightarrow{\sim} \GSpin_6 \twoheadrightarrow \SO_6 \hookrightarrow \GL_6
    \end{align*}
    can be identified with 
    \begin{align*}
         (g_1, z) \mapsto z^{-1} \cdot \wedge^2 g_1.
    \end{align*}
    See \cite{GanTakeda2011LLC-for-GSp4}*{4}.
\end{remark}

\begin{proposition}\label{proposition:refinement-GSO4}
    The following diagram commutes.
    % https://q.uiver.app/#q=WzAsNCxbMCwwLCJcXFBpX3tcXG1hdGhybXt0ZW1wfX0oXFxtYXRocm17R1NPXzR9KSJdLFsyLDAsIlxcUGlfe1xcbWF0aHJte3RlbXB9fSgoXFxtYXRocm17R0x9XzJcXHRpbWVzIFxcbWF0aHJte0dMfV8yKS9cXG1hdGhybXtHTH1fMSkiXSxbMiwxLCJcXFBoaV97XFxtYXRocm17dGVtcH19KChcXG1hdGhybXtHTH1fMlxcdGltZXMgXFxtYXRocm17R0x9XzIpL1xcbWF0aHJte0dMfV8xKSJdLFswLDEsIlxcUGhpKFxcbWF0aHJte1NPfV80KV97LyBcXG1hdGhybXtPfV80fSJdLFswLDEsImFjXzQiXSxbMSwyLCJcXG1hdGhjYWx7TH1fe1xcbWF0aHJte0dMfV8yIFxcdGltZXMgXFxtYXRocm17R0x9XzJ9Il0sWzAsMywie1xcbWF0aGNhbHtMfV97XFxtYXRocm17U099XzR9fV97LyBcXG1hdGhybXtPfV80fSIsMl0sWzIsMywiXFxvdGltZXMiXV0=
    \[\begin{tikzcd}
	{\Pi_{\mathrm{disc}}(\mathrm{GSO_4})} && {\Pi_{\mathrm{disc}}((\mathrm{GL}_2\times \mathrm{GL}_2)/\mathrm{GL}_1)} \\
	{\Phi(\mathrm{SO}_4)_{/ \mathrm{O}_4}} && {\Phi_{\mathrm{disc}}((\mathrm{GL}_2\times \mathrm{GL}_2)/\mathrm{GL}_1)}
	\arrow["{ac_4}", from=1-1, to=1-3]
	\arrow[from=1-1, to=2-1]
	\arrow["{\mathcal{L}_{\mathrm{GL}_2 \times \mathrm{GL}_2}}", from=1-3, to=2-3]
	\arrow[" \mathrm{Std} \otimes \mathrm{Std}^{\vee}", from=2-3, to=2-1]
    \end{tikzcd}\]
    The left vertical arrow is given by the 1st component of the morphism \cref{lemma:identification-parameter-GSO-and-SO-times-GL1} and the map of \cref{theorem:local-Langlands-correspondence-GSO2n}.
\end{proposition}

\begin{remark}\label{remark:unramified-place}
    For (not necessarily tempered) unramified representations, we can give a similar statement to the proposition.
    This statement follows from \cref{remark:accidental-isomorphism}.
\end{remark}

\begin{proof}
     We take a discrete series representation $\sigma$ of $\GSO_4(F)$ and globalize it to a cuspidal automorphic representation $\Sigma$ of $\GSO_4(\A_k)$ for a number field $k$ such that $F$ is a completion of $k$, see \cite{Takanashi2025-Sauvageot}*{Corollary 8.3, Theorem 8.5}.
     Then, the pullback $\Sigma \circ ac_4$ is a cuspidal automorphic representation $\Pi_1 \boxtimes \Pi_2$ of $\GL_2(\A_k) \times \GL_2(\A_k)$.
     We can form the Rankin-Selberg convolution $\Pi_1 \times \Pi_2^{\vee}$ by \cite{Ramakrishnan2000-Modularity-Rankin-Selberg}*{Theorem M}, which is an isobaric sum of cuspidal automorphic representations of general linear groups.
     On the other hand, the restriction of $\Sigma$ contains a cuspidal automorphic representation of $\SO_4(\A_k)$ by \cite{LabesseSchwermer2019}*{Theorem 1.1.1}.
     Thus, we can take the functorial lift $\Pi$ to $\GL_4(\A_k)$ by Arthur's theorem, which is an isobaric sum of discrete automorphic representations of general linear groups.
     By \cref{remark:unramified-place} and the strong multiplicity one theorem, we have $\Pi = \Pi_1 \times \Pi_2^{\vee}$.
     From the local component at the place above $F$, we obtain the proposition.
\end{proof}

\begin{corollary}\label{corollary:local-Langlands-correspondence-GSO4}
    \cref{conjecture:local-Langlands-conjecture} holds for $\GSO_4(F)$.
    Also, the correspondence gives a refinement of Xu's local Langlands correspondence.
\end{corollary}

There is a similar result for $ac_6$ in \cref{remark:accidental-isomorphism}, which can be proved by applying the main theorems in \cite{Kim2003-functoriality-exterior-square} and \cite{Henniart2009exterior} on the exterior square functoriality, instead of \cite{Ramakrishnan2000-Modularity-Rankin-Selberg}.
We omit the proof, which is analogous to the proof for \cref{proposition:refinement-GSO4}.

\begin{proposition}\label{proposition:refinement-GSO6}
    The following diagram commutes.
    % https://q.uiver.app/#q=WzAsNCxbMCwwLCJcXFBpX3tcXG1hdGhybXt0ZW1wfX0oXFxtYXRocm17R1NPXzR9KSJdLFsyLDAsIlxcUGlfe1xcbWF0aHJte3RlbXB9fSgoXFxtYXRocm17R0x9XzJcXHRpbWVzIFxcbWF0aHJte0dMfV8yKS9cXG1hdGhybXtHTH1fMSkiXSxbMiwxLCJcXFBoaV97XFxtYXRocm17dGVtcH19KChcXG1hdGhybXtHTH1fMlxcdGltZXMgXFxtYXRocm17R0x9XzIpL1xcbWF0aHJte0dMfV8xKSJdLFswLDEsIlxcUGhpKFxcbWF0aHJte1NPfV80KV97LyBcXG1hdGhybXtPfV80fSJdLFswLDEsImFjXzQiXSxbMSwyLCJcXG1hdGhjYWx7TH1fe1xcbWF0aHJte0dMfV8yIFxcdGltZXMgXFxtYXRocm17R0x9XzJ9Il0sWzAsMywie1xcbWF0aGNhbHtMfV97XFxtYXRocm17U099XzR9fV97LyBcXG1hdGhybXtPfV80fSIsMl0sWzIsMywiXFxvdGltZXMiXV0=
    \[\begin{tikzcd}
	{\Pi_{\mathrm{disc}}(\mathrm{GSO_6})} && {\Pi_{\mathrm{disc}}((\mathrm{GL}_4 \times \mathrm{GL}_1)/\mathrm{GL}_1)} \\
	{\Phi(\mathrm{SO}_6)_{/ \mathrm{O}_6}} && {\Phi_{\mathrm{disc}}((\mathrm{GL}_4 \times \mathrm{GL}_1)/\mathrm{GL}_1)}
	\arrow["{ac_6}", from=1-1, to=1-3]
	\arrow[from=1-1, to=2-1]
	\arrow["{\mathcal{L}_{\mathrm{GL}_4  \times \mathrm{GL}_1}}", from=1-3, to=2-3]
	\arrow["\wedge^2 \otimes \Std^{-1}", from=2-3, to=2-1]
    \end{tikzcd}\]
    The left vertical arrow is given by the 1st component of the morphism \cref{lemma:identification-parameter-GSO-and-SO-times-GL1} and the map of \cref{theorem:local-Langlands-correspondence-GSO2n}.
\end{proposition}

\begin{corollary}\label{corollary:local-Langlands-correspondence-GSO6}
    \cref{conjecture:local-Langlands-conjecture} holds for $\GSO_6(F)$.
    Also, the correspondence gives a refinement of Xu's correspondence.
\end{corollary}

\subsection{Gan--Savin's local Langlands correspondence for $G_2$}

We state the local Langlands correspondence for $G_2$ proved by Gan and Savin.

\begin{remark}
    The Levi subgroups of $G_2$ are isomorphic either to $\GL_{2}$ or $\GL_1 \times \GL_1$; the local Langlands correspondence for these groups is already known.
\end{remark}

\begin{theorem}[\cite{GS23LLC}*{1.1}]\label{theorem:local-Langlands-correspondence-G2}
    \cref{conjecture:local-Langlands-conjecture} holds for $G_2(F)$, except for Property $(3)$ when the residual characteristic of $F$ is equal to $p=3$ and $\phi$ is discrete, and for Property $(5)$.
    In that situation, we only have an injection 
    \begin{align*}
        \iota_{\fw} \colon \Pi_{[\phi]_G}(G) \hookrightarrow \Irr(\pi_0(S_{\phi})).
    \end{align*}
\end{theorem}

\begin{remark}
    We do not use Property $(3)$ when the residual characteristic of $F$ is equal to $p=3$ and $\phi$ is discrete, nor Property $(5)$ for $G_2(F)$ in this paper.
\end{remark}

\subsection{Relation with local theta correspondence}

In this paper, we utilize the similitude theta correspondence between $\PGO_{8}$ (or $\PGSO_{8}$) and $\PGSp_{6}$ and the exceptional theta correspondence for $G_2$ and $\PGSp_{6}$. 
We do not use the technical properties of these correspondences; therefore we summarize their main properties axiomatically.

Let $\iota^{\GSO_8}_{\GSp_6}$ denote the $L$-embedding ${}^L \GSp_6 \to {}^L \GSO_8$.
Also, let $\iota^{\GSp_6}_{G_2}$ denote the $L$-embedding ${}^L G_2 \to {}^L \GSp_6$ which factors through ${}^L \PGSp_6$.

\begin{theorem}[The similitude theta correspondence for $\PGSO_8$ and $\PGSp_6$]\label{theorem:similitude-theta}
    There exists a map 
    \begin{align*}
        \theta^{\PGSO_8}_{\PGSp_6} \colon \Pi_{\temp}(\PGSp_{6}) \to \Pi_{\temp}(\PGSO_8) \cup \{ 0 \}.
    \end{align*}
    The map satisfies the following properties.
    
    \begin{itemize}
    \item[(a)] 
    The image is equal to the disjoint union of $\Pi_{\temp}(\PGSO_8)_{\theta} =  \{ \pi \in \Pi_{\temp}(\PGSO_8) \mid \gamma(0, \iota^{\GL_{8}}_{\SO_8} \circ \phi^{\SO_8}_{\pi}, \psi) = 0  \}$ and $\{ 0 \}$.

    \item[(b)]
    The map is compatible with the local Langlands correspondence in the following sense:
    For any tempered irreducible representation $\sigma$ of $\PGSp_6(F)$ such that $\theta^{\PGSO_8}_{\PGSp_6}(\sigma) \neq 0$, we have 
    \begin{align*}
           {\cL_{\mathrm{GSO}_8}}_{/ \O_8 \times \GL_1} 
           (\theta^{\PGSO_8}_{\mathrm{PGSp}_6}(\sigma)) 
           =
           [
           \iota^{\GSO_8}_{\GSp_6}
           \circ  
           {\cL_{\GSp_6}}_{/ \Sp_6 \times \GL_1} (\sigma)]_{\O_8 \times \GL_1}.
    \end{align*}

    \item[(c)]
    If we set 
    \begin{align*}
        \Pi_{\temp}(\PGSp_6)_{\theta} = \{ \sigma \in \Pi_{\temp}(\PGSp_6) \mid  \theta^{\PGSO_8}_{\PGSp_6}(\sigma) \neq 0 \},
    \end{align*}
    then the map $\theta^{\PGSO_8}_{\PGSp_6}$ induces a bijection 
    \begin{align*}
        \theta^{\PGSO_8}_{\PGSp_6} \colon 
        \Pi_{\temp}(\PGSp_{6})_{\theta} \xrightarrow{\sim} \Pi_{\temp}(\PGSO_8)_{\theta}.
    \end{align*}

    \item[(d)] 
    The map $\theta^{\PGSO_8}_{\PGSp_6}$ maps generic representations to generic representations.
    The inverse image of any generic representation is generic.
    \end{itemize}
\end{theorem}

\begin{proof}
    The existence of the map follows from \cite{GanTakeda2016-Howeduality}*{Theorem 1.2} and \cite{bakic2023similitudeexceptionalthetacorrespondences}*{Proposition 5.1}.
    The statements (a) and (b) follow from \cite{bakic2023similitudeexceptionalthetacorrespondences}*{Lemma 5.3} and \cite{AtobeGan2017-localLanglandsorthogonal}*{Theorem 4.7}.
    The statement (c) follows from \cite{bakic2023similitudeexceptionalthetacorrespondences}*{Proposition 5.1} and \cite{GanTakeda2016-Howeduality}*{Theorem 1.2}.
    The statement (d) on the generic representations follows from \cite{bakic2023similitudeexceptionalthetacorrespondences}*{Lemma 5.3} and 
    \cite{AtobeGan2017-localLanglandsorthogonal}*{Theorem 4.7}.
\end{proof}

\begin{theorem}[The exceptional theta correspondence for $\PGSp_6$ and $G_2$]\label{theorem:exceptional-theta}
    There exists a map 
    \begin{align*}
        \theta^{\PGSp_6}_{G_2} \colon \Pi_{\temp}(G_2) \to \Pi_{\temp}(\PGSp_6) \cup \{ 0 \}.
    \end{align*}
    The map satisfies the following properties.

    \begin{itemize}
    \item[(a)]
    The image is equal to the disjoint union of $\Pi_{\temp}(\PGSp_6)_{\theta} =  \{ \pi \in \Pi_{\temp}(\PGSp_6) \mid \gamma(0, \pi, \Spin, \psi) = 0 \}$ and $\{ 0 \}$.
    Here, the local $\gamma$-factor $\gamma(s, \pi, \Spin, \psi)$ is the $\gamma$-factor for the spin representation of ${}^L \PGSp_6$, which is defined in \cite{GS23LLC}*{12.3}.

    \item[(b)]
    The map is compatible with the local Langlands correspondence in the following sense:
    For any tempered irreducible representation $\sigma$ of $G_2(F)$ such that $\theta^{\PGSp_6}_{G_2}(\sigma) \neq 0$, we have 
    \begin{align*}
           {\cL_{\mathrm{GSp_6}}}_{/ \Sp_6 \times \GL_1} 
           (\theta^{\PGSp_6}_{G_2}(\sigma)) 
           =
           [
           \iota^{\GSp_6}_{G_2}
           \circ  
           {\cL_{G_2}} (\sigma)
           ]_{\Sp_6 \times \GL_1}.
    \end{align*}

    \item[(c)] 
    Define
    \begin{align*}
        \Pi_{\temp}(G_2)_{\theta} = \{ \sigma \in \Pi_{\temp}(G_2) \mid  \theta^{\PGSp_6}_{G_2}(\sigma) \neq 0 \},
    \end{align*}
    then the map $\theta^{\PGSp_6}_{G_2}$ induces a bijection 
    \begin{align*}
        \theta^{\PGSp_6}_{G_2} \colon 
        \Pi_{\temp}(G_2)_{\theta} \xrightarrow{\sim} \Pi_{\temp}(\PGSp_6)_{\theta}.
    \end{align*}

    \item[(d)] 
    The map $\theta^{\PGSp_6}_{G_2}$ maps generic representations to generic representations.
    The inverse image of any generic representation is generic.
    \end{itemize}
\end{theorem}

\begin{proof}
    The existence of the map follows from \cite{GS23theta}*{Theorem 1.1, Theorem 1.2}.
    The statement (a) follows from \cite{GrossSavin1998-Motives}*{4, Proposition 1.10}, \cite{GS23LLC}*{Main theorem (vi), Theorem 12.7}, \cref{theorem:local-Langlands-correspondence-GSp2n} and \cref{theorem:local-Langlands-correspondence-G2}.
    Also, the statement (b) follows from \cite{GS23theta}*{15.1}.
    The statement (c) follows from \cite{GS23theta}*{Theorem 12.4}.
    The statement (d) follows from \cite{GS23theta}*{Corollary 11.2, Lemma 12.1}.
\end{proof}

\subsection{Automorphisms of $\PGSO_8$}

In this section, let $F$ be a local or global field.

In this subsection, we fix an $F$-splitting $\mathbf{spl}_{\PGSO_8}$ of $\PGSO_8(F)$. Also, we fix a $\Gamma_F$-splitting $\mathbf{spl}_{\Spin_8}$ of $\Spin_8(\C)$.

\begin{definition}\label{definition:automorphism-of-D4}
    We use the following notation for the automorphisms of the diagram of type $D_4$. 
    % https://q.uiver.app/#q=WzAsNCxbMSwwLCJcXGJldGFfMSJdLFsxLDEsIlxcYmV0YV8yIl0sWzAsMiwiXFxiZXRhXzMiXSxbMiwyLCJcXGJldGFfNCJdLFswLDEsIiIsMCx7InN0eWxlIjp7ImhlYWQiOnsibmFtZSI6Im5vbmUifX19XSxbMSwyLCIiLDAseyJzdHlsZSI6eyJoZWFkIjp7Im5hbWUiOiJub25lIn19fV0sWzEsMywiIiwwLHsic3R5bGUiOnsiaGVhZCI6eyJuYW1lIjoibm9uZSJ9fX1dLFswLDIsInNfNiIsMix7ImN1cnZlIjozLCJzdHlsZSI6eyJ0YWlsIjp7Im5hbWUiOiJhcnJvd2hlYWQifX19XSxbMCwzLCJzXzEiLDAseyJjdXJ2ZSI6LTMsInN0eWxlIjp7InRhaWwiOnsibmFtZSI6ImFycm93aGVhZCJ9fX1dXQ==
    
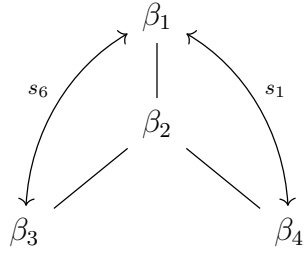
\[\begin{tikzcd}
	& {\beta_1} \\
	& {\beta_2} \\
	{\beta_3} && {\beta_4}
	\arrow[no head, from=1-2, to=2-2]
	\arrow["{s_6}"', curve={height=18pt}, tail reversed, from=1-2, to=3-1]
	\arrow["{s_1}", curve={height=-18pt}, tail reversed, from=1-2, to=3-3]
	\arrow[no head, from=2-2, to=3-1]
	\arrow[no head, from=2-2, to=3-3]
    \end{tikzcd}\] 
    These induce automorphisms of $\PGSO_8(F)$ and $\Spin_8(\C)$, by using the splitting $\mathbf{spl}_{\PGSO_8}$ and $\mathbf{spl}_{\widehat{\PGSO_8}}$.
    We denote the corresponding automorphisms of $G$ and $\widehat{G}$ (or ${}^L G$) by $s_i$ and $\widehat{s_i}$ for $i=1, 6$.  
    In general, any element $a$ of the automorphism group $\Aut(\mathrm{Dyn}(D_4)) \xrightarrow{\sim} S_3$ of the Dynkin diagram above gives a unique automorphism of $\PGSO_8(F)$ by the isomorphism theorem.
    We denote it also by $a$, and we use similar notation for the dual group.
\end{definition}

\subsection{Kret--Shin parameters for $\PGSO_8$} 

Following Gan--Savin \cite{GS23LLC}*{Theorem 12.3}, we construct certain refinements of Xu's $L$-parameters for $\PGSO_8$, which are called Kret--Shin parameters.  

We note that for a tempered irreducible representation $\pi$ of $\GSO_8(F)$, the $L$-parameter $\Std_{\GSO_8} \circ \phi^{\GSO_8}_{\pi}$ is independent of the choice of $\phi^{\GSO_8}_{\pi}$ as this is defined by the composition 
\begin{align*}
    \Pi(\GSO_8) \to \Phi(\GSO_8)_{/ \O_8 \times \GL_1} \to \Phi(\SO_8)_{/ \O_8} \to \Phi(\GL_8).
\end{align*}
We write the image as $\Std_{\GSO_8} \circ \phi^{\GSO_8}_{\pi}$ or $\phi^{\GL_{8}}_{\pi}$.
We use a similar notation for $\PGSO_8$.

\begin{proposition}\label{proposition:Kret--Shin-parameters}
    Let $\pi$ be an irreducible discrete series representation of $\PGSO_8(F)$.
    Then, there exists a discrete $L$-homomorphism $\phi \colon L_F \to {}^L \PGSO_{8}$, such that 
    \begin{align*}
        \begin{cases}
              \Std_{\PGSO_8} \circ \phi = \Std_{\PGSO_8} \circ \phi^{\PGSO_8}_{\pi}, \\ 
              \Std_{\PGSO_8} \circ \widehat{s_1} \circ \phi = \Std_{\PGSO_8} \circ \phi^{\PGSO_8}_{{}^{s_1}\pi}.
        \end{cases}
    \end{align*}
    Also, we can take the same $L$-parameters $\phi$ for all elements in the $L$-packet $\overline{\Pi}^{X}$ containing $\pi$.
\end{proposition}

\begin{proof}
    We can proceed as the proof of \cite{GS23LLC}*{Proposition 5.1}, by applying \cite{KretShin-24-general-orthogonal}*{Theorem A} instead of \cite{KretShin23-general-symplectic}*{Theorem A} and the globalization result \cite{Shi12}*{Theorem 5.8} as \cite{GS23LLC}*{Lemma 12.4}.
\end{proof}

\begin{definition}
    We call an $L$-parameter $\phi$ satisfying the condition of \cref{proposition:Kret--Shin-parameters} a Kret--Shin parameter of $\pi$.
    For any tempered irreducible representation $\pi$, we fix a Kret--Shin parameter $\phi^{KS}_{\pi}$ of $\pi$.
\end{definition}

\begin{remark}\label{remark:Kret--Shin-not-discrete-tempered}
    The proof of \cref{proposition:Kret--Shin-parameters} (or an application of \cref{corollary:local-Langlands-correspondence-GSO4} and \cref{corollary:local-Langlands-correspondence-GSO6}) shows that if $\pi$ is not a discrete series representation, we may and do take a canonical $L$-parameter for $\pi$ as a Kret--Shin parameter for $\pi$.
\end{remark}

\begin{lemma}\label{lemma:injectivity-G_2-to-GL7}
    The canonical map 
    \begin{align*}
        \Phi(G_2) \to \Phi(\GL_7)
    \end{align*}
    induced by the standard representation of $G_2(\C)$ is injective.
    Hence, the morphism
    \begin{align*}
        \Phi(G_2) \to \Phi(\GL_7) \to \Phi(\GL_8)
    \end{align*}
    obtained by adding the trivial representation is also injective.
\end{lemma}

\begin{proof}
    This follows from \cite{Larsen1994-conjugacy}*{Proposition 2.8} and \cite{Griess1995-G2}*{1.1 Theorem 1}. 
\end{proof}

\begin{proposition}\label{proposition:construction-of-distinguished-Xu-packet}
    The following statements hold true:
    \begin{enumerate}
        \item 
        Let $\pi$ be a tempered irreducible representation of $\PGSO_8(F)$ such that
        \begin{align}
            \begin{cases}
                \gamma(0, \Std_{\PGSO_8} \circ \phi^{\PGSO_8}_{\pi}, \psi) = 0, \\
                \gamma(0, \Std_{\PGSO_8} \circ \phi^{\PGSO_8}_{{}^{s_1}\pi}, \psi) = 0.
            \end{cases}
            \label{eq:condition-distinguished-packet}
        \end{align}
        Then, there exists a unique tempered irreducible representation $\pi^{G_2}$ of $G_2(F)$, such that 
        \begin{align*}
            \theta^{\PGSO_8}_{\PGSp_6} \circ \theta^{\PGSp_6}_{G_2}(\pi^{G_2}) = \pi.
        \end{align*}
        \item    
        By $(1)$ of this proposition, we have the map   
        \begin{align*}
            \{ \pi \in \Pi_{\temp}(\PGSO_8(F)) \mid \text{$\pi$ satisfies (3.3)} \}  
            \to  
            \Phi_{\temp}(G_2),
        \end{align*}   
        given by 
        \begin{align*}
           \cL_{\mathrm{dist}} \colon \pi \mapsto \phi^{G_2}_{\pi^{G_2}}.
        \end{align*}
        Then, this map is surjective and the fiber of the image of a representation $\pi$ is the $L$-packet of Xu containing the representation $\pi$.
        If $\pi$ is a discrete series representation, then $\pi$ has a unique Kret--Shin parameter given by $\iota^{\PGSO_8}_{G_2} \circ \phi^{G_2}_{\pi^{G_2}}$.
    \end{enumerate}
\end{proposition}

\begin{proof}
    We will prove the first statement.
    Let $\pi$ be a tempered irreducible representation of $\PGSO_8(F)$, such that 
        \begin{align*}
            \begin{cases}
                \gamma(0, \Std_{\PGSO_8} \circ \phi^{\PGSO_8}_{\pi}, \psi) = 0, \\
                \gamma(0, \Std_{\PGSO_8} \circ \phi^{\PGSO_8}_{{}^{s_1}\pi}, \psi) = 0.
            \end{cases}
        \end{align*}
    Then, by the first equation and the property (a) in \cref{theorem:similitude-theta}, the representation $\pi$ is the similitude theta lift of a uniquely-determined tempered irreducible representation $\pi'$ of $\PGSp_6(F)$.
    In this case, \cite{GS23LLC}*{12.3} implies that we have 
    \begin{align*}
        \gamma(s, \Std_{\PGSO_8} \circ \phi^{\PGSO_8}_{{}^{s_1}\pi}, \psi) 
        = 
        \gamma(s, \pi', \Spin, \psi).
    \end{align*}
    Thus, the property (a) in \cref{theorem:exceptional-theta} implies that there exists a unique irreducible tempered representation $\pi^{G_2}$ of $G_2(F)$, whose exceptional theta lift is equal to $\pi'$.
    Thus, we obtained the first statement.

    We now start the proof of surjectivity of the map $\cL_{\dist}$.
    The map $\cL_{\dist}$ is well-defined by the proof of the first statement.
    By the properties (d) in \cref{theorem:similitude-theta} and \cref{theorem:exceptional-theta} and \cite{GS23LLC}*{3.4, Proposition 8.1}, the map is surjective.

    We next prove that the fiber is a union of Xu $L$-packets.
    We first note that each discrete series $L$-packet of a proper Levi subgroup of $\PGSO_8(F)$ is a singleton, since such a Levi subgroup is a quotient of a product of general linear groups by a central split torus.
    Thus, for non-discrete-series representations, this follows from the fact that the parameters $\Std_{\GSO_8} \circ \phi^{\GSO_8}_{\pi}$ and $\Std_{\GSO_8} \circ \phi^{\GSO_8}_{{}^{s_1}\pi}$ are constant on each Xu $L$-packet by the compatibility of the local Langlands correspondence for $\GSO_8$ with parabolic induction in \cref{theorem:local-Langlands-correspondence-GSO2n}.
    For discrete series representations, since we can take a common Kret--Shin parameter for representations in each Xu $L$-packet of $\PGSO_8(F)$, the fiber of $\phi^{G_2}_{\pi^{G_2}}$ is a union of Xu $L$-packets of $\PGSO_8(F)$.

    We finally prove that the fiber consists of a single Xu $L$-packet.
    Let $\phi^{G_2}$ be the image of $\pi$.
    Thus, if two Xu $L$-packets $\Pi_1, \Pi_2$ are mapped to the parameter $\phi^{G_2}$ by $\cL_{\dist}$, then each of two $L$-packets contains a unique generic representation $\pi_1$ and $\pi_2$ respectively.
    This follows from \cite{Var17}*{Corollary 6.16} and \cref{theorem:local-Langlands-correspondence-GSO2n} (d).
    These two generic representations map to the same generic representation in the packet $\Pi^{G_2}_{\phi^{G_2}}$ of $G_2(F)$ by the property (d) in \cref{theorem:similitude-theta} and \cref{theorem:exceptional-theta}.
    Also, the Kret--Shin parameter is given by $\iota^{\PGSO_8}_{G_2} \circ \phi^{G_2}$ by the property $(b)$ in \cref{theorem:similitude-theta} and \cref{theorem:exceptional-theta}, and \cref{lemma:injectivity-G_2-to-GL7}.
    This concludes the proof of the second statement.
\end{proof}

\begin{definition}
    The $L$-packet of Xu which maps to $\phi^{G_2}$ by the map $\cL_{\mathrm{dist}}$ in \cref{proposition:construction-of-distinguished-Xu-packet} is called the distinguished Xu packet associated with $\phi^{G_2}_{\pi^{G_2}}$.
    We denote it by $\Pi_{\phi^{G_2}}^{X} = \Pi_{\phi^{G_2}}^{X}(\PGSO_{8})$.
\end{definition}

\begin{remark} \label{remark:uniqueness-distinguished-L-packet}
    From the proof, we can also see the following statement, whose proof is given below: 
    If a Xu $L$-packet $\Pi^{X}$ and a quadratic twist $\Pi^X \otimes \omega$ are distinguished, then we have $\Pi^X = \Pi^X \otimes \omega$.
    
    We give the proof. 
    Let $\phi_1^{G_2}$ and $\phi_2^{G_2}$ be the $L$-parameters for $G_2$ corresponding to $\Pi^{X}$ and $\Pi^X \otimes \omega$.
    From the assumption, we have 
    \begin{align*}
        \iota^{\GL_8}_{G_2} \circ \phi_1^{G_2}
        = 
        \iota^{\GL_8}_{G_2} \circ \phi_2^{G_2}.
    \end{align*}
    Then, the result follows from \cref{lemma:injectivity-G_2-to-GL7}.
\end{remark}

\section{The Hiraga--Ichino--Ikeda conjecture}\label{section:4}

Let $F$ be a non-archimedean local field of characteristic $0$.
Let $\psi$ be a nontrivial additive character of $F$.

\subsection{Normalization of measures}

\subsubsection{Regularized $\gamma$-factors}

For an $L$-parameter $\phi \colon L_F \to {}^L G$ and a finite-dimensional complex representation $r \colon {}^L G \to \GL_{\C}(V)$, we have the Artin $\gamma$-factor 
\begin{align*}
    \gamma(s, \phi, r, \psi).
\end{align*}

We define its regularized value at $s=0$ by 
\begin{align*}
    \gamma^*(0, \phi, r, \psi) = \lim_{s \to 0+} \zeta_F(s)^{n_{\phi}} \gamma(s, \phi, r, \psi),
\end{align*}
where the number $n_{\phi}$ is the order of zero of the function $\gamma(s, \phi, r, \psi)$ at $s=0$.
For the function $\zeta_F(s)$, see \cref{section:2}.

\subsubsection{Measures on $F$-varieties}

Let $\rd_{\psi}x$ be the self-dual measure on $\A^n(F) = F^n$ with respect to $\psi$.

Let $\bX$ be a smooth irreducible algebraic variety over $F$ of dimension $n$.
Let $\omega \in \Gamma(\bX, \wedge^{n}\Omega_{\bX/F})$ be a nowhere vanishing top differential form on $\bX$ over $F$.
Locally, there exists an \'etale coordinate $U \to \A^n$ for $\bX$ and we can write the differential form $\omega$ as 
\begin{align*}
    \omega = f dx_1 \wedge \ldots \wedge dx_n,
\end{align*}
on $U$.
Here, $x_1, \dots, x_n$ are the pullback of standard coordinates of $\A^n$ and $f$ is a regular function on $U$.
In this situation, we set 
\begin{align*}
    \abs{\omega}_{U, \psi} = \abs{f}_F \rd_{\psi}x_1 \ldots \rd_{\psi}x_n.
\end{align*}
By the formula for change of coordinates, the local measures $\abs{\omega}_{U, \psi}$ can be glued to the global measure on $X = \bX(F)$.
We denote by $\abs{\omega}_{\psi}$ this measure on $\bX(F)$.

\subsubsection{Measures on the reductive groups}

Let $\bG$ be a reductive group over $F$.
We follow the method of \cite{GrossGan1999Haarmeasure} to construct a normalized Haar measure on $G$.

Let $\bG_{\spl}$ be the split reductive group over $\Z$ which is isomorphic to $\bG$ over $\overline{F}$.
Then, there exists a generator $\omega_{\spl} \in \Gamma(\bG_{\spl}, \wedge^{\dim \bG} \Omega_{\bG_{\spl}/\Z})^{\bG}$ of the space of the invariant top differential forms on $\bG_{\spl}$, which is determined up to $\pm 1$.
We take an $\overline{F}$-isomorphism $\iota \colon \bG_{\overline{F}} \xrightarrow{\sim} (\bG_{\spl})_{\overline{F}}$.
Pulling back $(\omega_{\spl})_{\overline{F}}$, we obtain the differential form $\iota^{*} (\omega_{\spl})_{\overline{F}}$.
Then, there exists a scalar $\lambda \in \overline{F}^{\times}$ such that
\begin{align*}
    \omega = \lambda \iota^{*} (\omega_{\spl})_{\overline{F}}
\end{align*}
is defined over $F$. 
Thus, we obtain the Haar measure 
\begin{align*}
    \rd_{\psi} g = \abs{\lambda}^{-1}_F \abs{\omega}_F
\end{align*}
on $G = \bG(F)$.
It is easy to check that this measure depends only on $\psi$, and does not depend on the choices above.

\subsubsection{Measures on the unipotent groups}\label{subsubsection:measures-unipotent}

Via the exponential map and the character $\psi$ of $F$, we can obtain the measure on the unipotent group $U$, for example, by choosing a basis of $\Lie(U)$. 
Let $\mathbf{spl}_G = (T, B, \{X_{\alpha}\}_{\alpha})$ be an $F$-splitting of $G$.
Let $P = MU$ be a parabolic subgroup of $G$ and let $\overline{P} = M \overline{U}$ be its opposite.
We fix measures $du, d\overline{u}$ on the groups $U$ and $\overline{U}$, such that 
\begin{align*}
    \int_{U \times M \times \overline{U}} f(um\overline{u}) \delta_{M}(m) \rd u \rd_{\psi}m \rd \overline{u} 
    = 
    \int_{G} f(g) \rd_{\psi} g
\end{align*}
for all smooth functions $f \in \Cc(G)$.
We assume similar normalizations for Levi subgroups of $G$.
See \cite{Wal03}*{p. 240 (2), p.241 line 12}.
In the paper, Waldspurger defined the scalar $c(G | M) > 0$, but it is equal to $1$ under our normalization above.
It appears in the statement of the Plancherel formula; see \cite{Wal03}*{Th\'eor\`eme VIII.1.1}.

Let us assume that $G$ is split.
The results \cite{KalethaPrasad-Bruhat-Tits}*{Proposition 2.9.11, Lemma 2.9.14, Proposition 2.11.2, Proposition 2.11.3, Proposition 2.11.9} imply that the splitting $\mathbf{spl}_G$ gives the canonical measure $\rd_{\psi} u$ on $U$ and the canonical measure $\rd_{\psi} \overline{u}$ on $\overline{U}$, which satisfy the condition above.
Also, these results imply that the measure $\rd_{\psi}u$ is equal to the product measure of the measures $\rd_{\psi} u_{\alpha}$ on the root group $U_{\alpha}$ for each root $\alpha \in \Phi(A_T, U)$, i.e.,
\begin{align*}
      \rd_{\psi} u = \prod_{\alpha \in \Phi(A_T, U)} \rd_{\psi}u_{\alpha}.
\end{align*}
A similar result for roots for $A_M$ in $U$ holds.

\subsubsection{Measures on tempered spectra}\label{subsubsection:normalization-tempered-spectra}

Let $G$ be a reductive group over $F$.
In this section, we define and normalize the measure on the space of induced tempered representations of $G$.

We first assume that $G$ is a split torus. 
We use the notation in \cref{subsection:unramified-characters}.
We take the measure $\rd \chi$ on the space $\cX_{\unit}(G)$ such that the volume of $\cX_{\unit}(G)$ is equal to $1$; see \cite{Wal03}*{p.239 line 33}.
We take the measure on the space $i\fa^*_G$ such that the map in \Cref{eq:character-to-unr-character} is locally measure preserving.
Also, we have the exact sequence 
\begin{align*}
    1 \to G^1 \to G \to \Z^{\dim G} \to 1, 
\end{align*}
where the group $G^1$ is the maximal compact subgroup.
This gives the exact sequence of the Pontryagin dual 
\begin{align*}
    1 \to \cX_{\unit}(G) \to G^{\cD} \to (G^1)^{\cD} \to 1. 
\end{align*}
Here ${\cdot}^{\cD}$ denotes the Pontryagin dual.
We take a unique measure $d{\chi}$ on the group $G^{\cD}$ such that the local homeomorphism $\cX_{\unit}(G) \to G^{\cD}$ is locally measure preserving.  
Note that if we write the dual measure of $\rd_{\psi} g$ as
\begin{align*}
    \rd_{\psi} \chi,
\end{align*}
then, we have 
\begin{align*}
    d{\chi} = \gamma^*(0, \mathbf{1}, \psi)^{-\dim G}  \rd_{\psi} \chi.
\end{align*}

Let $G$ be a general quasi-split reductive group.
On the set of discrete series representations $\Pi_{\disc}(G)$, we take a unique topology such that the map taking the central characters on $A_G$
\begin{align*}
    \Pi_{\disc}(G) \to A_G^{\cD} \colon \pi \to {\chi_{\pi}}_{\restriction{A_G}}
\end{align*}
is a local homeomorphism.
Similarly, we can define a unique measure $\rd \pi$ on the set of discrete series representations $\Pi_{\disc}(G)$, such that the map above is locally measure preserving.
It is easy to see that if we fix a representation $\pi \in \Pi_{\disc}(G)$, then the map 
\begin{align*}
    i\fa_G^{*} \to \Pi_{\disc}(G) \colon \lambda \to \pi \otimes \chi_{\lambda}
\end{align*}
is a local homeomorphism and locally measure preserving; see \cite{Wal03}*{p.302 line 1}.
The image of this map is called a connected component of $\Pi_{\disc}(G)$.

We define the space of induced tempered representations $\Temp_{\ind}(G)$ and the canonical measure on the space. 
We have fixed a minimal parabolic pair $(P_0, A_0)$ of $G$.
Let $M$ be a standard Levi subgroup of $G$ and let $\sigma$ be a discrete series representation of $M$. 
By the theorem of Harish-Chandra \cite{Wal03}, we have the following result. 

\begin{proposition}[\cite{Wal03}*{Proposition III.4.1}]\label{proposition:induced-tempered-representations-HC}
    Let $G$ be a reductive group over $F$.
    Let $P = MN$ be a parabolic subgroup of $G$.
    \begin{enumerate}
        \item The isomorphism class of the representation $i^G_P(\sigma)$ does not depend on the choice of parabolic subgroup $\cP^G(M)$ of $G$. 
        We denote it by $i^G_{M}(\sigma)$.
        \item Two such induced representations $i^G_M(\sigma)$ and $i^G_{M'}(\sigma')$ have the same irreducible constituent if and only if there exists an element $g \in G$ such that $gMg^{-1}=M'$ and ${}^g \sigma = \sigma'$. 
        In this case, we have ${}^g i^G_{M}(\sigma) = i^G_{M'}(\sigma')$.
    \end{enumerate}
\end{proposition}

\begin{definition} 
    We define the set $\Temp_{\ind}(G)$ as the quotient set of the set of pairs $(M, \sigma)$, where $M$ is a standard Levi subgroup of $G$ and $\sigma$ is a discrete series representation of $M$, by the equivalence relation defined by the conjugation action of $G$.
    We can see that this set is bijective to the set of tempered induced representations of $G$ by applying \cref{proposition:induced-tempered-representations-HC}.
\end{definition}

Let $\cL^G_0$ denote a set of representatives of conjugacy classes of the standard Levi subgroups of $G$.
Then, we have a bijection 
\begin{align*}
      \coprod_{{M} \in \cL^G_0} \Pi_{\disc}(M)/W(G, M) \xrightarrow{\sim} \Temp_{\ind}(G) \colon [(M, \sigma)] \mapsto \pi = i^G_M(\sigma).
\end{align*}
On the left-hand side, the space $\Pi_{\disc}(M)/W(G, M)$ has a structure of a countable union of quotients of a compact torus by a finite group action. 
This gives the set $\Temp_{\ind}(G)$ a topology and the space of smooth functions.
Also, we take the quotient measure $\frac{1}{\abs{W(G, M)}}(i^G_M)_{*} d\sigma$ on the space $\Pi_{\disc}(M)/W(G, M)$, where $d\sigma$ is the canonical measure on the space $\Pi_{\disc}(M)$.
We denote this measure by $\rd \pi$.
We call the image of a connected component of $\Pi_{\disc}(M)$ by the above map, a connected component of $\Temp_{\ind}(G)$.

\subsection{The Harish-Chandra Plancherel formula}

\begin{theorem}[\cite{Wal03}*{Th\'eor\`eme VIII.1.1}]
    Let $G$ be a reductive group over $F$.
    \begin{enumerate}
        \item For any function $f \in \Cc(G)$, the function 
        \begin{align*}
            \pi \in \Temp_{\ind}(G) \to \C \colon \pi \mapsto \tr \pi(f) = \tr(\pi)(f \rd_{\psi}g)
         \end{align*}
         is smooth and compactly supported.
        \item
         There exists a smooth function $\mu^G_{ \psi}(\pi) \colon \Temp_{\ind}(G) \to \R_{>0}$ such that for any function $f \in \Cc(G)$, we have  
        \begin{align}
            f(1) = \int_{\Temp_{\ind}(G)} \tr \pi(f^{\vee}) \mu^G_{\psi}(\pi) \rd \pi. \label{eq:Plancherel-formula}
        \end{align}
    \end{enumerate}  
\end{theorem}

\begin{definition}
    We call the function $\mu^G_{\psi}(\pi)$ the Plancherel density function normalized with respect to $\psi$.
\end{definition}

\begin{proposition}[\cite{SakellaridisVenkatesh2017}*{Proposition 6.1.1}]\label{proposition:sakellaridis-venkatesh}
    The measurable function which satisfies \cref{eq:Plancherel-formula} for all functions $f \in \Cc(G)$ is unique up to coincidence at a conull set with respect to $\rd \pi$.
    In particular, any measurable function which satisfies \cref{eq:Plancherel-formula} for all functions $f \in \Cc(G)$ is non-negative almost everywhere with respect to $\rd \pi$.
\end{proposition}

\subsection{Fixing central characters}

For a unitary character $\chi_{A_G} \colon A_G \to \bS^{1}$, let $\Temp_{\ind}(G)_{\chi_{A_G}}$ denote the subset of $\Temp_{\ind}(G)$ consisting of representations with central character $\chi_{A_G}$ on $A_G$.
We can define the topology and measure on this space by using the space $\cX_{\unit}(A_M/A_G)$ for standard Levi subgroups $M$.
We denote this measure by $\rd_{\chi_{A_G}}\pi$.
We denote the restriction of the function $\mu^G_{\psi}(\pi)$ to this space by the same symbol.
The Fourier inversion theorem for $A_G$ and the Harish-Chandra Plancherel theorem imply the following proposition.

\begin{proposition}\label{proposition:Plancherel-formula-central-character}
    \begin{enumerate}    
    \item
    For any function $f \in \Cc(G, \chi_{A_G})$, we have 
    \begin{align*}
        f(1) 
        = 
        \gamma^{*}(0, \mathbf{1}, \psi)^{-\dim(A_G)}
        \int_{\pi \in \Temp_{\ind}(G)_{\chi_{A_G}}} \tr \pi (f^{\vee}) \mu^G_{\psi}(\pi) \rd_{\chi_{A_G}} \pi.
    \end{align*} 
    \item     
    For any function $f \in \Cc(G)$, we have 
    \begin{align*}
        f(1) = 
        \int_{\chi_{A_G} \in \widehat{A_G}} 
        (\int_{\pi \in \Temp_{\ind}(G)_{\chi_{A_G}}}      
        \tr \pi (f^{\vee}_{\chi_{A_G}})  
        \mu^G_{\psi}(\pi) 
        \rd_{\chi_{A_G}} \pi) \,
        \rd\chi_{A_G}, 
    \end{align*}
    where 
    \begin{align*}
        f_{\chi_{A_G}}(g) = \int_{z \in A_G} f(zg) \chi^{-1}_{A_G}(z) \rd_{\psi}z.
    \end{align*}
    \end{enumerate} 
\end{proposition}

\begin{remark}\label{remark:reduction-fixing-central-character}
    \cref{proposition:Plancherel-formula-central-character} implies that we can compute the density function $\mu^G_{\psi}(\pi)$ by fixing the central character $\chi_{A_G}$ of $A_G$.
\end{remark}

\subsection{The statement of the conjecture}

\begin{definition}
    Let $G$ be a reductive group over $F$.
    Let $\pi$ be a discrete series representation of $G$.
    Let $\langle \cdot, \cdot  \rangle$ be the canonical $G$-invariant pairing between $\pi$ and $\pi^{\vee}$.
    Then, the formal degree $d^G_{\psi}(\pi)$ is an element in $\R_{>0}$ which satisfies 
    \begin{align*}
        \int_{g \in G/A_G} 
        \langle \pi(g)v, v^{\vee} \rangle 
        \langle w, \pi^{\vee}(g) w^{\vee} \rangle 
        \rd_{\psi}g 
        = 
        d^G_{\psi}(\pi)^{-1}
        \langle v, w^{\vee} \rangle 
        \langle w, v^{\vee} \rangle.
    \end{align*}
\end{definition}

\begin{theorem}[\cite{Wal03}*{Th\'eor\`eme VIII.1.1}]
    Let $G$ be a reductive group over $F$.
    Then, the Plancherel density function $\mu^G_{\psi}$ can be written as the product of two factors: For any $\pi = i^G_M(\sigma) \in \Temp_{\ind}(G)$, the formal degree $d_{\psi}(\sigma)$ of $\sigma$ and the inverse of the Harish-Chandra $j$-function ${j_{\psi}(\sigma)} =  J_{P \vert \overline{P}}(\sigma)J_{\overline{P} \vert P}(\sigma)$ (see \cite{Wal03}*{IV.3} for details on $j$-function), where $J_{\overline{P} \vert P}(\sigma)$ is the unnormalized intertwining operator defined in \cref{subsection:normalizing-factors}.
\end{theorem}

\begin{conjecture}[The Hiraga--Ichino--Ikeda conjecture, the formal degree conjecture \cite{HII08}]\label{conjecture:formal-degree-conjecture}
    Let $G$ be a quasi-split reductive group over $F$. 
    Assume that \cref{conjecture:local-Langlands-conjecture} holds for $G$.
    \begin{enumerate}
        \item  
        For any Levi subgroup $M$ of $G$ and any discrete series representation $\sigma$ of $M$, the formal degree $d^{M}_{\psi}(\sigma)$ is equal to 
        \begin{align*}
            \langle \sigma, 1 \rangle 
            \frac{\abs{\gamma^*(0, \phi^M_{\sigma}, \Ad_{M}, \psi)}}{\abs{S^{\natural}_{\sigma}}}.
        \end{align*}   
        \item    
        For almost all $\pi = i^G_M(\sigma) \in \Temp_{\ind}(G)$, the $j$-function $j_{\psi}(\sigma)$ is equal to 
        \begin{align*}
             \abs{\gamma^*(0, \phi^M_{\sigma}, \Ad_{G/M}, \psi)}^{-1}.
        \end{align*}
    \end{enumerate}
\end{conjecture}

\begin{remark}\label{remark:reduction-to-G}
    We take a representation $\pi \in \Temp_{\ind}(G)$.
    By the property $(1)$ of \cref{conjecture:local-Langlands-conjecture}, the $L$-parameters of irreducible components $\pi'$ of $\pi$ are the same, and we call it the $L$-parameter of $\pi$ and we write it as $\phi^G_{\pi}$.
    Combining the two statements in the conjecture, the density function $\mu^G_{\psi}(\pi)$ is conjecturally equal to 
    \begin{align}
        \mu^G_{\HII, \psi}(\pi)
        =
        \langle \sigma, 1 \rangle 
        \frac{\abs{\gamma^*(0, \phi^G_{\pi}, \Ad_{G}, \psi)}}{\abs{S^{\natural}_{\sigma}}}
    \end{align}
    for almost all $\pi = i^G_P(\sigma) \in \Temp_{\ind}(G)$.
    Conversely, assuming this statement and $(1)$ of \cref{conjecture:formal-degree-conjecture}, we obtain Part $(2)$ of the conjecture.
\end{remark}

\begin{definition}
    Assume that \cref{conjecture:local-Langlands-conjecture} for $G$ holds.
    We define the equivalence relation $\sim_{st}$ on the set $\Temp_{\ind}(G)$ by $\pi \sim_{st} \pi'$ if and only if $\phi^G_{\pi} = \phi^G_{\pi'}$.
    We denote the push-forward measure of the measure $\rd \pi$ on $\Temp_{\ind}(G)$ to $\Temp_{\ind}(G)/\sim_{st}$ by $\rd \phi$.
\end{definition}

\begin{remark}\label{remark:stability-SO_2n}
    Considering \cref{theorem:local-Langlands-correspondence-SO2n}, we define the equivalence relation $\sim_{st,w}$ on $\Temp_{\ind}(\SO_{2n}(F))$ by 
    \begin{align*}
        \pi \sim_{st,w} \pi' \text{ if and only if } [\phi^{\SO_{2n}}_{\pi}]_{\O_{2n}} = [\phi^{\SO_{2n}}_{\pi'}]_{\O_{2n}},
    \end{align*}
    and similarly for $\GSO_{2n}(F)$ by using the transfer to $\GL_{2n}(F) \times \GL_1(F)$.
    Here, the notation $w$ denotes ``weak''.
    We have the canonical map 
    \begin{align}
        \Temp_{\ind}(\GSO_{2n}(F))/\sim_{st, w} 
        \to   
        \Temp_{\ind}(\SO_{2n}(F))/\sim_{st, w}, 
        \label{eq:transfer-by-restriction}
    \end{align}
    induced by the restriction of representations of $\GSO_{2n}(F)$.
\end{remark}

\subsection{The conjecture for classical groups}

\begin{theorem}\label{theorem:formal-degree-conjecture-GLn}
    \cref{conjecture:formal-degree-conjecture} for $\GL_n(F)$ holds.
\end{theorem}

\begin{proof}
    For example, see \cite{Beuzart-Plessis2021-Plancherel-GLnE-GLnF}*{Proposition 2.132}, which applies results by Silberger--Zink \cite{SilbergerZink1996-formaldegree} and Shahidi \cite{Shahidi1984-Fouriertransform}.
\end{proof}

Note that for even split orthogonal groups, we have seen that the invariants in the statement of \cref{conjecture:formal-degree-conjecture} can be defined by applying \cref{theorem:local-Langlands-correspondence-SO2n}; see \cref{remark:invariant-in-formal-degree-even-orthogonal}.

\begin{theorem}[\cite{beuzartplessis2025hiragaichinoikedaconjectureformaldegrees}*{Theorem 1.1}]\label{theorem:formal-degree-conjecture-SO2n}
    \cref{conjecture:formal-degree-conjecture} holds for any even split orthogonal group.
\end{theorem}

In fact, the conjecture has been proved uniformly for all classical groups \cite{beuzartplessis2025hiragaichinoikedaconjectureformaldegrees}.
The odd special orthogonal and unitary cases were already known before; see \cite{ILM17} and \cite{Beuzart-Plessis2021-Plancherel-GLnE-GLnF}.

\subsection{Other known cases}
 
We give some other known results on the Hiraga--Ichino--Ikeda conjecture here.

\begin{itemize}
    \item 
    Unipotent representations for unramified groups; \cite{FengOpdamSolleveld2022-unipotent}.
    \item 
    Simple supercuspidal representations; \cite{Mie2021-simplesupercuspidal}.
    \item
    For non-singular supercuspidal representations; \cite{Schwein2024-regular}, \cite{Ohara2023-non-singular}.
\end{itemize}

We also have some relation between the formal degree of a discrete series representation and the formal degree of its cuspidal support; see \cite{Wang2025-formal-degree}.

\section{The conjecture for $\GSO_{2n}$ and $\PGSO_{2n}$}\label{section:5}

In this section, we prove the Hiraga--Ichino--Ikeda conjecture for $\GSO_{2n}(F)$ and $\PGSO_{2n}(F)$, following the method of \cite{GI14}*{Lemma 13.2}.
See also \cref{remark:GSO2n-invariant-independence-of-choice}.

\subsection{Galois cohomology and abelian Galois cohomology}

In this section, we recall the results on the abelian Galois cohomology proved by Borovoi \cite{Borovoi1998abeliangalois}.

\begin{definition}[Abelian Galois cohomology]
    Let $\bG$ be a reductive group over $F$.
    Let $Z(\bG_{sc})$ be the center of $\bG_{sc}$, where $\bG_{sc}$ is the simply connected cover of the derived group of $\bG$.
    The $i$-th abelian Galois cohomology $H^{i}_{ab}(F, G)$ of $G$ is the $i$-th Galois hypercohomology $\H^{i}(F, Z^{\bullet})$ of the complex of $\Gamma_F$-modules
    \begin{align*}
       Z^{\bullet} = (\cdots \to 0 \to Z(\bG_{sc}) \to Z(\bG) \to 0 \to \cdots),
    \end{align*}
    where $Z(\bG_{sc})$ is placed at degree $-1$ and $Z(\bG)$ is placed at degree $0$.
\end{definition}

\begin{definition}[Fundamental group]
    Let $\bG$ be a reductive group over $F$.
    Let $(\bB, \bT)$ be a Borel pair of $\bG_{\overline{F}}$.
    Consider the group
    \begin{align*}
         X_*(\bT)/X_*(\bT_{sc}).
    \end{align*}
    Here, the group $\bT_{sc}$ is the maximal torus of $\bG_{sc}$ corresponding to $\bT \subset \bG$.
    The construction $(\bB, \bT) \mapsto X_*(\bT)/X_*(\bT_{sc})$ forms a projective system, and we define 
    \begin{align*}
        \pi_1(G) = \lim_{(\bB, \bT)}  X_*(\bT)/X_*(\bT_{sc}).
    \end{align*}
    This group becomes a $\Gamma_F$-module by the natural action of $\Gamma_F$.
\end{definition}

\begin{lemma}\label{lemma:facts-in-abelian-Galois-cohomology}
    Let $\bG, \bG'$ be reductive groups over $F$.
    Assume that we have a faithfully flat morphism $p \colon \bG' \to \bG$ such that it induces an isogeny map $\bG'_{\der} \to \bG_{\der}$.
    \begin{enumerate}
        \item (The Kottwitz isomorphism)      
        We have a functorial isomorphism
        \begin{align*}
            H^1(F, G) \xrightarrow{\sim} H^1_{ab}(F, G) \xrightarrow{\sim} \pi_0(Z(\widehat{G})^{\Gamma_F})^{\cD},
        \end{align*}
        and
        \begin{align*}
            H^2_{ab}(F, G) &\xrightarrow{\sim} (\pi_1(G)_{\Gamma_F})_{\mathrm{fr}} \otimes \Q/\Z,
        \end{align*}
        where $(\pi_1(G)_{\Gamma_F})_{\mathrm{fr}} = \pi_1(G)_{\Gamma_F}/(\pi_1(G)_{\Gamma_F})_{\mathrm{tors}}$ is the free part of the $\Gamma_F$-coinvariant of the fundamental group of $\bG$.
        \item    
        We have an exact sequence 
        \begin{align}
            1 \to H^1(F, G) \to H^1(F, G^{\natural}) \to H^2(F, A_G) \to H^2_{ab}(F, G).
        \end{align}
        \item    
        We have an exact sequence 
        \begin{multline}
            1 \to G/p(G') \to H^1(F, \Ker(p)) \to H^1(F, G') \to H^1(F, G) \\
            \to H^2(F, \Ker(p)) \to H^2_{ab}(F, G') \to H^2_{ab}(F, G).
        \end{multline}
        \item     
        Assume that $\Ker(p)$ is finite. 
        Then, we have 
        \begin{align*}
            \abs{ \# \Ker(p)(\overline{F}) }_F =  \frac{\abs{H^0(F, \Ker(p))}\abs{H^2(F, \Ker(p))}}{\abs{H^1(F, \Ker(p))}}.
        \end{align*}
        Here, We write $\abs{\Ker(p)(\overline{F})}$ as $\# \Ker(p)(\overline{F})$ to avoid confusion.
    \end{enumerate}
\end{lemma}

\begin{proof}
    The statement $(1)$ follows from \cite{Borovoi1998abeliangalois}*{Corollary 5.4.1}.
    The statement $(3)$ follows from the long exact sequence associated with the short exact sequence of 2-term complexes
    \begin{align*}
        1 \to (1 \to \Ker(p)) \to (Z(G'_{sc}) \to Z(G')) \to (Z(G_{sc}) \to Z(G)) \to 1.
    \end{align*}
    The statement $(2)$ follows from $(3)$ by taking $\bG$ as $\bG'$ and $\bG^{\natural}$ as $\bG$.
    The statement $(4)$ is well known as the local Euler characteristic formula.
\end{proof}

\begin{lemma}\label{lemma:canonical-isom-of-quotient-Levi}
    Let $\bG, \bG'$ be reductive groups over $F$.
    Assume that we have a faithfully flat morphism $p \colon \bG' \to \bG$, which induces the isogeny map $\bG'_{\der} \to \bG_{\der}$.
    Let $\bM$ be a Levi subgroup of $\bG$ and $\bM'$ be the inverse image in $\bG'$ under the map $p$.
    Then, we have an isomorphism
    \begin{align*}
        M/p(M') \xrightarrow{\sim} G/p(G')
    \end{align*}
    which is induced by the inclusion $\bM \hookrightarrow \bG$.
\end{lemma}

\begin{proof}
    Let $\bP$ be a parabolic subgroup of $\bG$ with a Levi factor $\bM$.
    There exists the following exact sequences (of pointed sets):
    \begin{align*}
        1 \to \bP(F) \to \bG(F) \to (\bG/\bP)(F) \to 1,
    \end{align*}
    and 
    \begin{align*}
        1 \to \bP'(F) \to \bG'(F) \to (\bG'/\bP')(F) \to 1.
    \end{align*}
    The last quotient sets are isomorphic.

    Let $g$ be an element of $G$.
    Then, the surjectivity in the second exact sequence implies that there exists an element $g' \in G'$ such that the image of $g$ and $p(g')$ in $(\bG/\bP)(F)$ are the same.
    By replacing $g'$ by the element $p(g')^{-1} g$, we may assume that $g$ is contained in $P$.
    Furthermore, we have the Levi decompositions $P = MN$ and $P' = M'N'$.
    Since $p(N') = N$, we may assume that $g$ is contained in $M$.
    Thus, we have shown the surjectivity of the map $M/p(M') \to G/p(G')$.
    The injectivity is clear and we have completed the proof.
\end{proof}

\begin{corollary}
    We have a canonical isomorphism 
    \begin{align*}
        W(G', M') \xrightarrow{\sim} W(G, M).
    \end{align*}
\end{corollary}

\begin{remark}
    Let $\bG, \bG'$ be reductive groups over $F$.
    Assume that we have a faithfully flat morphism $p \colon \bG' \to \bG$ which induces the isogeny map $\bG'_{\der} \to \bG_{\der}$. 
    We set $\cC_G = G/p(G')$ for any such pair.
    By \cref{lemma:canonical-isom-of-quotient-Levi}, we have the canonical isomorphism $\cC_G \xrightarrow{\sim} \cC_M$.
\end{remark}

\begin{proposition}\label{proposition:computation-scaler-isogeny}
    Let $G, G'$ be reductive groups over $F$.
    Assume that we have a central isogeny $p \colon G' \to G$.
    Then, we have 
    \begin{align*}
        \frac{\abs{\Ker(p)(F)}}{\abs{
            \# \Ker(p)(\overline{F})
            }_F \abs{\cC_G}} 
        = 
        \frac{\abs{H^1(F, (G')^{\natural})}}{\abs{H^1(F, G^{\natural})}}  
        \frac{1}{[X_*(A_G):X_*(A_{G'})]}.
    \end{align*}
\end{proposition}

\begin{proof}
    By $(4)$ of \cref{lemma:facts-in-abelian-Galois-cohomology}, we obtain 
    \begin{align*}
        \abs{\# \Ker(p)(\overline{F})}_F 
        = 
        \abs{H^0(F, \Ker(p))} \abs{H^2(F, \Ker(p))} / \abs{H^1(F, \Ker(p))}.
    \end{align*}
    Thus, we have 
    \begin{align*}
        \frac{\abs{\Ker(p)(F)}}
            {\abs{\# \Ker(p)(\overline{F})
            }_F \abs{\cC_G}}  
        = 
        \frac{\abs{H^1(F, \Ker(p))}}{\abs{H^2(F, \Ker(p))} \abs{\cC_G}}.
    \end{align*}
    Then, we can apply the long exact sequence in $(3)$ of \cref{lemma:facts-in-abelian-Galois-cohomology} to obtain the following:
    \begin{align*}
         \frac{\abs{H^1(F, \Ker(p))}}{\abs{H^2(F, \Ker(p))} \abs{\cC_G}} 
         = 
         \frac{\abs{H^1(F, G')}}{\abs{H^1(F, G)}} \frac{1}{\abs{\Ker(H^2_{ab}(p \colon G' \to G))}}.
    \end{align*}
    By a similar consideration by using $(2)$ of \cref{lemma:facts-in-abelian-Galois-cohomology}, we have 
    \begin{align*}
        \abs{H^1(F, G)} &= \abs{H^1(F, G^{\natural})} \frac{1}{\abs{\Ker H^2_{ab}(A_G \hookrightarrow G)}} \\
        \abs{H^1(F, G')} &= \abs{H^1(F, (G')^{\natural})} \frac{1}{\abs{\Ker H^2_{ab}(A_{G'} \hookrightarrow G')}}.
    \end{align*}
    By substituting these equations into the previous equation, we obtain
    \begin{multline*}
        \frac{\abs{H^1(F, G')}}{\abs{H^1(F, G)}} \frac{1}{\abs{\Ker(H^2_{ab}(p \colon G' \to G))}} \\
        = 
        \frac{\abs{H^1(F, (G')^{\natural})}}{\abs{H^1(F, G^{\natural})}}  
        \frac{1}{[X_*(A_G):X_*(A_{G'})]}
        \times \\
        \frac{[X_*(A_G):X_*(A_{G'})] \abs{\Ker H^2_{ab}(A_G \hookrightarrow G)}}{\abs{\Ker H^2_{ab}(A_{G'} \hookrightarrow G')} \abs{\Ker(H^2_{ab}(p \colon G' \to G))}}.
    \end{multline*}
    It suffices to show that the last term is equal to $1$. 
    We consider the following diagram: 
        % https://q.uiver.app/#q=WzAsNCxbMCwwLCJIXjIoRiwgQV97Ryd9KSJdLFswLDMsIkheMihGLCBBX3tHfSkiXSxbMiwwLCJIXjJfe2FifShGLCBHJykiXSxbMiwzLCJIXjJfe2FifShGLCBHKSJdLFswLDJdLFsyLDNdLFswLDFdLFsxLDNdXQ==
        \[
        \begin{tikzcd}
            {H^2(F, A_{G'})} && {H^2_{ab}(F, G')} \\
            \\
            \\
            {H^2(F, A_{G})} && {H^2_{ab}(F, G)}
            \arrow[from=1-1, to=1-3]
            \arrow[from=1-1, to=4-1]
            \arrow[from=1-3, to=4-3]
            \arrow[from=4-1, to=4-3].
        \end{tikzcd}
        \]
    All the maps appearing in this diagram are isogenies between groups consisting of the torsion points of some tori, by $(1)$ of \cref{lemma:facts-in-abelian-Galois-cohomology}.
    Thus, we obtain the result by considering the covering degree of the composition map $H^2(F, A_{G'}) \to H^2_{ab}(F, G)$.
\end{proof}

\subsection{The restriction of representations from $\GSO_{2n}(F)$ to $\SO_{2n}(F)$}

\begin{proposition}[\cite{GeeTaibi2019-GSp4}*{Proposition 8.3.1}] \label{proposition:restriction-GSO2n-to-SO2n}
    Let $\pi$ be an irreducible representation of $\GSO_{2n}(F)$.
    Then, the restriction $\pi_{\restriction{\SO_{2n}(F)}}$ of $\pi$ to $\SO_{2n}(F)$ is a multiplicity-free direct sum of irreducible representations of $\SO_{2n}(F)$.
\end{proposition}

We set $G' = \SO_{2n}(F) \times \GL_1(F)$ and $G = \GSO_{2n}(F)$.
Then, we have the natural isogeny $p \colon G' \to G$ induced by inclusion maps.
We set $Z = \Ker(p)$.
For any parabolic subgroup $P = MN$ of $G$, let $P'=M'N'$ denote the inverse image of $P = MN$ in $G'$.

\begin{lemma}\label{lemma:comparison-S-natural-groups}
    Let $M$ be a standard Levi subgroup of $G$.
    Let $\phi^{M}$ be a discrete $L$-parameter for $M$. 
    Let $\phi^{M'} \in \Phi_{\disc}(M')$ denote the push-forward ${}^L p \circ \phi^M$.
    \begin{enumerate}
        \item      
        We have a canonical isomorphism 
        \begin{align*}
            H^1(W_F, Z(\widehat{M})) &\xrightarrow{\sim} \Hom(M(F), \C^{\times}).
        \end{align*}
        \item     
        We have a canonical isomorphism 
        \begin{align*}
            \cC_M^{\cD} \xrightarrow{\sim} \Hom(W_F, \Ker({}^L p)).
        \end{align*}
        \item   
        We have a canonical isomorphism 
        \begin{align*}
            S^{\natural}_{\phi^M}/Z(\widehat{M^{\natural}})^{\Gamma_F} \xrightarrow{\sim} S_{\phi^M}/Z(\widehat{M})^{\Gamma_F}.
        \end{align*}
        \item     
        We have an exact sequence 
        \begin{align*}
            1 \to S_{\phi^M}/Z(\widehat{M})^{\Gamma_F} \to S_{\phi^{M'}}/Z(\widehat{M'})^{\Gamma_F} \to \mathrm{Stab}_{\cC^{\cD}_M} (\phi^{M}) \to 1,
        \end{align*}
        where $\mathrm{Stab}_{\cC^{\cD}_M} (\phi^{M})$ is the stabilizer in the group $\cC_M^{\cD}$ of the $L$-parameter $\phi^{M}$.
    \end{enumerate}
\end{lemma}

\begin{proof}
    The first statement follows from \cite{BinXu2016liftingproblem}*{Appendix}. 
    The second statement follows from the first. 
    The third statement holds trivially.
    The fourth statement follows from \cite{BinXu2016liftingproblem}*{Lemma 3.3}.
\end{proof}  

\begin{lemma}\label{lemma:Gan-Ichino-computation}
    For any irreducible representation $\pi$ of $G$ and any smooth function $f \in \Cc(G)$, then we have 
    \begin{align*}
        \tr (\pi \circ p) (p^*f)
         =
        \frac
        {\abs{\Ker(p)(F)}}
        {
            \abs{
                \# \Ker(p)(\overline{F})
            }_{F}\abs{\cC_G}} \sum_{\chi \in \cC_G^{\cD}
        } \tr (\pi \otimes \chi) (f).
    \end{align*}
\end{lemma}

\begin{proof}
    This can be proved in the same way as the proof of \cite{GI14}*{Lemma 13.2}.
\end{proof}

Let $\chi_{A_G}$ be a central character of $G$. 
Let $\chi_{Z(G')}$ (resp. $\chi_{A_{G'}}$) be the restriction of $\chi_{A_G}$ to $Z(G')$ (resp. $A_{G'}$).
Let $\Pi_{\disc}(M')_{\chi_{Z(G')}}$ be the subset of $\Pi_{\disc}(M')_{\chi_{A_{G'}}}$ which consists of representations with central character $\chi_{Z(G')}$ on $Z(G')$.
This set is a union of connected components of $\Pi_{\disc}(M')_{\chi_{A_{G'}}}$.
Similarly, we define the space ${\Temp_{\ind}(G')}_{\chi_{Z(G')}}$.

\begin{definition}
    We have the action of $\cC^{\cD}_{M}$ on the set $\Pi_{\disc}(M)_{\chi_{A_G}}$ defined by 
    \begin{align*}
        (\pi^M, \chi) \in \Pi_{\disc}(M)_{\chi_{A_G}} \times \cC_M^{\cD} \mapsto \pi^{M} \otimes \chi \in \Pi_{\disc}(M)_{\chi_{A_G}}.
    \end{align*} 
    We denote the quotient set by $\Pi_{\disc}(M)_{\chi_{A_G}}/\cC^{\cD}_M$ and the class of the representation $\pi^M$ by $[\pi^M]_{\cC^{\cD}_M}$ or $p_{\cC_M^{\cD}}(\pi^{M})$. 
    
    We have the action of $\cC_M$ on the set $\Pi_{\disc}(M')_{\chi_{Z(G')}}$ defined by 
    \begin{align*}
        (\pi^{M'}, m) \in \Pi_{\disc}(M')_{\chi_{Z(G')}} \times \cC_{M} \mapsto {}^{\Ad(m)}\pi^{M'} \in \Pi_{\disc}(M')_{\chi_{Z(G')}}.
    \end{align*}
    We denote the quotient set by $\Pi_{\disc}(M')_{\chi_{Z(G')}}/\cC_M$ and the class of the representation $\pi^{M'}$ by $[\pi^{M'}]_{{\cC_M}}$ or $p_{\cC_M}(\pi^{M'})$. 

    Let $d[\pi^M]_{\cC^{\cD}_M}$ be the quotient measure 
    \begin{align*}
        {p_{\cC_M^{\cD}}}_{*}
        (\frac{\abs{\mathrm{Stab}_{\cC_M^{\cD}}(\pi^M)}}{\abs{\cC_M^{\cD}}} \rd \pi^M)
    \end{align*}
    on $\Pi_{\disc}(M)_{\chi_{A_G}}/\cC^{\cD}_M$, of the measure $\rd \pi^M$ on $\Pi_{\disc}(M)_{\chi_{A_G}}$.
    Similarly, we define the quotient measure $d[\pi^{M'}]_{\cC_{M}}$ on the set $\Pi_{\disc}(M')_{\chi_{Z(G')}}/\cC_M$.
\end{definition}

\begin{remark}
    The quotient measure $d[\pi^M]_{\cC^{\cD}_M}$ satisfies the following property. 
    Let $\Phi \in L^1(\Pi_{\disc}(M)_{\chi_{A_G}})$, we have 
    \begin{align*}
        \int_{\Pi_{\disc}(M)_{\chi_{A_G}}/\cC^{\cD}_M} (\sum_{\pi^M \in [\pi^M]_{\cC^{\cD}_M}} \Phi(\pi^M)) d[\pi^M]_{\cC^{\cD}_M}  
        = 
        \int_{\Pi_{\disc}(M)_{\chi_{A_G}}} \Phi(\pi^M) \rd \pi^M.
    \end{align*}
\end{remark}

\begin{construction}
    We take a representation $\pi^M \in \Pi_{\disc}(M)_{\chi_{A_G}}$. 
    Then, by \cref{proposition:restriction-GSO2n-to-SO2n}, the pullback $\pi^M \circ p$ is a multiplicity-free direct sum of representations in $\Pi_{\disc}(M')_{\chi_{Z(G')}}$. 
    It is easy to see that these representations are in the same class of $\Pi_{\disc}(M')_{\chi_{Z(G')}}/\cC_M$. 
    Also, the restriction of $\pi^M$ does not change if we replace $\pi^M$ by  $\pi^M \otimes \chi$ for $\chi \in \cC_M^{\cD}$.
    Thus, we have obtained a map 
    \begin{align}
        \Pi_{\disc}(M)_{\chi_{A_G}}/\cC^{\cD}_M \to \Pi_{\disc}(M')_{\chi_{Z(G')}}/\cC_M.
    \end{align}
    
    We construct the inverse of this map.
    Let $\pi^{M'}$ be an element of $\Pi_{\disc}(M')_{\chi_{Z(G')}}$.
    Then, the representation $\pi^{M'}$ descends to the representation $\pi^{p(M')}$ of $p(M')$ and then we can take the induction $\ind^{M}_{p(M')}(\pi^{p(M')})$, which we simply denote by $\ind^{M}_{M'}(\pi^{M'})$. 
    By \cref{proposition:restriction-GSO2n-to-SO2n}, this induced representation is a multiplicity-free direct sum of irreducible representations in $\Pi_{\disc}(M)_{\chi_{A_G}}$. 
    It is easy to see that these representations are in the same class in $\Pi_{\disc}(M)_{\chi_{A_G}}/\cC^{\cD}_M$.
    This gives the inverse map of $(5.3)$ by the Frobenius reciprocity.

    In summary, we obtained a bijection 
    \begin{align*}
        r_{p} \colon 
        \Pi_{\disc}(M)_{\chi_{A_G}}/\cC^{\cD}_M \xrightarrow{\sim} \Pi_{\disc}(M')_{\chi_{Z(G')}}/\cC_M\colon [\pi^M]_{\cC^{\cD}_M} \mapsto [\pi^{M'}]_{\cC_M}.
    \end{align*}
    It is easy to see that this map is a homeomorphism.
\end{construction}

\begin{lemma}\label{lemma:push-forward-quotient-measure}
     We have 
     \begin{align*}
        {r_p}_{*}(\rd[\pi^M]_{\cC_M^{\cD}}) 
        =
        \frac{1}{\abs{[X_*(A_M) : X_*(A_{M'})]}}
        \rd[\pi^{M'}]_{\cC_M}.
     \end{align*}
\end{lemma}

\begin{proof}
    This follows from our normalization of the measures on the tempered spectra; see \cref{subsubsection:normalization-tempered-spectra}.
\end{proof}

\begin{remark}
    In the following, by abuse of notation, we write $\ind^{G}_{G'}$ for the induction $\ind^{G}_{p(G')}$ from $p(G')$ to $G$.
\end{remark}

The following lemma follows from a simple computation by utilizing the Frobenius reciprocity.

\begin{lemma}\label{lemma:associativity-induction}
    \begin{enumerate}
        \item     
        Let $\pi^M$ be a representation of $M$.
        The parabolic induction $i^{G'}_{P'} (\pi^M \circ p)$ is isomorphic to the representation $i^{G}_{P}(\pi^M) \circ p$.  
        \item     
        Let $\pi^{M'}$ be a representation of $p(M')$, which we view as a representation of $M'$.
        The parabolic induction $i^{G}_{P} (\ind^{M}_{M'}(\pi^{M'}))$  is isomorphic to the representation $\ind^{G}_{G'}(i^{G'}_{P'}(\pi^{M'}))$.
    \end{enumerate}
\end{lemma}

\begin{proof}
    By the Frobenius reciprocity, for any representation $\sigma'$ of $p(G')$, we have 
    \begin{align*}
        \Hom_{G'}(\sigma', i^{G}_{P}(\pi^M) \circ p) 
        &\xrightarrow{\sim} \Hom_{G}(\ind^{G}_{G'}(\sigma'), i^G_P(\pi^M)) \\
        &\xrightarrow{\sim} \Hom_{M}(r^G_P(\ind^{G}_{G'}(\sigma')), \pi^M).
    \end{align*}
    Similarly, we have 
    \begin{align*}
        \Hom_{G'}(\sigma', i^{G'}_{P'} (\pi^M \circ p)) 
        &\xrightarrow{\sim} \Hom_{G}(r^{G'}_{P'}(\sigma'), \pi^M \circ p) \\
        &\xrightarrow{\sim} \Hom_{M}(\ind^{M}_{M'}(r^{G'}_{P'}(\sigma')), \pi^M).
    \end{align*}
    Thus, it suffices to show that we have a natural isomorphism
    \begin{align*}
        r^G_P(\ind^{G}_{G'}(\sigma')) \xrightarrow{\sim} \ind^{M}_{M'}(r^{G'}_{P'}(\sigma')).
    \end{align*}
    This follows from the isomorphism $G/p(G') \xrightarrow{\sim} M/p(M')$ in \cref{lemma:canonical-isom-of-quotient-Levi}.
    This proves the first statement. 
    The second statement follows from the transitivity of the induction, since we have a natural isomorphism $p \colon N' \xrightarrow{\sim} N$.
\end{proof}

\subsection{Proof of the conjecture}

\begin{theorem}\label{theorem:formal-degree-conjecture-GSO2n}
    \cref{conjecture:formal-degree-conjecture} holds for $G = \GSO_{2n}$ and $\PGSO_{2n}$, i.e., we have 
    \begin{align*}
        \mu^G_{\psi}(\pi) = \mu^{G}_{\HII, \psi}(\pi)
    \end{align*}
    for almost all $\pi \in \Temp_{\ind}(G)$.
\end{theorem}

\begin{proof} 
    We fix a unitary central character $\chi_{A_G}$ of $A_G$ and prove the Plancherel formula for the fixed central character; see \cref{remark:reduction-fixing-central-character}.
    We obtain the character $\chi_{Z(G')}$ of $Z(G')$.
    
    We take a smooth function $f \in \Cc(G, \chi_{A_G})$.
    The pullback $p^*f$ is in the space $\Cc(G', \chi_{Z(G')})$.
    By \cref{theorem:formal-degree-conjecture-SO2n}, we have 
    \begin{align*}
        p^*f(1) 
        = 
        \gamma^*(0, \mathbf{1}, \psi)^{-1}
        \int_{{\Temp_{\ind}(G')}_{\chi_{Z(G')}}}
        \tr \pi' (p^*f^{\vee}) 
        \mu^{G'}_{\HII, \psi}(\pi')   
        \rd \pi'.
    \end{align*}
    By definition of the measure $\rd \pi'$, this can be rewritten as 
    \begin{multline*}
        \gamma^*(0, \mathbf{1}, \psi)^{-1} \\
        \sum_{M' \in \cL^{G'}_0}
        \frac{1}{\abs{W(G', M')}} 
        \int_{\pi^{M'} \in \Pi_{\disc}(M')_{\chi_{Z(G')}}} 
        \tr i^{G'}_{M'} (\pi^{M'}) (p^{*}f^{\vee})
        \mu_{\HII, \psi}^{G'}(i^{G'}_{M'}(\pi^{M'})) 
        \rd \pi^{M'}.
    \end{multline*}
    We also have 
    \begin{multline}
        \int_{\pi^{M'} \in \Pi_{\disc}(M')_{\chi_{Z(G')}}} 
        \tr i^{G'}_{M'} (\pi^{M'}) (p^{*}f^{\vee})
        \mu_{\psi}^{G'}(i^{G'}_{M'}(\pi^{M'})) 
        \rd \pi^{M'} \\
        = 
        \int_{\Pi_{\disc}(M')_{\chi_{Z(G')}}/\cC_M} 
        \lbrace
            \sum_{\pi^{M'} \in [\pi^{M'}]_{\cC_{M}}} 
            \tr i^{G'}_{M'} (\pi^{M'}) (p^*f^{\vee})
            \mu_{\HII, \psi}^{G'}(i^{G'}_{M'}(\pi^{M'})) 
        \rbrace 
        \rd [\pi^{M'}]_{\cC_M}. 
    \end{multline}
    By using \cref{remark:GSO2n-invariant-independence-of-choice}, it is easy to see that the function
    \begin{align*}
        \mu_{\HII, \psi}^{G'}(i^{G'}_{M'}(\pi^{M'})) 
    \end{align*}
    is constant on the class $[\pi^{M'}]_{\cC_M}$; we denote it by  $\mu_{\HII, \psi}^{G'}(i^{G'}_{M'}([\pi^{M'}]_{\cC_M}))$.
    The sum 
    \begin{align*}
        \sum_{\pi^{M'} \in [\pi^{M'}]_{\cC_{M}}} 
        \tr i^{G'}_{M'} (\pi^{M'}) (p^{*}f^{\vee})
    \end{align*}
    is equal to 
    \begin{multline*}
        \frac{\abs{\mathrm{Stab}_{\cC^{\cD}_M}(\pi^{M'})}}{\abs{\cC^{\cD}_{M}}}
        \tr i^{G}_{M} (\ind^{M}_{M'}(\pi^{M'}) \circ p ) (p^*f^{\vee}) \\
        = 
        \frac{\abs{\mathrm{Stab}_{\cC^{\cD}_M}(\pi^{M'})}}{\abs{\cC^{\cD}_{M}}}
        \tr (\ind^G_{G'} (i^{G'}_{M'} (\pi^{M'})) \circ p) (p^*f^{\vee}).
    \end{multline*}
    Here, we applied \cref{lemma:associativity-induction}.
    By \cref{lemma:Gan-Ichino-computation}, this is equal to  
    \begin{align*}
        \frac{\abs{\mathrm{Stab}_{\cC^{\cD}_M}(\pi^{M'})}}
        {\abs{\cC^{\cD}_{M}}}
        \frac{\abs{\Ker(p)(F)}}{\abs{\# \Ker(p)(\overline{F})}_{F}\abs{\cC_G}} 
        \sum_{\chi \in \cC_G^{\cD}}  
        \tr (\ind^G_{G'} (i^{G'}_{M'} (\pi^{M'})) \otimes \chi) (f^{\vee}) \\
        =   
        \frac{\abs{\mathrm{Stab}_{\cC^{\cD}_M}(\pi^{M'})}}{\abs{\cC^{\cD}_{M}}}
        \frac{\abs{\Ker(p)(F)}}{\abs{\# \Ker(p)(\overline{F})}_{F}}
        \tr (\ind^G_{G'} (i^{G'}_{M'} (\pi^{M'}))) (f^{\vee}).
    \end{align*}
    Here, we have used 
    \begin{align*}
        \ind^G_{G'} (i^{G'}_{M'} (\pi^{M'})) \otimes \chi 
        \xrightarrow{\sim} 
        \ind^G_{G'} (i^{G'}_{M'} (\pi^{M'})),
    \end{align*}
    for $\chi \in \cC^{\cD}_G$.
    By applying \cref{lemma:associativity-induction} once more, we obtain 
    \begin{align*}
        \tr (\ind^G_{G'} (i^{G'}_{M'} (\pi^{M'}))) (f^{\vee}) 
        = 
        \sum_{\pi^M \in r_p^{-1}([\pi^{M'}]_{\cC_M})} \tr (i^G_{M}(\pi^M)) (f^{\vee}).
    \end{align*}
    In summary, we have 
    \begin{multline}
        \gamma^*(0, \mathbf{1}, \psi)f(1) 
        = 
        \gamma^*(0, \mathbf{1}, \psi)p^*f(1)   \\
        = 
        \sum_{M \in \cL^G_0}   
        \frac{1}{\abs{W(G, M)}} 
        \int_{\Pi_{\disc}(M')_{\chi_{Z(G')}}/\cC_M}
        \lbrace
        \sum_{\pi^M \in r_p^{-1}([\pi^{M'}]_{\cC_M})} \tr (i^G_{M}(\pi^M)) (f^{\vee})
        \rbrace
        \\
        \times
        \frac{\abs{\mathrm{Stab}_{\cC^{\cD}_M}(\pi^{M})}}{\abs{\cC^{\cD}_{M}}}
        \frac{\abs{\Ker(p)(F)}}{\abs{\# \Ker(p)(\overline{F})}_{F}}
        \mu_{\HII, \psi}^{G'}(i^{G'}_{M'}([\pi^{M'}]_{\cC_M}))
        \rd[\pi^{M'}]_{\cC_M}. 
    \end{multline}
    Finally, we compute the remaining product of the density function and the measure. 
    By \cref{proposition:computation-scaler-isogeny} and \cref{lemma:facts-in-abelian-Galois-cohomology}, we have 
    \begin{multline*}
        \frac{\abs{\mathrm{Stab}_{\cC^{\cD}_M}(\pi^{M})}}{\abs{\cC^{\cD}_{M}}}
        \frac{\abs{\Ker(p)(F)}}{\abs{\# \Ker(p)(\overline{F})}_{F}}
        \mu_{\HII, \psi}^{G'}(i^{G'}_{M'}([\pi^{M'}]_{\cC_M})) \\
        =
        \abs{\mathrm{Stab}_{\cC^{\cD}_M}(\pi^{M})}
        \frac{\abs{H^1(F, (M')^{\natural})}}{\abs{H^1(F, M^{\natural})}}
        \frac{1}{[X_*(A_M) : X_*(A_{M'})]}
        \mu_{\HII, \psi}^{G'}(i^{G'}_{M'}([\pi^{M'}]_{\cC_M})) \\
        =
        \frac{1}{[X_*(A_M) : X_*(A_{M'})]}
        \frac{\abs{Z((M')^{\natural})^{\Gamma_F}}}{\abs{Z(M^{\natural})^{\Gamma_F}}}
        \abs{\mathrm{Stab}_{\cC^{\cD}_M}(\phi^{M}_{\pi^M})}
        \mu_{\HII, \psi}^{G'}(i^{G'}_{M'}([\pi^{M'}]_{\cC_M})). 
    \end{multline*}
    Furthermore, by \cref{lemma:comparison-S-natural-groups}, we have 
    \begin{align*}
        \frac{\abs{Z((M')^{\natural})^{\Gamma_F}}}{\abs{Z(M^{\natural})^{\Gamma_F}}}
        \abs{\mathrm{Stab}_{\cC^{\cD}_M}(\phi^{M}_{\pi^M})}
        \frac{1}{\abs{S^{\natural}_{\pi^{M'}}}}
        = 
        \frac{1}{\abs{S^{\natural}_{\pi^{M}}}}
    \end{align*}
    for any $\pi^{M'} \in [\pi^{M'}]_{\cC_M}$ and $\pi^M \in r_p^{-1}([\pi^{M'}]_{\cC_M})$.
    Thus, we obtain 
    \begin{multline*}
        \frac{1}{[X_*(A_M) : X_*(A_{M'})]}
        \frac{\abs{Z((M')^{\natural})^{\Gamma_F}}}{\abs{Z(M^{\natural})^{\Gamma_F}}}
        \abs{\mathrm{Stab}_{\cC^{\cD}_M}(\phi^{M}_{\pi^M})}
        \mu_{\HII, \psi}^{G'}(i^{G'}_{M'}([\pi^{M'}]_{\cC_M})) \\
        = 
        \frac{1}{[X_*(A_M) : X_*(A_{M'})]}
        \frac{\abs{\gamma^{*}(0, \pi^{M'}, \Ad_{G'}, \psi)}}{\abs{S^{\natural}_{\pi^M}}}
        d [\pi^{M'}]_{\cC_M}
    \end{multline*}
    for any $\pi^{M'} \in [\pi^{M'}]_{\cC_M}$ and $\pi^M \in r_p^{-1}([\pi^{M'}]_{\cC_M})$.
    Let $\mu^{G}_{\HII, \psi}(i^G_M([\pi^M]_{\cC^{\cD}_M}))$ denote the common value of $\mu^G_{\HII, \psi}(i^G_M(\pi^M))$ for any $\pi^M \in [\pi^M]_{\cC^{\cD}_M}$, see the stability condition in \cref{conjecture:local-Langlands-conjecture} and \cref{theorem:local-Langlands-correspondence-SO2n}.
    Lastly, by \cref{lemma:push-forward-quotient-measure} and the equalities
    \begin{align*}
        \gamma^{*}(0, \pi^{M'}, \Ad_{G'}, \psi) = \gamma^{*}(0, \pi^{M}, \Ad_{G}, \psi)
    \end{align*}
    and 
    \begin{align*}
        \begin{cases}
             \langle \pi^{M}, 1 \rangle = 1, \\
             \langle \pi^{M'}, 1 \rangle = 1,
        \end{cases} 
    \end{align*}
    the last expression in the previous equation is equal to 
    \begin{multline*}
        \sum_{M \in \cL^G_0}   
        \frac{1}{\abs{W(G, M)}} 
        \int_{\Pi_{\disc}(M)_{\chi_{A_G}}/\cC^{\cD}_M} 
        \lbrace
        \sum_{\pi^M \in [\pi^{M}]_{\cC^{\cD}_M}} \tr (i^G_{M}(\pi^M)) (f^{\vee})
        \rbrace  \\
        \times 
        \mu_{\HII, \psi}^{G}(i^{G}_{M}([\pi^{M}]_{\cC^{\cD}_M}))
        \rd[\pi^{M}]_{\cC^{\cD}_M}.
    \end{multline*}
    Thus, we obtain
    \begin{align*}
        f(1) 
        = 
        \gamma^*(0, \mathbf{1}, \psi)^{-1}
        \int_{\Temp_{\ind}(G)_{\chi_{A_G}}}  
        \tr \pi (f^{\vee})  
        \mu_{\HII, \psi}^{G}(\pi)
        \rd \pi.
    \end{align*}
    This completes the proof. 
\end{proof}

\subsection{Computations of $\gamma$-factors}

Recall that for a tempered induced representation $\pi \in \Temp_{\ind}(\SO_{2n})$, we write one of its $L$-parameters as $\phi^{\SO_{2n}}_{\pi}$. 
We denote the standard $L$-parameter $\Std_{\SO_{2n}} \circ \phi^{\SO_{2n}}_{\pi}$ by $\phi^{\GL_{2n}}_{\pi}$.
This standard parameter does not depend on the choice of $\phi^{\SO_{2n}}_{\pi}$.

We recall some explicit computations of the standard and the adjoint $\gamma$-factors for $\SO_{2n}(F)$, which we need later.

The following are well known.
\begin{lemma}\label{lemma:adjoint-orthogonal-wedge-standard}
    For any tempered $L$-parameter $\phi^{\SO_{2n}} \in \Phi_{\temp}(\SO_{2n})$, we have 
    \begin{align*}
        \Ad_{\SO_{2n}} \circ \phi^{\SO_{2n}} = \wedge^2 \circ \iota^{\GL_{2n}}_{\SO_{2n}} \circ \phi^{\SO_{2n}}.
    \end{align*}
\end{lemma}

\begin{corollary}\label{corollary:adjoint-factors-through-L-parameter}
     The adjoint $L$-function $\gamma(s, \pi, \Ad_{\SO_{2n}}, \psi)$ factors through $\Temp_{\ind}(\SO_{2n}(F))/\sim_{st, w}$, and we write it as 
     \begin{align*}
        \gamma(s, \phi, \Ad_{\SO_{2n}}, \psi)
     \end{align*}
     for $\phi \in \Temp_{\ind}(\SO_{2n}(F))/\sim_{st, w}$.
\end{corollary}

We now describe local behavior of the adjoint $\gamma$-factor for $\SO_{2n}(F)$.

\begin{lemma}\label{lemma:adjoint-L-explicit}
     We take a point $\phi_0 \in \Temp_{\ind}(\SO_{2n}(F))/\sim_{st, w}$.
     Then, there exist smooth functions $P(\phi)$ and $Q(\phi)$ defined in an open neighborhood $U_{\phi_0} \subset \Temp_{\ind}(\SO_{2n}(F))/\sim_{st, w}$ of $\phi_0$, such that $P(\phi)$ takes its values in $\R$, $Q(\phi_0) \neq 0$, and 
     \begin{align*}
         \gamma^{*}(0, \phi, \Ad_{\SO_{2n}}, \psi)  
         = 
         P(\phi)^2 Q(\phi)
     \end{align*}
     at almost all $\phi \in U_{\phi_0}$ with respect to $\rd \phi$.
\end{lemma}

Before proving this lemma, we give an application.

\begin{corollary}\label{corollary:adjoint-L-abs-value}
    In the notation of the previous lemma, there exists a smooth function $c(\phi)$ on $U_{\phi_0}$ which takes its value in $\bS^1$, and we have 
    \begin{align*}
        \abs{\gamma^{*}(0, \phi, \Ad_{\SO_{2n}}, \psi)}  
        = 
        c(\phi) P(\phi)^2 Q(\phi)
    \end{align*}
    at almost all $\phi \in U_{\phi_0}$ with respect to $\rd \phi$. 
    Indeed, $c(\phi)$ is given as the function 
    \begin{align*}
        c(\phi) = \frac{Q(\phi)}{\abs{Q(\phi)}}
    \end{align*}
    on $U_{\phi_0}$.
    In particular, the functions $\abs{\gamma^{*}(0, \phi, \Ad_{\SO_{2n}}, \psi)}$ and $c(\phi)$ coincide with smooth functions up to null set in $U_{\phi_0}$.
\end{corollary}

\begin{proof}[The proof of \cref{lemma:adjoint-L-explicit}]
    The standard $L$-parameter $\phi_0^{\GL_{2n}}$ for $\phi_0$ can be written as 
    \begin{align}
        \bigoplus_{i \in I} (\phi_i \oplus \phi_i^{\vee})^{\oplus m_i} 
        \oplus 
        \bigoplus_{j \in J} \phi_j,
    \end{align}
    where 
    \begin{itemize}
        \item    
        $I, J$ are totally ordered finite sets.
        \item     
        $\phi_i \in \Phi_{\disc}(\GL_{k_i})$ and $\phi_j \in \Phi_{\disc}(\GL_{l_j})$ and $\sum_{i} 2k_i + \sum_{j} l_j = 2n$.
        Furthermore, $\phi_j$ is self-dual and orthogonal.
        \item     
        For distinct $i_1, i_2 \in I$, we have 
        \begin{align*}
            \phi_{i_1} \neq \phi_{i_2}, \phi_{i_2}^{\vee}.
        \end{align*}
        For distinct $j_1, j_2 \in J$,  we have 
        \begin{align*}
            \phi_{j_1} \neq \phi_{j_2}.
        \end{align*}
    \end{itemize}

    The standard $L$-function of any point in a sufficiently small open neighborhood is given by 
    \begin{align}
        \phi^{\GL_{2n}}_{\lambda} =
        \bigoplus_{i \in I} 
        \lbrace
        \bigoplus_{k_i = 1}^{m_i} (\phi_i \otimes \chi_{i\lambda_{k_i}} \oplus \phi_i^{\vee} \otimes \chi_{-i\lambda_{k_i}})
        \rbrace
        \oplus 
        \bigoplus_{j \in J} \phi_j, \label{eq:parameter-standard}
    \end{align}
    for sufficiently small $\lambda_{k_i} \in \R$ and we set $\lambda = ((\lambda_{k_i})_{k_i=1}^{m_i})_{i \in I}$ and $\chi_{s}(w) = \abs{w}_F^s$ for $w \in L_F$.
    We will compute 
    \begin{align*}
        \wedge^2 \phi^{\GL_{2n}}_{\lambda}.
    \end{align*}
    We set 
    \begin{align*}
        \psi_i = \bigoplus_{k_i = 1}^{m_i} (\phi_i \otimes \chi_{i\lambda_{k_i}} \oplus \phi_i^{\vee} \otimes \chi_{-i\lambda_{k_i}}),
    \end{align*}
    for $i \in I$.
    Then, we can see that 
    \begin{align}
        \wedge^2 \phi^{\GL_{2n}}_{\lambda} 
        &= 
        \wedge^2 (\bigoplus_{i \in I} \psi_i \oplus \bigoplus_{j \in J} \phi_j) \\
        &= 
        \wedge^2(\bigoplus_{i \in I} \psi_i ) 
        \oplus 
        \wedge^2(\bigoplus_{j \in J} \phi_j)
        \oplus   
        (\bigoplus_{i \in I} \psi_i) \otimes (\bigoplus_{j \in J} \phi_j). \label{eq:parameter-wedge}
    \end{align}
    We have obtained the three terms. 

    The first term of \cref{eq:parameter-wedge} is equal to 
    \begin{align}
        \wedge^2(\bigoplus_{i \in I} \psi_i)  
        = 
        \bigoplus_{i \in I} \wedge^2 \psi_i 
        \oplus   
        \bigoplus_{(i_1, i_2) \in I^2, i_1 < i_2} \psi_{i_1} \otimes \psi_{i_2}.
    \end{align}
    We also have 
    \begin{multline}
        \wedge^2 \psi_i  
        =
        \bigoplus_{k_i=1}^{m_i} 
        \lbrace
        (\wedge^2\phi_i) \otimes \chi_{2i\lambda_{k_i}} \oplus 
        (\wedge^2\phi_i^{\vee}) \otimes \chi_{-2i\lambda_{k_i}}  
        \oplus 
        (\phi_i \otimes \phi_i^{\vee})
        \rbrace
        \\
        \oplus   
        \bigoplus_{1 \leq k_i^{1} < k_i^{2} \leq m_i} 
        (\phi_i \otimes \chi_{i\lambda_{k^{1}_i}} \oplus \phi_i^{\vee} \otimes \chi_{-i\lambda_{k^{1}_i}})
        \otimes 
        (\phi_i \otimes \chi_{i\lambda_{k^{2}_i}} \oplus \phi_i^{\vee} \otimes \chi_{-i\lambda_{k^{2}_i}}).
    \end{multline}
    
    The second term of \cref{eq:parameter-wedge} is equal to 
    \begin{align}
        \wedge^2 (\bigoplus_{j \in J} \phi_j) 
        = 
        \bigoplus_{j \in J}  \wedge^2 \phi_j 
        \oplus  
        \bigoplus_{(j_1, j_2) \in J^2, j_1 < j_2} \phi_{j_1} \otimes \phi_{j_2}.
    \end{align}

    The third term of \cref{eq:parameter-wedge} is equal to
    \begin{align}
        \bigoplus_{(i, j) \in I \times J} 
        \bigoplus_{k_i = 1}^{m_i}
        \lbrace
        (\phi_i \otimes \chi_{i\lambda_{k_i}} \oplus \phi_i^{\vee} \otimes \chi_{-i\lambda_{k_i}}) \otimes \phi_j
        \rbrace
        .
    \end{align}

    From now on, we will ignore null sets with respect to $\rd \phi$.
    We have 
    \begin{align*}
        \gamma^*(0, \phi, \Ad_{\SO_{2n}}, \psi)   
        =   
        \lim_{s \to 0+} 
        \gamma(s, \mathbf{1}, \psi)^{n_{\phi_0}}   
        \gamma(s, \phi, \Ad_{\SO_{2n}}, \psi),
    \end{align*}
    where 
    \begin{align*}
        n_{\phi_0}  
        = 
        \sum_{i \in I} m_i.
    \end{align*}

    We now consider the regularized standard $\gamma$-factor for each term of \cref{eq:parameter-wedge}.
    We introduce some notation.
    \begin{itemize}
        \item  
        Let $I_o$ be the subset of $I$ consisting of $i \in I$ such that $\gamma(0, \wedge^2 \phi_i, \psi) \neq 0$.   
        \item  
        Let $I_s$ be the subset of $I$ consisting of $i \in I$ such that $\gamma(0, \wedge^2 \phi_i, \psi) = 0$.
        \item     
        Let $I_{nd}$ be the subset of $I$ consisting of $i \in I$ such that $\phi_i$ is not self-dual.
    \end{itemize}

    The regularized $\gamma$-factor for the first term can be written as 
    \begin{align*}
        \{
        \prod_{i \in I_s} \prod_{k_i = 1}^{m_i} 4\lambda_{k_i}^2  
        \prod_{i \in I_s \sqcup I_o} 
        \prod_{k_i^1 < k_i^2}
        (\lambda_{k^1_i}^2 - \lambda_{k^2_i}^2 )^2
        \prod_{i \in I_{nd}} 
        \prod_{k_i} \lambda_{k_i}^2 \}
        \times
        Q_1(\lambda),
    \end{align*}
    for a function $Q_1(\lambda)$ non-vanishing at $\lambda = 0$.
    The $\gamma$-factor for the second term is smooth and nonzero at $\lambda = 0$, which we write $Q_2(\lambda)$.
    The $\gamma$-factor for the third term can be written as 
    \begin{align*}
        \prod_{(i, j) \in I \times J, \phi_i = \phi_j} 
        \prod_{k_i = 1}^{m_i}    
        \lambda_{k_i}^2 
        \times Q_3(\lambda),
    \end{align*}
    with a function $Q_3(\lambda)$ non-vanishing at $\lambda = 0$.
    Thus, we have
    \begin{multline}
        \{
        \prod_{i \in I_s} \prod_{k_i = 1}^{m_i} 2\lambda_{k_i}
        \prod_{i \in I_s \sqcup I_o} 
        \prod_{k_i^1 < k_i^2}
        (\lambda_{k^1_i}^2 - \lambda_{k^2_i}^2 ) \\
        \times
        \prod_{i \in I_{nd}} 
        \prod_{k_i} \lambda_{k_i}
        \prod_{(i, j) \in I \times J, \phi_i = \phi_j} 
        \prod_{k_i = 1}^{m_i}    
        \lambda_{k_i}
        \}^2  \\
        \times Q_1(\lambda)Q_2(\lambda)Q_3(\lambda).
        \label{eq:adjoint-L-explicit}
    \end{multline}
    and thus it has the form 
    \begin{align*}
        P(\lambda)^2 Q(\lambda)
    \end{align*}
    as in the statement of this lemma.
\end{proof}

\begin{remark}\label{remark:standard-poles}
    We describe the standard $\gamma$-factor $\gamma(s, \phi^{\GL_{2n}}_{\lambda}, \psi)$ by using the notation in \cref{eq:parameter-standard}.
    Let $I_{\triv}$ (resp. $J_{\triv}$) be the subset of $I$ (resp. $J$) consisting of $i \in I$ (resp. $j \in J$) with $\phi_i = \mathbf{1}$ (resp. $\phi_j = \mathbf{1}$).
    Note that $\abs{J_{\triv}} \leq 1$.
    Then, there exists a smooth function $R(s, \lambda)$ which is nonzero at $(0, 0) \in \C_{+} \times \R^{\sum_{i \in I} m_i}$, such that $\gamma(s, \phi^{\GL_{2n}}_{\lambda}, \psi)$ is equal to
    \begin{align*}
         s^{\abs{J_{\triv}}}
         \times 
         \prod_{i \in I_{\triv}} \prod_{k_i = 1}^{m_i} (\lambda_{k_i}^2 + s^2)
         \times    
         R(s, \lambda).
    \end{align*}
    We set $\gamma(s, \mathbf{0}, \psi) = 1$.
    Then, we have $I_{\triv} \subset I_{o}$ in the notation in the proof of \cref{lemma:adjoint-L-explicit}.
\end{remark}

\section{$\PGSO_8$ with triality}\label{section:6}

\subsection{The extension of representations of $\PGSO_8$ to $\PGSO_8$ with the triality automorphism} 

By \cref{definition:automorphism-of-D4}, we have the group $\PGSO_8(F) \rtimes S_3$.

\begin{lemma}
    We have $H^2(S_3, \C^{\times}) = 0$.
\end{lemma}

\begin{proof}
    This follows from the Lyndon-Hochschild-Serre spectral sequence.
\end{proof}

\begin{corollary}\label{corollary:extension-to-S_3}
    If an irreducible representation $\pi$ of $\PGSO_8(F)$ satisfies the condition 
    \begin{align*}
        \pi^{a} \xrightarrow{\sim} \pi 
    \end{align*} 
    for any $a \in S_3$ (i.e., $\pi$ is $S_3$-invariant), the representation $\pi$ has an extension $\widehat{\pi}$ to $\PGSO_8(F) \rtimes S_3$ with the same underlying vector space. 
\end{corollary}

\begin{proof}
    For any $a \in S_3$, we fix $I_a \colon \pi^a \xrightarrow{\sim} \pi$.
    Then, for any $a, b \in S_3$, we have 
    \begin{align*}
        I_{ab} = c(a, b) I_{a} I_{b} \colon \pi^{ab} \to \pi^a \to \pi,
    \end{align*}
    for a constant $c(a, b) \in \C^{\times}$.
    The function $c \colon S_3 \times S_3 \to \C^{\times}$ defines a $2$-cocycle on $S_3$ valued in $\C^{\times}$.
    This proves the result.
\end{proof}

\begin{remark}
    If $\pi$ is unitary in the previous corollary, then $\widehat{\pi}$ can be taken as a unitary representation.
\end{remark}

\begin{lemma}\label{lemma:canonical-extension-to-A3}
    Let $(\pi, V)$ be an irreducible representation of $\PGSO_8(F)$ which is $S_3$-invariant. 
    Then, the restriction of any extension $\widehat{\pi}$ as in \cref{corollary:extension-to-S_3} to the group $\PGSO_8(F) \rtimes A_3$ is uniquely determined by $\pi$.
\end{lemma}

\begin{proof}
    Let $\widehat{\pi}$ be an extension.
    Then, for any involution $a$ in $S_3$, the operator $\widehat{\pi}(a) \colon \pi \to \pi^{a}$ is determined up to $\pm 1$, since the order of $a$ is two and by Schur's lemma.
    Let $\sigma_1, \sigma_2$ be generators of $S_3$ with the relation $(\sigma_1 \sigma_2)^3 = 1$.
    Then, the possible choice of the pair $\widehat{\pi}(\sigma_1), \widehat{\pi}(\sigma_2)$ is determined up to $\pm 1$ by the Coxeter relation.
    In each case, we obtain the same operator $\widehat{\pi}(\sigma_1)\widehat{\pi}(\sigma_2)$ as $(\pm 1)^2=1$.
    This completes the proof.
\end{proof}

\begin{definition}
    Let $(\pi, V)$ be an irreducible representation of $\PGSO_8(F)$ which is $S_3$-invariant.  
    We call the restriction of any extension $\widehat{\pi}$ as in \cref{corollary:extension-to-S_3} to the group $\PGSO_8(F) \rtimes A_3$, the canonical extension of $\pi$ to $\PGSO_8(F) \rtimes A_3$.
    We write it as $\widetilde{\pi}$.
\end{definition}

\subsection{The canonical extension for functorial lifts from $G_2$}

\begin{proposition}\label{proposition:invariance-S3}
    Let $\phi^{G_2} \in \Phi_{\temp}(G_2)$ be a tempered $L$-parameter for $G_2$.
    Let $\Pi^{X}_{\phi^{G_2}}$ be the corresponding distinguished Xu $L$-packet.
    Then, any representation $\pi \in \Pi^{X}_{\phi^{G_2}}$ is invariant under the action of $S_3$ and it has the canonical extension $\widetilde{\pi}$ to $\PGSO_8(F) \rtimes A_3$. 
\end{proposition}

\begin{proof}
    We first note that the action of nontrivial element in the coset $\O_{2n}(F)/\SO_{2n}(F)$ is equal to $s_1s_6s_1$.
    The equivalence between (i) and (ii) in \cref{theorem:local-Langlands-correspondence-GSO2n}(c) and the condition \cref{eq:condition-distinguished-packet} imply that $\pi$ and ${}^{s_1}\pi$ are invariant under the action of $s_1 s_6 s_1$. 
    Thus, we have 
    \begin{align*}
        \begin{cases}
            \pi \xrightarrow{\sim} {}^{s_1s_6s_1}\pi, \\ 
            {}^{s_1}\pi \xrightarrow{\sim} {}^{s_1s_6}\pi.
        \end{cases}
    \end{align*}
    The second condition implies that we have $\pi \xrightarrow{\sim} {}^{s_6}\pi$.
    The elements $s_6$ and $s_1s_6s_1$ generate $S_3$ as $(s_1s_6)^3 = 1$ and thus we obtained the result.
\end{proof}

\section{On normalization of intertwining operators}\label{section:7}

\subsection{On central morphisms}

Let $\bG$ and $\bG'$ be reductive groups over $F$.
Assume that we have a central morphism $p \colon \bG' \to \bG$ which induces the isogeny $\bG'_{\der} \to \bG_{\der}$.

\begin{lemma}\label{lemma:central-morphism-image-closed} 
    The image $p(G') \subset G$ is closed with respect to the $p$-adic topology.
\end{lemma}

\begin{proof}
    The map $p$ decomposes into $\bG' \to p(\bG') \to \bG$.
    The first map is smooth and hence open on $F$-points.
    Thus, the map $\bG'(F) \to p(\bG')(F)$ has an open image and thus it has a closed image. 
    The second map is a closed immersion.
    This completes the proof.
\end{proof}

For any Levi subgroup $M$ of $G$, we write the corresponding Levi subgroup of $G'$ as $M'$. 
We set $M^{\sharp} = p(M')Z(G)$. 
We define the group $P^{\sharp}$ by $P^{\sharp} = p(P')Z(G)$.
The map $\bM' \times Z(\bG) \to \bM$ is smooth and thus the group $M^{\sharp}$ is an open subgroup of $M$ as the image of $M' \times Z(G)$.
The group $p(M') \cap Z(G) = p(Z(G'))$ is a closed subgroup of both $Z(G)$ and $p(M')$ by \cref{lemma:central-morphism-image-closed}.

\begin{lemma}\label{lemma:identification-character-space}
    We have an isomorphism 
    \begin{align*}
          \fa^{G, *}_{M} \xrightarrow{\sim}  \fa^{G', *}_{M'}
    \end{align*}
    which is induced by the map $X^*(A_M/A_G) \to X^*(A_{M'}/A_{G'})$.
\end{lemma}

\begin{proof}
    Both spaces are generated by the relative roots $\Phi(G, A_M) \xrightarrow{\sim} \Phi(G', A_{M'})$.
\end{proof}

\subsection{Normalizing factors}\label{subsection:normalizing-factors}

Let $G$ be a reductive group over $F$ and let $M$ be a Levi subgroup over $F$ of $G$.
Let $P = M N_P, Q = M N_{Q} \in \cP^G(M)$ be parabolic subgroups.
Let $(\sigma, W)$ be a tempered representation of $M$ of finite length. 
We denote an element of $\fa^{G, *}_{M, \C}$ by $\lambda$.

We recall the facts on unnormalized intertwining operators; see \cite{Wal03}*{IV.1, IV.2} for details.
We define the unnormalized intertwining operator 
\begin{align*}
    J_{Q \vert P}(\sigma_{\lambda}) \colon i^G_{P}(\sigma_{\lambda}) \to i^{G}_{Q}(\sigma_{\lambda})
\end{align*}
by the absolute convergent integral
\begin{align*}
    (J_{Q \vert P}(\sigma_{\lambda})\phi)(g) 
    = 
    \int_{N_P \cap N_Q \backslash N_{Q}} \phi(ng) dn
\end{align*}
for $g \in G$, when $\langle \mathrm{Re}(\lambda), \alpha^{\vee} \rangle \gg 0$ for $\alpha \in \Phi(A_M, P) \cap \Phi(A_M, \overline{Q})$. 
If we fix a special maximal compact subgroup $K$ which is in a good position relative to $M$, then the representations $i^G_{P}(\sigma_{\lambda})$ are realized in a fixed space.
The operator $J_{Q \vert P}(\sigma_{\lambda})$ is meromorphically continued to $\fa^{G, *}_{M, \C}$.

\begin{definition}\label{definition:normalizing-factor}
    A normalizing factor 
    \begin{align*}
        r_{Q \vert P} (\sigma_{\lambda})
    \end{align*}    
    is a meromorphic function on the orbit of $\sigma$ under the action of $\fa^{G, *}_{M, \C}$ which satisfies conditions $(R_1)$ to $(R_7)$ of \cite{KuokFaiLi2014-R-group}*{Definition 2.2.1} for the tempered representation $\sigma$ and $\lambda \in \fa^{G, *}_{M, \C}$.
    We define the normalized intertwining operator with respect to the normalizing factor $r_{Q \vert P} (\sigma_{\lambda})$, by 
    \begin{align*}
        R_{Q \vert P}(\sigma_{\lambda}) = r_{Q \vert P} (\sigma_{\lambda})^{-1} J_{Q \vert P}(\sigma_{\lambda}).
    \end{align*}
    This operator is holomorphic and unitary at $\lambda = 0$.
\end{definition}

\begin{remark}\label{remark:reduction-corank-one}
    By the reduction of Arthur \cite{Art89IntResI} which is also explained in \cite{KuokFaiLi2014-R-group}*{Remark 2.2.3}, we can reduce the computation and construction of normalizing factors to the corank-one cases, i.e., the cases in which we have $\dim_{\R}(\fa^{G, *}_{M}) = 1$. 
    For example, the normalizing factor $r_{Q \vert P}(\sigma_{\lambda})$ is obtained by
    \begin{align*}
        r_{Q \vert P} (\sigma_{\lambda})
        = 
        \prod_{\beta \in 
        \Phi(P, A_M)_{\red}
        \cap
        \Phi(\overline{Q}, A_M)_{\red}
        }
        r^{M_{\beta}}_{\overline{P^{M_{\beta}}} \vert P^{M_{\beta}}} (\sigma_{\lambda}), 
    \end{align*}
    where $r^{M_{\beta}}_{\overline{P^{M_{\beta}}} \vert P^{M_{\beta}}} (\sigma_{\lambda})$ is a normalizing factor for the triple $M_{\beta}, M, \sigma_{\lambda}$.
\end{remark}

\begin{definition}\label{definition:rho-Q-vert-P}
    We take $P, Q \in \cP^G(M)$.
    We denote the representation of ${}^L M$ on the space $\widehat{\fn_Q}/\widehat{\fn_Q} \cap \widehat{\fn_P}$ by $\rho_{Q \vert P}$.
\end{definition}

\subsection{R-groups}

In this section, we review the theory of R-groups, following \cite{KuokFaiLi2014-R-group}*{2.3}.

Let $M$ be a Levi subgroup of $G$ and let $\sigma \in \Pi_{\disc}(M)$ be a discrete series representation.
We consider the stabilizer
\begin{align*}
    \mathrm{Stab}_{W(M)}(\sigma) = W(M)_{\sigma} = W_{\sigma}
\end{align*}
in $W(M) = W(G, M)$ of $\sigma$.
We fix $P \in \cP^G(M)$.
Then, for any $w \in W(M)_{\sigma}$, we have a map 
\begin{align*}
    R_P(w, \sigma, \psi) \colon i^{G}_{P}(\sigma) \overset{R_{w^{-1}Pw \vert P}}{\to} i^G_{w^{-1}Pw}(\sigma)  \overset{l(\widetilde{w})}{\to} i^G_{P}({}^{\widetilde{w}}\sigma).
\end{align*}
If we choose an isomorphism $\sigma(\widetilde{w}) \colon {}^{\widetilde{w}}\sigma \to \sigma$, then we obtain an isomorphism 
\begin{align*}
    R'_{P}(w, \sigma, \psi, \sigma(\widetilde{w})) \in \End_{G}(i^G_P(\sigma)).
\end{align*}

For a root $\alpha \in \Phi(P, A_M)_{\red}$, we can construct the Levi subgroup $M_{\alpha}$ of $G$, which contains $M$ as a corank-one Levi subgroup and whose unique simple root is equal to $\alpha$.
We denote the parabolic subgroup $P \cap M_{\alpha}$ by $P^{M_{\alpha}}$.
Then, the inverse of the $j$-function associated to $M_{\alpha}, M, \sigma$ is defined and is denoted by $\mu_{\alpha}(\sigma)$.
Then, there exists a root system $\Phi_{\sigma}$ on $\fa^{*}_M$ which consists of the multiples of a root $\alpha \in \Phi(P, A_M)_{\red}$ which satisfies the condition $\mu_{\alpha}(\sigma)=0$.
The Weyl group of this root system $\Phi_{\sigma}$ is denoted by $W^{0}_{\sigma}$.

\begin{definition}[\cite{KuokFaiLi2014-R-group}*{Proposition 2.3.1}]
    The R-group $R_{\sigma}$ of $\sigma$ is defined by $R_{\sigma} = W_{\sigma}/W^{0}_{\sigma}$.
    By fixing a Weyl chamber $C_{\sigma} \subset \fa^{*}_M$, we obtain the splitting 
    \begin{align*}
        s_{C_{\sigma}}  \colon R_{\sigma} \hookrightarrow W_{\sigma}, 
    \end{align*}
    whose image is the subgroup which stabilizes $C_{\sigma}$.
\end{definition}

\begin{theorem}[Harish-Chandra, Silberger \cite{Silberger1978Knapp-Stein}]
    Fix a section $s=s_{C_{\sigma}}  \colon R_{\sigma} \hookrightarrow W_{\sigma}$ as above.
    The operators $R'_P(w, \sigma, \psi, \sigma(\widetilde{w}))$ for $w \in W^0_{\sigma}$ are scalar multiplication operators.
    The operators $R'_P(s(r), \sigma, \psi, \sigma(\widetilde{w}))$ for $r \in R_{\sigma}$ form a basis of $\End_G(i^G_{P}(\sigma))$.
\end{theorem}

\begin{corollary}
    The parabolic induction $i^G_P(\sigma)$ is irreducible if and only if $R_{\sigma} = 1$.
\end{corollary}

\begin{proposition}\label{proposition:normalizing-factor-and-mu-function}
    Let $r_{\overline{P^{M_{\alpha}}} \vert P^{M_{\alpha}}}(\sigma_{\lambda})$ and $r_{{P^{M_{\alpha}}} \vert \overline{P^{M_{\alpha}}}}(\sigma_{\lambda})$ be normalizing factors for $M_{\alpha}, M, \sigma$.
    Then, the following are equivalent.
    \begin{enumerate}
        \item $\mu_{\alpha}(\sigma) = 0$.
        \item $r_{\overline{P^{M_{\alpha}}} \vert P^{M_{\alpha}}}(\sigma)^{-1} = 0$.
        \item $r_{{P^{M_{\alpha}}} \vert \overline{P^{M_{\alpha}}}}(\sigma)^{-1} = 0$.
    \end{enumerate}
\end{proposition}

\begin{proof}
    As we have 
    \begin{align*}
    \mu_{\alpha}(\sigma) = r_{\overline{P^{M_{\alpha}}} \vert P^{M_{\alpha}}}(\sigma)^{-1}r_{{P^{M_{\alpha}}} \vert \overline{P^{M_{\alpha}}}}(\sigma)^{-1}
    \end{align*}
    and 
    \begin{align*}
        r_{\overline{P^{M_{\alpha}}} \vert P^{M_{\alpha}}}(\sigma) 
        = 
        \overline{r_{{P^{M_{\alpha}}} \vert \overline{P^{M_{\alpha}}}}(\sigma)},
    \end{align*}
    We obtain the result.
\end{proof}

\subsection{Central morphisms and normalizing factors}

Let $\bG$ and $\bG'$ be reductive groups over $F$.
Assume that we have a central morphism $p \colon \bG' \to \bG$ which induces the isogeny $\bG'_{\der} \to \bG_{\der}$.

For any Levi subgroup $M$ of $G$, we write the corresponding Levi subgroup of $G'$ as $M'$. 
Recall that we set $M^{\sharp} = p(M')Z(G)$. 
Then, we can define the parabolic induction functor $i^{G^{\sharp}}_{P^{\sharp}}$ and the unnormalized intertwining operator $J_{Q^{\sharp} \vert P^{\sharp}}$ in the same way as $i^G_P$ and $J_{Q \vert P}$. 

Giving a representation $\sigma^{\sharp}$ of $M^{\sharp}$ is equivalent to giving representations $(\sigma', \chi)$ of $M'$ and $Z(G)$ such that, for any $z' \in Z(G')$, we have $\sigma'(z') = \chi(p(z'))$, since we have $p(G') \cap Z(G) = p(M') \cap Z(G) = p(Z(G'))$.

\begin{lemma}\label{lemma:condition-existence-extension-to-M}
    Let $\sigma'$ be an irreducible representation  of $M'$.
    Then, the following are equivalent:
    \begin{enumerate}
        \item  There exists a character $\chi \colon Z(G) \to \C^{\times}$ such that for any $z' \in Z(G')$ we have $\sigma'(z') = \chi(p(z'))$.
        \item  We have $\sigma'(z') = 1$ for all $z' \in \Ker(p) \subset Z(M')$.
    \end{enumerate}

    If the central character of $\sigma'$ is unitary, then we can take $\chi$ as a unitary character.
\end{lemma}

\begin{proof}
    The sufficiency is trivial.
    We prove the necessity.
    If $\sigma'$ satisfies the condition of the statement, then the representation $\sigma'$ can be seen as the representation of $p(M')$.
    The restriction of the central character of $Z(p(M'))$ to $p(Z(G'))$ can be extended to the character $\chi$ of $Z(G)$, as the subgroup $p(Z(G')) \subset Z(G)$ is closed by \cref{lemma:central-morphism-image-closed}.
    This completes the proof.
\end{proof}

Let $\sigma^{\sharp}$ be a tempered irreducible representation of $M^{\sharp}$.
This corresponds to a pair of representations $(\sigma', \chi)$ as above.
The induction $\ind^{M}_{M^{\sharp}}(\sigma^{\sharp})$ is semisimple and of finite length. 
Let $\sigma \subset \ind^{M}_{M^{\sharp}}(\sigma^{\sharp})$ be one of its irreducible components. 
By the Frobenius reciprocity, we have an injection
\begin{align*}
    \sigma^{\sharp} \hookrightarrow \sigma_{\restriction{M^{\sharp}}} .
\end{align*}

Then, we have the following diagram:
% https://q.uiver.app/#q=WzAsNSxbMCwwLCJcXG1hdGhybXtpbmR9Xkdfe0dee1xcc2hhcnB9fShpXntHXntcXHNoYXJwfX1fe1Bee1xcc2hhcnB9fShcXHNpZ21hXntcXHNoYXJwfSkpIl0sWzQsMCwiaV57R157XFxzaGFycH19X3tQXntcXHNoYXJwfX0oXFxzaWdtYV57XFxzaGFycH0pIl0sWzAsMiwiaV5HX3tQfShcXG1hdGhybXtpbmR9Xk1fe01ee1xcc2hhcnB9fShcXHNpZ21hXntcXHNoYXJwfSkpIl0sWzIsMiwiaV5HX1AoXFxzaWdtYSkiXSxbNCwyLCJpXntHXntcXHNoYXJwfX1fe1Bee1xcc2hhcnB9fShcXHNpZ21hX3tcXHJlc3RyaWN0aW9ue01ee1xcc2hhcnB9fX0pIl0sWzEsMCwiXFxtYXRocm17aW5kfV5HX3tHXntcXHNoYXJwfX0iLDAseyJzdHlsZSI6eyJib2R5Ijp7Im5hbWUiOiJzcXVpZ2dseSJ9fX1dLFsyLDMsIiIsMCx7InN0eWxlIjp7ImhlYWQiOnsibmFtZSI6ImVwaSJ9fX1dLFszLDQsIlxcdGV4dHJte3Jlc3RyaWN0aW9ufSJdLFsxLDQsIiIsMix7InN0eWxlIjp7InRhaWwiOnsibmFtZSI6Imhvb2siLCJzaWRlIjoiYm90dG9tIn19fV0sWzAsMiwiXFxjb25nIl1d
\[\begin{tikzcd}
	{\mathrm{ind}^G_{G^{\sharp}}(i^{G^{\sharp}}_{P^{\sharp}}(\sigma^{\sharp}))} &&&& {i^{G^{\sharp}}_{P^{\sharp}}(\sigma^{\sharp})} \\
	\\
	{i^G_{P}(\mathrm{ind}^M_{M^{\sharp}}(\sigma^{\sharp}))} && {i^G_P(\sigma)} && {i^{G^{\sharp}}_{P^{\sharp}}(\sigma_{\restriction{M^{\sharp}}})}
	\arrow["\varphi", from=1-1, to=3-1]
	\arrow["{\mathrm{ind}^G_{G^{\sharp}}}", squiggly, from=1-5, to=1-1]
	\arrow[hook', from=1-5, to=3-5]
	\arrow[two heads, from=3-1, to=3-3]
	\arrow["{\textrm{restriction}}", from=3-3, to=3-5].
\end{tikzcd}\]
Here, the morphism $\varphi$ is the functorial isomorphism constructed in a similar way to \cref{lemma:associativity-induction} (2).
Also, as in the proof of \cref{lemma:canonical-isom-of-quotient-Levi}, we have 
\begin{align*}
    P \backslash G \xrightarrow{\sim} P^{\sharp} \backslash G^{\sharp}. 
\end{align*}
Thus, the restriction map in the diagram is also an isomorphism.
The following lemma follows from a simple computation.

\begin{lemma}
    The isomorphism $\varphi$ in the above diagram satisfies
    \begin{align*}
        J_{Q \vert P}(\ind^M_{M^{\sharp}}(\sigma^{\sharp}_{\lambda})) \circ \varphi
        =
        \varphi \circ \ind^{G}_{G^{\sharp}}(J_{Q^{\sharp} \vert P^{\sharp}}(\sigma^{\sharp}_{\lambda}))
    \end{align*}
    for any $\lambda \in \fa^{G, *}_{M, \C}$.
\end{lemma}

%Just apply \ind^{G}_{G^{\sharp}} i^{G^{\sharp}}_{P^{\sharp}} \xrightarrow{\sim} i^{G}_{P^{\sharp}} and so on.

\begin{corollary}\label{corollary:central-morphism-normalizing-factor-compatible}
    A family $r_{Q' \vert P'}(\sigma'_{\lambda})$ is a normalizing factor for $G', M', \sigma'$ if and only if it is a normalizing factor for $G, M, \sigma$ through the isomorphism in \cref{lemma:identification-character-space}. 
    In this case, we set 
    \begin{align*}
        r_{Q \vert P}(\sigma_{\lambda}) = r_{Q' \vert P'}(\sigma'_{\lambda}).
    \end{align*}
\end{corollary}

\section{Relation between adjoint $E_6$ and $\PGSO_8$ with triality}\label{section:8}

\subsection{Preliminaries on adjoint $E_6$}

\begin{lemma}\label{lemma:center-connected}
    Let $G$ be a split adjoint group over $F$.
    Then, for any Levi subgroup $L$ of $G$, the center $Z(L)$ is connected and split, and hence it is equal to $A_L$. 
    The character group $X^*(A_L)$ has the basis $\Delta_Q$, where $Q$ is any parabolic subgroup $Q \in \cP^G(L)$.
\end{lemma}

\begin{proof}
    Let $(B, T)$ be a Borel pair over $F$ of $G$ such that $T \subset L$.
    Since the group $\bG$ is adjoint and split, the maximal torus $\bT$ is split and the character group $X^*(\bT)$ is generated by the set of simple roots $\Delta$.
    Then, the character group $X^*(Z(L))$ is generated by the image of simple roots in $\Delta \setminus \Delta^L$, where $\Delta^L$ is the set of simple roots of $L$ with respect to $T$.
    Hence, the group $Z(L)$ is connected and split.
    The second assertion follows easily from this point of view.
\end{proof}

Let $\bG$ be the split semisimple adjoint group of type $E_6$ over $F$.
We fix a Borel pair $(B, T)$ of $G$.
We denote the Dynkin diagram of $\bG$ as follows:
% https://q.uiver.app/#q=WzAsNixbMCwxLCJcXGFscGhhXzEiXSxbMSwxLCJcXGFscGhhXzMiXSxbMiwxLCJcXGFscGhhXzQiXSxbMywxLCJcXGFscGhhXzUiXSxbNCwxLCJcXGFscGhhXzYiXSxbMiwwLCJcXGFscGhhXzIiXSxbMCwxLCIiLDAseyJzdHlsZSI6eyJoZWFkIjp7Im5hbWUiOiJub25lIn19fV0sWzEsMiwiIiwwLHsic3R5bGUiOnsiaGVhZCI6eyJuYW1lIjoibm9uZSJ9fX1dLFsyLDMsIiIsMCx7InN0eWxlIjp7ImhlYWQiOnsibmFtZSI6Im5vbmUifX19XSxbMyw0LCIiLDAseyJzdHlsZSI6eyJoZWFkIjp7Im5hbWUiOiJub25lIn19fV0sWzUsMiwiIiwwLHsic3R5bGUiOnsiaGVhZCI6eyJuYW1lIjoibm9uZSJ9fX1dXQ==
\[\begin{tikzcd}
	&& {\alpha_2} \\
	{\alpha_1} & {\alpha_3} & {\alpha_4} & {\alpha_5} & {\alpha_6}
	\arrow[no head, from=1-3, to=2-3]
	\arrow[no head, from=2-1, to=2-2]
	\arrow[no head, from=2-2, to=2-3]
	\arrow[no head, from=2-3, to=2-4]
	\arrow[no head, from=2-4, to=2-5].
\end{tikzcd}\]
We denote the root $\sum_{i=1, \dots, 6} n_i \alpha_i$ by the labeled diagram
\[
  \esixroot{n_1}{n_2}{n_3}{n_4}{n_5}{n_6}
\].

The highest root is equal to 
\begin{align*}
    \widetilde{\alpha} 
    = \alpha_1 + 2\alpha_2 + 2\alpha_3 + 3\alpha_4 + 2\alpha_5 + \alpha_6 
    =
    \esixroot{1}{2}{2}{3}{2}{1}.
\end{align*}

Let $M$ be the standard Levi subgroup of type $D_4$.
Then, we have an isomorphism $\bM_{\der} \xrightarrow{\sim} \Spin_8$.
For any root $\alpha \in \Phi(G, T)$, let $\alpha_M$ denote its restriction to $A_M$.

\begin{lemma}
     The set of roots $\Phi(G, A_M)$ consists of $6$ elements. 
     Also, this set forms a root system in $\fa^{G, *}_M$ with the Weyl group $W(G, M) \xrightarrow{\sim} S_3$.
\end{lemma}

\begin{proof}
    From the description of the highest root as above, we can see that the set of positive roots is given by 
    \begin{align*}
        \alpha_{1, M}, \alpha_{6, M}, \widetilde{\alpha}_{M}.
    \end{align*}
    The group $M$ is the only standard Levi subgroup of $G$ of type $D_4$.
    Thus, the minimal length element in the coset $s_{\alpha_i} W(M, T)$, for each $i=1, 6$, preserves the Levi subgroup $M$.
    This implies that $\Phi(G, A_M)$ is a root system in $\fa^{G, *}_M$ with the Weyl group $W(G, M)$.
\end{proof}

\begin{remark}
    We have the simple reflections $s_1$ and $s_6$ corresponding to $\alpha_{1, M}$ and $\alpha_{6, M}$. 
    They are given by the minimal length elements in the coset $s_{\alpha_i} W(M, T)$ for $i = 1, 6$.
    We have 
    \begin{align*}
        s_1 \alpha_{6, M} = s_6 \alpha_{1, M} = \widetilde{\alpha}_{M}.
    \end{align*}
\end{remark}

\begin{lemma}
     The map 
     \begin{align*}
        W(G, M) \to \Aut_{F}(\bM) \to \Out_{F}(\bM_{\ad}) \xrightarrow{\sim} S_3
     \end{align*}
     is an isomorphism.
     We can take this isomorphism in the way that the elements $s_1$ and $s_6$ in $W(G, M)$ map to the automorphism $s_1$ and $s_6$ in $\Out_{F}(\bM_{\ad})$ respectively; see \cref{definition:automorphism-of-D4}.
\end{lemma}

\begin{proof}
    It is easy to see that the action of $W(G, M)$ on $X^*(Z(M))$ is isomorphic to the standard representation of $S_3$ on $\Z^2$.
    Thus, the action on $Z(M_{\der}) \cong \mu_2^2$ is also given by the standard representation, which is faithful.
    The lemma follows from this observation.
\end{proof}

Let $W(G, M)_{ev}$ be a unique normal subgroup of $W(G, M)$ of index $2$.
Let $\H$ be the split semisimple group of type $G_2$.

\begin{corollary}
    We have $\bM_{\ad}^{W(G, M)_{ev}} \xrightarrow{\sim} \H$. 
    Here, the action of $W(G, M)_{ev}$ is given by Tits liftings.
\end{corollary}

For each root $\alpha \in \Phi(A_M, G)_{\red}$, we have a Levi subgroup $M_{\alpha} \subset G$, which contains $M$ as a Levi subgroup of corank-one.

\begin{lemma}\label{lemma:M-alpha-isogeny-GSpin10}
    We have an isomorphism 
    \begin{align*}
        \bM_{\alpha} \xrightarrow{\sim} (\GL_1 \times \Spin_{10}) / \mu_4.
    \end{align*}
    Thus, we have an isogeny
    \begin{align*}
        \bM_{\alpha} \twoheadrightarrow \PGSO_{10} \times \GL_1.
    \end{align*}
\end{lemma}

\begin{proof}
    We may assume that $\alpha_M = \alpha_{6, M}$.
    By computing the character group $X^*(T \cap (M_{\alpha})_{\der})$, we can see that the center of $(M_{\alpha})_{\der}$ is isomorphic to $\mu_4$.
    In fact, we have 
    \begin{align*}
        X^*(T \cap (M_{\alpha})_{\der}) \xrightarrow{\sim} \bigoplus_{i=1, \dots, 6} \Z\alpha_i / \Z(\esixroot{2}{3}{4}{6}{5}{4})
    \end{align*}
    and 
    \begin{align*}
        X^*(Z((M_{\alpha})_{\der})) \xrightarrow{\sim} \bigoplus_{i=1, \dots, 6} \Z\alpha_i / \{\Z(\esixroot{2}{3}{4}{6}{5}{4}) + \bigoplus_{i=1, \dots, 5} \Z \alpha_i\}.
    \end{align*}
    Since the center of $M_{\alpha}$ is connected by \cref{lemma:center-connected}, the intersection $A_{M_{\alpha}} \cap (M_{\alpha})_{\der}$ equals the center of $(M_{\alpha})_{\der}$. 
    This completes the proof.
\end{proof}

\begin{remark}\label{remark:fix-isomorphism-PGSO8}
    We note that we do not have the canonical isomorphism $M_{\der} \xrightarrow{\sim} \Spin_8(F)$.
    Thus, we fix one such isomorphism in the following way.
    We fix the isomorphism $\Spin_{8} \to \bM_{\der}$ such that the action of $\Spin_8$ on the space $\Lie(N_{\widetilde{\alpha}_M})$ is the standard representation of $\Spin_8$.
    Then, the adjoint action of $M_{\der}$ on the spaces $N_{\alpha_{1, M}}$ and $N_{\alpha_{6, M}}$ gives spin representations.
\end{remark}

\subsection{Notes on root systems and their Weyl groups}\label{subsection:notes-on-Weyl-group}

Let $\Phi = (\Phi, \Phi^{\vee}, V)$ be a root system.  
Let $\Phi_1$ be a root subsystem of $\Phi$.
We denote the Weyl group of $\Phi$ (resp. $\Phi_1$) by $W$ (resp. $W_1$).
We fix a set of simple roots $I_1$ of $\Phi_1$.

In this section, we review the computation of the normalizer $N_{W}(W_1) = \Norm_{W}(W_1)$ of $W_1$ in $W$, following \cite{Carter1972-Conjugacy-Weyl-group}.

\begin{lemma}[\cite{Carter1972-Conjugacy-Weyl-group}*{Proposition 28}]\label{lemma:Weyl-group-normalizer}
    Let $\Phi_2$ be the root subsystem in $\Phi_1^{\perp} =  \bigcap_{\alpha \in \Phi_1} \Ker(\alpha^{\vee})$ defined as $\Phi_1^{\perp} \cap \Phi$.
    \begin{enumerate}
        \item     
        The Weyl group $W_2$ of $\Phi_2$ is contained in $N_{W}(W_1)$.  
        The element $w \in W$ is in $W_2$ if and only if it fixes the roots in $\Phi_1$.
        \item    
        The product $W_1 \times W_2$ is a normal subgroup of $N_{W}(W_1)$. 
        The quotient $N_{W}(W_1)/(W_1 \times W_2)$ is isomorphic to the group of automorphisms $\Aut_{W}(I_1)$ of the Dynkin diagram of $\Phi_1$ induced by $W$.
    \end{enumerate}
\end{lemma}

 \begin{lemma}\label{lemma:computation-normalizer}
     We fix a set of simple roots $I_2$ of $\Phi_2$.
     Then, the normalizer $N_{W}(I_1, I_2) = \Norm_{W}(I_1) \cap \Norm_{W}(I_2)$ of the subsets $I_1$ and $I_2$ in $W$ gives a splitting of the map 
     \begin{align*}
          N_{W}(W_1) \twoheadrightarrow  N_{W}(W_1)/(W_1 \times W_2).
     \end{align*}
     Thus, we have an isomorphism
     \begin{align*}
        W_2 \rtimes N_{W}(I_1, I_2) \xrightarrow{\sim} N_{W}(W_1)/W_1
     \end{align*}
     as subquotients of $W$.
 \end{lemma}

 \begin{proof}
    Let $w$ be an element of $N_{W}(I_1, I_2)$.
    Then, we have $wI_1 = I_1$ and $w W_1 w^{-1} = W_1$.
    Thus, we have an inclusion $N_{W}(I_1, I_2) \subset N_{W}(W_1)$.
    Of course, we also have $w W_2 w^{-1} = W_2$ in the above situation.

    We will show the injectivity of the map 
    \begin{align*}
        W_2 \rtimes N_{W}(I_1, I_2) \xrightarrow{\sim}
         N_{W}(W_1)/W_1.
    \end{align*}
    We take $(w_2, w) \in W_2 \rtimes N_{W}(I_1, I_2)$ which maps to the identity in $N_{W}(W_1)/W_1$.
    Then, we have $w_1 w_2 w  = 1$ for some $w_1 \in W_1$.
    Acting with the left-hand side on $I_1$, we have
    \begin{align*}
        w_1w_2w(I_1) = w_1(I_1)
    \end{align*}
    and $w_1 = 1$.
    Similarly, we have $w_2 = 1$.
    Thus, we have proved the injectivity of the map.

    We will prove the surjectivity.
    Let $w \in N_W(W_1)$ be any element.
    Then, by translating by an element of $W_1$, we may assume that $wI_1 = I_1$. 
    In this situation, $w$ also normalizes $\Phi_2$.
    Thus, there exists an element $w_2 \in W_2$ such that $w_2 w I_2 = I_2$ and $w_2 w I_1 = w_2 I_1 = I_1$.
    Thus, we obtained the result.
 \end{proof}

\begin{remark}\label{remark:levi-weyl}
    Let $\bG$ be a reductive group over $F$.
    Let $(P_0, A_0)$ be a minimal parabolic pair of $G$ over $F$.
    Then, we have the root system $\Phi(G, A_0)$.
    Let $M$ be a Levi subgroup of $G$ containing $A_0$.
    Then, the set of roots $\Phi(M, A_0)$ is a root subsystem of $\Phi(G, A_0)$.
    In this case, there exists a canonical isomorphism
    \begin{align*}
        N_{W(G, A_0)}(W(M, A_0)) = W(G, M).
    \end{align*}
    Thus, we can apply \cref{lemma:Weyl-group-normalizer,lemma:computation-normalizer} to this situation.
\end{remark}

\subsection{Relation between splittings and Tits liftings}

Let $\mathbf{spl}_G = (T, B, \{X_{\alpha}\}_{\alpha})$ be an $F$-splitting of $G$.
This gives the splitting $\mathbf{spl}_M$ of $M$ which can be identified with the splitting of $M_{\ad}$.
We denote it by $\mathbf{spl}_{M_{\ad}}$. 

In this section, we prove the following.
\begin{proposition}\label{proposition:Tits-lifting-preserves-splitting}
    The Tits lifting $\widetilde{w}$ of any element $w \in W(G, M)$ preserves the splitting $\mathbf{spl}_{M}$.
\end{proposition}

\begin{remark}
    As any two $F$-splittings of $G$ are $G(F)$-conjugate, it suffices to show the proposition for an $F$-splitting of $G$.
\end{remark}

\begin{lemma}
    We have 
    \begin{align*}
        \widetilde{s_1s_6} &= \widetilde{s_1}\widetilde{s_6} \\
        \widetilde{s_6s_1} &= \widetilde{s_6}\widetilde{s_1} \\
        \widetilde{s_1s_6s_1} &= \widetilde{s_1}\widetilde{s_6}\widetilde{s_1}.
    \end{align*}
\end{lemma}

\begin{proof}
    This easily follows from the computation of the length of each element.
\end{proof}

\begin{corollary}
     To prove \cref{proposition:Tits-lifting-preserves-splitting}, it suffices to check that the elements $\widetilde{s_1}$ and $\widetilde{s_6}$ preserve the splitting of $M$.
\end{corollary}

\begin{lemma}\label{lemma:reduction-splitting-preserved}
    To prove \cref{proposition:Tits-lifting-preserves-splitting}, it suffices to check the following. 
    We consider the Levi subgroup $\GL_1 \times \SO_{8}$ of $\SO_{10}$. 
    If we fix the splitting $\mathbf{spl}_{\SO_{10}}$ of $\SO_{10}$, then the Tits lifting $\widetilde{s}$ of the nontrivial element $s \in W(\SO_{10}, \GL_1 \times \SO_8)$ preserves the splitting.
\end{lemma}

\begin{proof}
    This follows from \cref{lemma:M-alpha-isogeny-GSpin10} and $\Lie(\PGSO_{10}) \xrightarrow{\sim} \Lie(\SO_{10})$.
\end{proof}

\begin{proof}[The proof of \cref{proposition:Tits-lifting-preserves-splitting}]\label{proof:Tits-lifting-preserves-splitting}

    By \cref{lemma:reduction-splitting-preserved}, it suffices to show that if we fix the splitting $\mathbf{spl}_{\SO_{10}}$ of $\SO_{10}$, then the Tits lifting $\widetilde{s}$ of the nontrivial element $s \in W(\SO_{10}, \GL_1 \times \SO_8)$ preserves our splitting. 
    We will do this by taking an explicit splitting for $\SO_{10}$.
    
    Let $(V, q)$ be the split quadratic space defining the group $\SO_{10}$. 
    We identify the quadratic form with the symmetric bilinear form on $V$.
    Let $e_1, e_2, \dots, e_9, e_{10}$ be the hyperbolic basis of $V$ such that 
    \begin{align*}
        q(e_i, e_j) = \delta_{i, 11-j}
    \end{align*}
    for all $i, j = 1, \dots, 10$.
    Then, the space spanned by the vectors $e_2, \dots, e_9$ (resp. $e_1, e_{10}$) is non-degenerate and is denoted by $W \subset V$ (resp. $W' \subset V$).
    We denote the restriction of $q$ to $W$ (resp. $W'$) by $q_W$ (resp. $q_{W'}$).
    Then, the inclusion $\GL_1 \times \SO_8 \hookrightarrow \SO_{10}$ is identified with the inclusion $\SO(W') \times \SO(W) \hookrightarrow \SO(V)$.

    We fix the maximal torus $T$ of $\SO(V)$ defined by the basis $e_1, \dots, e_{10}$ and we fix the Borel subgroup $B$ of $\SO(V)$ defined by the flag $\langle e_1 \rangle \subset \langle e_1, e_2 \rangle \subset \dots \subset V$.
    We will take an explicit splitting for the Borel pair $(B, T)$. The Lie algebra $\mathfrak{so}(W)$ can be identified with the space of the skew symmetric endomorphism of $W$.
    For $i = 1, 2, 3, 4$, we define a root vector for the simple root $\gamma_i$ in $\Phi(\SO_{10}, T)$ by 
    \begin{align*}
        X_{\gamma_i} (e_j) =
        \begin{cases}
            0,        &j \neq i+1, 11-i, \\
            e_i,      &j = i+1, \\
            -e_{10-i} &j = 11-i.
        \end{cases}
    \end{align*}
    and for $i=5$, we define a root vector for the simple root $\gamma_5$ in $\Phi(\SO_{10}, T)$ by
    \begin{align*}
        X_{\gamma_5}(e_j) = 
        \begin{cases}
            0,        &j \neq 6, 7, \\
            e_4,      &j = 6, \\
            -e_5      &j = 7.
        \end{cases}
    \end{align*}
    This defines an $F$-splitting $(T, B, \{X_{\gamma_i}\}_{i=1, \dots, 5})$.
    We denote the simple reflection corresponding to $\gamma_i$ by $\sigma_{i}$.
    Then, by \cite{Stumbo2000minimallength}*{Theorem 4}, the element $\widetilde{s}$ is equal to 
    \begin{align*}
        \widetilde{\sigma_5}\widetilde{\sigma_4}\widetilde{\sigma_3}\widetilde{\sigma_1}\widetilde{\sigma_2}\widetilde{\sigma_3}\widetilde{\sigma_4}\widetilde{\sigma_5}.
    \end{align*}
    A simple computation shows that this product is equal to the element
    \begin{align*}
        \begin{cases}
           e_1 \to -e_{10}, \\ 
           e_{10} \to -e_1, \\
           e_5 \to -e_6,  \\
           e_6 \to -e_5, \\
           e_i \to -e_i, \quad i=2, 3, 4, 7, 8, 9
        \end{cases}
   \end{align*}
   in $\SO(V)$.
   This implies the result.
\end{proof}

\begin{remark}
    This proposition also follows from \cite{Springer1984-linearalgebraicgroups}*{Proposition 9.3.5}.
\end{remark}

\begin{corollary}
    For any $w_1, w_2 \in W(G, M)$, there exists an element $z(w_1, w_2) \in Z(M)$ such that 
    \begin{align*}
        \widetilde{w_1w_2} = z(w_1, w_2) \widetilde{w_1} \widetilde{w_2}.
    \end{align*}
\end{corollary}

\subsection{Normalized intertwining operators for $D_4 \subset D_5$}

Let $\alpha$ be an element of $\Phi(G, A_M)_{\red}$.
We take a tempered representation $\sigma$ of $M$ such that the central character of $\sigma$ on $A_{M_{\alpha}}$ is trivial.
Note that the group $M_{\alpha}/A_{M_{\alpha}}$ is isomorphic to $\PGSO_{10}(F)$ by \cref{lemma:center-connected}.

By \cref{lemma:M-alpha-isogeny-GSpin10} and \cref{corollary:central-morphism-normalizing-factor-compatible}, the computation of a normalizing factors for $M_{\alpha}, M, \sigma$ is reduced to the computation of the normalizing factors for $L'=\SO_{10}(F), M'=\GL_1(F) \times \SO_{8}(F)$.
For general $\lambda \in \fa^{G, *}_{M, \C}$ and $\sigma \in \Pi_{\temp}(M)$, we set $\phi^M_{\sigma_{\lambda}} = \phi^M_{\sigma} \otimes \chi_{\lambda}$.

\begin{theorem}[\cite{Art13}*{Proposition 2.3.1}]
    Let $\sigma'$ be an irreducible tempered representation of $M'=\GL_1(F) \times \SO_{8}(F)$.
    We take $P' \in \cP^{L'}(M')$. 
    Then, the function
    \begin{align*}
       \epsilon(\frac{1}{2}, \phi^{\GL_1 \times \SO_8}_{\sigma'_{\lambda}}, \rho_{\overline{P'} \vert P'}^{\vee}, \psi)
        \gamma_A(0, \phi, \rho_{\overline{P'} \vert P'}^{\vee}, \psi)^{-1}
    \end{align*}
    is well-defined, and we can take this function as a normalizing factor $r_{\overline{P'} \vert P'}(\sigma'_{\lambda})$.
\end{theorem}

For any $\alpha \in \Phi(G, A_M)_{\red}$, there exists a unique $w_{\alpha} \in W(G, M)$ such that $w_{\alpha}(\widetilde{\alpha}_M) = \alpha$.
As before, we take the isomorphism $M_{\widetilde{\alpha}_M}/A_{M_{\widetilde{\alpha}_M}} \xrightarrow{\sim} \PGSO_{10}(F)$ such that, via this isomorphism, the action of $\Spin_8 \xrightarrow{\sim} \widehat{M}_{\mathrm{sc}}$ on the space $\Lie(N_{\widetilde{\alpha}_M})$ is equal to the standard representation of $\Spin_8$.
We fix an isomorphism $M_{\alpha}/A_{M_{\alpha}} \xrightarrow{\sim} \PGSO_{10}(F)$ by using the Tits lifting $\widetilde{w_{\alpha}}$.

\begin{corollary}\label{corollary:normalizing-factor-D5-D4}
    We can take the function 
    \begin{align*}
        \epsilon(\frac{1}{2},\phi^{M/A_{M_{\alpha}}}_{{w_{\alpha}^{-1}}(\sigma_{\lambda})}, \Std_{\GL_1 \times \GSO_8}, \psi)
        \gamma_A(0, \phi^{M/A_{M_{\alpha}}}_{{w_{\alpha}^{-1}}(\sigma_{\lambda})}, \Std_{\GL_1 \times \GSO_8}, \psi)^{-1}
    \end{align*}
    as a normalizing factor $r_{\overline{P^{M_{\alpha}}} \vert P^{M_{\alpha}}}(\sigma_{\lambda})$.
\end{corollary}

\begin{proof}
    This follows from 
    \begin{align*}
        \begin{cases}
        \widetilde{w_{\alpha}} M \widetilde{w_{\alpha}}^{-1} = M, \\
        \widetilde{w_{\alpha}} N_{{\widetilde{\alpha}_M}} \widetilde{w_{\alpha}}^{-1} = N_{\alpha}, \\
        \widetilde{w_{\alpha}} \overline{N}_{\widetilde{\alpha}_M} \widetilde{w_{\alpha}}^{-1} = \overline{N}_{\alpha}.
        \end{cases}
    \end{align*}
\end{proof}

\subsection{Normalized intertwining operators for $D_4 \subset E_6$}

Let $\sigma$ be a representation of $M$ with trivial central character.
Then, by \cref{remark:reduction-corank-one}, we can obtain a normalizing factor $r_{Q \vert P}(\sigma_{\lambda})$ in general.
In particular, we have the following result.

\begin{proposition}\label{proposition:normalizing-factor-for-s1s6}
    Let $P = MN$ be the standard parabolic subgroup of $G$ with standard Levi factor $M$. 
    Let $\sigma$ be a representation of $M$ with a trivial central character.
    We set $w = s_1 s_6$.
    Then, we have 
    \begin{align*}
        r_{w^{-1}Pw \vert P}(\sigma_{\lambda})
        = 
        r^{M_{\widetilde{\alpha}_M}}_{\overline{P^{M_{\widetilde{\alpha}_M}}} \vert P^{M_{\widetilde{\alpha}_M}}} (\sigma_{\lambda})  
        r^{M_{{\alpha_{6, M}}}}_{\overline{P^{M_{{\alpha_{6, M}}}}} \vert P^{M_{{\alpha_{6, M}}}}}(\sigma_{\lambda}).
    \end{align*}
    If we take $\lambda = 2s \rho_P$, then we have 
    \begin{align*}
        r^{M_{\widetilde{\alpha}_M}}_{\overline{P^{M_{\widetilde{\alpha}_M}}} \vert P^{M_{\widetilde{\alpha}_M}}} (\sigma_{\lambda})    
        = 
        \epsilon(\frac{1}{2} + 16s,\sigma, \Std_{\PGSO_8}, \psi)
        \gamma_A(16s, \sigma, \Std_{\PGSO_8}, \psi)^{-1} 
    \end{align*}
    and 
    \begin{align*}
        r^{M_{{\alpha_{6, M}}}}_{\overline{P^{M_{{\alpha_{6, M}}}}} \vert P^{M_{{\alpha_{6, M}}}}}(\sigma_{\lambda}) 
        = 
        \epsilon(\frac{1}{2} + 8s,{}^{s_1}\sigma, \Std_{\PGSO_8}, \psi)
        \gamma_A(8s, {}^{s_1}\sigma, \Std_{\PGSO_8}, \psi)^{-1}.
    \end{align*}  
\end{proposition}

\begin{proof}
    The first statement follows from \cref{remark:reduction-corank-one}.
    Since we have 
    \begin{align*}
        2\rho_P = 16 \widetilde{\alpha}_M,
    \end{align*}
    we obtain the first computation from \cref{corollary:normalizing-factor-D5-D4}.
    For the second computation, we have
    \begin{align*}
        s_1(2 \rho_P) = 16 \alpha_{6, M} = 8 \widetilde{\alpha}_M + 24(\varpi_{\alpha_6}-\varpi_{\alpha_1}) \in \fa^{M_{\widetilde{\alpha}}, *}_{M} \oplus \fa^{*}_{M_{\widetilde{\alpha}}}
    \end{align*}
    since the roots $\alpha_{1, M}$ and $\alpha_{6, M}$ spans a root system of type $A_2$ in $\fa^{*}_M$.
    Thus, we obtain the second computation by \cref{corollary:normalizing-factor-D5-D4}.
\end{proof}

\begin{remark}\label{remark:replace-arthur-by-artin}
    The ratio of  $\gamma$-factors 
    \begin{align*}
        \frac{
        \gamma_A(s, \sigma, \Std_{\PGSO_8}, \psi)}
        {\gamma(s, \sigma, \Std_{\PGSO_8}, \psi)}
    \end{align*}
    converges to $1$, as $s$ goes to $0$ if $\sigma$ is a tempered representation of $M_{\ad} = \PGSO_8$.
\end{remark}

\section{On twisted endoscopy for $\PGSO_8$ with triality and $G_2$}\label{section:9}

Let $\bG$ be the split semisimple adjoint group of type $E_6$ over $F$.
Let $\bP = \bM \ltimes \N$ be the standard parabolic subgroup of $\bG$ of type $D_4$.

\subsection{The setting and explicit $L$-embeddings from endoscopic groups}\label{subsection:setting-and-L-embeddings}

In this section, we will give explicit description of the $L$-embeddings related to the functorial lifts from $G_2$.
To fix the notation, we will give an explicit model of $\PGSO_{8}$.

\begin{construction}\label{construction:explicit-model-D_4}
    Let $k$ be the $p$-adic field $F$ or $\C$.

    \begin{enumerate}
    \item
    Let $(W, q_W)$ be the split quadratic space over $k$ of dimension $8$.
    Let $\{ f_1, \dots, f_8 \}$ be a hyperbolic basis such that 
    \begin{align*}
        q_{W}(f_i, f_j) = \delta_{i, 9-j}.
    \end{align*}
    
    \item
    We fix a Borel pair of $\SO(W, q_W)$ defined as follows.
    Let $T_{\SO(W, q_W)}$ be the maximal torus corresponding to the ordered basis $f_1, \dots, f_8$.
    Let $B_{\SO(W, q_W)}$ be the Borel subgroup corresponding to the isotropic flag 
    \begin{align*}
        \langle f_1 \rangle \subset \langle f_1, f_2 \rangle \subset \dots \subset \langle f_1, f_2, f_3, f_4 \rangle.
    \end{align*}
    This also gives a Borel pair for $\Spin(W, q_W)$ or $\PGSO(W, q_W)$.
    The basis $\{ f_1, \dots, f_8 \}$ gives an isomorphism 
    \begin{align*}
        W(D_4) \xrightarrow{\sim} (\Z/2\Z)^4_0 \rtimes S_4, 
    \end{align*}
    where $(\Z/2\Z)^4_0$ is the subgroup of $(\Z/2\Z)^4$ which consists of the elements $(a_1, a_2, a_3, a_4)$ with $\sum_{i=1, 2, 3, 4} a_i = 0$.
    
    \item 
    We denote the Dynkin diagram of type $D_4$ by 
    % https://q.uiver.app/#q=WzAsNCxbMSwwLCJcXGJldGFfMSJdLFsxLDEsIlxcYmV0YV8yIl0sWzAsMiwiXFxiZXRhXzMiXSxbMiwyLCJcXGJldGFfNCJdLFswLDEsIiIsMCx7InN0eWxlIjp7ImhlYWQiOnsibmFtZSI6Im5vbmUifX19XSxbMSwyLCIiLDAseyJzdHlsZSI6eyJoZWFkIjp7Im5hbWUiOiJub25lIn19fV0sWzEsMywiIiwwLHsic3R5bGUiOnsiaGVhZCI6eyJuYW1lIjoibm9uZSJ9fX1dXQ==
    
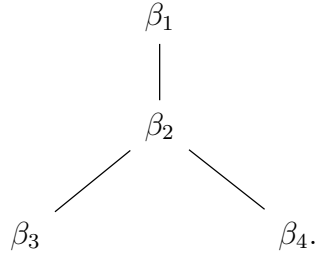
\[
    \begin{tikzcd}
	& {\beta_1} \\
	& {\beta_2} \\
	{\beta_3} && {\beta_4}
	\arrow[no head, from=1-2, to=2-2]
	\arrow[no head, from=2-2, to=3-1]
	\arrow[no head, from=2-2, to=3-3].
    \end{tikzcd}
    \]
    We will give an explicit splitting $\{ X_{\beta_i} \}_{\{i=1, 2, 3, 4\}}$  with the explicit Borel pair above.
    We identify $\mathfrak{so}(W, q_W)$ with the space of the skew-symmetric endomorphisms of $W$. 
    For $i=1,2,3$, we define 
    \begin{align*}
        X_{\beta_{i}} (e_j) =
        \begin{cases}
            0 \quad j \neq i+1, 9-i, \\
            e_j \quad j = i+1, \\
            -e_{8-i} \quad j=9-i
        \end{cases}
    \end{align*}
    and for $i=4$, we set 
    \begin{align*}
        X_{\beta_{4}} (e_j) = 
        \begin{cases}
            0 \quad j=1, 2, 3, 4, 7, 8, \\
            e_3 \quad j=5, \\
            -e_4 \quad j=6. 
        \end{cases}
    \end{align*}
    \item    
    We use the same notation as \cref{definition:automorphism-of-D4} for the group of automorphisms $\Aut(D_4) \xrightarrow{\sim} S_3$ of $D_4$.
    This gives the splitting-preserving automorphisms of $\PGSO(W, q_W)$ and $\Spin(W, q_W)$.
    We set 
    \begin{align*}
        \tau = s_6s_1.
    \end{align*}
    Dually, we set $\tau = \widehat{\tau}^{-1} = \widehat{s_6}\widehat{s_1}$.
    Recall that we have the following diagram:
    
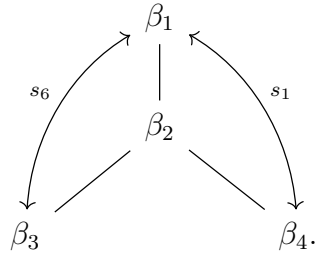
\[\begin{tikzcd}
	& {\beta_1} \\
	& {\beta_2} \\
	{\beta_3} && {\beta_4}
	\arrow[no head, from=1-2, to=2-2]
	\arrow["{s_6}"', curve={height=18pt}, tail reversed, from=1-2, to=3-1]
	\arrow["{s_1}", curve={height=-18pt}, tail reversed, from=1-2, to=3-3]
	\arrow[no head, from=2-2, to=3-1]
	\arrow[no head, from=2-2, to=3-3].
    \end{tikzcd}\] 

    \item     
    The group $\Spin(W,q_W)^{S_3}$ is the split semisimple group of type $G_2$. 
    If $k=\C$, we identify this group with $\widehat{H}$.
    Thus, we obtained the $L$-embedding
    \begin{align*}
         \iota^{M_{\ad}}_{H} \colon {}^L H \hookrightarrow {}^L M_{\ad}.
    \end{align*}
    Regarding the folding construction, the simple short root of $G_2$ corresponds to the $S_3$-orbit of the root $\beta_1$ and the simple long root of $G_2$ corresponds to the orbit of the root $\beta_2$.
    \item
    The root subsystem of type $A_2$ in the root system of type $G_2$, whose simple roots are given by $\{ \beta^{\vee}_2, \sum_{i=1,\dots,4}\beta^{\vee}_i \}$, determines the subgroup  $\SL_3 = \widehat{\PGL_3} \subset \Spin(W,q_W)^{S_3}$.
    \end{enumerate}
\end{construction}

Let $\phi^{H}$ be a tempered $L$-parameter.
We set $\phi^{M_{\ad}} = \iota^{M_{\ad}}_{H} \circ \phi^{H}$.
Let $\Std_{H} \colon {}^L H \to {}^L \GL_7$ be the $L$-embedding corresponding to the standard representation of $G_2(\C)$.

\begin{remark} \label{remark:divide-five-cases}
    We divide the discussion into five cases:
    \begin{enumerate}
        \item     
        The parameter $\phi^{H}$ is discrete and $\gamma(0, \phi^{H}, \Std_{H}, \psi) \neq 0$.
        \item   
        The parameter $\phi^{H}$ is discrete and $\gamma(0, \phi^{H}, \Std_{H}, \psi) = 0$.
        \item      
        The parameter $\phi^{H}$ factors through the short root ${}^L \GL_2 \subset {}^L H$ as a discrete $L$-parameter.
        \item     
        The parameter $\phi^{H}$ factors through the long root ${}^L \GL_2 \subset {}^L H$ as a discrete $L$-parameter.   
        \item     
        The parameter $\phi^{H}$ factors through the maximal torus ${}^L T_H \subset {}^L H$ as a discrete $L$-parameter.  
    \end{enumerate}
\end{remark}

\begin{definition}\label{definition:Levi-subgroup-notation}
    Cases $(1)$ and $(2)$ are treated in \cite{GS23LLC}*{Lemma 2.4}.
    In Cases $(3), (4), (5)$ in \cref{remark:divide-five-cases}, we write the corresponding standard Levi subgroup of $H$ as $I$.
    We also take a corresponding discrete $L$-parameter for $I$ and we denote it by $\phi^I$.
    The representation of $I$ corresponding to $\phi^I$ is written by $\sigma^I$.
\end{definition}

In the following subsections, we will give explicit description of the endoscopic lift $\Pi^{X}_{\phi^H}$ of $\phi^H$ for each case in \cref{remark:divide-five-cases}.
We also prove some irreducibility results for the parabolic induction $i^G_P(\pi)$ for each element $\pi \in \Pi^{X}_{\phi^H}$.
Note that the group $G$ is of type $E_6$ and of adjoint type.

\subsubsection{The endoscopic lift from $H$: Case (1)}

\begin{proposition}[\cite{GS23LLC}*{Lemma 2.4(ii)}]\label{proposition:description-of-endoscopic-lift-case-1}
    In Case $(1)$, the $L$-packet $\Pi^X_{\phi^H}$ consists of discrete series representations.
\end{proposition}

\begin{proof}
    It suffices to see that $\iota^{M_{\ad}}_{H} \circ \phi^H$ is discrete. 
    However, this follows from the facts that the lift $\iota^{\PGSp_6}_{H} \circ \phi^H$ is discrete by \cite{GS23LLC}*{Lemma 2.4(ii)} and that the parameter $\iota^{\GL_7}_{H} \circ \phi^H$ does not contain a trivial representation of $L_F$.
\end{proof}

\begin{proposition}\label{proposition:irreducibility-of-parabolic-induction-case-1}
    The parabolic induction $i^{G}_{P}(\pi)$ of any element $\pi \in \Pi^X_{\phi^H}$ is irreducible.
\end{proposition}

\begin{proof}
    We have $W_{\pi} = W(G, M)$ by \cref{proposition:invariance-S3}.
    For any reduced relative root $\alpha \in \Phi(G, A_M)$, the inverse of the normalizing factor $r^{M_{\alpha}}_{\overline{P^{M_{\alpha}}} \vert P^{M_{\alpha}}}(\pi)^{-1}$ is equal to $\gamma(0, \pi, \Std_{\GSO_8}, \psi) = 0$ by \cref{proposition:construction-of-distinguished-Xu-packet} and \cref{corollary:normalizing-factor-D5-D4}.
    Thus, by \cref{proposition:normalizing-factor-and-mu-function}, the group $W^0_{\pi}$ is generated by $s_1, s_6$ and it equals the group $W(G, M)$.
    Hence $R_{\pi} = W_{\pi}/W^0_{\pi} = 1$.
\end{proof}

\subsubsection{The endoscopic lift from $H$: Case (2)}

We set $H' = \PGL_3$.
This group is one of the elliptic endoscopic groups of $H$.
We have the $L$-embedding
\begin{align*}
    \iota^{H}_{H'} \colon {}^L H' \to {}^L H
\end{align*}
described in \cref{construction:explicit-model-D_4}.
In this case, the $L$-parameter $\phi^H$ can be expressed as $\iota^{H}_{H'} \circ \phi^{H'}$ for some discrete $L$-parameter $\phi^{H'}$ for $H'$.
Let $\sigma^{H'}$ be the corresponding discrete series representation of $H'$.
This representation is determined up to taking the contragredient representation from $\phi^{H}$; this follows from \cite{GS23LLC}*{Lemma 2.4(i)}.

On the other hand, we have the Levi subgroup $L/A_M$ of $M_{\ad}$ which corresponds to the simple roots $\beta_2, \beta_3$.
This group corresponds to the Levi subgroup $L \subset G$ determined by the simple roots $\alpha_3, \alpha_4$.
Let $P_L \subset P = P_M$ be the standard parabolic subgroup of $G$ with Levi factor $L$.
Note that we have the isomorphism 
\begin{align*}
    L/A_M \xrightarrow{\sim} (\GL_1 \times \GL_3 \times \GL_1)/ \GL_1,
\end{align*}
where the group $\GL_1$ is embedded by the map 
\begin{align*}
    \GL_1 \hookrightarrow \GL_1 \times \GL_3 \times \GL_1 \colon z \mapsto (z, z, z^{2}).
\end{align*}

\begin{lemma}\label{lemma:description-of-L-homomorphism-case-2}
    The composition $\iota^{M_{\ad}}_{H} \circ \iota^{H}_{H'}$ factors through $\iota^{M_{\ad}}_{L/A_M}$.
    The image of $\widehat{\PGL_3} = \SL_3$ is conjugate to the derived group of $\widehat{L/A_M}$ by a Weyl element $\cW = (1, 0, 0, 1) \rtimes (14) \in W(D_4) \cong (\Z/2\Z)^4_0 \rtimes S_4$.
\end{lemma}

\begin{proof}
    It suffices to show that in the group $\widehat{M_{\ad}}$, the group $\widehat{\PGL_3}$ is conjugate to the derived group of $\widehat{L/A_M}$ by the Weyl element $\cW = (1, 0, 0, 1) \rtimes (14) \in W(D_4)$.
    We have 
    \begin{align*}
        \cW(\beta^{\vee}_2) &= \beta^{\vee}_2 \\
        \cW(\beta^{\vee}_3) &= \sum_{i} \beta^{\vee}_i.
    \end{align*} 
    Hence we have the result.
\end{proof}

\begin{proposition}\label{proposition:description-of-endoscopic-lift-case-2}
    The direct sum of the elements in $\Pi^X_{\phi^H}$ is given by $i^{M_{\ad}}_{L/A_M}(\sigma^{L})$, where $\sigma^{L}$ is given by 
    \begin{align*}
        \sigma^{L} = \mathbf{1} \boxtimes \sigma^{H'} \boxtimes \mathbf{1},
    \end{align*}
    as the representation of $L/A_M \xrightarrow{\sim} (\GL_1 \times \GL_3 \times \GL_1)/ \GL_1$ or $L$.
    The cardinality of $\Pi^X_{\phi^H}$ is given by 
    \begin{align*}
        \begin{cases}
            1 & \text{ if } \sigma^{H'} \text{ is not self-dual}, \\
            2 & \text{ if } \sigma^{H'} \text{ is self-dual}.
        \end{cases}
    \end{align*}
\end{proposition}

\begin{proof}
    The first statement is a direct consequence of the previous lemma.
    We can prove the second statement by computing the group $S_{\phi^{M_{\ad}}}$, but we prove the statement in another way: We will compute the $R$-group of $\sigma^{L}$.

    We will apply \cref{lemma:Weyl-group-normalizer} to the set $\{ \beta_2, \beta_3 \}$.
    Then, the set $I_2$ is empty and the group $\Aut_{W}(I_1)$ is isomorphic to $\Z/2\Z$ which switches the two roots $\beta_2, \beta_3$.
    Thus, the group $W(M_{\ad}, L/A_M)$ is isomorphic to $\Z/2\Z$, see \cref{remark:levi-weyl}.
    The class representing the nontrivial element is given by the class of the element $w_0^{M_{\ad}} w_0^{L/A_M}$, where the element $w_0^{M_{\ad}}$ is the longest element of the Weyl group of $M_{\ad}$ and the element $w_0^{L/A_M}$ is the longest element of the Weyl group of $L/A_M$.
    The automorphism of $L_{\ad} \xrightarrow{\sim} \PGL_3$ induced by the element $w_0^{M_{\ad}}$ is the self-dual involution.

    Assume that the representation $\sigma^{H'}$ is not self-dual. Then, the group $W(M_{\ad}, L/A_M)_{\sigma^L}$ is trivial, and thus we obtain the result.
    Thus, we obtain the result for this case.

    If the representation $\sigma^{H'}$ is self-dual, the normalizer $W(M_{\ad}, L/A_M)_{\sigma^L}$ in the Weyl group is equal to $\Z/2\Z$.
    We will show that the group $W(M_{\ad}, L/A_M)_{\sigma^L}^{0}$ is trivial.
    By \cref{proposition:normalizing-factor-and-mu-function}, it suffices to compute the normalizing factors for $\sigma^{L}$ and each reduced relative root in $\Phi(M_{\ad}, A_L/A_M)$.
    For each root $\beta$ of $M_{\ad}$, we write its restriction to $A_L/A_M$ by $\beta_L$.
    The set of relative roots $\Phi(M_{\ad}, A_L/A_M)$ consists of the roots $\beta_{1, L}, \beta_{4, L}, \beta_{1, L} + \beta_{4, L}$ and their negations.
    By \cref{corollary:central-morphism-normalizing-factor-compatible}, we can compute each normalizing factor by replacing $M_{\ad}$ by $\SO_8$.
    It is easy to check that the reciprocals of corresponding normalizing factors are given as follows:
    \begin{align*}
        \begin{cases}
        &\gamma_A(s, \sigma^{H'}, \psi), \\
        &\gamma_A(s, \sigma^{H'}, \wedge^2, \psi), \\
        &\gamma_A(s, (\sigma^{H'})^{\vee}, \psi) = \gamma_A(s, \sigma^{H'}, \psi).
        \end{cases}
    \end{align*}
    These functions do not vanish at $s = 0$, as the representation $\sigma^{H'}$ is a discrete series representation of $\PGL_3(F)$.
    Hence, the group $W(M_{\ad}, L/A_M)_{\sigma^L}^{0}$ is trivial and we have the result.
\end{proof}

\begin{proposition}\label{proposition:irreducibility-of-parabolic-induction-case-2}
    The parabolic induction $i^{G}_{P}(\pi)$ of any element $\pi \in \Pi^X_{\phi^H}$ is irreducible.
\end{proposition}

\begin{proof}
    Recall the parabolic subgroup $P_L \subset P$ with Levi factor $L$.
    Then, the parabolic induction $i^{G}_{P}(i^{M}_{L}(\sigma^L))$ is equal to the parabolic induction $i^{G}_{P_L}(\sigma^L)$.
    Thus, it suffices to show that the $R$-group of $\sigma^L$ and $L \subset G$ is isomorphic to the $R$-group of $\sigma^L$ and $L/A_M \subset M_{\ad}$.
    We denote the former by $R_{\sigma^L}$ and the latter by $R_{\sigma^L}^{M_{\ad}}$.

    We will compute the $R$-group $R_{\sigma^L}$ by applying the results in \cref{subsection:notes-on-Weyl-group} and \cref{proposition:normalizing-factor-and-mu-function} as the previous proposition. 
    We will use the notation in \cref{subsection:notes-on-Weyl-group}.
    Of course, we will set $I_1 = \{\alpha_3, \alpha_4\}$ which  corresponds to the set of simple roots of $A_2 \subset D_4$.

    We now compute $\Phi_2$ in $\Phi = \Phi(G, T)$.
    We first note that the highest root is given by 
    \begin{align*}
        \esixroot{1}{2}{2}{3}{2}{1}.
    \end{align*}
    Also, if the root 
    \begin{align*}
        \esixroot{n_1}{n_2}{n_3}{n_4}{n_5}{n_6}
    \end{align*}
    is orthogonal to $\alpha_3, \alpha_4$, 
    then we have 
    \begin{align*}
        2n_3 &= n_1 + n_4 \\
        2n_4 &= n_2 + n_3 + n_5. 
    \end{align*}
    Thus, a simple computation using two facts above shows that the root system $\Phi_2$ in \cref{subsection:notes-on-Weyl-group} is isomorphic to $A_2 \times A_2$, and the simple roots are given by the roots 
    \begin{align*}
        \esixroot{0}{0}{0}{0}{0}{1}, \quad \esixroot{1}{0}{1}{1}{1}{0}
    \end{align*}
    and 
    \begin{align*}
        \esixroot{0}{1}{1}{2}{2}{1}, \quad \esixroot{1}{1}{1}{1}{0}{0}.
    \end{align*}
    Also, it is proved in the previous proposition that the group $\Aut_{W}(I_1)$ is isomorphic to $\Z/2\Z$ and the nontrivial element is given by the class of the element $w_0^{M_{\ad}} w_0^{L/A_M} \in W(M_{\ad}, L/A_M)$.

    We finally show that the group $W(G, L)^0_{\sigma^L}$ is isomorphic to the group $W_2 = W(A_2) \times W(A_2)$.
    This implies the proposition, as this shows that $\abs{R_{\sigma^L}} \leq 2$ and a priori we have $\abs{R_{\sigma^L}} \geq \abs{R_{\sigma^L}^{M_{\ad}}}$.
    To show this, it suffices to show that for each simple root $\gamma \in \Phi_2$, the corresponding reflection $s_{\gamma}$ gives an element of $W(G, L)^{0}_{\sigma^L}$.
    As $\gamma$ is orthogonal to $I_1$, the Levi group $L_{\gamma} \subset G$ is of type $A_1 \times A_2$.
    The first and third simple roots are contained in the root system generated by $\alpha_2, \alpha_3, \alpha_4, \alpha_5, \alpha_6$.
    The second and fourth simple roots are contained in the root system generated by $\alpha_1, \alpha_2, \alpha_3, \alpha_4, \alpha_5$.
    In the both cases, the group $L_{\gamma}$ is contained in a Levi subgroup of type $D_5 \subset E_6$.

    Now, we apply \cref{corollary:central-morphism-normalizing-factor-compatible}.
    Then, we can compute the normalizing factor for $\sigma^L$ and $\gamma$ by using $\GL_1 \times \SO_8 \subset \SO_{10}$.
    We will show that the reciprocal of the corresponding normalizing factor for each simple root is equal to
    \begin{align*}
        \gamma_A(s, \mathbf{1}, \psi).
    \end{align*}
    This implies that the class of $s_{\gamma}$ is contained in $W(G, L)^{0}_{\sigma^L}$ and we have the result. 

    For the first simple root, we consider the group $\SO_{10}$.
    The Levi subgroup corresponding to $L_{\gamma}$ (resp. $L$) is standard and isomorphic to $\GL_2 \times \GL_3$ (resp. $\GL_1 \times \GL_1 \times \GL_3$).
    The simple root $\esixroot{0}{0}{0}{0}{0}{1}$ corresponds to the simple root of $\GL_2$ and the simple roots of $\GL_3$ correspond to $\alpha_4, \alpha_3$.
    The representation corresponds to $\sigma^L$ is given by
    \begin{align*}
        \sigma^L = \mathbf{1} \boxtimes \mathbf{1} \boxtimes (\sigma^{H'})^{\vee}.
    \end{align*}
    In this case, we can easily check that the normalizing factor is given by $\gamma(s, \mathbf{1}, \psi)$.
    In the third case, the computation is similar and the root $\esixroot{0}{1}{1}{2}{2}{1}$ corresponds to the highest root of $\SO_{10}$.

    In the second case, we take a basis $e_1, \dots, e_5$ for $F^5$ and the dual basis $e_1^{\vee}, \dots, e_5^{\vee}$ for $F^5$ with respect to the quadratic form defining $\SO_{10}$.
    Then, the group $L_{\gamma}$ is isomorphic to the Levi subgroup $\GL_2 \times \GL_3$ corresponding to the subspaces spanned by  $\{e_1, e_5\}$, $\{e_1^{\vee}, e_5^{\vee}\}$, $\{e_2, e_3, e_4\}$, and $\{e_2^{\vee}, e_3^{\vee}, e_4^{\vee}\}$.
    This easily shows the result. 
    By replacing $e_5$ with $e_5^{\vee}$ and $e_1$ with $e_1^{\vee}$, we obtain the result for the fourth case, similarly.
\end{proof}

\subsubsection{The endoscopic lift from $H$: Case (3)}

In Case $(3)$, the short root of $G_2(\C)$ corresponds to the orbit of the root $\beta^{\vee}_1$ under the triality automorphism $\tau$.
Thus, the group ${}^L I$ factors through the Levi subgroup ${}^L (L/A_M) \subset {}^L M_{\ad}$, where the Levi subgroup $L/A_M$ is of type $A_1^3$ corresponding to the simple roots $\beta_1, \beta_3, \beta_4$.

We have an isomorphism
\begin{align*}
    L/A_M \xrightarrow{\sim} (\GL_2 \times \GL_2 \times \GL_2)/(\GL_1^{(1, 1, 0)} \cdot \GL_1^{(0, 1, -1)}).
\end{align*}
Also, through this isomorphism, the action of $\tau$ is given by 
\begin{align*}
    (g_1, g_2, g_3) \mapsto (g_2^{\vee}, g_3, g_1^{\vee}).
\end{align*}
On the dual side, the $L$-embedding $\iota^{L/A_M}_{I} \colon {}^L I \hookrightarrow {}^L (L/A_M)$ is given by the map
\begin{multline*}
    \GL_2 \hookrightarrow \{ (g_1, g_2, g_3) \in \GL_2^{3} \mid \det(g_1)^{-1} = \det(g_2) = \det(g_3) \}   \\
    \colon g \mapsto (g^{\vee}, g, g),
\end{multline*}
as the image is $\tau = \widehat{\tau}^{-1}$-invariant.

\begin{proposition}\label{proposition:description-of-endoscopic-lift-case-3}
    The direct sum of the elements in $\Pi^X_{\phi^H}$ is given by $i^{M_{\ad}}_{L/A_M}(\sigma^{L})$, where $\sigma^{L}$ is given by 
    \begin{align*}
        \sigma^{L} = (\sigma^I)^{\vee} \boxtimes \sigma^I \boxtimes \sigma^I,
    \end{align*}
    as the representation of $L/A_M \xrightarrow{\sim} (\GL_2 \times \GL_2 \times \GL_2)/(\GL_1^{(1, 1, 0)} \cdot \GL_1^{(0, 1, -1)})$ or $L$.

    If the representation $\sigma^I$ is not self-dual, the representation $i^{M_{\ad}}_{L/A_M}(\sigma^L)$ is irreducible.
\end{proposition}

\begin{proof}
    Again, we apply \cref{subsection:notes-on-Weyl-group} to the case $I_1 = \{ \beta_1, \beta_3, \beta_4 \}$.
    By the maximality of $L \subset M$, we have 
    \begin{align*}
        \abs{W(M_{\ad}, L/A_M)} \leq 2.
    \end{align*}
    As the set $I_2$ consists of the highest root $\widetilde{\beta}$ of the root system of $M_{\ad}$, we have 
    \begin{align*}
        \abs{W(M_{\ad}, L/A_M)} = 2,
    \end{align*} 
    and this group is generated by the class of the reflection $s_{\widetilde{\beta}}$.
    As we have $s_{\widetilde{\beta}} \sigma^{L} = (\sigma^{L})^{\vee}$, we obtain the result.
\end{proof}

\begin{proposition}\label{proposition:irreducibility-of-parabolic-induction-case-3}
    If $\sigma^I$ is not self-dual, then the parabolic induction $i^{G}_{P}(\pi)$ of any element $\pi \in \Pi^X_{\phi^H}$ is irreducible.
\end{proposition}

\begin{proof}
    By the previous proposition, the $L$-packet $\Pi^X_{\phi^H}$ is a singleton.
    Also, we have 
    \begin{align*}
        i^{G}_{P}(\pi) = i^{G}_{P_L}(\sigma^L).
    \end{align*}
    We have to prove the irreducibility of this representation.

    In this case, a similar computation to \cref{proposition:irreducibility-of-parabolic-induction-case-2} shows that we have $I_2 = \{ \sum_{i=2, 3, 5}\alpha_i + 2 \alpha_4 = \esixroot{0}{1}{1}{2}{1}{0} \}$.
    Also, we have 
    \begin{align*}
        \Aut_{W(G, L)}(I_1) \subset \Aut(I_1) =  S_3.
    \end{align*}
    However, all the elements of $S_3$ are  realized as Tits liftings of the elements of $W(G, M) \xrightarrow{\sim} \Norm_{W(G, T)}(\{\alpha_2, \alpha_3, \alpha_4, \alpha_5\})$. 
    The Tits liftings fix the root $\alpha_4$.
    Thus, we have the equality
    \begin{align*}
        W(G, L) \xrightarrow{\sim} \langle s_{\sum_{i=2, 3, 5}\alpha_i + 2 \alpha_4} \rangle \times S_3.
    \end{align*} 
    If $\sigma^I$ is not self-dual, then we have 
    \begin{align*}
        W(G, L)_{\sigma^L} \xrightarrow{\sim} S_3.
    \end{align*}
    It suffices to show that we have 
    \begin{align*}
        W(G, L)_{\sigma^L}^{0} = W(G, L)_{\sigma^L}.
    \end{align*}
    In our situation, each element of the right-hand side is represented by the Tits liftings of the elements of $W(G, M)$.
    Thus, by the same argument as \cref{proposition:irreducibility-of-parabolic-induction-case-1}, all the elements of $W(G, L)_{\sigma^L}$ are in $W(G, L)_{\sigma^L}^{0}$ if we check that the classes of Tits liftings of $s_{\alpha_{1, M}}$ and $s_{\alpha_{6, M}}$ are in $W(G, L)_{\sigma^L}^{0}$.
    However, the same computation in \cref{proposition:irreducibility-of-parabolic-induction-case-1} shows the inclusion.
    This completes the proof.
\end{proof}

\subsubsection{The endoscopic lift from $H$: Case (4)}

In Case $(4)$, we note that the group ${}^L I$ is contained in ${}^L H'$ in Case $(2)$ and $I$ is a Levi subgroup of $H'$ and $H$ simultaneously.

Let $L$ be the Levi subgroup of $G$ defined by $\alpha_4$.
Then, the group $L/A_M$ is the Levi subgroup of $M_{\ad}$ corresponding to $\beta_2$.
We have 
\begin{align*}
    L/A_M \xrightarrow{\sim} (\GL_1 \times \GL_2 \times \GL_1 \times \GL_1) / \GL_1^{(1, 1, 1, 2)}.
\end{align*}
We have the $L$-homomorphism
\begin{align*}
    \iota^{L/A_M}_{I} \colon {}^L I \hookrightarrow {}^L (L/A_M)
\end{align*}
induced by $\iota^{M_{\ad}}_{H}$.

\begin{lemma}\label{lemma:finite-index}
    Let $A$ be a split torus. 
    Let $\alpha \colon A \to \bG_m$ be a nontrivial character over $F$. 
    Then, the image $\alpha(A(F)) \subset F^{\times}$ is an open subgroup of finite index.
\end{lemma}

\begin{proof}
    This follows from the finiteness of Galois cohomology of the kernel of $\alpha$ and the implicit function theorem.
\end{proof}

\begin{proposition}\label{proposition:description-of-endoscopic-lift-case-4}
    The direct sum of the elements in $\Pi^X_{\phi^H}$ is given by $i^{M_{\ad}}_{L/A_M}(\sigma^{L})$, where $\sigma^{L}$ is given by 
    \begin{align*}
        \sigma^{L} = \mathbf{1} \boxtimes \sigma^I \boxtimes \chi_{\sigma^I}^{-1}  \boxtimes \mathbf{1},
    \end{align*}
    as the representation of $L/A_M \xrightarrow{\sim} (\GL_1 \times \GL_2 \times \GL_1 \times \GL_1)/\GL_1^{(1, 1, 1, 2)}$ or $L$.

    Let $\cO_{\sigma^I} \subset \Pi_{\disc}(I)$ be the connected component of $\sigma^I$.
    Then, there exists a dense open subset $\cO' \subset \cO_{\sigma^I}$ such that the representation $i^{M_{\ad}}_{L/A_M}(\mathbf{1} \boxtimes \rho \boxtimes \chi_{\rho}^{-1} \boxtimes \mathbf{1})$ is irreducible for $\rho \in \cO'$.
\end{proposition}

\begin{proof}
    The first statement is trivial. 
    We will prove the second statement.   
    The basis of the center $A_L/A_M$ is given by the restriction $\beta_{1, L}, \beta_{3, L}, \beta_{4, L}$ of $\beta_{1}, \beta_{3}, \beta_{4}$.
    A simple direct computation shows that the central character of $\sigma^L$ is given by 
    \begin{align*}
       a \in A_L/A_M \mapsto \chi_{\sigma^I} \circ \beta_{3, L}(a).
    \end{align*}
    Thus, if $w \in W(G, L)_{\sigma^L}$, then we have 
    \begin{align*}
        \chi_{\sigma^I} \circ (w \beta_3 - \beta_3) (a) = 1 
    \end{align*}
    for all $a \in A_L/A_M$.
    By \cref{lemma:finite-index}, if $\chi_{\sigma^I}$ is in a dense subset of $\widehat{A_I}$, then this condition holds only if $w \beta_3 = \beta_3$. 
    As we may assume $w \beta_2 = \beta_2$, now we are in a similar situation to \cref{proposition:description-of-endoscopic-lift-case-2}.

    By the same reasoning as the proof for the non-self-dual case in \cref{proposition:description-of-endoscopic-lift-case-2}, we have 
    \begin{align*}
        W(M_{\ad}, L/A_M)_{\sigma^L} = 1.
    \end{align*}
    This completes the proof.
\end{proof}

By the same argument as \cref{proposition:irreducibility-of-parabolic-induction-case-2}, we can also obtain the following result. As the proof is easier than that, we omit it.

\begin{proposition}\label{proposition:irreducibility-of-parabolic-induction-case-4}
    Let $\cO_{\sigma^I} \subset \Pi_{\disc}(I)$ be the connected component of $\sigma^I$.
    Then, there exists a dense open subset $\cO' \subset \cO_{\sigma^I}$ such that
    the parabolic induction $i^{G}_{P}(\pi)$ of any element $\pi \in \Pi^X_{\phi^H}$ is irreducible.
\end{proposition}

\subsubsection{The endoscopic lift from $H$: Case (5)}

In Case $(5)$, the group ${}^L I$ is the maximal torus of ${}^L H$.
It is also contained in ${}^L H'$ in Case $(2)$.

As the group ${}^L I$ can be seen as the maximal torus of ${}^L H'$ = ${}^L \PGL_3$, we can view $I$ as a maximal torus of $\PGL_3$.
Any character of $I$ is given by a triple $(\chi_1, \chi_2, \chi_3)$ of characters $F^{\times} \to \bS^1$ or a character of $(F^{\times})^3$, such that $\prod_{i} \chi_i = 1$.
We denote this character by $\sigma^{I}$.

Let $L = T$ be the maximal torus of $G$.

\begin{proposition}\label{proposition:description-of-endoscopic-lift-case-5}
    The direct sum of the elements in $\Pi^X_{\phi^H}$ is given by $i^{M_{\ad}}_{L/A_M}(\sigma^{L})$, where $\sigma^{L}$ is given by 
    \begin{align*}
        \sigma^{L}(t) 
        &= \chi_1(\beta_2(t)) \cdot \chi_1 \chi_2(\beta_3(t)) \\
        &= \chi_1 \circ (\beta_2 + \beta_3) (t) \cdot \chi_2 \circ \beta_3(t),
    \end{align*}
    for $t \in L/A_M$ or $t \in L$.

    Let $\cO_{\sigma^I} \subset \Pi_{\disc}(I) = I^{\cD}$ be the connected component of $\sigma^I$.
    Then, there exists a dense open subset $\cO' \subset \cO_{\sigma^I}$ such that the representation $i^{M_{\ad}}_{L/A_M}(\rho_1 \circ (\beta_2 + \beta_3) \cdot \rho_2 \circ \beta_3)$ is irreducible for $\rho = (\rho_1, \rho_2, \rho_3) \in \cO'$.
\end{proposition}

\begin{proof}
    In the statement of the proposition, we write 
    \begin{align*}
        \rho^L = \rho_1 \circ (\beta_2 + \beta_3) \cdot \rho_2 \circ \beta_3.
    \end{align*}
    The first statement follows from a simple computation.
    We utilize a similar reasoning to \cref{proposition:description-of-endoscopic-lift-case-4} using \cref{lemma:finite-index}.
    Then, for any $w \in W(M_{\ad}, L/A_M)_{\rho^L}$, we have 
    \begin{align*}
        w (\beta_2 + \beta_3) &= \beta_2 + \beta_3, \\
        w \beta_3 &= \beta_3,
    \end{align*}
    if $\rho$ lies in a dense open subset of $I^{\cD}$.
    By the same reasoning as \cref{proposition:description-of-endoscopic-lift-case-4}, we obtain the result.
\end{proof}

By the same argument as \cref{proposition:irreducibility-of-parabolic-induction-case-2}, we can also obtain the following result. As the proof is easier than that, we omit it.

\begin{proposition}\label{proposition:irreducibility-of-parabolic-induction-case-5}
    Let $\cO_{\sigma^I} \subset \Pi_{\disc}(I) = I^{\cD}$ be the connected component of $\sigma^I$.
    Then, there exists a dense open subset $\cO' \subset \cO_{\sigma^I}$ such that
    the parabolic induction $i^{G}_{P}(\pi)$ of any element $\pi \in \Pi^X_{\phi^H}$ is irreducible.
\end{proposition}

\subsection{An irreducibility theorem and an application}\label{subsection:irreducibility-theorem-and-application}

Recall that $H$ is the split semisimple group of type $G_2$ over $F$.

\begin{proposition}\label{proposition:smoothness-of-endoscopic-lift}
    The map 
    \begin{align*}
        \Temp_{\ind}(H)/\sim_{st} \to \Temp_{\ind}(M_{\ad})/\sim_{st} 
    \end{align*}
    obtained by the endoscopic lift is smooth.
\end{proposition}

\begin{proof}
    This proposition is a direct consequence of \cref{proposition:description-of-endoscopic-lift-case-2}, \cref{proposition:description-of-endoscopic-lift-case-3}, \cref{proposition:description-of-endoscopic-lift-case-4} and \cref{proposition:description-of-endoscopic-lift-case-5}.
\end{proof}

\begin{theorem}\label{theorem:generic-irreducibility-functorial-lifts-G2}
    For representations $\sigma \in \Temp_{\ind}(H)$ in a dense open subset, we have the following two results:
    \begin{enumerate}
    \item The distinguished Xu packet $\Pi^{X}_{\phi^{H}_{\sigma}}$ is a singleton if $\iota^{M_{\ad}}_{H} \circ \phi^H_{\sigma}$ is not discrete.
    \item The representations $\pi$ in the corresponding distinguished Xu packet $\Pi^{X}_{\phi^{H}_{\sigma}}$ have the irreducible parabolic induction $i^G_{P}(\pi)$ to $G$.
    \end{enumerate}
\end{theorem}

\begin{proof}
    This theorem is a direct consequence of \cref{proposition:irreducibility-of-parabolic-induction-case-1}, \cref{proposition:irreducibility-of-parabolic-induction-case-2}, \cref{proposition:irreducibility-of-parabolic-induction-case-3}, \cref{proposition:irreducibility-of-parabolic-induction-case-4} and \cref{proposition:irreducibility-of-parabolic-induction-case-5}.
\end{proof}

\begin{remark}
    It seems that an irreducibility result of this type holds for any endoscopic lift from $H$.
    However, we neither need nor prove this general result in this paper.
\end{remark}

The following two lemmas are formal consequences of the normalization of intertwining operators.

\begin{lemma}\label{lemma:normalized-intertwiner-homomorphism}
    For any tempered representation $\sigma$ with trivial central character, we have 
    \begin{align*}
        R_P(w', {}^w \sigma, \psi) \circ R_P(w, \sigma, \psi) = R_P(w'w, \sigma, \psi) 
    \end{align*}
    for any $w, w' \in W(G, M)$.
\end{lemma}

\begin{lemma}\label{lemma:intertwiner-intertwines-induction-of-isom}
    For any homomorphism $A \colon \sigma \to \sigma'$ of irreducible representations of $M$ and $P, Q \in \cP^G(M)$, we have 
    \begin{align*}
        R_{Q \vert P}(\sigma') \circ i^G_P(A) = i^G_Q(A) \circ  R_{Q \vert P}(\sigma). 
    \end{align*}
\end{lemma}

\begin{theorem}\label{theorem:canonical-extension-PGSO8-S3}
    Let $\phi^{H}$ be a tempered $L$-parameter for $H$.
    Then, any representation $\pi$ in the corresponding distinguished Xu packet $\Pi^{X}_{\phi^{H}}$ has a unique unitary extension $\widehat{\pi}$ to $M_{\ad} \rtimes W(G, M)$ such that
    \begin{align*}
         i^G_P(\widehat{\pi}(w)) = R_{P}(w^{-1}, \pi, \psi) 
    \end{align*} 
    for any element $w \in W(G, M)$ and any parabolic subgroup $P \in \cP^G(M)$.
\end{theorem}

\begin{proof}
    We note that the uniqueness is trivial, as the parabolic induction functor is faithful. 
    We first show the existence for fixed $P$.
    We take an element $\sigma \in \Temp_{\ind}(H)$.
    Then, for $\sigma$ in a dense open subset, the parabolic induction $i^G_{P}(\pi)$ of any element $\pi$ in $\Pi^X_{\phi^H_{\sigma}}$ is irreducible.
    We note that we have $\pi \xrightarrow{\sim} \pi^{\widetilde{w}}$ for $w \in W(G, M)$.
    Then, we have a unique unitary automorphism $\widehat{\pi}(w) \colon \pi \to \pi^{\widetilde{w}}$ as in the statement of the theorem by Schur's lemma.
    In general, we take the direct sum $\Pi_{\sigma}$ of elements in $\Pi^X_{\phi^H_{\sigma}}$.
    By \cref{proposition:smoothness-of-endoscopic-lift}, the unitary operators $R_P(w^{-1}, \Pi_{\sigma}, \psi)$ depend smoothly on $\sigma \in \Temp_{\ind}(H)$. 
    By the limiting argument explained below, we obtain the unitary operator $\widehat{\Pi_{\sigma}}(w) \colon \Pi_{\sigma} \to \Pi_{\sigma}^{w}$ such that, for $\sigma$ in a dense open subset, the operator is the sum of the operators $\widehat{\pi}(w)$ for $\pi \in \Pi^X_{\phi^H_{\sigma}}$.
    Indeed, if we fix an open compact subgroup $K_M$ which is $W(G, M)$-invariant, then the space $\Pi^{K_M}_{\sigma}$ is realized in a fixed finite-dimensional $\C$-vector space $V$ with the positive definite hermitian product.
    Thus, the space of the unitary operators on $V$ is compact and we can apply the limiting argument.
    As the sum $\Pi_{\sigma}$ is multiplicity-free, we obtain the operator $\widehat{\pi}(w)$ by restricting $\widehat{\Pi_{\sigma}}(w)$ to $\pi \in \Pi^X_{\phi^H_{\sigma}}$.
    This proves the existence for fixed $P$.

    We denote the isomorphism constructed above for $P$ by $\widehat{\pi}_P(w)$.
    We have to show that $\widehat{\pi}_P(w) = \widehat{\pi}_Q(w)$.
    We have 
    \begin{align*}
        R_{Q \vert P}({}^{w^{-1}}\pi) \circ i^G_P(\widehat{\pi}_P(w)) 
        =
        i^G_Q(\widehat{\pi}_P(w)) \circ R_{Q \vert P}(\pi),
    \end{align*}
    by \cref{lemma:intertwiner-intertwines-induction-of-isom}.
    On the other hand, we have 
    \begin{align*}
        R_{Q \vert P}({}^{w^{-1}}\pi) \circ R_P(w^{-1}, \pi, \psi)
        =
        R_Q(w^{-1}, \pi, \psi) \circ R_{Q \vert P}(\pi).
    \end{align*}
    This implies that $\widehat{\pi}_P(w) = \widehat{\pi}_Q(w)$.

    Lastly, we prove that 
    \begin{align*}
        \widehat{\pi}(w') \widehat{\pi}(w) 
        = 
        \widehat{\pi}(w'w)
    \end{align*}
    for any $w', w \in W(G, M)$.
    By \cref{lemma:normalized-intertwiner-homomorphism}, we have 
    \begin{align*}
        \widehat{\pi^w}(w') \widehat{\pi}(w) 
        = 
        \widehat{\pi}(ww').
    \end{align*}
    We will show that 
    \begin{align*}
        \widehat{\pi^w}(w') = \widehat{\pi}(ww'w^{-1}).
    \end{align*}
    This obviously implies the result.
    From the independence of the definition of $\widehat{\pi^w}(w')$ on parabolic subgroups $P$, we have 
    \begin{align*}
        i^G_P(\widehat{\pi^w}(w')) 
        = 
        R_{P}(w'^{-1}, \pi^w, \psi).
    \end{align*}
    By conjugating $l_{w}$, we have 
    \begin{align*}
        i^G_{wPw^{-1}}(\widehat{\pi^w}(w')) 
        = 
        R_{wPw^{-1}}(ww'^{-1}w^{-1}, \pi, \psi)
        =
        i^G_{wPw^{-1}}(\widehat{\pi}(ww'w^{-1})).
    \end{align*}
    This implies the result.
\end{proof}

\subsection{On twisted endoscopy for $\PGSO_8$ with triality}

As we have $\Spin_8^{S_3} = \Spin_8^{A_3} = \widehat{H}$, the group $H$ is an endoscopic group of $M_{\ad} \rtimes \tau$.
Strictly speaking, the datum given by the group $H$, the $L$-group $\widehat{H} \rtimes W_F$, the element $s=1$ and the $L$-embedding ${}^L H \to {}^L M_{\ad}$ constitutes an endoscopic datum for $(M_{\ad}, \tau, 1)$. 
For a general definition of endoscopic data, see \cite{KS99}*{2.1}.
We have the following setup:
\begin{itemize}
    \item 
    The norm mapping, \cite{KS99}*{Theorem 3.3 A}.
    We have the canonical morphism of $F$-varieties
    \begin{align*}
        \cA_{M_{\ad}/H} \colon H // H   \to    M_{\ad} // (M_{\ad}, \tau),  
    \end{align*}
    where the space $M_{\ad} // (M_{\ad}, \tau)$ is the quotient of $M_{\ad}$ by the $\tau$-twisted conjugacy action.
    This morphism is called the norm mapping, and it satisfies the following condition:
    If $T_H \xrightarrow{\sim} T_{\tau} = T/(1-\tau)T$ is an admissible embedding in the sense of \cite{KS99}*{p.28, line 13}, then we have the commutative diagram 
    \[
         % https://q.uiver.app/#q=WzAsNCxbMCwwLCJUX0giXSxbMiwwLCJUX3tcXHRhdX0iXSxbMCwyLCJILy9IIl0sWzIsMiwiTS8vX3tcXHRhdX0gTSJdLFswLDJdLFsyLDNdLFswLDFdLFsxLDNdXQ==
        \begin{tikzcd}
            {T_H} && {T_{\tau}} \\
            \\
            {H//H} && {M_{\ad}// (M_{\ad}, \tau)}
            \arrow[from=1-1, to=1-3]
            \arrow[from=1-1, to=3-1]
            \arrow[from=1-3, to=3-3]
            \arrow[from=3-1, to=3-3].
        \end{tikzcd}
    \]
    For semisimple elements $\delta \in M_{\ad}(F)$ and $\gamma \in H(F)$, we say that $\gamma$ is a norm of $\delta$ if the image of $\gamma$ in $H//H(F)$ maps to the image of $\delta$ in $M_{\ad} // (M_{\ad}, \tau)$.
    We give explicit description of this mapping in \cref{subsection:explicit-matching} below.
    \item 
    The Whittaker normalized transfer factor $\Delta = \Delta_{M_{\ad} \rtimes \tau^{-1}, H}$, see \cite{kottwitz2012splittinginvariantssignconventions}*{5.5} for the details.
    In our particular case, it is equal to the constant function $1$, see \cref{lemma:transfer-factor}.
    \item 
    The existence of a transfer $f^H \in \Cc(H)$ for any smooth function $f^{M_{\ad}} \in \Cc(M_{\ad})$. 
    By the definition of transfers, then, we have the equality of stable orbital integrals (normalized by the Weyl discriminant)
    \begin{align*}
        SI^{H}_{\gamma}(f^{H}) 
        = 
        \sum_{\delta} \Delta(\delta, \gamma) I^{M_{\ad} \rtimes \tau}_{\delta}(f^{M_{\ad}}),
    \end{align*}
    if $\gamma \in H(F)$ is a strongly regular element and $\delta$ runs over the set of $\tau$-twisted conjugacy classes in $M_{\ad}(F)$ whose norms are $\gamma$.
\end{itemize}

We also consider the twisted endoscopy for $M_{\ad} \rtimes \tau^{-1}$ and the analogous constructions in this setting.
We have the following results for these constructions.

\begin{lemma}\label{lemma:transfer-factor}
    We have $\Delta = 1$. 
\end{lemma}

\begin{proof}
    This follows from the same argument as \cite{Yamauchi2024}*{Theorem 7.1} as we have $s=1$ for our endoscopic datum.
\end{proof}

\begin{lemma}
    If the strongly $M_{\ad}$-regular element $\gamma \in H(F)$ is a norm of $\delta \in M(F)$ in the setting of the twisted endoscopy for $M_{\ad} \rtimes \tau$, then the element $\gamma^{-1}$ is a norm of $\tau(\delta)^{-1} \in M(F)$ in the setting of the twisted endoscopy for $M_{\ad} \rtimes \tau^{-1}$.
\end{lemma}

\begin{proof}
    See \cite{atobe2024localintertwiningrelationscotempered}*{Lemma F.8.5}.
\end{proof}

\begin{corollary}\label{corollary:matching-inversion}
    Let $f^{M_{\ad}}$ be any smooth function on $M_{\ad}$ and let $f^{H}$ be any transfer of $M_{\ad}$ with respect to the space $M_{\ad} \rtimes \tau^{-1}$.
    Then, the function ${f^H}^{\vee}$ is a transfer of the function ${{}^{\tau}f^{M_{\ad}}}^{\vee}$ with respect to the space $M_{\ad} \rtimes \tau$.
\end{corollary}

\subsection{Explicit matching}\label{subsection:explicit-matching}

We describe the matching of orbits explicitly in our particular case.
We fix a Borel pair $(B_H, T_H)$ of $H = M_{\ad}^{\tau}$ over $F$.
Let $S_H$ be a maximal torus of $H$ defined over $F$. 
Then, we can find an element $h \in H(\overline{F})$ such that $S_H = h T_H h^{-1}$.
In this case, corresponding to $S_H, C_H = h B_H h^{-1}$, we obtain the $\tau$-stable Borel pair $S, C$ of $M_{\ad}$. 
For example, we have $S = \Cent_{M_{\ad}}(S_H)$.
We have the isomorphism 
\begin{align*}
   S_H \xrightarrow{\sim} S/(1-\tau)S
\end{align*}
induced from $T \to T_H$.
This morphism induces the matching of orbits.

For any regular element $\widetilde{\delta} \in S \rtimes \tau$, we have 
\begin{align*}
    \Cent_{M_{\ad}}(\widetilde{\delta}) = S^{\tau}.
\end{align*}
Thus, the conjugacy classes in the stable conjugacy class of $\widetilde{\delta}$ are parametrized by the Galois cohomology set $H^1(F, S^{\tau}) = H^1(F, S_H)$.
This set also parametrizes the conjugacy classes in the stable conjugacy class of any norm $\gamma$ of $\widetilde{\delta}$ in $H(F)$.

\subsection{Matching at a certain singular orbit}

Let $f^{M_{\ad}}$ be any smooth function on $M_{\ad}$ and let $f^{H}$ any transfer of $f^{M_{\ad}}$ with respect to the space $M_{\ad} \rtimes \tau^{-1}$.
We use a generalized result on Shalika's germ expansion by Vigneras \cite{Vigneras1981orbitalintegral}*{2.5 Proposition}.
The following proposition is an analog of \cite{beuzartplessis2025hiragaichinoikedaconjectureformaldegrees}*{Proposition 4.1}.

\begin{proposition}\label{proposition:transfer-singular-orbit}
    We have 
    \begin{align*}
        f^H(1) 
        = 
        \int_{M^{\tau}_{ad} \backslash M_{\ad}} 
        f^{M_{\ad}}(m^{-1}\tau^{-1}(m))
        dm
        = 
        \int_{M^{\tau}_{ad} \backslash M_{\ad}}
        f^{M_{\ad}}(\tau(m)^{-1}m)
        dm.
    \end{align*}
\end{proposition}

\begin{proof}
    We will follow the argument of Kottwitz in \cite{Kottwitz1988-Tamagawa}*{p.638 line 24} in combination with the Shalika germ expansion for disconnected groups proved by Vigneras \cite{Vigneras1981orbitalintegral}*{Proposition 2.5}.
    By the definition of transfers, we have the equality of stable orbital integrals
    \begin{align*}
        SI^{H}_{\gamma}(f^{H}) 
        = 
        \sum_{\delta} I^{M_{\ad} \rtimes \tau}_{\delta}(f^{M_{\ad}}),
    \end{align*}
    where a strongly regular element $\gamma \in H(F)$ is a norm of elements $\delta \in M_{\ad}(F)$.
    We may assume that $\delta \in M_{\ad}^{\tau}(F)$ is sufficiently close to $1 \in M_{\ad}(F)$ and $\delta, \gamma$ are elliptic strongly regular elements which matches the construction as \cref{subsection:explicit-matching}.
    By applying the germ expansion, we obtain 
    \begin{align*}
        (-1)^{q(H)}c^H f^H(1) = (-1)^{q(M^{\tau}_{\ad})} c^{M_{\ad} \rtimes \tau^{-1}} \int_{M^{\tau}_{ad} \backslash M_{\ad}}
        f^{M_{\ad}}(m^{-1}\tau^{-1}(m))
        dm,
    \end{align*}
    where we set
    \begin{align*}
        \begin{cases}
        c^H = \abs{\Ker(H^1(F, S_H) \to H^1(F, H))} = \abs{H^1(F, S_H)}, \\
        c^{M_{\ad} \rtimes \tau^{-1}} = \abs{\Ker(H^1(F, S) \to H^1(F, M_{\ad}))}
        \end{cases}
    \end{align*}
    for the anisotropic subtorus $S =\Cent_{M_{\ad}}(\delta \rtimes \tau^{-1})$ of $M_{\ad}$ and $S_H = \Cent_H(\gamma)$ of $H$.
    Note that we have $S \subset M_{\ad}^{\tau}$ as $\delta$ is sufficiently close to $1$ by the lemma below and the map 
    \begin{align*}
        H^1(F, S) \to H^1(F, M_{\ad})
    \end{align*}
    factors through $H^1(F, M_{\ad}^{\tau}) = 1$.
    Thus, we have $c^H = c^{M_{\ad} \rtimes \tau^{-1}}$ since $S \cong S_H$.
    Hence, we obtain the result by the change of variables $m \mapsto \tau(m)$.
\end{proof}

\begin{lemma}
    Let $k$ be a field.
    Let $A$ be a semisimple matrix in $\GL_n$.
    Then, for any semisimple element $B \in \Cent_{\GL_n(k)}(A)$ which is sufficiently close to $1$ in the Zariski topology, we have $\Cent_{\GL_n(k)}(AB) = \Cent_{\GL_n(k)}(A)^{B}$.
\end{lemma}

\begin{proof}
    By extending the field if necessary, we may assume that $A$ is diagonalizable.
    We write $A$ as a block diagonal matrix
    \begin{align*}
        A = \begin{pmatrix}
            a_1 I_{n_1} & 0 & \cdots & 0 \\
            0 & a_2 I_{n_2} & \cdots & 0 \\
            \vdots & \vdots & \ddots & \vdots \\
            0 & 0 & \cdots & a_r I_{n_r}
        \end{pmatrix},
    \end{align*}
    with distinct eigenvalues $a_1, a_2, \ldots, a_r$.
    Then, any element in $\Cent_{\GL_n(k)}(A)$ is also written as a block diagonal matrix with the same block decomposition, i.e.,
    \begin{align*}
        B = \begin{pmatrix}
            B_1 & 0 & \cdots & 0 \\
            0 & B_2 & \cdots & 0 \\
            \vdots & \vdots & \ddots & \vdots \\
            0 & 0 & \cdots & B_r
        \end{pmatrix}.
    \end{align*}
    The adjoint action of the matrix $AB$ is then given by 
    \begin{align*}
        M = (M_{ij})_{1 \leq i, j \leq r} \mapsto (a_i a_j^{-1} B_i M_{ij} B_j^{-1})_{1 \leq i, j \leq r}.
    \end{align*}
    If we set 
    \begin{align*}
        l_{ij}(B) \colon M_{ij} \mapsto a_i a_j^{-1} B_i M_{ij} B_j^{-1},
    \end{align*}
    then, in the Zariski open set 
    \begin{align*}
        \bigcap_{i \neq j} D(\det(l_{ij}( \cdot ) - \id)),
    \end{align*}
    we have $\Cent_{\GL_n}(AB) = \Cent_{\GL_n}(A)^{B}$.
    We note that this open set contains $1$ and we obtain the result.
\end{proof}

\subsection{Conjectural character identities}

Let $f^{M_{\ad}}$ be any smooth function on $M_{\ad}$ and let $f^{H}$ be any transfer of $f^{M_{\ad}}$ with respect to the space $M_{\ad} \rtimes \tau$.

\begin{conjecture}\label{conjecture:endoscopic-character-relation}
    Let $\phi^H$ be a tempered $L$-parameter for $H$.
    Let $f^{M_{\ad}}$ be any smooth function on $M_{\ad}$ and let $f^{H}$ be any transfer of $f^{M_{\ad}}$ with respect to the space $M_{\ad} \rtimes \tau$.
    Then, we have 
    \begin{align*}
        \sum_{\pi \in \Pi^X_{\phi^H}} \tr (\widetilde{\pi}(\tau) \pi(f^{M_{\ad}}))  
        = 
        \sum_{\pi^{H} \in \Pi_{\phi^H}(H)} 
        \dim_{\C} \rho_{\pi^H} \cdot
        \tr \pi^H (f^H).
    \end{align*}
\end{conjecture}

\begin{remark}\label{remark:character-relation-inversion}
    This conjecture is equivalent to the following conjecture.
    Let $\phi^H$ be a tempered $L$-parameter for $H$.
    Let $f^{M_{\ad}}$ be any smooth function on $M_{\ad}$ and let $f^{H}$ be any transfer of $f^{M_{\ad}}$ with respect to the space $M_{\ad} \rtimes \tau^{-1}$. 
    Then, we have
    \begin{align*}
        \sum_{\pi \in \Pi^X_{\phi^H}} \tr (\widetilde{\pi}(\tau) \pi({f^{M_{\ad}}}^{\vee}))  
        = 
        \sum_{\pi^{H} \in \Pi_{\phi^H}(H)} 
        \dim_{\C} \rho_{\pi^H} \cdot
        \tr \pi^H ({f^H}^{\vee}).
    \end{align*}
    The equivalence follows from \cref{corollary:matching-inversion}.
\end{remark}

\begin{theorem}\label{lemma:transfer-surjection}
    For any function $f^H \in \Cc(H)$, there exists a function $f^{M_{\ad}} \in \Cc(M_{\ad})$ such that $f^H$ is a transfer of $f^{M_{\ad}}$.
\end{theorem}

\begin{proof}
    See \cite{gan2025trialityadjointliftinggl3}*{Corollary 4.8}.
\end{proof}

\section{Formal degree conjecture for $G_2$}\label{section:10}

\subsection{Preparatory computations} 

We consider convenient smooth sections on $A_M N \backslash G$ which are used in the construction of the key limit.

\begin{construction}
    We fix a function $\Psi \in \Cc(\overline{N})$. 
    We assume that $\Psi(1) = 1$.
    For any function $f \in \Cc(M_{\ad})$, we define the function $\phi_{s, f} \in \Cc(A_M N \backslash G, \delta_P^{\frac{1}{2}} \chi_{2s\rho_P})$ as follows.
    On the big cell $N M \overline{N} \xrightarrow{\sim} N \times M \times \overline{N}$, we set 
    \begin{align*}
        \phi_{s, f} (n m \overline{n}) = (\delta_P^{\frac{1}{2}} \chi_{2s\rho_P})(m) f(m) \Psi(\overline{n}).
    \end{align*}
    We take the extension by zero to $G$.
    It is easy to see that this function is actually smooth.
\end{construction}

Recall that we have $w = s_1s_6$ and $\widetilde{w} = \widetilde{s_1} \widetilde{s_6}$.
Later, we will consider the limit 
\begin{align*}
     \lim_{s \to 0+} \gamma(s, \mathbf{1}, \psi)^2 
     \int_{n \in N \cap w\overline{N}w^{-1}} 
     \phi_{s, f}(\widetilde{w}^{-1}n)
     dn.
\end{align*}

We consider the Levi subgroup $M_{\alpha}$ for $\alpha \in \Phi(P, A_M)$. 
We have the parabolic subgroup $P^{M_{\alpha}} = M N^{M_{\alpha}} \in \cP^{M_{\alpha}}(M)$ associated to $P$.

\begin{lemma}\label{lemma:intersection-density-simple-roots}
    The subset $N^{M_{\alpha}} M \overline{N^{M_{\alpha}}}$ is a Zariski dense open subset of $M_{\alpha}$. 
    Furthermore, this set intersects the subgroup $\widetilde{s_{\alpha}}^{-1}N^{M_{\alpha}}$ nontrivially. 
    Thus, the subspace $N^{M_{\alpha}} M \overline{N^{M_{\alpha}}} \cap \widetilde{s_{\alpha}}^{-1}N^{M_{\alpha}}$ is a Zariski dense open subset of $\widetilde{s_{\alpha}}^{-1}N^{M_{\alpha}}$.
\end{lemma}

Note that we have $N^{M_{\alpha}} = N_{\alpha}$.

\begin{proof}
    The first statement is well known.
    We will move on to the second statement.
    By \cref{lemma:M-alpha-isogeny-GSpin10}, it suffices to show the corresponding statement for $\GL_1 \times \SO_8 \subset \SO_{10}$.
    We give explicit description of the intersection in \cref{lemma:bigcell-intersection-SO10} below.
\end{proof}

\begin{lemma}\label{lemma:bigcell-intersection-SO10}
    We use the notation in \cref{proof:Tits-lifting-preserves-splitting}.
    In particular, we have the standard parabolic subgroup $P' = M'N'$ of $L' = \SO_{10} = \SO(V)$ with the Levi component $M' = \SO_{8} \times \GL_1 = \SO(W) \times \SO(W')$.
    We have a unique simple root $\alpha' \in \Phi(P', A_{M'})$ and the Tits lifting $\widetilde{s_{\alpha'}}$ of the nontrivial element $s_{\alpha'} \in W(L', M')$.
    In this setting, we have the following statements:

    \begin{enumerate}
    \item 
    The intersection $\widetilde{s_{\alpha'}}^{-1}N_{\alpha'} \cap N_{\alpha'}M'N_{-\alpha'}$ is dense in $\widetilde{s_{\alpha'}}^{-1}N_{\alpha'}$.

    \item
    We identify the space $N_{\alpha'}$ with the vector space $W$ by the isomorphism which sends $x \in W$ to the element $n(x)$ given by 
    \begin{align*}
        n(x) e_i =
        \begin{cases}
            e_1 \quad i = 1, \\
            q(x, e_{11-i}) e_1 + e_i \quad i = 2, \dots, 9, \\
            \frac{1}{2}q(x, x) e_1 - q(x, e_i) e_i + e_{10}  \quad i = 10.
        \end{cases}
    \end{align*}
    We use a similar identification $\overline{n} \colon W \xrightarrow{\sim} \overline{N}_{\alpha'}$.
    Then, the subspace $N_{\alpha'}^{\reg} =  \widetilde{s_{\alpha'}}^{-1}N_{\alpha'} \cap N_{\alpha'}M'N_{-\alpha'}$ is identified with the subset $W^{\reg} = \{ x \in W \mid q(x, x) \neq 0 \}$.

    \item    
    Thus, we have the map    
    \begin{align*}
        \widetilde{s_{\alpha'}}^{-1}N_{\alpha'} \cap N_{\alpha'}M'N_{-\alpha'}   
        \hookrightarrow  
        N_{\alpha'}M'N_{-\alpha'} \\
        \colon  
        \widetilde{s_{\alpha'}}^{-1} n(x) 
        \mapsto   
        (n'(x), m'(x), \overline{n}'(x)).
    \end{align*}
    \item        
    For any element $v \in V$ with $q(v, v) \neq 0$, let $r_v \in \SO(V)$ be the corresponding orthogonal reflection.
    We set $r = r_{e_5-e_6}$.
    Then, the map $n'$ is equal to the map 
    \begin{align*}
        \widetilde{s_{\alpha'}}^{-1}n(x) \mapsto n(-rx'), 
    \end{align*}
    where 
    \begin{align*}
        x' = -\frac{2x}{q(x, x)} = -\frac{x}{q(x)}.
    \end{align*}
    \item 
    The map $\overline{n}'$ is given by 
    \begin{align*}
        \widetilde{s_{\alpha'}}^{-1}n(x) \mapsto \overline{n}(x').
    \end{align*}
    \item         
    For any $x \in W^{\reg}$, the element $m'(x)$ is equal to 
    \begin{align*}
    (-\frac{1}{q(x)},  -(r \circ r_{x})) \in \GL_1 \times \SO_8.
    \end{align*}
    \end{enumerate}
\end{lemma}

\begin{proof}
   These statements are checked by a straightforward calculation, which we omit.
\end{proof}

\begin{corollary}
    The intersection $\widetilde{w}^{-1} (N \cap w \overline{N} w^{-1})^{\reg} = \widetilde{w}^{-1} (N \cap w \overline{N} w^{-1}) \cap NM\overline{N}$ is a Zariski dense open subset of the space $\widetilde{w}^{-1} (N \cap w \overline{N} w^{-1})$.
\end{corollary}

\begin{proof}
    This follows from the decomposition 
    \begin{align*}
        \widetilde{w}^{-1} (N \cap w \overline{N} w^{-1}) 
        = 
        \widetilde{s_6}^{-1} N_{\alpha_6} \widetilde{s_1}^{-1} N_{\alpha_1}
    \end{align*}
    and the product formula
    \begin{align*}
        N_{\alpha_6} M N_{-\alpha_6} N_{\alpha_1} M N_{-\alpha_1}  
        = 
        N_{\alpha_6} N_{\alpha_1} M N_{-\alpha_6} N_{-\alpha_1},
    \end{align*}
    since the $-\alpha_{6, M} + \alpha_{1, M}$ is not in $\Phi(G, A_M)$.
\end{proof}

\begin{definition}\label{definition:equivariant-map}
    We set $(N \cap w \overline{N} w^{-1})^{\reg} = \widetilde{w}\{ \widetilde{w}^{-1} (N \cap w \overline{N} w^{-1})^{\reg} \}$.
    We denote the above natural injection $(N \cap w \overline{N} w^{-1})^{\reg} \hookrightarrow N \times M \times \overline{N}$ by $n \mapsto (n'(n), m'(n), \overline{n}'(n))$.
    The image is contained in the space $N_{\alpha_6} N_{\alpha_1} \times M \times N_{-\alpha_6} N_{-\alpha_1}$, and we set $n'(n) = (n_6(n), n_1(n)) \in N_{\alpha_6} \times N_{\alpha_1}$ and $\overline{n}'(n) = (\overline{n_1}'(n), \overline{n_6}'(n)) \in N_{-\alpha_6} \times N_{-\alpha_1}$.
\end{definition}

\begin{lemma}\label{lemma:equivariant-map-general}
    The following statements hold.
    \begin{enumerate}
        \item 
        The space $(N \cap w \overline{N} w^{-1})^{\reg}$ is a homogeneous space under the conjugation action of $M$ with  stabilizers isomorphic to $H$.
        \item     
        The image of the map $m'$ contains a central element $z_M$ of $M$.
        \item
        We have
        \begin{align*}
            &n'(m^{-1}nm) = \Ad(w^{-1}(m)^{-1})n'(n), \\
            &\overline{n}'(m^{-1}nm) = \Ad(m)^{-1}\overline{n}'(n)
        \end{align*}
        for any $m \in M$ and 
        \begin{align*}
            m'(m^{-1}nm) = \widetilde{w}^{-1}(m)^{-1} m'(n) m
        \end{align*}
        for any $m \in M$. 
        \item    
        The map $m'$ is proper. 
        \end{enumerate}
\end{lemma}

\begin{proof}
    By \cref{remark:fix-isomorphism-PGSO8}, the action of $M_{\der}$ on $N \cap w \overline{N} w^{-1}$ is isomorphic to the direct sum of the standard representation and one of the spin representations of $M_{\der} \xrightarrow{\sim} \Spin_8$.
    Then, the first statement follows from \cite{Kimura1983prehomogeneousscalar}*{p.96 3.A.(15), Proposition 2.18}.

    The second statement follows from \cref{lemma:bigcell-intersection-SO10}, by taking $x = e_5 - e_6 \in W$ and considering the isogenies.

    The third statement follows from the definition of these maps.

    We will prove the fourth statement. 
    By the first, second and third statements, the image of $m'$ is the orbit 
    \begin{align*}
         \{ z_M \widetilde{w}^{-1}(m) m^{-1} \},
    \end{align*}
    under the $\widetilde{w}^{-1}$-twisted action of $M$.
    This orbit is isomorphic to $M/M^{\widetilde{w}}$ and we also have 
    \begin{align*}
        M^{\widetilde{w}} \xrightarrow{\sim} H \times \mu_3
    \end{align*}
    by a simple direct computation.
    Thus, the morphism $m'$ is isomorphic to 
    \begin{align*}
        M/H \twoheadrightarrow M/(H \times \mu_3)
    \end{align*}
    and it is finite.
    This completes the proof.
\end{proof}

\begin{proposition}\label{proposition:normalization-measure-twisted-orbit}
    We fix the Haar measure $\rd_{\psi}n$ on the group $N \cap w\overline{N}w^{-1}$ obtained from a splitting $(B, T, \{ X_{\alpha} \}_{\alpha})$ of $G$ and a nontrivial additive character $\psi \colon F \to \bS^1$.
    Then, the measure $\delta^{\frac{1}{2}}_P(m'(n)) \rd_{\psi}n$ is equal to the pullback of the quotient measure $\rd_{\psi}m/\rd_{\psi}h$ on $M/H$ by the isomorphism $m'$ in \cref{definition:equivariant-map} up to positive scalar multiplication.
\end{proposition}

\begin{proof}
    This follows from the uniqueness property of quotient measures and the equivariance property in \cref{lemma:equivariant-map-general} (3).
    Indeed, we have 
    \begin{align*}
        \delta^{\frac{1}{2}}_P(m'(\Ad(m)^{-1}n)) & \rd_{\psi} \Ad(m)^{-1}n \\
        &=
        \delta^{\frac{1}{2}}_P(\widetilde{w}^{-1}(m)^{-1} m'(n) m) \rd_{\psi} \Ad(m)^{-1}n \\
        &= 
        (\delta^{\frac{1}{2}}_P/w \delta^{\frac{1}{2}}_P)(m) \abs{\det(\Ad(m) \colon \Lie(N \cap w\overline{N}w^{-1}))}_F^{-1} \\
        & \quad \quad \times \delta^{\frac{1}{2}}_P(m'(n)) \rd_{\psi} n
    \end{align*}
    for any $m \in M$. 

    It suffices to show that we have 
    \begin{align*}
        (\delta^{\frac{1}{2}}_P/w \delta^{\frac{1}{2}}_P)(m) 
        = 
        \abs{\det(\Ad(m) \colon \Lie(N \cap w\overline{N}w^{-1}))}_F.
    \end{align*}
    This follows from the lemma below.
\end{proof}

\begin{lemma}
    Let $G$ be a connected reductive group over $F$.
    Let $M$ be a Levi subgroup of $G$ and let $P, Q$ be parabolic subgroups of $G$ with a Levi component $M$. 
    Let $A_M$ be the maximal split central torus of $M$.
    Then, we have 
    \begin{align*}
        \rho_{P} - \rho_{Q} = \sum_{\alpha \in \Phi(P, A_M) \cap \Phi(\overline{Q}, A_M)} m(\alpha) \alpha \in X^*(A_M) \otimes_{\Z} \Q,
    \end{align*}
    where $m(\alpha)$ is the multiplicity of the root $\alpha$ in $\Lie(G)$.
\end{lemma}

\begin{proof}
    We have 
    \begin{align*}
        \rho_{P} 
        &= \frac{1}{2} \sum_{\alpha \in \Phi(P, A_M)} m(\alpha) \alpha, \\
        &= \frac{1}{2} (\sum_{\alpha \in \Phi(P, A_M) \cap \Phi(Q, A_M)} m(\alpha) \alpha 
        + \sum_{\alpha \in \Phi(P, A_M) \cap \Phi(\overline{Q}, A_M)} m(\alpha) \alpha), \\
        \rho_{Q} 
        &= \frac{1}{2} \sum_{\alpha \in \Phi(Q, A_M)} m(\alpha) \alpha, \\
        &= \frac{1}{2} (\sum_{\alpha \in \Phi(P, A_M) \cap \Phi(Q, A_M)} m(\alpha) \alpha 
        - \sum_{\alpha \in \Phi(P, A_M) \cap \Phi(\overline{Q}, A_M)} m(\alpha) \alpha).
    \end{align*}
    Thus, we obtain 
    \begin{align*}
        \rho_{P} - \rho_{Q} = \sum_{\alpha \in \Phi(P, A_M) \cap \Phi(\overline{Q}, A_M)} m(\alpha) \alpha.
    \end{align*}
\end{proof}

\begin{definition}
    We set $\rd^{\times}_{\psi} n = \delta^{\frac{1}{2}}_P(m'(n)) \rd_{\psi}n = \rho_P(m'(n)) \rd_{\psi}n$.
\end{definition}

\begin{lemma}\label{lemma:change-of-variables-by-isogeny}
    Let $F$ be a $p$-adic field. 
    Let $A$ be a split torus and let $\phi \colon A \to A$ be an isogeny over $F$. 
    Then, we have 
    \begin{align*}
        C_{\phi} \int_{a \in A(F)} f(a) \rd a = \sum_{t \in A(F)/\phi(A(F))} \int_{a \in A(F)} f(t \phi(a)) \rd a,
    \end{align*} 
    for any function $f \in \Cc(A)$, where 
    \begin{align*}
         C_{\phi} = \frac{\abs{\Ker \phi(F)}}{\abs{\#\Ker \phi(\overline{F})}_F} > 0.
    \end{align*}
\end{lemma}

\begin{remark}
    We will not use the explicit value of $C_{\phi}$.
\end{remark}

\begin{proof}
    By the uniqueness property of Haar measures, it suffices to check the value on the both sides at $f = \mathbf{1}_{A(\fo_F)}$.
    We note that we have 
    \begin{align*}
        \phi(A(F)) \cap A(\fo_F) = \phi(A(\fo_F))
    \end{align*}
    as the group $A(\fo_F)$ is maximally compact in $A(F)$.
    Then, the statement is equivalent to 
    \begin{align*}
        C_{\phi} \vol(A(\fo_F)) 
        = 
        \abs{A(\fo_F)/(\phi(A(F)) \cap A(\fo_F))}
        \vol(A(\fo_F)). 
    \end{align*}

    By the snake lemma, we have the exact sequence 
    \begin{multline*}
        0 \to A(\fo_F)/(\phi(A(F)) \cap A(\fo_F)) \to A(F)/\phi(A(F)) \\ \to X_*(A)/\phi(X_*(A)) \to 0.
    \end{multline*}
    The size of the last term is equal to 
    \begin{align*}
        \# \Ker\phi(\overline{F}).
    \end{align*}
    Also, as $A$ is split, we have 
    \begin{align*}
        \abs{A(F)/\phi(A(F))} = \abs{H^1(F, \Ker\phi)}. 
    \end{align*}
    From the local Euler--Poincar\'e characteristic formula, we have 
    \begin{align*}
       \frac{\abs{H^1(F, \Ker\phi)}}{\# \Ker\phi(\overline{F})} 
       = 
       \frac{\abs{H^0(F, \Ker\phi)}\abs{H^2(F, \Ker\phi)}}{\abs{\# \Ker\phi(\overline{F})}_F \#\Ker\phi(\overline{F})}.
    \end{align*}
    As the group $\Ker\phi$ is a product of $\mu_n$, we have 
    \begin{align*}
        \abs{H^2(F, \Ker\phi)} 
        = 
        \# \Ker\phi(\overline{F}).
    \end{align*}
    Hence, we obtain the result.
\end{proof}

The following lemma is well known as a computation in `Tate's thesis'.
\begin{lemma}\label{lemma:Tate's-thesis}
    Let $f \in \Cc(F)$ be a smooth function.
    Then, the integral
    \begin{align*}
        \int_{z \in F^{\times}} f(z) \abs{z}^s_F \rd_{\psi} z
    \end{align*}
    defines a meromorphic function in $s \in \C$, which is holomorphic on $\C_{+}$.
    We have 
    \begin{align*}
        \lim_{s \to 0+} \gamma(s, \mathbf{1}, \psi) \int_{z \in F^{\times}} f(z) \abs{z}^s_F \, \rd_{\psi} z
        =
        f(0).
    \end{align*}
    The convergence is uniform in the support of $f$, smoothness of $f$, and the supremum of $f$.
\end{lemma}

\begin{proof}
    We write the function $f$ as 
    \begin{align*}
        f = \sum_{t \in \varpi^{-m}\fo_F/\varpi^n\fo_F} c_t  \mathbf{1}_{t + \varpi^n\fo_F},
    \end{align*}
    for some $m, n \in \Z_{\geq 1}$.
    We set 
    \begin{align*}
        f_0 &= c_0 \mathbf{1}_{\varpi^n\fo_F}, \\
        f_1 &= \sum_{t \in \varpi^{-m}\fo_F/\varpi^n\fo_F, t \neq 0} c_t  \mathbf{1}_{t + \varpi^n\fo_F}.
    \end{align*}
    We can see that
    \begin{align*}
        \| f \|_{\infty} = \max_{t} \abs{c_t}.
    \end{align*}
    Then, we have 
    \begin{align*}
        \abs{\int_{z \in F^{\times}} f_1(z) \abs{z}^s_F \, \rd_{\psi} z} 
        \leq 
        \| f \|_{\infty}(1-q^{-1})(q^{-ns} + \cdots + q^{ms}),
    \end{align*}
    and 
    \begin{align*}
        \int_{z \in F^{\times}} f_0(z) \abs{z}^s_F \, \rd_{\psi} z
        = 
        f(0) q^{-ns} \frac{1-q^{-1}}{1-q^{-s}}.
    \end{align*}
    This implies the result.
\end{proof}

\begin{lemma}\label{lemma:extension-functions-on-tori-to-vector-spaces}
    Let $V$ be a finite-dimensional vector space over $F$ on which a split torus $\GL_1$ acts by multiplication.
    Let $C$ be a compact subset in $V \setminus \{ 0 \}$.
    Let $h \in \Cc(V)$ be a smooth function on $V$.
    Then, the family of functions $h_v \colon t \mapsto h(tv)$ for $v \in C$ defines the family of elements in $\Cc(F)$.
    The supports of the functions $\{ h_v \}_{v \in C}$ are contained in a compact set which only depends on the support of $h$ and $C$.
    If the function $h$ is $L$-invariant for an $\fo_F$-lattice $L \subset V$, then there exists an $\fo_F$-lattice $L' \subset F$ such that the functions $h_v$ is $L'$-invariant.
\end{lemma}

\begin{proof}
    We take a coordinate of $V$ and denote its elements in the form $v = (v_1, \dots, v_n)$.
    Set $m_C = \min_{v \in C} \max_{i} \abs{v_i}_F > 0$.
    If the support of $h$ is contained in the subset $K_{< M} = \{ v \in V \mid \abs{v_i} \leq M \}$ for some $M > 0$ and we have $tv \in K_{< M}$, then we have $\abs{t}_F \leq \frac{M}{m_C}$.
    Thus, the supports of the functions $\{ h_v \}_{v \in C}$ are contained in a compact set which only depends on the support of $h$ and $C$.
    If the function $h$ is $L$-invariant for an $\fo_F$-lattice $L \subset V$, then the continuity of the action $F \times V \to V$ implies that there exists an $\fo_F$-lattice $L' \subset F$ such that $L' \cdot C \subset L$. 
    Then, we have
    \begin{align*}
        h_v(t+a) = h(tv + av) = h(tv) = h_v(t),
    \end{align*}
    for any $t \in F, a \in L'$ and $v \in C$.
    This completes the proof.
\end{proof}

\subsection{The limit formulas} 

In the rest of this section, we will prove the following two formulas for the limit 
\begin{align*}
    \lim_{s \to 0+} \gamma(s, \mathbf{1}, \psi)^2 
    \int_{n \in N \cap w\overline{N}w^{-1}} 
    \phi_{s, f}(\widetilde{w}^{-1}n)
    dn.
\end{align*}

We note that $H = M^{\widetilde{w}}_{\ad}$ is connected and simply-connected.

\begin{theorem}\label{theorem:geometric-limit-formula} 
    For any function $f \in \Cc(M_{\ad})$, we have 
    \begin{multline*}
    \lim_{s \to 0+} \gamma(s, \mathbf{1}, \psi)^2 
    \int_{n \in N \cap w\overline{N}w^{-1}} 
    \phi_{s, f}(\widetilde{w}^{-1}n)
    \rd_{\psi}n \\
    = 
    \frac{K}{512}
    \int_{H \backslash M_{\ad}} 
    f(\widetilde{w}^{-1}(m)^{-1} m) \rd_{\psi}\overline{m},
    \end{multline*}
    where $\rd_{\psi} \overline{m}$ is the quotient measure of $\rd_{\psi}m$ by $\rd_{\psi}h$ and $K$ is a positive constant.
\end{theorem}

\begin{theorem}\label{theorem:spectral-limit-formula}
    There exists a smooth function $c \colon \Temp_{\ind}(H)/\sim_{st} \to \bS^1$ such that, for any function $f \in \Cc(M_{\ad})$, we have 
    \begin{multline*}
    \lim_{s \to 0+} \gamma(s, \mathbf{1}, \psi)^2 
    \int_{n \in N \cap w\overline{N}w^{-1}} 
    \phi_{s, f}(\widetilde{w}^{-1}n)
    \rd_{\psi}n \\
    = 
    \frac{1}{512}
    \int_{\phi \in \Temp_{\ind}(H)/{\sim_{st}}}
    \lbrace
    \sum_{\pi \in \Temp_{\ind}(M_{\ad}), \pi \subset \Pi^X_\phi(M_{\ad})} \tr (\widetilde{\pi}(w^{-1}) \pi(f)) 
    \rbrace   \\
    \times
    c(\phi)
    \mu^{H}_{\HII, \psi}(\phi) 
    \rd \phi.
    \end{multline*}
    Here, the notation $\pi \subset \Pi^X_\phi(M_{\ad})$ means that the representation $\pi$ is a tempered subrepresentation of $\bigoplus_{\pi' \in \Pi^X_{\phi}(M_{\ad})} \pi'$.
\end{theorem}

Recall that $w = s_1s_6$ and $\tau = s_6s_1$.

\subsection{Proof of the geometric limit formula}

\begin{proof}[Proof of \cref{theorem:geometric-limit-formula}]
    First, we have 
    \begin{align*}
        \int_{n \in N \cap w\overline{N}w^{-1}} 
        \phi_{s, f}(\widetilde{w}^{-1}n)
        \rd_{\psi}n
        = 
        \int_{n \in (N \cap w\overline{N}w^{-1})^{\reg}} 
        \phi_{s, f}(\widetilde{w}^{-1}n)
        \rd_{\psi}n,
    \end{align*}
    by Zariski density of $(N \cap w\overline{N}w^{-1})^{\reg}$.
    We take a function $F_f \in \Cc(M)$ such that 
    \begin{align*}
        \int_{A_M} F_f(am) \rd_{\psi}a = f(m).
    \end{align*}
    From \cref{definition:equivariant-map}, we have 
    \begin{multline*}
        \int_{n \in N \cap w\overline{N}w^{-1}} 
        \phi_{s, f}(\widetilde{w}^{-1}n)
        \rd_{\psi}n \\
        =
        \int_{(N \cap w\overline{N}w^{-1})^{\reg} \times A_M}
        \delta_P^{\frac{1}{2}}(m'(n)) F_f(am'(n)) \Psi(\overline{n}'(n)) 
        \rd_{\psi}n 
        \rd_{\psi}a.
    \end{multline*}
    Applying \cref{lemma:change-of-variables-by-isogeny} with $A = A_M, \phi = \Ad(\widetilde{w})^{-1} - 1$, we have 
    \begin{multline*}
        \int_{(N \cap w\overline{N}w^{-1})^{\reg} \times A_M}
        \delta_P^{\frac{1}{2}}(m'(n)) F_f(am'(n)) \Psi(\overline{n}'(n)) 
        \rd_{\psi}n 
        \rd_{\psi}a \\
        =
        C_{\phi}^{-1}
        \int_{(N \cap w\overline{N}w^{-1})^{\reg} \times A_M}
        (\delta_P^{\frac{1}{2}}\chi_{2s\rho_P})(m'(n)) \\
        \times
        \sum_{t \in A_M/\phi(A_M)} F_f(t m'(ana^{-1})) \Psi(\overline{n}'(n)) 
        \rd_{\psi}n 
        \rd_{\psi}a,
    \end{multline*}
    by \cref{lemma:equivariant-map-general}.
    For any $t \in A_M$, we have 
    \begin{multline*}
        \int_{(N \cap w\overline{N}w^{-1})^{\reg} \times A_M}
        (\delta_P^{\frac{1}{2}}\chi_{2s\rho_P})(m'(n)) F_f(tm'(ana^{-1})) \Psi(\overline{n}'(n)) 
        \rd_{\psi}n 
        \rd_{\psi}a 
        \\
        = 
        \int_{(N \cap w\overline{N}w^{-1})^{\reg} \times A_M}
        \chi_{2s\rho_P}(m'(a^{-1}na)) 
        F_f(tm'(n)) 
        \Psi(\overline{n}'(a^{-1}na)) 
        \rd^{\times}_{\psi}n 
        \rd_{\psi}a,
    \end{multline*}
    by the invariance of $d^{\times}_{\psi}n$.
    We can see that 
    \begin{align*}
        \chi_{2s\rho_P}(\widetilde{w}^{-1}(a)^{-1}a) 
        = 
        \abs{\alpha_1(a)}_F^{32s} \abs{\alpha_6(a)}_F^{16s} 
    \end{align*}
    for any $a \in A_M$ and 
    \begin{align*}
        (\overline{n}'_6(a^{-1}na),\overline{n}'_1(a^{-1}na))
        = 
        (\alpha_6(a)\overline{n}'_6(n), \alpha_1(a)\overline{n}'_1(n))
    \end{align*}
    for any $a \in A_M$ and $n \in (N \cap w\overline{N}w^{-1})^{\reg}$.
    As the map $m'$ is proper, the support $C_F$ of the function $F \circ m' \colon (N \cap w\overline{N}w^{-1})^{\reg} \to \C$ is compact. 
    Now, we apply \cref{lemma:extension-functions-on-tori-to-vector-spaces} to the function $\Psi \in \Cc(N_{-\alpha_6} \times N_{-\alpha_1})$.
    Also applying \cref{lemma:Tate's-thesis}, we can see that, each limit 
    \begin{multline*}
        \lim_{s \to 0+} \gamma(s, \mathbf{1}, \psi)^2 
        \int_{(N \cap w\overline{N}w^{-1})^{\reg} \times A_M}
        \chi_{2s\rho_P}(m'(a^{-1}na))
        F_f(t m'(n))  \\
        \times 
        \Psi(\overline{n}'(a^{-1}na)) 
        \rd^{\times}_{\psi}n 
        \rd_{\psi}a
    \end{multline*}
    for $t \in A_M(F)/\phi(A_M(F))$, uniformly converges to 
    \begin{align*}
        \frac{1}{512}
        \int_{(N \cap w\overline{N}w^{-1})^{\reg}} 
        F_f(t m'(n))
        \rd^{\times}_{\psi}n.
    \end{align*}
    We applied the normalization $\Psi(1) = 1$ of the function $\Psi$ here.
    Also, there exists a central element $z_M \in M$ such that
    \begin{align*}
        \int_{(N \cap w\overline{N}w^{-1})^{\reg}} 
        F_f(t m'(n))
        d^{\times}_{\psi}n 
        = 
        K
        \int_{H \backslash M} 
        F_f(t z_M \widetilde{w}^{-1}(m)^{-1}m)
        \rd_{\psi}\overline{m}
    \end{align*}
    by \cref{proposition:normalization-measure-twisted-orbit}, where $K > 0$ is a positive constant.
    Again, we apply \cref{lemma:change-of-variables-by-isogeny} to rewrite the integral
    \begin{multline*}
        \frac{K}{512}
        C_{\phi}^{-1}
        \int_{H \backslash M_{\ad} \times A_M} 
        \sum_{t \in A_M/\phi(A_M)} F_f(t \widetilde{w}^{-1}(a)^{-1}a z_M  \widetilde{w}^{-1}(m)^{-1}m)
        \rd_{\psi}\overline{m}
        \rd_{\psi}a
    \end{multline*}
    by
    \begin{align*}
        \frac{K}{512}
        \int_{H \backslash M_{\ad}} 
        \int_{a \in A_M}
        F_f(a \widetilde{w}^{-1}(m)^{-1}m)
        \,
        \rd_{\psi} a 
        \,
        \rd_{\psi}\overline{m}.
    \end{align*}
    This is equal to
    \begin{align*}
        \frac{K}{512}
        \int_{H \backslash M_{\ad}} 
        f(\widetilde{w}^{-1}(m)^{-1}m) 
        \rd_{\psi}\overline{m}.
    \end{align*}
    This completes the proof.
\end{proof}

\subsection{Proof of the spectral limit formula} 

In this section, we will prove \cref{theorem:spectral-limit-formula}.

\subsubsection{A preparatory computation}

We have the isomorphism 
\begin{align*}
    \Cc(A_M N \backslash G, \delta_P^{\frac{1}{2}} \chi_{2s\rho_P}) 
    \xrightarrow{\sim} 
    i^G_P(\Cc(M_{\ad})\otimes \chi_{2s\rho_P})
\end{align*}
given as follows:
For any function $\phi \in \Cc(A_M N \backslash G, \delta_P^{\frac{1}{2}} \chi_{2s\rho_P})$, we define the section $\widetilde{\phi} \in i^G_P(\Cc(M_{\ad})\otimes \chi_{2s\rho_P})$ by
\begin{align*}
    (\widetilde{\phi}(g))(m)
     = (\delta_P^{\frac{1}{2}} \chi_{2s\rho_P}(m))^{-1} \phi(mg).
\end{align*}
The inverse map is given by 
\begin{align*}
    \psi \mapsto (g \mapsto (\psi(g))(1)).
\end{align*}

We see the function $\phi_{s, f}$ as the element of the right-hand side and we write it as $\widetilde{\phi_{s, f}}$.
We have 
\begin{align*}
    \widetilde{\phi_{s, f}}(\widetilde{w}^{-1}n)(1) 
    = 
    \phi_{s, f}(\widetilde{w}^{-1}n).
\end{align*}
This implies that 
\begin{multline*}
    \int_{n \in N \cap w\overline{N}w^{-1}} 
    \phi_{s, f}(\widetilde{w}^{-1}n)
    \rd_{\psi}n 
    = 
    \int_{n \in N \cap w\overline{N}w^{-1}} 
    \widetilde{\phi_{s, f}}(\widetilde{w}^{-1}n)(1) 
    \rd_{\psi}n  \\
    = 
    \int_{n \in N \cap w\overline{N}w^{-1}} 
    \lbrace
    \int_{\pi \in \Temp_{\ind}(M_{\ad})} 
    \tr \pi ( \widetilde{\phi_{s, f}}(\widetilde{w}^{-1}n)^{\vee} )
    \mu^{M_{\ad}}_{\psi}(\pi) \rd \pi
    \rbrace
    \rd_{\psi}n.
\end{multline*}

By \cref{theorem:formal-degree-conjecture-GSO2n}, we may and do take 
\begin{align*}
    \mu^{M_{\ad}}_{\psi}(\pi) 
    = 
    \mu^{M_{\ad}}_{\HII, \psi}(\pi).
\end{align*}

For any $\pi \in \Temp_{\ind}(M_{\ad})$ and $v \in \pi$, we define the section $\widetilde{\phi_{s, v}}$ of $i^G_{P}(\pi \otimes \chi_{2s \rho_P})$ by 
\begin{align*}
    \widetilde{\phi_{s, v}}(nm\overline{n}) 
    = 
    (\delta_P^{\frac{1}{2}} \chi_{2s\rho_P})(m)  \Psi(\overline{n}) \pi(m)v
\end{align*}
for $n m \overline{n} \in NM\overline{N}$ and take the zero extension to $G$.
Then, we have 
\begin{multline*}
    \int_{n \in N \cap w\overline{N}w^{-1}} 
    \lbrace
    \int_{\pi \in \Temp_{\ind}(M_{\ad})} 
    \tr \pi ( \widetilde{\phi_{s, f}}(\widetilde{w}^{-1}n)^{\vee} )
    \mu^{M_{\ad}}_{\psi}(\pi) \rd \pi
    \rbrace
    \rd_{\psi}n \\
    = 
    \int_{n \in N \cap w\overline{N}w^{-1}} 
    \lbrace
    \int_{\pi \in \Temp_{\ind}(M_{\ad})} 
    \tr ( v \in \pi \mapsto \widetilde{\phi_{s, \pi(f^{\vee})v}}(\widetilde{w}^{-1}n))
    \mu^{M_{\ad}}_{\psi}(\pi) \rd \pi
    \rbrace
    \rd_{\psi}n.
\end{multline*}

We note the following facts:
\begin{itemize}
    \item    
    The function $\mu^{M_{\ad}}_{\psi}(\pi)$ is smooth.
    \item    
    The integral 
    \begin{align*}
        \int_{n \in N \cap w\overline{N}w^{-1}}  
        \widetilde{\phi_{s, \pi(f^{\vee})v}}(\widetilde{w}^{-1}n)
        \rd_{\psi}n
    \end{align*}
    is equal to the unnormalized intertwining operator
    \begin{align*}
        (l(\widetilde{w})J_{w^{-1}Pw \vert P}(\pi \otimes \chi_{2s\rho_P}, \psi)\widetilde{\phi_{s, \pi(f^{\vee})v}})(1),
    \end{align*}
    which is compactly uniformly convergent in $(s, \pi) \in \C_+ \times \Temp_{\ind}(M_{\ad})$.
    Thus, we can interchange the order of two integrals.
    \item      
    The unnormalized intertwining operator is normalized by the factor 
    \begin{align*}
        r_{w^{-1}Pw \vert P}(\pi \otimes \chi_{2s\rho_P})^{-1}.
    \end{align*}
    The normalized operator
    \begin{align*}
         r_{w^{-1}Pw \vert P}(\pi \otimes \chi_{2s\rho_P})^{-1}
        l(\widetilde{w}) J_{w^{-1}Pw \vert P}(\pi \otimes \chi_{2s\rho_P}, \psi)
    \end{align*}
    is denoted by $R_{P}(w, \pi \otimes \chi_{2s\rho_P}, \psi)$.
    This operator is holomorphic in $s$ and analytic in $\pi \in \Temp_{\ind}(M_{\ad})$.
\end{itemize}
Thus, the last expression is equal to 
\begin{align*}
    \int_{\pi \in \Temp_{\ind}(M_{\ad})} 
    \tr(v \in \pi \mapsto R(w, \pi \otimes \chi_{2s\rho_P}, \psi) \widetilde{\phi_{s, \pi(f^{\vee})v}} (1)) \\
    \times 
    r_{w^{-1}Pw \vert P}(\pi \otimes \chi_{2s\rho_P}) \mu^{M_{\ad}}_{\psi}(\pi) \rd \pi.
\end{align*}
We give the summary of this computation.
\begin{lemma}\label{lemma:summary}
    We have 
    \begin{multline*}
        \lim_{s \to 0+}
        \gamma(s, \mathbf{1}, \psi)^2
        \int_{n \in N \cap w\overline{N}w^{-1}} 
        \phi_{s, f}(\widetilde{w}^{-1}n)
        \rd_{\psi}n = 
        \lim_{s \to 0+}
        \gamma(s, \mathbf{1}, \psi)^2 \\
        \times 
        \int_{\pi \in \Temp_{\ind}(M_{\ad})} 
        \tr(v \in \pi \mapsto R_{P}(w, \pi \otimes \chi_{2s\rho_P}, \psi) \widetilde{\phi_{s, \pi(f^{\vee})v}} (1)) \\
        \times
        r_{w^{-1}Pw \vert P}(\pi \otimes \chi_{2s\rho_P} )\mu^{M_{\ad}}_{\psi}(\pi)\rd \pi.
    \end{multline*}
\end{lemma}

The following computation of the limit is the key to prove our main theorem:
\begin{theorem}\label{theorem:key-computation}
     There exists a smooth function $c(\phi^H)$ on the space  $\Temp_{\ind}(H)/\sim_{st}$ which takes its value in $\bS^{1}$ and satisfying the following condition:
     Let $\Phi$ be a compactly supported smooth function on $\Temp_{\ind}(G)$.
     Then, we have 
     \begin{multline*}
        \lim_{s \to 0+}
        \gamma(s, \mathbf{1}, \psi)^2
        \int_{\pi \in \Temp_{\ind}(M_{\ad})} 
        \Phi(\pi)    
        r_{w^{-1}Pw \vert P}(\pi \otimes \chi_{2s\rho_P} )\mu^{M_{\ad}}_{\psi}(\pi)
        \rd \pi \\
        = 
        \frac{1}{512}
        \int_{\phi^H \in \Temp_{\ind}(H)/\sim_{st}} 
        \lbrace
        \sum_{\pi \in \Pi^X_{\phi^H}} 
        \Phi(\pi)    
        \rbrace
        c(\phi^H)  
        \mu^H_{\HII, \psi}(\phi^H)
        \rd \phi^H.
     \end{multline*}
\end{theorem}

\begin{proof}[Proof that \cref{theorem:key-computation} implies \cref{theorem:spectral-limit-formula}]
    We apply \cref{theorem:canonical-extension-PGSO8-S3} which states that
    \begin{align*}
        R(w^{-1}, \pi, \psi) = i^G_P(\widetilde{\pi}(w))
    \end{align*}
    for the canonical extension $\widetilde{\pi}$ of $\pi$.
    Thus, we have
    \begin{align*}
        R_{P}(w, \pi) \widetilde{\phi_{0, \pi(f^{\vee})v}} (1)
        = 
        i^G_P(\widetilde{\pi}(w^{-1})) \widetilde{\phi_{0, \pi(f^{\vee})v}} (1)
        = 
        \widetilde{\pi}(w^{-1}) \pi(f^{\vee}) v
    \end{align*}
    for any $\pi \in \Temp_{\ind}(M_{\ad})$ and any $v \in \pi$.
    By \cref{lemma:summary} and \cref{theorem:key-computation}, we have 
    \begin{multline*}
        \lim_{s \to 0+}
        \gamma(s, \mathbf{1}, \psi)^2
        \int_{n \in N \cap w\overline{N}w^{-1}} 
        \phi_{s, f}(\widetilde{w}^{-1}n)
        \rd_{\psi}n  \\
        = 
        \int_{\phi^H \in \Temp_{\ind}(H)/\sim_{st}} 
        \lbrace
        \sum_{\pi \in \Pi^X_{\phi^H}} 
        \tr(\widetilde{\pi}(w^{-1}) \pi(f^{\vee}))
        \rbrace
        c(\phi^H)  
        \mu_{\HII, \psi}^H(\phi^H)
        \rd \phi^H.
    \end{multline*}
    This completes the proof.
\end{proof}

\begin{remark}\label{remark:localization}
   By taking partitions of unity, we can reduce the computation of the limit for general functions $\Phi$ in \cref{theorem:key-computation}  to the computation for functions supported in a small neighborhood of each point $\pi \in \Temp_{\ind}(M_{\ad})$.
   In particular, we may and do assume that $\Phi$ is supported in a small neighborhood of a representation $\pi_0 \in \Temp_{\ind}(M_{\ad})$.
\end{remark}

\subsubsection{Proof of \cref{theorem:key-computation} for non-distinguished cases}

\begin{lemma}\label{lemma:residue-theorem}
    As meromorphic functions in the variables $\alpha_1, \dots, \alpha_n$, we have 
    \begin{align*}
        \frac{1}{\prod_{i=1}^{n} \alpha_i} 
        = 
        (-1)^{n-1} 
        \sum_{i=1}^n 
        \frac{1}{\alpha_i} 
        \times    
        \frac{1}{\prod_{j, j \neq i} (\alpha_i-\alpha_j)}.  
    \end{align*}
\end{lemma}

\begin{proof}
    Let $\alpha_1, \dots, \alpha_n$ be pairwise distinct nonzero complex numbers.
    We consider the meromorphic function 
    \begin{align*}
        \frac{1}{\prod_{i = 1}^{n} (z-\alpha_i)}. 
    \end{align*}
    By the residue theorem, this function can be written as 
    \begin{align*}
        \sum_{i = 1}^{n} \frac{1}{z-\alpha_i} 
        \times    
        \frac{1}{\prod_{i \neq j}(\alpha_i - \alpha_j)}.
    \end{align*} 
    Evaluating both functions at $z=0$, we obtain the result.
\end{proof}

\begin{proposition}\label{proposition:non-distinguished}
    Let $\pi_0 \in \Temp_{\ind}(M_{\ad})$ be an induced representation of $M_{\ad}$ such that $\gamma(0, \pi_0, \Std, \psi) \neq 0$ or $\gamma(0, {}^{s_1}\pi_0, \Std, \psi) \neq 0$.
    Then, \cref{theorem:key-computation} holds for $\Phi$ supported on a sufficiently small open neighborhood of $\pi_0$.
\end{proposition}

\begin{proof}
    By \cref{proposition:normalizing-factor-for-s1s6}, 
    we have 
    \begin{multline*}
        r_{w^{-1}Pw \vert P}(\pi \otimes \chi_{2s\rho_P}) \\
        =
        \epsilon(\frac{1}{2} + 16s, \pi, \Std_{\GSO_8}, \psi)
        \epsilon(\frac{1}{2} + 8s, {}^{s_1}\pi, \Std_{\GSO_8}, \psi) \\
        \times
        \gamma_A(16s, \pi, \Std_{\GSO_8}, \psi)^{-1} 
        \gamma_A(8s, {}^{s_1}\pi, \Std_{\GSO_8}, \psi)^{-1}.
    \end{multline*} 
    By the symmetry under $\pi \mapsto {}^{s_1}\pi$, we may assume $\gamma(0, {}^{s_1}\pi_0, \Std, \psi) \neq 0$.
    We use the notation in the proof of \cref{lemma:adjoint-L-explicit} and \cref{remark:standard-poles}.
    In particular, we set $\phi^{\GL_{8}}_0 = \phi^{\GL_{8}}_{\pi_0}$.
    We also set $m=\abs{I_{\triv}}$.

    We first assume $\abs{J_{\triv}} = 1$.
    Then, using the explicit description of regularized $\gamma$-factors in the proof of \cref{lemma:adjoint-L-explicit} and \cref{remark:standard-poles} with \cref{remark:replace-arthur-by-artin}, we see that if $\phi^{\GL_{8}}_{\pi} = \phi^{\GL_{8}}_{\lambda}$, then the integrand on the left-hand side of \cref{theorem:key-computation} is equal to 
    \begin{align*}
        \frac{1}{s}
        \prod_{i \in I_{\triv}} \frac{\lambda_i^2 }{\lambda_i^2 + (16s)^2} 
        \Phi(\pi)
        S(s, \lambda) 
    \end{align*} 
    for some smooth function $S(s, \lambda)$ at $(s, \lambda) = (0, 0)$. 
    As we have
    \begin{align*}
    \gamma(s, \mathbf{1}, \psi) = \frac{1-q^{-s}}{1-q^{-(1-s)}},
    \end{align*}
    and \cref{remark:replace-arthur-by-artin}, the resulting limit
    \begin{align*}
        \lim_{s \to 0+}
        \gamma(s, \mathbf{1}, \psi)^2
        \int_{\pi \in \Temp_{\ind}(M_{\ad})} 
        \Phi(\pi)    
        r_{w^{-1}Pw \vert P}(\pi \otimes \chi_{2s\rho_P} )
        \mu^{M_{\ad}}_{\psi}(\pi)\rd \pi 
    \end{align*}
    is zero if $\Phi$ is supported on a sufficiently small neighborhood of $\pi_0$.
    In fact, the limit 
    \begin{align*}
        \lim_{s \to 0+}
        \gamma(s, \mathbf{1}, \psi)
        \int_{\pi \in \Temp_{\ind}(M_{\ad})} 
        \Phi(\pi)    
        r_{w^{-1}Pw \vert P}(\pi \otimes \chi_{2s\rho_P} )
        \mu^{M_{\ad}}_{\psi}(\pi)\rd \pi 
    \end{align*}
    exists.

    We now assume $\abs{J_{\triv}} = 0$.
    Then, the integrand on the left-hand side of \cref{theorem:key-computation} is equal to 
    \begin{align*}
        \frac{\prod_{(i_1, i_2) \in I_{\triv}^2, i_1 \neq i_2}  
        (\lambda_{i_1}^2 - \lambda_{i_2}^2)}{\prod_{i \in I_{\triv}}  (\lambda_i^2 + (16s)^2)} 
        \Phi(\pi)
        S(s, \lambda) 
    \end{align*} 
    for some smooth function $S(s, \lambda)$ at $(s, \lambda) = (0, 0)$. 
    By \cref{lemma:residue-theorem}, this expression is equal to 
    \begin{align*}
         (-1)^{\abs{I_{\triv}}-1} 
         \sum_{i \in I_{\triv}}   
         \frac{1}{\lambda_i^2 + (16s)^2} \times 
         \frac{\prod_{(i_1, i_2) \in I_{\triv}^2, i_1 \neq i_2}  
         (\lambda_{i_1}^2 - \lambda_{i_2}^2)}
         {\prod_{j \in I_{\triv},  j \neq i} (\lambda_i^2 - \lambda_j^2)}
         \Phi(\pi) S(s, \lambda).
    \end{align*}
    Note that the factor 
    \begin{align*}
        \frac{\prod_{(i_1, i_2) \in I_{\triv}^2, i_1 \neq i_2}  
         (\lambda_{i_1}^2 - \lambda_{i_2}^2)}
         {\prod_{j \in I_{\triv},  j \neq i} (\lambda_i^2 - \lambda_j^2)}
         \Phi(\pi) S(s, \lambda)
    \end{align*}
    is smooth.
    As the distribution 
    \begin{align*}
           \frac{s}{\lambda_i^2 + (16s)^2} 
    \end{align*}
    converges to the delta distribution
    \begin{align*}
       \frac{1}{16} \pi \delta_0(\lambda_i)
    \end{align*}
    as $s \to 0+$, the resulting limit
    \begin{align*}
        \lim_{s \to 0+}
        \gamma(s, \mathbf{1}, \psi)^2
        \int_{\pi \in \Temp_{\ind}(M_{\ad})} 
        \Phi(\pi)    
        r_{w^{-1}Pw \vert P}(\pi \otimes \chi_{2s\rho_P} )\mu^{M_{\ad}}_{\psi}(\pi)\rd \pi 
    \end{align*}
    is zero if $\Phi$ is supported on a sufficiently small neighborhood of $\pi_0$.
\end{proof}

\subsubsection{Proof of \cref{theorem:key-computation} for distinguished cases}

By \cref{remark:localization} and \cref{proposition:non-distinguished}, it suffices to compute the limit at each point $\pi$ satisfying
\begin{align*}
    \begin{cases}
     \gamma(0, \pi, \Std, \psi) = 0, \\
     \gamma(0, {}^{s_1}\pi, \Std, \psi) = 0.
    \end{cases}
\end{align*}
Let $\Temp_{\ind, \mathrm{dist}}(M_{\ad})$ be the subset of $\Temp_{\ind}(M_{\ad})$ consisting of the representations satisfying the condition.
An element in $\Temp_{\ind, \mathrm{dist}}(M_{\ad})$ is said to be distinguished.

By \cref{proposition:construction-of-distinguished-Xu-packet} and its proof, we have a bijection
\begin{align*}
    \Temp_{\ind, \mathrm{dist}}(M_{\ad})/\sim_{st, w} \xrightarrow{\sim} \Phi_{\temp}(H).
\end{align*}
On the other hand, we have a bijection  
\begin{align*}
    \Temp_{\ind}(H)/\sim_{st} \xrightarrow{\sim} \Phi_{\temp}(H).
\end{align*}
By composing these two maps, we obtain the isomorphism
\begin{align*}
    \theta \colon \Temp_{\ind}(H)/\sim_{st} \xrightarrow{\sim} \Temp_{\ind, \mathrm{dist}}(M_{\ad})/\sim_{st, w}.
\end{align*}
By \cref{proposition:smoothness-of-endoscopic-lift}, this map is smooth.
Through this map $\theta$, we will identify the space $\Temp_{\ind}(H)/\sim_{st}$ as a closed subspace in the space $\Temp_{\ind}(M_{\ad})/\sim_{st, w}$.
Since every proper Levi subgroup of $H$ is a product of general linear groups, the relation $\sim_{st}$ is trivial on the subspace of non-discrete-series representations of $H$.

In this situation, we have 
\begin{align}
    \gamma(s, \theta(\phi^H), \Ad_{\PGSO_8}, \psi)
    = 
    \gamma(s, \phi^H, \Ad_{H}, \psi) \gamma(s, \phi^H, \Std_{H}, \psi)^2 \label{eq:decomposition-adjoint-restriction-to-G2}
\end{align}
by decomposing the restriction of the adjoint representation of $\widehat{\PGSO_8}$ to $\widehat{H}$.

We finally start the proof of \cref{theorem:key-computation}.
We follow the method in \cite{Beuzart-Plessis2021-Plancherel-GLnE-GLnF}*{Proposition 3.41}.

\begin{proof}[The proof of \cref{theorem:key-computation}]
    Let $\pi_0 \in \Temp_{\ind, \dist}(M_{\ad})$ be a distinguished induced representation of $M_{\ad}$. 
    By \cref{proposition:non-distinguished} and \cref{remark:localization}, it suffices to prove the theorem for any function $\Phi$ supported on a sufficiently small open neighborhood of $\pi_0 \in  \Temp_{\ind}(M_{\ad})$.

    Note that if $\pi_0$ is not a discrete series representation, we can attach an $L$-parameter for $M_{\ad}$ to any representation $\pi$ in the connected component of $\pi_0$ in $\Temp_{\ind}(M_{\ad})$ by \cref{proposition:refinement-GSO4} and \cref{theorem:local-Langlands-GLn}.
    If $\pi_0$ is discrete series, by \cref{proposition:construction-of-distinguished-Xu-packet}, we can attach an $L$-parameter of $\PGSO_8(F)$ of the form $\iota^{\PGSO_8}_{G_2} \circ \phi^{G_2}$ for a uniquely determined element $\phi^{G_2} \in \Phi_{\disc}(G_2)$.
    Let $\phi_{\pi}$ denote the $L$-parameter of $\pi$ in the cases above.
    From now on, we assume that the function $\Phi$ is supported on the connected component $\cO = \cO_{\pi_0}$ of $\pi_0$ if $\pi_0$ is not a discrete series representation.
    Also, we introduce the equivalence relation $\sim_{st}$ on $\cO$, defined by $\pi \sim_{st} \pi'$ if and only if $\phi_{\pi} = \phi_{\pi'}$.
    We will identify the class of $\pi$ with $\phi_{\pi}$.
    This relation is trivial on $\cO$ if $\pi_0$ is not a discrete series representation.
    Let $\phi_0$ be the class of $\pi_0$.
 
    First, it is easy to see that the functions 
    \begin{align*}
        \mu^{M_{\ad}}_{\psi}(\pi), r_{w^{-1}Pw \vert P}(\pi \otimes \chi_{2s\rho_P})
    \end{align*}
    only depend on the $L$-parameter of $\pi$, by \cref{lemma:adjoint-L-explicit} and \cref{proposition:normalizing-factor-for-s1s6}.
    We denote them by 
    \begin{align*}
        \mu^{M_{\ad}}_{\psi}(\phi_{\pi}), r_{w^{-1}Pw \vert P}(\phi_{\pi} \otimes \chi_{2s\rho_P}).
    \end{align*}
    Let $\phi_0$ denote the class of $\pi_0$.
    Then, we obtain 
    \begin{multline}
        \lim_{s \to 0+}
        \gamma(s, \mathbf{1}, \psi)^2
        \int_{\pi \in \Temp_{\ind}(M_{\ad})} 
        \Phi(\pi)    
        r_{w^{-1}Pw \vert P}(\pi \otimes \chi_{2s\rho_P} )\mu^{M_{\ad}}_{\psi}(\pi)\rd \pi  \\
        = 
        \lim_{s \to 0+}
        \gamma(s, \mathbf{1}, \psi)^2
        \int_{\phi \in \cO/\sim_{st}}
        \{
        \sum_{\phi_{\pi} = \phi}
        \Phi(\pi)    
        \}
        r_{w^{-1}Pw \vert P}(\phi \otimes \chi_{2s\rho_P} ) \mu^{M_{\ad}}_{\psi}(\phi)\rd \phi. 
    \end{multline}

    From now on, for the $L$-parameter $\phi_0$, we use the notation $I_{\triv}, J_{\triv}$ and so on  in \cref{remark:divide-five-cases},  \cref{lemma:adjoint-L-explicit} and \cref{remark:standard-poles}.
    In Cases $(1), (3)$, we have $\abs{J_{\triv}} = 1$ and in Cases $(2), (4), (5)$, we have $\abs{J_{\triv}} = 0$. 
    In any case, $\abs{I_{\triv}}$ is even.

    We also use the notation in \cref{subsection:setting-and-L-embeddings}, \cref{proposition:description-of-endoscopic-lift-case-2}, \cref{proposition:description-of-endoscopic-lift-case-3}, \cref{proposition:description-of-endoscopic-lift-case-4}, and \cref{proposition:description-of-endoscopic-lift-case-5}.
    In particular, we can write 
    \begin{align*}
    &\phi^H_0 = \iota^H_I \circ \phi^I_0 \\
    &\phi_0 = \iota^{M_{\ad}}_L \circ \phi^L_0,
    \end{align*}
    where $\phi^I_0$ (resp. $\phi^L_0$) is a discrete $L$-parameter for $I$ (resp. $L$). 

    We first consider Cases $(1)$ and $(3)$.
    We have $\abs{J_{\triv}} = 1$.
    Let $\cO^H$ be the connected component of $\Temp_{\ind}(H)$ containing $\phi_0^H = \theta^{-1}(\phi_0)$.
    Then, it is easy to see that, in these cases $(1), (3)$, the map $\theta$ induces the measure-preserving isomorphism 
    \begin{align*}
        \theta \colon (\cO^{H}, \rd \phi^{H}) \xrightarrow{\sim}
     (\cO, \rd \phi).
    \end{align*}
    Also, using \cref{proposition:smoothness-of-endoscopic-lift}, \cref{lemma:adjoint-L-explicit}, \cref{eq:decomposition-adjoint-restriction-to-G2}, and the equation
    \begin{align*}
        4\abs{S^{\natural}_{\phi^{I}}} = \abs{S^{\natural}_{\phi^{L/A_M}}},
    \end{align*}
    which follows from \cite{GS23LLC}*{Lemma 2.4(ii)} and a simple calculation, it is easy to see that 
    \begin{multline*}
        \lim_{s \to 0+} r_{w^{-1}Pw \vert P}(\theta(\phi^H) \otimes \chi_{2s\rho_P})\gamma(s, \mathbf{1}, \psi)^2 \mu^{M_{\ad}}_{\psi}(\theta(\phi^H)) \\
        =
        \frac{1}{512}
        c'(\theta(\phi^H)) \frac{\gamma^{*}(0, \phi^H, \Ad_{H}, \psi)}{\abs{S_{\phi^I}^{\natural}}},
    \end{multline*}
    for a smooth function $c'$ on $\cO^H/\sim_{st}$ valued in $\bS^1$.
    In this case, the Levi subgroup $I$ of $H$ is isomorphic to $H$ or $\GL_2(F)$ corresponding to the short root, we can easily see that the right-hand side can be written as
    \begin{align*}
        \frac{1}{512}
        c(\phi^H) 
        \mu^H_{\HII, \psi}(\phi^H),
    \end{align*}
    for a smooth function $c$ on $\cO^H/\sim_{st}$ valued in $\bS^1$ up to a null set in $\cO^H/\sim_{st}$ as in \cref{lemma:adjoint-L-explicit} and \cref{corollary:adjoint-L-abs-value}.
    Here, we can use, for example, the following computation of the restriction of the adjoint representation in Case $(3)$ based on \cref{eq:decomposition-adjoint-restriction-to-G2}: 
    \begin{multline*}
        \Ad_{H} \circ \iota^{H}_I \circ \phi^I
        = 
        \Ad(\phi^I) 
        \oplus 
        \mathbf{1} 
        \oplus 
        \det(\phi^I) 
        \oplus 
        \det(\phi^I)^{-1} \\
        \oplus \Sym^3(\phi^I) \otimes \det(\phi^I)^{-1} 
        \oplus \Sym^3(\phi^{I})^{\vee} \otimes \det(\phi^I).
    \end{multline*}
    We omit a similar computation.
    Hence, by the Lebesgue convergence theorem, we have 
    \begin{multline*}
        \lim_{s \to 0+}
        \gamma(s, \mathbf{1}, \psi)^2
        \int_{\phi \in \cO/\sim_{st}}
        \{
        \sum_{\phi_{\pi} = \phi}
        \Phi(\pi)    
        \}
        r_{w^{-1}Pw \vert P}(\phi \otimes \chi_{2s\rho_P} )\mu^{M_{\ad}}_{\psi}(\phi)\rd \phi \\
        =
        \frac{1}{512}
        \int_{\phi^H \in \cO^H/\sim_{st}}
        \{
        \sum_{\phi_{\pi} = \theta(\phi^H)}
        \Phi(\pi)    
        \}
        c(\phi^H) 
        \mu^H_{\HII, \psi}(\phi^H)
        \rd \phi^H,
    \end{multline*}
    for a smooth function $c$ on $\cO^H/\sim_{st}$ valued in $\bS^1$.
    This implies the theorem.

    We move on to the proof for cases $(2), (4), (5)$.
    In this case, we have the map 
    \begin{align*}
       \theta \colon \cO^H/{\sim_{st}} \hookrightarrow \cO,
    \end{align*}
    whose image is a proper closed subset of $\cO$.
   
    To compute explicitly, we take a local coordinate at $\pi_0$.
    Let $\cU$ be a sufficiently small neighborhood of $0 \in \fa_L^{*}$ which is $W(M_{\ad}, L)_{\sigma_0^L}$ invariant. 
    Then, there exists a sufficiently small neighborhood $\cV \subset \Temp_{\ind}(M_{\ad})$ of $\pi_0$ such that the map
    \begin{align*}
        \lambda \in \cU \mapsto \pi_{\lambda} =  i^{M_{\ad}}_{L} (\sigma^L_0 \otimes \chi_{i\lambda}) \in \cV,
    \end{align*}
    is an isomorphism between $\cU/W(M_{\ad}, L/A_M)_{\sigma_0^L}$ and $\cV$.
    If $\Psi$ is any $L^1$ function on $\cV$, then we have 
    \begin{align*}
        \int_{\pi \in \cV} \Psi(\pi) \rd \pi 
        = 
        \frac{1}{\abs{W(M_{\ad}, L/A_M)_{\sigma_0^L}}}
        \int_{\lambda \in \cU} \Psi(\pi_{\lambda}) \rd \lambda.
    \end{align*}
    From now on, we assume that $\Phi$ is supported on $\cV$.
    Then, we have
    \begin{multline}
        \gamma(s, \mathbf{1}, \psi)^2
        \int_{\pi \in \Temp_{\ind}(M_{\ad})} 
        \Phi(\pi)    
        r_{w^{-1}Pw \vert P}(\pi \otimes \chi_{2s\rho_P} )\mu^{M_{\ad}}_{\psi}(\pi)\rd \pi  \\
        = 
        \gamma(s, \mathbf{1}, \psi)^2
        \frac{1}
        {\abs{W(M_{\ad}, L/A_M)_{\sigma_0^L}}} 
        \int_{\lambda \in \cU} 
        \Phi(\pi_{\lambda})    \\
        \times 
        r_{w^{-1}Pw \vert P}(\pi_{\lambda} \otimes \chi_{2s\rho_P} )\mu^{M_{\ad}}_{\psi}(\pi_{\lambda})\rd \lambda.
        \label{eq:localization}
    \end{multline}

    In Cases $(2), (4), (5)$, the group ${}^L I$ factors through ${}^L H' = {}^L \PGL_3$. 
    Let $\iota^{H'}_{I}$ be the embedding ${}^L I \hookrightarrow {}^L H'$.
    We set $\phi^{H'}_0 = \iota^{H'}_I \circ \phi^{I}_0$.
    Also, by \cref{lemma:description-of-L-homomorphism-case-2}, we have a Weyl  element $\widetilde{\cW} \in \Spin_8$ such that $\widetilde{\cW}(\widehat{H'})$ is the derived group of the standard Levi subgroup of $\widehat{M_{\ad}}$ corresponding to $\{ \alpha_2^{\vee}, \alpha_3^{\vee} \}$.
  
    \begin{lemma}
        On the maximal torus $T_{M_{\ad}} = T/A_M \subset M_{\ad}$, the action of $\cW$ in \cref{lemma:description-of-L-homomorphism-case-2} commutes with the action of $s_1$. 
        in particular, we have 
        \begin{align*}
             s_1(\widetilde{\cW}) = \widetilde{\cW},
        \end{align*}
        for the Tits lifting $\widetilde{\cW}$ for $\cW$.
        We have a similar result for $\widehat{M_{\ad}}$.
    \end{lemma}

    \begin{proof}
        If we take the basis $\alpha_1, \dots, \alpha_4$ for $X^{*}(T_{M_{\ad}})$, then the action of $\cW$ is given by the matrix
        \begin{align*}
            \begin{pmatrix}
                0 & 0 & 1 & -1 \\
                -1 & 1 & 1 & -1 \\
                0 & 0 & 1 & 0 \\
                -1 & 0 & 1 & 0
            \end{pmatrix}.
        \end{align*}
        It is easy to show that this matrix commutes with the matrix 
        \begin{align*}
            \begin{pmatrix}
                0 & 0 & 0 & 1 \\
                0 & 1 & 0 & 0 \\
                0 & 0 & 1 & 0 \\
                1 & 0 & 0 & 0
            \end{pmatrix},
        \end{align*}
        which represents the action of $s_1$ on $X^{*}(T_{M_{\ad}})$.
    \end{proof}

    Recall that the group $H'$ is equal to $\PGL_3$.
    The standard $L$-parameter of $\phi_0 = \iota^{M_{\ad}}_{I} \circ \phi^I_0$ is given by 
    \begin{align}
        \mathbf{1} \oplus \phi^{H'}_0 \oplus (\phi^{H'}_0)^{\vee} \oplus \mathbf{1}.
        \label{eq:parameter-from-PGL3}
    \end{align}
    In Case $(4)$, the group $I$ can be seen as both a Levi subgroup of $H'$ and a Levi subgroup of $H$.
    The parameter $\phi^{H'}_{0}$ is equal to 
    \begin{align*}
           (\det \phi^I_0)^{-1} \oplus \phi^I_0.
    \end{align*}
    Similarly, in Case $(5)$, the group $I$ is a maximal torus of $H' = \PGL_3$ and $\phi^I_0$ can be seen as the triple $\chi_1, \chi_2, \chi_3 \in \Phi_{\temp}(\GL_1)$ such that $\prod_{i} \chi_i = 1$.
    In this case, we have 
    \begin{align*}
        \phi^{H'}_{0} = \chi_1 \oplus \chi_2 \oplus \chi_3.
    \end{align*}
    Thus, we have the three possibilities:
    \begin{itemize}
        \item    
        $\chi_i \neq \mathbf{1}$ for $i=1, 2, 3$.
        \item     
        $\chi_1 = \mathbf{1}$ and $\chi_2 = \chi_3^{-1} \neq \mathbf{1}$.
        \item    
        $\chi_i = \mathbf{1}$ for $i=1, 2, 3$.
    \end{itemize}
    The $L$-parameter for $\pi_{\lambda}$ is given by $\widetilde{\cW}(\phi_0^L \otimes \chi_{\lambda})$.

    Depending on the form of parameters, we define the following set of weights in the standard representation of $\Spin_8(\C)$.
    \begin{definition}
        Let $w(\Std_{M_{\ad}})$ be the set of weights of the action of $\widehat{T_{M_{\ad}}}$ on the standard representation of $\widehat{M_{\ad}} = \Spin_8$.
        Let $w(\phi^H_0) \subset w(\Std_{M_{\ad}})$ be a subset of weights in the standard representation of $\Spin_8(\C)$ defined as follows:
        \begin{itemize}
          \item   
          In Case $(2)$, we define $w(\phi^H_0)$ as 
          \begin{align*}
             \{ \varpi_{\beta_1}^{\vee}, -\varpi_{\beta_1}^{\vee} \}.
          \end{align*}
          \item    
          In Case $(4)$, we define $w(\phi^H_0)$ as 
          \begin{align*} 
            \begin{cases}
                \{ \varpi_{\beta_1}^{\vee}, -\varpi_{\beta_1}^{\vee} \} &\quad \det(\phi^I_0) \neq \mathbf{1}, \\
                \{ \varpi_{\beta_1}^{\vee}, \varpi_{\beta_1}^{\vee} - \alpha_1, -\varpi_{\beta_1}^{\vee} + \alpha_1, -\varpi_{\beta_1}^{\vee} \} &\quad \det(\phi^I_0) = \mathbf{1}.
            \end{cases}
          \end{align*}
          In Case $(5)$, we define $w(\phi^H_0)$ as 
          \begin{align*}
            \begin{cases}
                \{ \varpi_{\beta_1}^{\vee}, -\varpi_{\beta_1}^{\vee} \} &\quad \chi_1, \chi_2, \chi_3 \neq \mathbf{1}, \\
                \{ \varpi_{\beta_1}^{\vee}, \varpi_{\beta_1}^{\vee} - \alpha_1, -\varpi_{\beta_1}^{\vee} + \alpha_1, -\varpi_{\beta_1}^{\vee} \} &\quad \chi_1 = \mathbf{1},\chi_2 = \chi_3^{-1} \neq \mathbf{1}, \\  
                w(\Std_{M_{\ad}})
                &\quad \chi_1  = \chi_2 = \chi_3 = \mathbf{1}.
            \end{cases}
          \end{align*}
        \end{itemize}
        We will see an elements of the set $w(\phi^H_0)$ as a linear function on $\cW \fa^*_{L, \C}$.
        Let $w(\phi^H_0)/\{\pm 1\}$ be the quotient of the set $w(\phi^H_0)$ by $\pm 1$.
        For an element $\varpi \in w(\phi^H_0)$, let $[\varpi]$ denote the class of $\varpi$ in $w(\phi^H_0)/\{\pm 1\}$.
    \end{definition}

    \begin{lemma}\label{lemma:transitivity}
        The action of $W(M_{\ad}, L/A_M)_{\sigma^L_0}$ on the space $\fa^*_{L, \C}$ induces an action on the product set $w(\phi^H_0)/\{\pm 1\} \times s_1w(\phi^H_0)/\{\pm 1\}$ through the isomorphism $\cW \colon \fa^*_{L, \C} \xrightarrow{\sim} \cW\fa^*_{L, \C}$.
        Also, this action is transitive.
    \end{lemma}

    \begin{proof}
    We prove this lemma also by case-by-case consideration. 
    If the cardinality of the set $w(\phi^H_0)/\{\pm 1\}$ is $1$ or $2$, the group generated by $s_{\beta_1}$ and $s_{\beta_4}$ acts the set $w(\phi^H_0)/\{\pm 1\} \times s_1w(\phi^H_0)/\{\pm 1\}$ transitively.

    So, we assume that the set $w(\phi^H_0)/\{\pm 1\}$ has the cardinality $4$.
    This occurs only if we are in Case $(5)$ and $\chi_1 = \chi_2 = \chi_3 = \mathbf{1}$.
    Let $(\Z/2\Z)^4_0 \subset (\Z/2\Z)^4$ be the subgroup consisting of elements $(v_i)_{i=1, 2, 3, 4}$ with $\sum_i v_i = 0$. 
    Then, the Weyl group of $M_{\ad}$ and $T/A_M$ is isomorphic to $(\Z/2\Z)^4_0 \rtimes S_4$.
    In this case, it is easy to see that the action of $S_4 \subset (\Z/2\Z)^4 \rtimes S_4$ on the set $w(\Std_{M_{\ad}})/{\pm 1}$ is transitive.
    Thus, to prove the transitivity stated in this lemma, it suffices to show that the map 
    \begin{align*}
        (\Z/2\Z)^4_0 \rtimes (S_1 \times S_3) 
        \overset{s_1}{\hookrightarrow} 
        (\Z/2\Z)^4_0 \rtimes S_4 
        \twoheadrightarrow 
        S_4
    \end{align*} 
    is surjective.
    Here, $s_1$ is the automorphism of $W(M_{\ad}, T/A_M)$ induced by $s_1$.

    Let $\sigma_i$ be the simple reflection corresponding to the simple root $\beta_i$ for $i= 1, 2, 3, 4$.
    Then, the group 
    \begin{align*}
        (\Z/2\Z)^4_0 \rtimes (S_1 \times S_3) 
    \end{align*}
    is equal to 
    \begin{align*}
        \langle \sigma_3 \sigma_4 \rangle_{\mathrm{norm}} \rtimes 
        \langle \sigma_2, \sigma_3 \rangle,
    \end{align*}
    where the group $\langle \sigma_3 \sigma_4 \rangle_{\mathrm{norm}}$ is the normal subgroup of the Weyl group generated by $\sigma_3 \sigma_4$.
    Thus, the image of the injection 
    \begin{align*}
        (\Z/2\Z)^4_0 \rtimes (S_1 \times S_3) 
        \overset{s_1}{\hookrightarrow} 
        (\Z/2\Z)^4_0 \rtimes S_4
    \end{align*}
    is isomorphic to the group
    \begin{align*}
            \langle \sigma_3 \sigma_1 \rangle_{\mathrm{norm}} \rtimes 
            \langle \sigma_2, \sigma_3 \rangle,
    \end{align*}    
    as the automorphism $s_1$ switches $\sigma_1$ and $\sigma_4$. 
    The group $\langle \sigma_2, \sigma_3 \rangle$ is isomorphic to $S_1 \times S_3$. 
    Also, the image of $\langle \sigma_3 \sigma_1 \rangle_{\mathrm{norm}}$ by the map 
    \begin{align*}
        (\Z/2\Z)^4_0 \rtimes S_4 
        \twoheadrightarrow 
        S_4
    \end{align*}
    is the Klein four-group in $S_4$.
    As these two groups generate $S_4$, we obtain the result.
    \end{proof}

    \begin{remark}
        The set $w(\phi^H_0)$ is bijective to the set $I_{\triv}$ for the parameter $\phi_0^{\GL_8} = \iota^{\GL_8}_{M_{\ad}} \circ \phi^{M_{\ad}}$.
        Thus, this set $w(\phi^H_0)$ parametrizes trivial representations in the standard $L$-parameter $\phi_0^{\GL_8} = \iota^{\GL_8}_{M_{\ad}} \circ \phi^{M_{\ad}}$ for $\phi^{M_{\ad}}$.
    \end{remark}

    From now on, for the sake of simplicity, let $\cW$ denote the action of $\widetilde{\cW}$.
    We define the polynomial related to the standard $\gamma$-factor of $\pi \otimes \chi_{2s \rho_P}$.
    \begin{definition}
        For an element $\varpi \in w(\phi^H_0)$, let $P_{[\varpi]}$ be the polynomial on $\fa_{L,\C}^{*}$ defined by 
        \begin{align*}
            P_{[\varpi]}(\lambda) 
            = 
            \prod_{[\varpi'] \in w(\phi^H_0)/\{\pm 1\}, [\varpi'] \neq [\varpi]} (\varpi(\cW\lambda)^2 -  \varpi'(\cW\lambda)^2).
        \end{align*}
        We also set 
        \begin{align*}
            R_{s}(\lambda) 
            = 
            \prod_{[\varpi] \in w(\phi^H_0)/\{\pm 1\}} (\varpi(\cW \lambda)^2 + s^2).
        \end{align*}
    \end{definition}
    We note that in \cref{lemma:adjoint-L-explicit}, we can write the adjoint $\gamma$-factor of $\pi_{\lambda}$ as 
    \begin{align*}
        P(\pi_{\lambda})^2 Q(\pi_{\lambda})
    \end{align*}
    with a function $P$ which is a polynomial in $\fa_{L,\C}^{*}$ with real coefficients with respect to $\fa^{*}_L$, and a non-vanishing smooth function $Q(\pi_{\lambda})$ at $\lambda = 0$.
    
    \begin{lemma}\label{lemma:regularity}
        Let $\varpi$ be an element in $w(\phi^H_0)$.
        Then, the rational function 
        \begin{align*}
            \frac{P(\pi_{\lambda})^2}
            {P_{[\varpi]}(\cW\lambda) 
            \, P_{[\varpi]}(s_1\cW\lambda)}. 
        \end{align*}
        is a polynomial on $\fa_{L,\C}^{*}$.
    \end{lemma}

    \begin{proof}
    We first note that the action of $\cW$ and $s_1$ on $T$ commutes, which is used without comments.

    As in the proof of the previous lemma, we omit the easy argument for the cases where $w(\phi^H_0)/\{ \pm 1\} = 1, 2$.

    We consider the case where $w(\phi^H_0)/\{ \pm 1\} = 4$, which is the most difficult.
    We take a factor 
    \begin{align*}
        \varpi(\cW\lambda)^2 - \varpi'(\cW \lambda)^2
    \end{align*}
    in the definition of $P_{[\varpi]}$.
    Then, by the explicit description of the polynomial $P(\pi_{\lambda})$ in \cref{eq:adjoint-L-explicit}, it is easy to see that this factor divides the polynomial $P(\pi_{\lambda})$. 

    Assume first that the polynomial 
    \begin{align*}
        \varpi(s_1\cW\lambda)^2 - \varpi'(s_1\cW \lambda)^2
    \end{align*}
    is not a scalar multiple of $\varpi(\cW\lambda)^2 - \varpi'(\cW \lambda)^2$.
    Then, as the adjoint representation is stable under the action of $s_1$, we have 
    \begin{align*}
        P({}^{s_1} \pi_{\lambda}) =  P(\pi_{\lambda}).
    \end{align*}
    Thus, the polynomial $\varpi(s_1\cW\lambda)^2 - \varpi'(s_1\cW \lambda)^2
    =
    \varpi(\cW s_1\lambda)^2 - \varpi'(\cW s_1\lambda)^2
    $ divides the polynomial $P(\pi_{\lambda})$. 
    If the polynomial 
    \begin{align*}
        \varpi(s_1\cW\lambda)^2 - \varpi'(s_1\cW \lambda)^2
    \end{align*}
    is a scalar multiple of $\varpi(\cW\lambda)^2 - \varpi'(\cW \lambda)^2$, then, the product 
    \begin{align*}
        (\varpi(\cW\lambda)^2 - \varpi'(\cW \lambda)^2)
        (\varpi(s_1\cW\lambda)^2 - \varpi'(s_1\cW \lambda)^2)
    \end{align*}
    divides the polynomial $P(\pi_{\lambda})^2$.
    In any case, for any linear factor $\varpi(\cW\lambda)^2 - \varpi'(\cW \lambda)^2$ in the definition of $P_{\varpi}$, the product 
    \begin{align*}
        (\varpi(\cW\lambda)^2 - \varpi'(\cW \lambda)^2)
        (\varpi(s_1\cW\lambda)^2 - \varpi'(s_1\cW \lambda)^2)
    \end{align*}
    divides the polynomial $P(\pi_{\lambda})^2$.
    As the polynomial  $P_{[\varpi]}(\cW\lambda)$ is a product of linearly independent factors as above, the polynomial 
    \begin{align*}
        P_{[\varpi]}(\cW\lambda)
        P_{[\varpi]}(s_1\cW\lambda)
    \end{align*}
    divides the polynomial $P(\pi_{\lambda})^2$.
    This completes the proof.
\end{proof}

    We now present a key computation. 
    The integrand on the right-hand side of \cref{eq:localization} is equal to 
    \begin{align}
          \frac{1}{\abs{S_{\phi^L_0}^{\natural}}}
          \Phi(\pi_{\lambda}) 
          T(s, \pi_{\lambda})
          \frac{P(\pi_{\lambda})^2 \abs{Q(\pi_{\lambda})}}{R_{16s}(\cW\lambda) R_{8s}(s_1\cW \lambda)} 
          \label{eq:limitand}
    \end{align}
    for a smooth function $T(s, \pi)$ on $\C_+ \times \cV$. 
    We apply \cref{lemma:residue-theorem} for the product in the definition of $R_{s}(\cW\lambda)$, to rewrite this as 
    \begin{multline*}
        \frac{1}{\abs{S_{\phi^L_0}^{\natural}}}
        \Phi(\pi_{\lambda}) 
        T(s, \pi_{\lambda}) 
        \abs{Q(\pi_{\lambda})} \\
        \times    
        \sum_{[\varpi], [\varpi']} 
        \frac{1}{\varpi(\cW\lambda)^2 + (16s)^2} 
        \frac{1}{\varpi'(s_1\cW\lambda)^2 + (8s)^2} 
        \frac{P(\pi_{\lambda})^2}{P_{[\varpi]}(\cW\lambda)P_{[\varpi]}(s_1\cW\lambda)}.
    \end{multline*}
    By \cref{lemma:transitivity} and the invariance of measure $d\lambda$ under the action of Weyl groups, the integral of this function is equal to the integral of the following function:
    \begin{multline*}
        \frac{\abs{w(\phi_0^H)/{\pm 1}}^2}{\abs{S_{\phi^L_0}^{\natural}}}
        \Phi(\pi_{\lambda}) 
        T(s, \pi_{\lambda})
        \abs{Q(\pi_{\lambda})}  \\
        \times 
        \frac{1}{\varpi_{\beta_1}^{\vee}(\cW\lambda)^2 + (16s)^2} 
        \frac{1}{\varpi_{\beta_4}^{\vee}(\cW\lambda)^2 + (8s)^2} 
        \frac{P(\pi_{\lambda})^2}{P_{[\varpi^{\vee}_{\beta_1}]}(\cW\lambda)P_{[\varpi^{\vee}_{\beta_1}]}(s_1\cW\lambda)}.
    \end{multline*}
    By a similar reasoning to \cref{proposition:non-distinguished} by using the delta distribution, the limit as $s \to 0+$ of the integral of this function is equal to the integral of the function
    \begin{multline}
        \frac{1}{512}
        (\lim_{s \to 0+}s^{-1} \gamma(s, \mathbf{1}, \psi))^2
        (\lim_{s \to 0+}s^{-1} \zeta_F(s)^{-1})^2
        \frac{\abs{w(\phi_0^H)/{\pm 1}}^2}{\abs{W(M_{\ad}, L/A_M)_{\sigma^L_0}}\abs{S_{\phi^L_0}^{\natural}}}
         \\
        \times 
        \Phi(\pi_{\lambda}) 
        T(s, \pi_{\lambda})
        \abs{Q(\pi_{\lambda})}  
        \times   
        \frac{P(\pi_{\lambda})^2}{P_{[\varpi_{\beta_1}^{\vee}]}(\cW\lambda)P_{[\varpi_{\beta_1}^{\vee}]}(s_1\cW\lambda)}
        \label{eq:result-computation}
    \end{multline}
    on the subset of $\cU$ defined by $\varpi_{\beta_1}^{\vee}(\cW\lambda) = \varpi_{\beta_4}^{\vee}(\cW\lambda) = 0$.
    Here, we explain why we have the factor 
    \begin{align*}
        \frac{1}{512}
        (\lim_{s \to 0+}s^{-1} \gamma(s, \mathbf{1}, \psi))^2
        (\lim_{s \to 0+}s^{-1} \zeta_F(s)^{-1})^2.
    \end{align*}
    By \cref{lemma:center-connected} which gives a basis of the space of characters on $A_L$ and the normalization of measure in \cref{subsubsection:normalization-tempered-spectra}, we have the factor 
    \begin{align*}
        (\frac{\log(q)}{2\pi})^{\dim(\fa_L^{*})} 
        (\frac{\log(q)}{2\pi})^{-\dim(\fa_I^{*})} 
        = 
        \frac{\log(q)^2}{(2\pi)^2}.
    \end{align*}
    Also, by the description in \cref{eq:parameter-from-PGL3}, it is easy to see that the action of $\cW$ maps $\theta(\cO^H) \cap \cU$ onto the subset of $\cU$ defined by $\varpi_{\beta_1}^{\vee}(\cW\lambda) = \varpi_{\beta_4}^{\vee}(\cW\lambda) = 0$ and this map is measure-preserving.
    Furthermore, we obtain the scalar 
    \begin{align*}
        (\lim_{s \to 0+}s^{-1} \gamma(s, \mathbf{1}, \psi))^2 \frac{1}{128} \pi^2
    \end{align*}
    from the computation of the limit of the distribution
    \begin{align*}
        \gamma(s, \mathbf{1}, \psi)^2 
        \frac{1}{\varpi_{\beta_1}^{\vee}(\cW\lambda)^2 + (16s)^2} 
        \frac{1}{\varpi_{\beta_4}^{\vee}(\cW\lambda)^2 + (8s)^2},
    \end{align*}
    as $s \to 0+$.
    This completes the computation of scalar multiplication we have obtained.

    We now prove the following lemma by hand, which relates the orders of Weyl group for $M_{\ad}$ and $H$.
    \begin{lemma}\label{lemma:comparison-Weyl-and-S-group}
        In Cases $(2), (4), (5)$, we have 
        \begin{align*}
            \frac
            {(\abs{I_{\triv}}/2)^2}
            {\abs{W(M_{\ad}, L/A_M)_{\sigma^L_0}} \abs{S^{\natural}_{\phi_{\sigma^L_0}}}} 
            = 
            \frac{1}{\abs{W(H, I)_{\sigma^I_0}} \abs{S^{\natural}_{\phi^I_0}}}.
        \end{align*}
    \end{lemma}

    \begin{proof}
    As the author has no elegant proof for this lemma, we will prove the lemma by case-by-case consideration.
    We use the isomorphism
    \begin{align*}
        W(M_{\ad}, T/A_M) \cong (\Z/2\Z)^4_0 \rtimes S_4,
    \end{align*}
    where $(\Z/2\Z)^4_0$ is the subgroup of $(\Z/2\Z)^4$ consisting of elements $(v_i)_{i=1, 2, 3, 4}$ with $\sum_i v_i = 0$.

    Before starting the computation, we give a comment on the action of Weyl groups for the similitude orthogonal group.
    If we take an element $t = (t_1, t_2, t_3, t_4, z) \in \GL_1(F)^4 \times \GL_1(F)$ in the maximal torus of $\GSO_8(F)$, the corresponding element in $\GL_8(F)$ is given by the diagonal matrix 
    \begin{align*}
        \diag(t_1, t_2, t_3, t_4, zt_4^{-1}, zt_3^{-1}, zt_2^{-1}, zt_1^{-1}).
    \end{align*}
    Thus, for example the action of the Weyl element $w = (1, 1, 0, 0) \rtimes 1 \in (\Z/2\Z)^4 \rtimes S_4$ is equal to 
    \begin{align*}
        (t_1, t_2, t_3, t_4, z) \mapsto (zt_1^{-1}, zt_2^{-1}, t_3, t_4, z).
    \end{align*}
    Thus, for example, the character $\chi = \chi_1 \otimes \chi_2 \otimes \chi_3 \otimes \chi_4 \otimes \mu$ of $\GL_1(F)^4 \times \GL_1(F)$ transforms into the character $\chi_1^{-1} \otimes \chi_2^{-1} \otimes \chi_3 \otimes \chi_4 \otimes (\mu \chi_1 \chi_2)$ and the character on the torus corresponding to the similitude factor changes.
    Similar changes will occur in the computation for other Levi subgroups.

    We start with Case $(2)$.
    In this case, we have 
    \begin{align*}
        \abs{w(\phi^H_0)/\{\pm 1\}} = \abs{I_{\triv}}/2 = 1
    \end{align*}
    and 
    \begin{align*}
        \abs{S^{\natural}_{\phi_{\sigma^L_0}}} = 3.
    \end{align*}
    Also, we have 
    \begin{align*}
        \abs{W(M_{\ad}, L/A_M)_{\sigma^L_0}} 
        = 
        \begin{cases}
            1 \quad \text{if $\sigma^L_0$ is not self-dual}, \\
            2 \quad \text{if $\sigma^L_0$ is self-dual},
        \end{cases}
    \end{align*}
    and the nontrivial element in the second case is given by the longest element of $M_{\ad}$.
    On the other hand, we have 
    \begin{align*}
        \abs{W(H, I)_{\sigma^I_0}} = 1
    \end{align*}
    and 
    \begin{align*}
        \abs{S^{\natural}_{\phi^I_0}}
        = 
        \begin{cases}
            3 \quad \text{if $\sigma^L_0$ is not self-dual}, \\
            6 \quad \text{if $\sigma^L_0$ is self-dual},
        \end{cases}
    \end{align*}
    because the longest element of $\widehat{H}$ acts on the group $\widehat{\PGL_3}$ as the self-dual involution.
    Hence, we obtain the result in Case $(2)$.
    
    Next, we consider Case $(4)$.
    In this case, we have 
    \begin{align*}
        \abs{S^{\natural}_{\phi_{\sigma^L_0}}}
        =\abs{S^{\natural}_{\phi^I_0}}
        =2.
    \end{align*}
    We first assume that 
    \begin{align*}
        \abs{w(\phi^H_0)/\{\pm 1\}} = \abs{I_{\triv}}/2 = 1.
    \end{align*}
    In this case, we have 
    \begin{align*}
        \abs{W(M_{\ad}, L/A_M)_{\sigma^L_0}} 
        = 
        \abs{W(H, I)_{\sigma^I_0}}
        = 
        \begin{cases}
            1 \quad \text{if $\sigma^I_0$ is not self-dual}, \\
            2 \quad \text{if $\sigma^I_0$ is self-dual}.
        \end{cases}
    \end{align*}
    The nontrivial element is given by the longest element in this case.
    This implies the lemma.
    We then assume that 
    \begin{align*}
        \abs{w(\phi^H_0)/\{\pm 1\}} = \abs{I_{\triv}}/2 = 2.
    \end{align*}
    In this case, the representation $\sigma^I_0$ is self-dual and we have 
    \begin{align*}
        \abs{W(M_{\ad}, L/A_M)_{\sigma^L_0}} 
        = 
        8.
    \end{align*}
    The group $W(M_{\ad}, L/A_M)_{\sigma^L_0}$ is generated by the image of the elements 
    \begin{align*}
        0 \rtimes (12), (1, 1, 0, 0) \rtimes 1, (0, 0, 1, 1) \rtimes 1
    \end{align*}
    in $W(M_{\ad}, T/A_M) \subset (\Z/2\Z)^4 \rtimes S_4$.
    Also, we have
    \begin{align*}
        \abs{W(H, I)_{\sigma^I_0}}
        = 
        2.
    \end{align*}
    This completes the proof.
    
    Finally, we consider case $(5)$.
    In this case, we have 
    \begin{align*}
        \abs{S^{\natural}_{\phi_{\sigma^L_0}}}
        =\abs{S^{\natural}_{\phi^I_0}}
        =1.
    \end{align*}
    We first assume that 
    \begin{align*}
        \abs{w(\phi^H_0)/\{\pm 1\}} = 4.
    \end{align*}
    In this case, we have 
    \begin{align*}
        \abs{W(M_{\ad}, L/A_M)_{\sigma^L_0}} 
        = 
        \abs{W(M_{\ad}, T/A_M)}
        = 
        192. 
    \end{align*}
    Also, we have 
    \begin{align*}
        \abs{W(H, I)_{\sigma^I_0}}
        =
        \abs{W(H, T_H)}
        =
        12.
    \end{align*}
    Thus, we obtain the result. 
    Next, we assume that 
    \begin{align*}
        \abs{w(\phi^H_0)/\{\pm 1\}} = \abs{I_{\triv}}/2 = 2.
    \end{align*}
    In this case, we have $\chi_1=1, \chi_2 = \chi_3^{-1}$ in the notation of case $(5)$ in \cref{remark:divide-five-cases}.
    We have 
    \begin{align*}
        \abs{W(M_{\ad}, L/A_M)_{\sigma^L_0}} 
        = 
        \begin{cases}
            8 \quad \text{if $\chi_2^2 \neq 1$}, \\
            16 \quad \text{if $\chi_2^2 = 1$}.
        \end{cases}
    \end{align*}
    In the first case, the group $W(M_{\ad}, L/A_M)_{\sigma^L_0}$ is generated by the image of the elements 
    \begin{align*}
        0 \rtimes (12), (1, 1, 0, 0) \rtimes 1, (0, 0, 1, 1) \rtimes (3 4)
    \end{align*}
    in $W(M_{\ad}, T) \subset (\Z/2\Z)^4 \rtimes S_4$.
    In the second case, the group $W(M_{\ad}, L)_{\sigma^L_0}$ is generated by the image of the elements 
    \begin{align*}
        0 \rtimes (12), (1, 1, 0, 0) \rtimes 1, (0, 0, 1, 1) \rtimes (3 4), 0 \rtimes (3 4)
    \end{align*}
    in $W(M_{\ad}, T) \subset (\Z/2\Z)^4 \rtimes S_4$.
    A similar consideration implies that 
    \begin{align*}
        \abs{W(H, I)_{\sigma^I_0}}
        = 
        \begin{cases}
            2  \quad \text{if $\chi_2^2 \neq 1$}, \\
            4  \quad \text{if $\chi_2^2 = 1$}.
        \end{cases}
    \end{align*}

    We note that we have an isomorphism $W(H, T_H) \xrightarrow{\sim} S_3 \times \{ \pm 1 \} = D_6$, where the group $S_3$ acts on the set $\chi_1, \chi_2, \chi_3$ by permutation and $-1$ acts by inversion.
    The group $W(H, I)_{\sigma^I_0}$ is generated by 
    \begin{align*}
        (2 3) \times -1 
    \end{align*}
    in the group $W(H, T_H) \xrightarrow{\sim} S_3 \times \{ \pm 1 \}$ in the first case, and 
    \begin{align*}
        (2 3) \times 1, 1 \times -1
    \end{align*}
    in the second case.
    This completes the proof. 
    Lastly, we assume that 
    \begin{align*}
        \abs{w(\phi^H_0)/\{\pm 1\}} = \abs{I_{\triv}}/2 = 1.
    \end{align*}
    In this case, note that we have 
    \begin{align*}
        \chi_i \chi_j \neq 1
    \end{align*}
    for $i, j=1,2,3$ with $i \neq j$.
    We take an element 
    \begin{align*}
        w = (v_1, v_2, v_3, v_4) \rtimes \sigma \in (\Z/2\Z)^4 \rtimes S_4,
    \end{align*}
    with $\sum_{i=1, 2, 3, 4}v_i = 0 \in \Z/2\Z$.
    If this element is in the subgroup $W(M_{\ad}, L)_{\sigma^L_0}$, then as the characters $\chi_i$ are nontrivial, we have $\sigma \in S_1 \times S_3$.
    Also, the condition above and the comments in the beginning of the proof imply that we have 
    \begin{align*}
        (v_1, v_2, v_3, v_4) \in \{ (0, 0, 0, 0), (1, 1, 1, 1) \}.
    \end{align*}
    Thus, the element $w$ uniquely corresponds to an element of $W(H, T_H)_{\sigma^I_0}$, the trivial element or the longest element of $W(H, T_H)$.
    This completes the proof.
\end{proof}

     We note that, on the subset of $\cU$ defined by the equations $\varpi_{\beta_1}^{\vee}(\cW\lambda) = \varpi_{\beta_4}^{\vee}(\cW\lambda) = 0$, we have 
    \begin{align*}
        P_{[\varpi_{\beta_1}^{\vee}]}(\cW\lambda) 
        &= 
        \prod_{[\varpi] \in w(\phi^H_0)/\{\pm 1\}, [\varpi] \neq [\varpi_{\beta_1}^{\vee}]}
        ((\varpi_{\beta_1}^{\vee}(\cW\lambda))^2 - \varpi(\cW\lambda)^2) \\
        &= 
        \prod_{[\varpi] \in w(\phi^H_0)/\{\pm 1\}, [\varpi] \neq [\varpi_{\beta_1}^{\vee}]}
        (- \varpi(\cW\lambda)^2) \\
        &= (-1)^{\abs{w(\phi^H_0)}/2 - 1} \lim_{s \to 0+} s^{-2}R_s(\lambda)
    \end{align*}
    Also, we have 
    \begin{align*}
        \abs{w(\phi^H_0)} = \abs{I_{\triv}}
    \end{align*}
    by definition.
    Thus, on the subspace of $\cU$ defined by the equations $\varpi_{\beta_1}^{\vee}(\cW\lambda) = \varpi_{\beta_4}^{\vee}(\cW\lambda) = 0$, regarding \cref{eq:limitand}, we have 
    \begin{multline*}
        \lim_{s \to 0+} 
          (16s)^2(8s)^2
          \Phi(\pi_{\lambda}) 
          T(s, \pi_{\lambda})
          \frac{P(\pi_{\lambda})^2 \abs{Q(\pi_{\lambda})}}{R_{16s}(\cW\lambda) R_{8s}(s_1\cW \lambda)} \\
        = 
        \Phi(\pi_{\lambda})
        T(0, \pi_{\lambda})
        \frac{P(\pi_{\lambda})^2\abs{Q(\pi_{\lambda})}}{P_{[\varpi_{\beta_1}^{\vee}]}(\cW\lambda)P_{[\varpi_{\beta_1}^{\vee}]}(s_1\cW\lambda)},
    \end{multline*}
    which appears in \cref{eq:result-computation}.
    Thus, by \cref{lemma:comparison-Weyl-and-S-group}, 
    the limit of \cref{eq:localization} can be computed by the restricting the domain of integration to $\theta^{-1}(\cV)$, multiplying
    \begin{align*}
       \frac{1}{512} 
       \frac{\abs{S^{\natural}_{\phi^L_0}}}{\abs{S^{\natural}_{\phi^I_0}}}
       (\lim_{s \to 0+}s^{-1} \gamma(s, \mathbf{1}, \psi))^2
       (\lim_{s \to 0+}s^{-1} \zeta_F(s)^{-1})^2  (16s)^2(8s)^2
    \end{align*}
    and formally taking the limit as $s \to 0+$ inside the integral.
    Instead of multiplying this factor, we can  multiply
    \begin{align*}
        \frac{1}{512}
        \frac{\abs{S^{\natural}_{\phi^L_0}}}{\abs{S^{\natural}_{\phi^I_0}}}
        \zeta_F(16s)^{-1} 
        \zeta_F(8s)^{-1}
        \gamma(16s, \mathbf{1}, \psi) 
        \gamma(8s, \mathbf{1}, \psi).
    \end{align*}
    The limit is equal to 
    \begin{multline*}
        \frac{1}{512}
        \int_{\phi^H \in \theta^{-1}(\cV)} 
        \Phi(\theta(\phi^H))
        \\
        \times
        \{
        \lim_{s \to 0+} 
        \zeta_F(16s)^{-1} 
        \zeta_F(8s)^{-1}
        \gamma(16s, \mathbf{1}, \psi) 
        \gamma(8s, \mathbf{1}, \psi) \\
        \times r_{w^{-1}Pw \vert P}(\theta(\phi^H) \otimes \chi_{2s\rho_P} )
        \frac{\abs{S^{\natural}_{\phi^L_0}}}{\abs{S^{\natural}_{\phi^I_0}}}
        \mu^{M_{\ad}}_{\psi}(\theta(\phi^H))
        \} 
        \rd \phi^H.
    \end{multline*}
    By \cref{eq:decomposition-adjoint-restriction-to-G2} and \cref{proposition:normalizing-factor-for-s1s6}, the limit in the integrand is equal to 
    \begin{align*}
        c(\phi^H) 
        \mu^H_{\HII, \psi}(\phi^H)
    \end{align*} 
    for a smooth function $c$ on $\cV$ with absolute value $1$ as in \cref{lemma:adjoint-L-explicit} and \cref{corollary:adjoint-L-abs-value}. 
    Here, we can use, for example, the following computation of the restriction of the adjoint representation in Case $(4)$ based on \cref{eq:decomposition-adjoint-restriction-to-G2}: 
    \begin{multline*}
        \Ad_{H} \circ \iota^{H}_I \circ \phi^I
        = 
        \Ad(\phi^I) 
        \oplus 
        \mathbf{1} 
        \oplus 
        \det(\phi^I) 
        \oplus 
        \det(\phi^I)^{-1} \\
        \oplus \phi^I 
        \oplus \phi^{I, \vee}
        \oplus \phi^I \otimes \det(\phi^I)
        \oplus \phi^{I, \vee} \otimes \det(\phi^I)^{-1}.
    \end{multline*}
    We omit a similar computation.
    Thus, we finally obtain 
    \begin{align*}
        \lim_{s \to 0+}
        \gamma(s, \mathbf{1}, \psi)^2
        \int_{\pi \in \Temp_{\ind}(M_{\ad})} 
        \Phi(\pi)    
        r_{w^{-1}Pw \vert P}(\pi \otimes \chi_{2s\rho_P} )\mu^{M_{\ad}}_{\psi}(\pi)\rd \pi  \\
        = 
        \frac{1}{512}
        \int_{\Temp_{\ind}(H)/\sim_{st}} 
        \Phi(\theta(\phi^H))
        c(\phi^H)
        \mu^H_{\HII, \psi}(\phi^H)
        \rd \phi^H,
    \end{align*}
    for any smooth function $\Phi$ supported on $\cV$.
    This completes the proof.
\end{proof}

\subsection{The proof of the Hiraga--Ichino--Ikeda conjecture for $G_2$}

\begin{theorem}\label{theorem:main-theorem}
    Assume \cref{conjecture:endoscopic-character-relation}.
    Then, \cref{conjecture:formal-degree-conjecture} holds for $H = G_2$.
\end{theorem}

\begin{proof}
    We assume \cref{conjecture:endoscopic-character-relation}.
    By \cref{theorem:generic-irreducibility-functorial-lifts-G2}, we can replace the index of the sum in the right-hand side of \cref{theorem:spectral-limit-formula} with the set $\{ \pi \in \Pi^X_{\phi}(M_{\ad})  \}$.
    By \cref{proposition:transfer-singular-orbit}, \cref{theorem:geometric-limit-formula} and \cref{theorem:spectral-limit-formula}, there exists a positive constant $K>0$ such that we have 
    \begin{align*}
        K
        f^H(1) 
        = 
        \int_{\pi^H \in \Temp_{\ind}(H)} \tr\pi^H({f^H}^{\vee}) c(\phi^H_{\pi^H}) \mu^H_{\HII, \psi}(\pi^H)\rd \pi^H,
    \end{align*}
    for any transfer $f^H$ of the function $f$.
    By \cref{lemma:transfer-surjection}, every compactly supported smooth function on $H$ arises as the  transfer of a function $f \in \Cc(M_{\ad})$.
    Thus, the formula above holds for any function $f^H \in \Cc(H)$.
    We see that $c(\phi^H_{\pi^H}) = 1$ for almost all $\pi^H \in \Temp_{\ind}(H)$ by the positivity of the Plancherel measure; see \cref{proposition:sakellaridis-venkatesh}.
    Thus, we have 
    \begin{align*}
        Kf^H(1) 
        = 
        \int_{\pi^H \in \Temp_{\ind}(H)} \tr\pi^H({f^H}^{\vee}) \mu^H_{\HII, \psi}(\pi^H) \rd \pi^H,
    \end{align*}
    for any function $f^H \in \Cc(H)$.
    This completes the proof up to determining the constant $K$.
    However, the constant $K$ must be equal to $1$ since the Hiraga--Ichino--Ikeda conjecture holds for the Steinberg representation of $G_2$; see \cite{HII08}*{3.3}.
    Thus, we obtained the result.
\end{proof}

We also obtain the following refinement of \cref{theorem:key-computation}.
\begin{theorem}\label{theorem:refined-key-computation}
     Let $\Phi$ be a compactly supported smooth function on $\Temp_{\ind}(M_{\ad})$.
     Then, we have 
     \begin{multline*}
        \lim_{s \to 0+}
        \gamma(s, \mathbf{1}, \psi)^2
        \int_{\pi \in \Temp_{\ind}(M_{\ad})} 
        \Phi(\pi)    
        r_{w^{-1}Pw \vert P}(\pi \otimes \chi_{2s\rho_P} )\mu^{M_{\ad}}_{\psi}(\pi)\rd \pi \\
        = 
        \frac{1}{512}
        \int_{\phi^H \in \Temp_{\ind}(H)/\sim_{st}} 
        \lbrace
        \sum_{\pi \in \Pi^X_{\phi^H}} 
        \Phi(\pi)    
        \rbrace
        \mu^H_{\HII, \psi}(\phi^H)
        \rd \phi^H.
    \end{multline*}
\end{theorem}

\begin{proof}
    From the proof of the previous theorem,
    \begin{align*}
        c(\phi^H) = 1 
    \end{align*}
    in \cref{theorem:key-computation}.
    This completes the proof of the theorem.
\end{proof}

\bibliographystyle{plain}
\bibliography{thesis}

@book {Art13,
    AUTHOR = {Arthur, James},
     TITLE = {The endoscopic classification of representations},
    SERIES = {American Mathematical Society Colloquium Publications},
    VOLUME = {61},
      NOTE = {Orthogonal and symplectic groups},
 PUBLISHER = {American Mathematical Society, Providence, RI},
      YEAR = {2013},
     PAGES = {xviii+590},
      ISBN = {978-0-8218-4990-3},
   MRCLASS = {22E55 (11F66 11F70 11F72 11R37 20G25 22E50)},
  MRNUMBER = {3135650},
MRREVIEWER = {Dihua\ Jiang},
       DOI = {10.1090/coll/061},
       URL = {https://doi.org/10.1090/coll/061},
}

@article {Beuzart-Plessis2021-Plancherel-GLnE-GLnF,
    AUTHOR = {Beuzart-Plessis, Rapha\"el},
     TITLE = {Plancherel formula for {${\rm GL}_n(F)\backslash {\rm
              GL}_n(E)$} and applications to the {I}chino-{I}keda and formal
              degree conjectures for unitary groups},
   JOURNAL = {Invent. Math.},
  FJOURNAL = {Inventiones Mathematicae},
    VOLUME = {225},
      YEAR = {2021},
    NUMBER = {1},
     PAGES = {159--297},
      ISSN = {0020-9910,1432-1297},
   MRCLASS = {22E50 (11F70)},
  MRNUMBER = {4270666},
MRREVIEWER = {Alexandre\ Afgoustidis},
       DOI = {10.1007/s00222-021-01032-6},
       URL = {https://doi.org/10.1007/s00222-021-01032-6},
}

@article {BP21a,
    AUTHOR = {Beuzart-Plessis, Rapha\"{e}l},
     TITLE = {On the formal degree conjecture for classical groups},
   JOURNAL = {Oberwolfach Report},
    VOLUME = {18},
      YEAR = {2021},
    NUMBER = {3},
     PAGES = {2096--2101},
       DOI = {10.4171/OWR/2021/39},
}

@article {HII08,
    AUTHOR = {Hiraga, Kaoru and Ichino, Atsushi and Ikeda, Tamotsu},
     TITLE = {Formal degrees and adjoint {$\gamma$}-factors},
   JOURNAL = {J. Amer. Math. Soc.},
  FJOURNAL = {Journal of the American Mathematical Society},
    VOLUME = {21},
      YEAR = {2008},
    NUMBER = {1},
     PAGES = {283--304},
      ISSN = {0894-0347,1088-6834},
   MRCLASS = {22E50},
  MRNUMBER = {2350057},
       DOI = {10.1090/S0894-0347-07-00567-X},
       URL = {https://doi.org/10.1090/S0894-0347-07-00567-X},
}

@article {Mok15,
    AUTHOR = {Mok, Chung Pang},
     TITLE = {Endoscopic classification of representations of quasi-split
              unitary groups},
   JOURNAL = {Mem. Amer. Math. Soc.},
  FJOURNAL = {Memoirs of the American Mathematical Society},
    VOLUME = {235},
      YEAR = {2015},
    NUMBER = {1108},
     PAGES = {vi+248},
      ISSN = {0065-9266,1947-6221},
      ISBN = {978-1-4704-1041-4; 978-1-4704-2226-4},
   MRCLASS = {22E55 (11R42 22E50)},
  MRNUMBER = {3338302},
MRREVIEWER = {Neven\ Grbac},
       DOI = {10.1090/memo/1108},
       URL = {https://doi.org/10.1090/memo/1108},
}

@article {Sha90,
    AUTHOR = {Shahidi, Freydoon},
     TITLE = {A proof of {L}anglands' conjecture on {P}lancherel measures;
              complementary series for {$p$}-adic groups},
   JOURNAL = {Ann. of Math. (2)},
  FJOURNAL = {Annals of Mathematics. Second Series},
    VOLUME = {132},
      YEAR = {1990},
    NUMBER = {2},
     PAGES = {273--330},
      ISSN = {0003-486X,1939-8980},
   MRCLASS = {11R39 (11F70 11S37 22E35 22E55)},
  MRNUMBER = {1070599},
MRREVIEWER = {Stephen\ Gelbart},
       DOI = {10.2307/1971524},
       URL = {https://doi.org/10.2307/1971524},
}

@article {Shi12,
    AUTHOR = {Shin, Sug Woo},
     TITLE = {Automorphic {P}lancherel density theorem},
   JOURNAL = {Israel J. Math.},
  FJOURNAL = {Israel Journal of Mathematics},
    VOLUME = {192},
      YEAR = {2012},
    NUMBER = {1},
     PAGES = {83--120},
      ISSN = {0021-2172,1565-8511},
   MRCLASS = {22E35 (43A85)},
  MRNUMBER = {3004076},
MRREVIEWER = {Valeri\u{\i}\ Vladimirovich\ Volchkov},
       DOI = {10.1007/s11856-012-0018-z},
       URL = {https://doi.org/10.1007/s11856-012-0018-z},
}

@article {Wal03,
    AUTHOR = {Waldspurger, J.-L.},
     TITLE = {La formule de {P}lancherel pour les groupes {$p$}-adiques
              (d'apr\`es {H}arish-{C}handra)},
   JOURNAL = {J. Inst. Math. Jussieu},
  FJOURNAL = {Journal of the Institute of Mathematics of Jussieu. JIMJ.
              Journal de l'Institut de Math\'{e}matiques de Jussieu},
    VOLUME = {2},
      YEAR = {2003},
    NUMBER = {2},
     PAGES = {235--333},
      ISSN = {1474-7480,1475-3030},
   MRCLASS = {22E35 (22E50)},
  MRNUMBER = {1989693},
MRREVIEWER = {Rebecca\ Herb},
       DOI = {10.1017/S1474748003000082},
       URL = {https://doi.org/10.1017/S1474748003000082},
}

@article {Xu18local,
    AUTHOR = {Xu, Bin},
     TITLE = {L-packets of quasisplit {$\mathrm{GSp}(2n)$} and {$\mathrm{GO}(2n)$}},
   JOURNAL = {Math. Ann.},
  FJOURNAL = {Mathematische Annalen},
    VOLUME = {370},
      YEAR = {2018},
    NUMBER = {1-2},
     PAGES = {71--189},
      ISSN = {0025-5831,1432-1807},
   MRCLASS = {22E50 (11F70)},
  MRNUMBER = {3747484},
MRREVIEWER = {Ivan\ Mati\'{c}},
       DOI = {10.1007/s00208-016-1515-x},
       URL = {https://doi.org/10.1007/s00208-016-1515-x},
}

@article {Var17,
    AUTHOR = {Varma, Sandeep},
     TITLE = {On descent and the generic packet conjecture},
   JOURNAL = {Forum Math.},
  FJOURNAL = {Forum Mathematicum},
    VOLUME = {29},
      YEAR = {2017},
    NUMBER = {1},
     PAGES = {111--155},
      ISSN = {0933-7741,1435-5337},
   MRCLASS = {22E35 (22E50)},
  MRNUMBER = {3592596},
MRREVIEWER = {Jeffrey\ D.\ Adler},
       DOI = {10.1515/forum-2015-0113},
       URL = {https://doi.org/10.1515/forum-2015-0113},
}

@article {GI14,
    AUTHOR = {Gan, Wee Teck and Ichino, Atsushi},
     TITLE = {Formal degrees and local theta correspondence},
   JOURNAL = {Invent. Math.},
  FJOURNAL = {Inventiones Mathematicae},
    VOLUME = {195},
      YEAR = {2014},
    NUMBER = {3},
     PAGES = {509--672},
      ISSN = {0020-9910,1432-1297},
   MRCLASS = {11F67 (11F70 22E50)},
  MRNUMBER = {3166215},
MRREVIEWER = {Ivan\ Mati\'{c}},
       DOI = {10.1007/s00222-013-0460-5},
       URL = {https://doi.org/10.1007/s00222-013-0460-5},
}

@article {GS23LLC,
    AUTHOR = {Gan, Wee Teck and Savin, Gordan},
     TITLE = {The local {L}anglands conjecture for {$G_2$}},
   JOURNAL = {Forum Math. Pi},
  FJOURNAL = {Forum of Mathematics. Pi},
    VOLUME = {11},
      YEAR = {2023},
     PAGES = {Paper No. e28, 42},
      ISSN = {2050-5086},
   MRCLASS = {11S37 (11F27 11F70 22E50)},
  MRNUMBER = {4658199},
MRREVIEWER = {Yiwen\ Ding},
       DOI = {10.1017/fmp.2023.27},
       URL = {https://doi.org/10.1017/fmp.2023.27},
}

@article {KS99,
    AUTHOR = {Kottwitz, Robert E. and Shelstad, Diana},
     TITLE = {Foundations of twisted endoscopy},
   JOURNAL = {Ast\'{e}risque},
  FJOURNAL = {Ast\'{e}risque},
    NUMBER = {255},
      YEAR = {1999},
     PAGES = {vi+190},
      ISSN = {0303-1179,2492-5926},
   MRCLASS = {22E55 (11F70 11R34 22-02 22E50)},
  MRNUMBER = {1687096},
MRREVIEWER = {Volker\ J.\ Heiermann},
}

@article {Art89IntResI,
    AUTHOR = {Arthur, James},
     TITLE = {Intertwining operators and residues. {I}. {W}eighted
              characters},
   JOURNAL = {J. Funct. Anal.},
  FJOURNAL = {Journal of Functional Analysis},
    VOLUME = {84},
      YEAR = {1989},
    NUMBER = {1},
     PAGES = {19--84},
      ISSN = {0022-1236},
   MRCLASS = {22E55 (11F70 11F72)},
  MRNUMBER = {999488},
MRREVIEWER = {David\ Joyner},
       DOI = {10.1016/0022-1236(89)90110-9},
       URL = {https://doi.org/10.1016/0022-1236(89)90110-9},
}

@article {ILM17,
    AUTHOR = {Ichino, Atsushi and Lapid, Erez and Mao, Zhengyu},
     TITLE = {On the formal degrees of square-integrable representations of
              odd special orthogonal and metaplectic groups},
   JOURNAL = {Duke Math. J.},
  FJOURNAL = {Duke Mathematical Journal},
    VOLUME = {166},
      YEAR = {2017},
    NUMBER = {7},
     PAGES = {1301--1348},
      ISSN = {0012-7094,1547-7398},
   MRCLASS = {11F70},
  MRNUMBER = {3649356},
MRREVIEWER = {Peter\ Humphries},
       DOI = {10.1215/00127094-0000001X},
       URL = {https://doi.org/10.1215/00127094-0000001X},
}

@misc{bakic2023similitudeexceptionalthetacorrespondences,
      title={Similitude exceptional theta correspondences}, 
      author={Petar Bakic and Wee Teck Gan and Gordan Savin},
      year={2023},
      eprint={2308.13339},
      archivePrefix={arXiv},
      primaryClass={math.RT},
      url={https://arxiv.org/abs/2308.13339}, 
}

@article {Borovoi1998abeliangalois,
    AUTHOR = {Borovoi, Mikhail},
     TITLE = {Abelian {G}alois cohomology of reductive groups},
   JOURNAL = {Mem. Amer. Math. Soc.},
  FJOURNAL = {Memoirs of the American Mathematical Society},
    VOLUME = {132},
      YEAR = {1998},
    NUMBER = {626},
     PAGES = {viii+50},
      ISSN = {0065-9266,1947-6221},
   MRCLASS = {20G10 (11E72 14E20 18G50)},
  MRNUMBER = {1401491},
MRREVIEWER = {James\ E.\ Humphreys},
       DOI = {10.1090/memo/0626},
       URL = {https://doi.org/10.1090/memo/0626},
}

@article {BinXu2016liftingproblem,
    AUTHOR = {Xu, Bin},
     TITLE = {On a lifting problem of {L}-packets},
   JOURNAL = {Compos. Math.},
  FJOURNAL = {Compositio Mathematica},
    VOLUME = {152},
      YEAR = {2016},
    NUMBER = {9},
     PAGES = {1800--1850},
      ISSN = {0010-437X,1570-5846},
   MRCLASS = {22E50 (11F70 20G25)},
  MRNUMBER = {3568940},
MRREVIEWER = {Ramin\ Takloo-Bighash},
       DOI = {10.1112/S0010437X16007545},
       URL = {https://doi.org/10.1112/S0010437X16007545},
}

@article {GS23theta,
    AUTHOR = {Gan, Wee Teck and Savin, Gordan},
     TITLE = {Howe duality and dichotomy for exceptional theta
              correspondences},
   JOURNAL = {Invent. Math.},
  FJOURNAL = {Inventiones Mathematicae},
    VOLUME = {232},
      YEAR = {2023},
    NUMBER = {1},
     PAGES = {1--78},
      ISSN = {0020-9910,1432-1297},
   MRCLASS = {11F27 (11F70 22E50)},
  MRNUMBER = {4557399},
       DOI = {10.1007/s00222-022-01165-2},
       URL = {https://doi.org/10.1007/s00222-022-01165-2},
}

@article {Stumbo2000minimallength,
    AUTHOR = {Stumbo, Fabio},
     TITLE = {Minimal length coset representatives for quotients of
              parabolic subgroups in {C}oxeter groups},
   JOURNAL = {Boll. Unione Mat. Ital. Sez. B Artic. Ric. Mat. (8)},
  FJOURNAL = {Bollettino della Unione Matematica Italiana. Serie VIII.
              Sezione B. Articoli di Ricerca Matematica},
    VOLUME = {3},
      YEAR = {2000},
    NUMBER = {3},
     PAGES = {699--715},
      ISSN = {0392-4041},
   MRCLASS = {20F55 (20F36 51F15)},
  MRNUMBER = {1801609},
MRREVIEWER = {Stephen\ P.\ Humphries},
}

@misc{atobe2024localintertwiningrelationscotempered,
      title={Local Intertwining Relations and Co-tempered $A$-packets of Classical Groups}, 
      author={Hiraku Atobe and Wee Teck Gan and Atsushi Ichino and Tasho Kaletha and Alberto Mínguez and Sug Woo Shin},
      year={2024},
      eprint={2410.13504},
      archivePrefix={arXiv},
      primaryClass={math.NT},
      url={https://arxiv.org/abs/2410.13504}, 
}

@article {Silberger1978Knapp-Stein,
    AUTHOR = {Silberger, Allan J.},
     TITLE = {The {K}napp-{S}tein dimension theorem for {$p$}-adic groups},
   JOURNAL = {Proc. Amer. Math. Soc.},
  FJOURNAL = {Proceedings of the American Mathematical Society},
    VOLUME = {68},
      YEAR = {1978},
    NUMBER = {2},
     PAGES = {243--246},
      ISSN = {0002-9939,1088-6826},
   MRCLASS = {22E50},
  MRNUMBER = {492091},
       DOI = {10.2307/2041781},
       URL = {https://doi.org/10.2307/2041781},
}

@article {Kimura1983prehomogeneousscalar,
    AUTHOR = {Kimura, Tatsuo},
     TITLE = {A classification of prehomogeneous vector spaces of simple
              algebraic groups with scalar multiplications},
   JOURNAL = {J. Algebra},
  FJOURNAL = {Journal of Algebra},
    VOLUME = {83},
      YEAR = {1983},
    NUMBER = {1},
     PAGES = {72--100},
      ISSN = {0021-8693},
   MRCLASS = {32M10 (20G20)},
  MRNUMBER = {710588},
MRREVIEWER = {A.\ L.\ Onishchik},
       DOI = {10.1016/0021-8693(83)90138-2},
       URL = {https://doi.org/10.1016/0021-8693(83)90138-2},
}

@article{Yamauchi2024,
author = {Takuya Yamauchi},
title = {{Transfers of some Hecke elements for possibly ramified base change in ${\mathrm{GL}}_n$}},
journal = {Kyoto Journal of Mathematics},
publisher = {Duke University Press},
pages = {1--15},
year = {2024},
doi = {10.1215/21562261-2024-0012},
URL = {https://doi.org/10.1215/21562261-2024-0012}
}

@article {KretShin23-general-symplectic,
    AUTHOR = {Kret, Arno and Shin, Sug Woo},
     TITLE = {Galois representations for general symplectic groups},
   JOURNAL = {J. Eur. Math. Soc. (JEMS)},
  FJOURNAL = {Journal of the European Mathematical Society (JEMS)},
    VOLUME = {25},
      YEAR = {2023},
    NUMBER = {1},
     PAGES = {75--152},
      ISSN = {1435-9855,1435-9863},
   MRCLASS = {11R39 (11F70 11F80 11G18)},
  MRNUMBER = {4556781},
MRREVIEWER = {Jack\ Shotton},
       DOI = {10.4171/jems/1179},
       URL = {https://doi.org/10.4171/jems/1179},
}

@article {KretShin-24-general-orthogonal,
    AUTHOR = {Kret, Arno and Shin, Sug Woo},
     TITLE = {Galois representations for general orthogonal groups},
   JOURNAL = {J. Inst. Math. Jussieu},
  FJOURNAL = {Journal of the Institute of Mathematics of Jussieu. JIMJ.
              Journal de l'Institut de Math\'ematiques de Jussieu},
    VOLUME = {23},
      YEAR = {2024},
    NUMBER = {5},
     PAGES = {1959--2050},
      ISSN = {1474-7480,1475-3030},
   MRCLASS = {11F80 (11F70 11G18 11R39)},
  MRNUMBER = {4821551},
       DOI = {10.1017/S1474748023000427},
       URL = {https://doi.org/10.1017/S1474748023000427},
}

@article {GanTakeda2011LLC-for-GSp4,
    AUTHOR = {Gan, Wee Teck and Takeda, Shuichiro},
     TITLE = {The local {L}anglands conjecture for {${\rm GSp}(4)$}},
   JOURNAL = {Ann. of Math. (2)},
  FJOURNAL = {Annals of Mathematics. Second Series},
    VOLUME = {173},
      YEAR = {2011},
    NUMBER = {3},
     PAGES = {1841--1882},
      ISSN = {0003-486X,1939-8980},
   MRCLASS = {22E50 (20G25 22E35)},
  MRNUMBER = {2800725},
MRREVIEWER = {B.\ Sury},
       DOI = {10.4007/annals.2011.173.3.12},
       URL = {https://doi.org/10.4007/annals.2011.173.3.12},
}

@article {Ramakrishnan2000-Modularity-Rankin-Selberg,
    AUTHOR = {Ramakrishnan, Dinakar},
     TITLE = {Modularity of the {R}ankin-{S}elberg {$L$}-series, and
              multiplicity one for {${\rm SL}(2)$}},
   JOURNAL = {Ann. of Math. (2)},
  FJOURNAL = {Annals of Mathematics. Second Series},
    VOLUME = {152},
      YEAR = {2000},
    NUMBER = {1},
     PAGES = {45--111},
      ISSN = {0003-486X,1939-8980},
   MRCLASS = {11F70 (11F55 11F66 11F80 11G18 14G10 20G35 22E55)},
  MRNUMBER = {1792292},
MRREVIEWER = {Solomon\ Friedberg},
       DOI = {10.2307/2661379},
       URL = {https://doi.org/10.2307/2661379},
}

@article {Kim2003-functoriality-exterior-square,
    AUTHOR = {Kim, Henry H.},
     TITLE = {Functoriality for the exterior square of {${\rm GL}_4$} and
              the symmetric fourth of {${\rm GL}_2$}},
      NOTE = {With appendix 1 by Dinakar Ramakrishnan and appendix 2 by Kim
              and Peter Sarnak},
   JOURNAL = {J. Amer. Math. Soc.},
  FJOURNAL = {Journal of the American Mathematical Society},
    VOLUME = {16},
      YEAR = {2003},
    NUMBER = {1},
     PAGES = {139--183},
      ISSN = {0894-0347,1088-6834},
   MRCLASS = {11F70 (11R39 22E46)},
  MRNUMBER = {1937203},
MRREVIEWER = {Mahdi\ Asgari},
       DOI = {10.1090/S0894-0347-02-00410-1},
       URL = {https://doi.org/10.1090/S0894-0347-02-00410-1},
}

@article {Kottwitz1988-Tamagawa,
    AUTHOR = {Kottwitz, Robert E.},
     TITLE = {Tamagawa numbers},
   JOURNAL = {Ann. of Math. (2)},
  FJOURNAL = {Annals of Mathematics. Second Series},
    VOLUME = {127},
      YEAR = {1988},
    NUMBER = {3},
     PAGES = {629--646},
      ISSN = {0003-486X,1939-8980},
   MRCLASS = {11F70 (11E72 22E55)},
  MRNUMBER = {942522},
MRREVIEWER = {Stephen\ Gelbart},
       DOI = {10.2307/2007007},
       URL = {https://doi.org/10.2307/2007007},
}

@article {GeeTaibi2019-GSp4,
    AUTHOR = {Gee, Toby and Ta\"ibi, Olivier},
     TITLE = {Arthur's multiplicity formula for {${\bf GSp}_4$} and
              restriction to {${\bf Sp}_4$}},
   JOURNAL = {J. \'Ec. polytech. Math.},
  FJOURNAL = {Journal de l'\'Ecole polytechnique. Math\'ematiques},
    VOLUME = {6},
      YEAR = {2019},
     PAGES = {469--535},
      ISSN = {2429-7100,2270-518X},
   MRCLASS = {11F72 (11F46 11F55)},
  MRNUMBER = {3991897},
MRREVIEWER = {Han\ Wu},
       DOI = {10.5802/jep.99},
       URL = {https://doi.org/10.5802/jep.99},
}

@article {SakellaridisVenkatesh2017,
    AUTHOR = {Sakellaridis, Yiannis and Venkatesh, Akshay},
     TITLE = {Periods and harmonic analysis on spherical varieties},
   JOURNAL = {Ast\'erisque},
  FJOURNAL = {Ast\'erisque},
    NUMBER = {396},
      YEAR = {2017},
     PAGES = {viii+360},
      ISSN = {0303-1179,2492-5926},
      ISBN = {978-2-85629-871-8},
   MRCLASS = {22E50 (11F67)},
  MRNUMBER = {3764130},
MRREVIEWER = {Arnab\ Mitra},
}

@article {SilbergerZink2018-Classification-L-parameter,
    AUTHOR = {Silberger, Allan J. and Zink, Ernst-Wilhelm},
     TITLE = {Langlands classification for {$L$}-parameters},
   JOURNAL = {J. Algebra},
  FJOURNAL = {Journal of Algebra},
    VOLUME = {511},
      YEAR = {2018},
     PAGES = {299--357},
      ISSN = {0021-8693,1090-266X},
   MRCLASS = {11S37 (22E40)},
  MRNUMBER = {3834776},
MRREVIEWER = {Jack\ Shotton},
       DOI = {10.1016/j.jalgebra.2018.06.012},
       URL = {https://doi.org/10.1016/j.jalgebra.2018.06.012},
}

@incollection {Haines2014-StableBernstein,
    AUTHOR = {Haines, Thomas J.},
     TITLE = {The stable {B}ernstein center and test functions for {S}himura
              varieties},
 BOOKTITLE = {Automorphic forms and {G}alois representations. {V}ol. 2},
    SERIES = {London Math. Soc. Lecture Note Ser.},
    VOLUME = {415},
     PAGES = {118--186},
 PUBLISHER = {Cambridge Univ. Press, Cambridge},
      YEAR = {2014},
      ISBN = {978-1-107-69363-0},
   MRCLASS = {11G18 (20G25 22E35 22E50)},
  MRNUMBER = {3444233},
MRREVIEWER = {Shuichiro\ Takeda},
}

@book {KalethaPrasad-Bruhat-Tits,
    AUTHOR = {Kaletha, Tasho and Prasad, Gopal},
     TITLE = {Bruhat-{T}its theory---a new approach},
    SERIES = {New Mathematical Monographs},
    VOLUME = {44},
 PUBLISHER = {Cambridge University Press, Cambridge},
      YEAR = {2023},
     PAGES = {xxx+718},
      ISBN = {978-1-108-83196-3},
   MRCLASS = {20E42 (11F70 20G25 22E50)},
  MRNUMBER = {4520154},
MRREVIEWER = {Corina\ Ciobotaru},
}

@book {HarrisTaylor2001-simple-Shimura,
    AUTHOR = {Harris, Michael and Taylor, Richard},
     TITLE = {The geometry and cohomology of some simple {S}himura
              varieties},
    SERIES = {Annals of Mathematics Studies},
    VOLUME = {151},
      NOTE = {With an appendix by Vladimir G. Berkovich},
 PUBLISHER = {Princeton University Press, Princeton, NJ},
      YEAR = {2001},
     PAGES = {viii+276},
      ISBN = {0-691-09090-4},
   MRCLASS = {11G18 (11F70 11S37 14G35 22E45)},
  MRNUMBER = {1876802},
MRREVIEWER = {James\ Milne},
}

@article {Henniart2000-simpleproof,
    AUTHOR = {Henniart, Guy},
     TITLE = {Une preuve simple des conjectures de {L}anglands pour {${\rm
              GL}(n)$} sur un corps {$p$}-adique},
   JOURNAL = {Invent. Math.},
  FJOURNAL = {Inventiones Mathematicae},
    VOLUME = {139},
      YEAR = {2000},
    NUMBER = {2},
     PAGES = {439--455},
      ISSN = {0020-9910,1432-1297},
   MRCLASS = {11F70 (11R39 11S37 22E50 22E55)},
  MRNUMBER = {1738446},
MRREVIEWER = {Dihua\ Jiang},
       DOI = {10.1007/s002220050012},
       URL = {https://doi.org/10.1007/s002220050012},
}

@incollection {Borel1979-AutomorphicLfunctions,
    AUTHOR = {Borel, A.},
     TITLE = {Automorphic {$L$}-functions},
 BOOKTITLE = {Automorphic forms, representations and {$L$}-functions
              ({P}roc. {S}ympos. {P}ure {M}ath., {O}regon {S}tate {U}niv.,
              {C}orvallis, {O}re., 1977), {P}art 2},
    SERIES = {Proc. Sympos. Pure Math.},
    VOLUME = {XXXIII},
     PAGES = {27--61},
 PUBLISHER = {Amer. Math. Soc., Providence, RI},
      YEAR = {1979},
      ISBN = {0-8218-1437-0},
   MRCLASS = {10D40 (12A67 22E50)},
  MRNUMBER = {546608},
MRREVIEWER = {Yasuhiro\ Asoo},
}

@incollection {LapidRallis2005-localfactors,
    AUTHOR = {Lapid, Erez M. and Rallis, Stephen},
     TITLE = {On the local factors of representations of classical groups},
 BOOKTITLE = {Automorphic representations, {$L$}-functions and applications:
              progress and prospects},
    SERIES = {Ohio State Univ. Math. Res. Inst. Publ.},
    VOLUME = {11},
     PAGES = {309--359},
 PUBLISHER = {de Gruyter, Berlin},
      YEAR = {2005},
      ISBN = {978-3-11-017939-2; 3-11-017939-3},
   MRCLASS = {11F70 (22E45)},
  MRNUMBER = {2192828},
MRREVIEWER = {Mahdi\ Asgari},
       DOI = {10.1515/9783110892703.309},
       URL = {https://doi.org/10.1515/9783110892703.309},
}

@article{Deligne1976-localconstant,
author = {Deligne, Pierre},
journal = {Inventiones mathematicae},
pages = {299-316},
title = {Les constantes locales de l'équation fonctionelle de la fonction L d'Artin d'une représentation orthogonale.},
url = {http://eudml.org/doc/142410},
volume = {35},
year = {1976},
}

@article {GrossSavin1998-Motives,
    AUTHOR = {Gross, Benedict H. and Savin, Gordan},
     TITLE = {Motives with {G}alois group of type {$G_2$}: an exceptional
              theta-correspondence},
   JOURNAL = {Compositio Math.},
  FJOURNAL = {Compositio Mathematica},
    VOLUME = {114},
      YEAR = {1998},
    NUMBER = {2},
     PAGES = {153--217},
      ISSN = {0010-437X,1570-5846},
   MRCLASS = {11F70 (11F80 22E35)},
  MRNUMBER = {1661756},
MRREVIEWER = {Dipendra\ Prasad},
       DOI = {10.1023/A:1000456731715},
       URL = {https://doi.org/10.1023/A:1000456731715},
}

@article {KuokFaiLi2014-R-group,
    AUTHOR = {Chao, Kuok Fai and Li, Wen-Wei},
     TITLE = {Dual {$R$}-groups of the inner forms of {${\rm SL}(N)$}},
   JOURNAL = {Pacific J. Math.},
  FJOURNAL = {Pacific Journal of Mathematics},
    VOLUME = {267},
      YEAR = {2014},
    NUMBER = {1},
     PAGES = {35--90},
      ISSN = {0030-8730,1945-5844},
   MRCLASS = {22E50 (11F70)},
  MRNUMBER = {3163476},
MRREVIEWER = {Dmitry\ Gourevitch},
       DOI = {10.2140/pjm.2014.267.35},
       URL = {https://doi.org/10.2140/pjm.2014.267.35},
}

@article {Carter1972-Conjugacy-Weyl-group,
    AUTHOR = {Carter, R. W.},
     TITLE = {Conjugacy classes in the {W}eyl group},
   JOURNAL = {Compositio Math.},
  FJOURNAL = {Compositio Mathematica},
    VOLUME = {25},
      YEAR = {1972},
     PAGES = {1--59},
      ISSN = {0010-437X,1570-5846},
   MRCLASS = {20H15 (17B20)},
  MRNUMBER = {318337},
MRREVIEWER = {Juri\ A.\ Bahturin},
}

@article{LabesseSchwermer2019,
title = {Central morphisms and cuspidal automorphic representations},
journal = {Journal of Number Theory},
volume = {205},
pages = {170-193},
year = {2019},
issn = {0022-314X},
doi = {https://doi.org/10.1016/j.jnt.2019.05.005},
url = {https://www.sciencedirect.com/science/article/pii/S0022314X19301842},
author = {Jean-Pierre Labesse and Joachim Schwermer}
}

@misc{kottwitz2012splittinginvariantssignconventions,
      title={On Splitting Invariants and Sign Conventions in Endoscopic Transfer}, 
      author={R. Kottwitz and D. Shelstad},
      year={2012},
      eprint={1201.5658},
      archivePrefix={arXiv},
      primaryClass={math.RT},
      url={https://arxiv.org/abs/1201.5658}, 
}

@article {GanTakeda2016-Howeduality,
    AUTHOR = {Gan, Wee Teck and Takeda, Shuichiro},
     TITLE = {A proof of the {H}owe duality conjecture},
   JOURNAL = {J. Amer. Math. Soc.},
  FJOURNAL = {Journal of the American Mathematical Society},
    VOLUME = {29},
      YEAR = {2016},
    NUMBER = {2},
     PAGES = {473--493},
      ISSN = {0894-0347,1088-6834},
   MRCLASS = {22E50 (11F27)},
  MRNUMBER = {3454380},
MRREVIEWER = {Ivan\ Mati\'c},
       DOI = {10.1090/jams/839},
       URL = {https://doi.org/10.1090/jams/839},
}

@article {AtobeGan2017-localLanglandsorthogonal,
    AUTHOR = {Atobe, Hiraku and Gan, Wee Teck},
     TITLE = {On the local {L}anglands correspondence and {A}rthur
              conjecture for even orthogonal groups},
   JOURNAL = {Represent. Theory},
  FJOURNAL = {Representation Theory. An Electronic Journal of the American
              Mathematical Society},
    VOLUME = {21},
      YEAR = {2017},
     PAGES = {354--415},
      ISSN = {1088-4165},
   MRCLASS = {11F70 (11R39)},
  MRNUMBER = {3708200},
MRREVIEWER = {Kimball\ L.\ Martin},
       DOI = {10.1090/ert/504},
       URL = {https://doi.org/10.1090/ert/504},
}

@article {Vigneras1981orbitalintegral,
    AUTHOR = {Vign\'eras, Marie-France},
     TITLE = {Caract\'erisation des int\'egrales orbitales sur un groupe
              r\'eductif {$p$}-adique},
   JOURNAL = {J. Fac. Sci. Univ. Tokyo Sect. IA Math.},
  FJOURNAL = {Journal of the Faculty of Science. University of Tokyo.
              Section IA. Mathematics},
    VOLUME = {28},
      YEAR = {1981},
    NUMBER = {3},
     PAGES = {945--961 (1982)},
      ISSN = {0040-8980},
   MRCLASS = {22E35},
  MRNUMBER = {656066},
MRREVIEWER = {L.\ Corwin},
}

@article {Henniart2009exterior,
    AUTHOR = {Henniart, Guy},
     TITLE = {Sur la fonctorialit\'e, pour {$\rm GL(4)$}, donn\'ee par le
              carr\'e{} ext\'erieur},
   JOURNAL = {Mosc. Math. J.},
  FJOURNAL = {Moscow Mathematical Journal},
    VOLUME = {9},
      YEAR = {2009},
    NUMBER = {1},
     PAGES = {33--45, back matter},
      ISSN = {1609-3321,1609-4514},
   MRCLASS = {22E55 (11F70 11R39)},
  MRNUMBER = {2567395},
       DOI = {10.17323/1609-4514-2009-9-1-33-45},
       URL = {https://doi.org/10.17323/1609-4514-2009-9-1-33-45},
}

@article {Griess1995-G2,
    AUTHOR = {Griess, Jr., Robert L.},
     TITLE = {Basic conjugacy theorems for {$G_2$}},
   JOURNAL = {Invent. Math.},
  FJOURNAL = {Inventiones Mathematicae},
    VOLUME = {121},
      YEAR = {1995},
    NUMBER = {2},
     PAGES = {257--277},
      ISSN = {0020-9910,1432-1297},
   MRCLASS = {20G20 (17A35 20D06 22E10)},
  MRNUMBER = {1346206},
MRREVIEWER = {James\ E.\ Humphreys},
       DOI = {10.1007/BF01884298},
       URL = {https://doi.org/10.1007/BF01884298},
}

@article {Larsen1994-conjugacy,
    AUTHOR = {Larsen, Michael},
     TITLE = {On the conjugacy of element-conjugate homomorphisms},
   JOURNAL = {Israel J. Math.},
  FJOURNAL = {Israel Journal of Mathematics},
    VOLUME = {88},
      YEAR = {1994},
    NUMBER = {1-3},
     PAGES = {253--277},
      ISSN = {0021-2172,1565-8511},
   MRCLASS = {20G20 (20C15 22E15)},
  MRNUMBER = {1303498},
MRREVIEWER = {James\ E.\ Humphreys},
       DOI = {10.1007/BF02937514},
       URL = {https://doi.org/10.1007/BF02937514},
}

@article {Scholze2013-localLanglandsGLn,
    AUTHOR = {Scholze, Peter},
     TITLE = {The local {L}anglands correspondence for {$\mathrm{GL}_n$} over
              {$p$}-adic fields},
   JOURNAL = {Invent. Math.},
  FJOURNAL = {Inventiones Mathematicae},
    VOLUME = {192},
      YEAR = {2013},
    NUMBER = {3},
     PAGES = {663--715},
      ISSN = {0020-9910,1432-1297},
   MRCLASS = {22E50 (11G18 14G35)},
  MRNUMBER = {3049932},
MRREVIEWER = {Maarten\ Sander\ Solleveld},
       DOI = {10.1007/s00222-012-0420-5},
       URL = {https://doi.org/10.1007/s00222-012-0420-5},
}

@article {GrossGan1999Haarmeasure,
    AUTHOR = {Gross, Benedict H. and Gan, Wee Teck},
     TITLE = {Haar measure and the {A}rtin conductor},
   JOURNAL = {Trans. Amer. Math. Soc.},
  FJOURNAL = {Transactions of the American Mathematical Society},
    VOLUME = {351},
      YEAR = {1999},
    NUMBER = {4},
     PAGES = {1691--1704},
      ISSN = {0002-9947,1088-6850},
   MRCLASS = {20G30 (11G99)},
  MRNUMBER = {1458303},
MRREVIEWER = {Stefan\ K\"uhnlein},
       DOI = {10.1090/S0002-9947-99-02095-4},
       URL = {https://doi.org/10.1090/S0002-9947-99-02095-4},
}

@misc{kaletha2014endoscopicclassificationrepresentationsinner,
      title={Endoscopic Classification of Representations: Inner Forms of Unitary Groups}, 
      author={Tasho Kaletha and Alberto Minguez and Sug Woo Shin and Paul-James White},
      year={2014},
      eprint={1409.3731},
      archivePrefix={arXiv},
      primaryClass={math.NT},
      url={https://arxiv.org/abs/1409.3731}, 
}

@article {HiragaIchinoIkeda2008-Correction,
    AUTHOR = {Hiraga, Kaoru and Ichino, Atsushi and Ikeda, Tamotsu},
     TITLE = {Correction to: ``{F}ormal degrees and adjoint
              {$\gamma$}-factors'' [{J}. {A}mer. {M}ath. {S}oc. {\bf 21}
              (2008), no. 1, 283--304; MR2350057]},
   JOURNAL = {J. Amer. Math. Soc.},
  FJOURNAL = {Journal of the American Mathematical Society},
    VOLUME = {21},
      YEAR = {2008},
    NUMBER = {4},
     PAGES = {1211--1213},
      ISSN = {0894-0347,1088-6834},
   MRCLASS = {22E50},
  MRNUMBER = {2425185},
       DOI = {10.1090/S0894-0347-08-00605-X},
       URL = {https://doi.org/10.1090/S0894-0347-08-00605-X},
}

@article {Harish-Chandra1976-HarmonicanalysisIII,
    AUTHOR = {Harish-Chandra},
     TITLE = {Harmonic analysis on real reductive groups. {III}. {T}he
              {M}aass-{S}elberg relations and the {P}lancherel formula},
   JOURNAL = {Ann. of Math. (2)},
  FJOURNAL = {Annals of Mathematics. Second Series},
    VOLUME = {104},
      YEAR = {1976},
    NUMBER = {1},
     PAGES = {117--201},
      ISSN = {0003-486X},
   MRCLASS = {22E45},
  MRNUMBER = {439994},
MRREVIEWER = {P.\ C.\ Trombi},
       DOI = {10.2307/1971058},
       URL = {https://doi.org/10.2307/1971058},
}

@article {Knapp-Stein1971intertwiningoperators,
    AUTHOR = {Knapp, A. W. and Stein, E. M.},
     TITLE = {Intertwining operators for semisimple groups},
   JOURNAL = {Ann. of Math. (2)},
  FJOURNAL = {Annals of Mathematics. Second Series},
    VOLUME = {93},
      YEAR = {1971},
     PAGES = {489--578},
      ISSN = {0003-486X},
   MRCLASS = {22E45},
  MRNUMBER = {460543},
MRREVIEWER = {G.\ I.\ Ol\cprime shanski\u i},
       DOI = {10.2307/1970887},
       URL = {https://doi.org/10.2307/1970887},
}

@book {Rogawski1990-U3,
    AUTHOR = {Rogawski, Jonathan D.},
     TITLE = {Automorphic representations of unitary groups in three
              variables},
    SERIES = {Annals of Mathematics Studies},
    VOLUME = {123},
 PUBLISHER = {Princeton University Press, Princeton, NJ},
      YEAR = {1990},
     PAGES = {xii+259},
      ISBN = {0-691-08586-2; 0-691-08587-0},
   MRCLASS = {22E55 (11F70 11R39 22-02)},
  MRNUMBER = {1081540},
MRREVIEWER = {David\ Joyner},
       DOI = {10.1515/9781400882441},
       URL = {https://doi.org/10.1515/9781400882441},
}

@article {Heiermann2006-unipotentpoles,
    AUTHOR = {Heiermann, Volker},
     TITLE = {Orbites unipotentes et p\^oles d'ordre maximal de la fonction
              {$\mu$} de {H}arish-{C}handra},
   JOURNAL = {Canad. J. Math.},
  FJOURNAL = {Canadian Journal of Mathematics. Journal Canadien de
              Math\'ematiques},
    VOLUME = {58},
      YEAR = {2006},
    NUMBER = {6},
     PAGES = {1203--1228},
      ISSN = {0008-414X,1496-4279},
   MRCLASS = {11F70 (11F80 22E50)},
  MRNUMBER = {2270924},
MRREVIEWER = {David\ Goldberg},
       DOI = {10.4153/CJM-2006-043-8},
       URL = {https://doi.org/10.4153/CJM-2006-043-8},
}

@misc{beuzartplessis2025hiragaichinoikedaconjectureformaldegrees,
      title={The {H}iraga-{I}chino-{I}keda conjecture on formal degrees for classical groups}, 
      author={Raphaël Beuzart-Plessis},
      year={2025},
      eprint={2508.08470},
      archivePrefix={arXiv},
      primaryClass={math.RT},
      url={https://arxiv.org/abs/2508.08470}, 
}

@article {FengOpdamSolleveld2022-unipotent,
    AUTHOR = {Feng, Yongqi and Opdam, Eric and Solleveld, Maarten},
     TITLE = {On formal degrees of unipotent representations},
   JOURNAL = {J. Inst. Math. Jussieu},
  FJOURNAL = {Journal of the Institute of Mathematics of Jussieu. JIMJ.
              Journal de l'Institut de Math\'ematiques de Jussieu},
    VOLUME = {21},
      YEAR = {2022},
    NUMBER = {6},
     PAGES = {1947--1999},
      ISSN = {1474-7480,1475-3030},
   MRCLASS = {22E50 (11S37 20G25)},
  MRNUMBER = {4515286},
MRREVIEWER = {Petar\ Baki\'c},
       DOI = {10.1017/S1474748021000062},
       URL = {https://doi.org/10.1017/S1474748021000062},
}

@article {Mie2021-simplesupercuspidal,
    AUTHOR = {Mieda, Yoichi},
     TITLE = {On the formal degree conjecture for simple supercuspidal
              representations},
   JOURNAL = {Math. Res. Lett.},
  FJOURNAL = {Mathematical Research Letters},
    VOLUME = {28},
      YEAR = {2021},
    NUMBER = {4},
     PAGES = {1227--1242},
      ISSN = {1073-2780,1945-001X},
   MRCLASS = {11F70 (22E50)},
  MRNUMBER = {4344703},
MRREVIEWER = {Ivan\ Mati\'c},
       DOI = {10.4310/MRL.2021.v28.n4.a11},
       URL = {https://doi.org/10.4310/MRL.2021.v28.n4.a11},
}

@article {Ohara2023-non-singular,
    AUTHOR = {Ohara, Kazuma},
     TITLE = {On the formal degree conjecture for non-singular supercuspidal
              representations},
   JOURNAL = {Int. Math. Res. Not. IMRN},
  FJOURNAL = {International Mathematics Research Notices. IMRN},
      YEAR = {2023},
    NUMBER = {13},
     PAGES = {10997--11034},
      ISSN = {1073-7928,1687-0247},
   MRCLASS = {22E50 (11F70)},
  MRNUMBER = {4609777},
MRREVIEWER = {Rongqing\ Ye},
       DOI = {10.1093/imrn/rnac154},
       URL = {https://doi.org/10.1093/imrn/rnac154},
}

@article {Schwein2024-regular,
    AUTHOR = {Schwein, David},
     TITLE = {Formal degree of regular supercuspidals},
   JOURNAL = {J. Eur. Math. Soc. (JEMS)},
  FJOURNAL = {Journal of the European Mathematical Society (JEMS)},
    VOLUME = {26},
      YEAR = {2024},
    NUMBER = {10},
     PAGES = {3685--3737},
      ISSN = {1435-9855,1435-9863},
   MRCLASS = {22E50 (11S37)},
  MRNUMBER = {4768407},
MRREVIEWER = {Maarten\ Sander\ Solleveld},
       DOI = {10.4171/jems/1412},
       URL = {https://doi.org/10.4171/jems/1412},
}

@article {Wang2025-formal-degree,
    AUTHOR = {Wang, Yiyang},
     TITLE = {Formal degrees and parabolic induction: the maximal generic
              case},
   JOURNAL = {Manuscripta Math.},
  FJOURNAL = {Manuscripta Mathematica},
    VOLUME = {176},
      YEAR = {2025},
    NUMBER = {48},
      ISSN = {0025-2611,1432-1785},
   MRCLASS = {22E50},
  MRNUMBER = {4932679},
       DOI = {10.1007/s00229-025-01632-z},
       URL = {https://doi.org/10.1007/s00229-025-01632-z},
}

@article {Shahidi1984-Fouriertransform,
    AUTHOR = {Shahidi, Freydoon},
     TITLE = {Fourier transforms of intertwining operators and {P}lancherel
              measures for {${\rm GL}(n)$}},
   JOURNAL = {Amer. J. Math.},
  FJOURNAL = {American Journal of Mathematics},
    VOLUME = {106},
      YEAR = {1984},
    NUMBER = {1},
     PAGES = {67--111},
      ISSN = {0002-9327,1080-6377},
   MRCLASS = {22E50 (11S37)},
  MRNUMBER = {729755},
MRREVIEWER = {Allan\ J.\ Silberger},
       DOI = {10.2307/2374430},
       URL = {https://doi.org/10.2307/2374430},
}

@article{SilbergerZink1996-formaldegree,
  author = {Silberger, A. J. and E.-W. Zink},
  journal = {Max-Planck-Institut für Mathematik},
  number = {},
  title = {The formal degree of discrete series representations of
central simple algebras over p-adic fields},
  volume = {},
  year = {1996}
}

@incollection {Springer1984-linearalgebraicgroups,
    AUTHOR = {Springer, T. A.},
     TITLE = {Linear algebraic groups},
 BOOKTITLE = {Perspectives in mathematics},
     PAGES = {455--495},
 PUBLISHER = {Birkh\"auser, Basel},
      YEAR = {1984},
      ISBN = {3-7643-1624-1},
   MRCLASS = {20G15},
  MRNUMBER = {779686},
MRREVIEWER = {Andy\ R.\ Magid},
}

@article {Takanashi2025-Sauvageot,
    AUTHOR = {Takanashi, Yugo},
     TITLE = {A note on the {S}auvageot density principle},
   JOURNAL = {Manuscripta Math.},
  FJOURNAL = {Manuscripta Mathematica},
    VOLUME = {176},
      YEAR = {2025},
    NUMBER = {57},
      ISSN = {0025-2611,1432-1785},
   MRCLASS = {22E50 (11F70 22E55)},
  MRNUMBER = {4941997},
       DOI = {10.1007/s00229-025-01650-x},
       URL = {https://doi.org/10.1007/s00229-025-01650-x},
}

@misc{gan2025trialityadjointliftinggl3,
      title={Triality and adjoint lifting for ${{\mathrm{GL}}(3)}$}, 
      author={Wee Teck Gan},
      year={2025},
      eprint={2512.08307},
      archivePrefix={arXiv},
      primaryClass={math.NT},
      url={https://arxiv.org/abs/2512.08307}, 
}

@misc{Takanashi2026-EndoscopicDescriptionG2,
    AUTHOR = {Takanashi, Yugo},
    TITLE = {Endoscopic description of the local {Langlands} correspondence for {$\mathrm{G}_2$}},
    YEAR = {2026},
    NOTE = {Preprint, \href{https://arxiv.org/abs/2609.19439}{arXiv:2609.19439}},
}

\end{document}